\documentclass[a4paper,12pt,american]{article}

\usepackage{lmodern}
\usepackage[charter]{mathdesign}
\usepackage[scaled=.8]{helvet}

\usepackage{SymLexSubsetRS_packages}

\usepackage[unicode,pdfusetitle]{hyperref}
\usepackage{hyperxmp}

\usepackage[
type={CC},
modifier={by},
version={4.0},
lang={american},
imagewidth=8em,
]{doclicense}

\hypersetup{
  keeppdfinfo,
  pdfauthor={Jörg Rambau (University of Bayreuth)},
  pdfsubject={A new method to enumerate symmetric subsets of a finite set up to symmetry based on the reverse-search paradigm},
  pdfkeywords={enumeration, reverse search, symmetry, circuits, cocircuits, triangulations},
}

\usepackage{authblk}

\newtheorem{theorem}{Theorem}
\newtheorem{lemma}{Lemma}
\newtheorem{corollary}{Corollary}
\newtheorem{observation}{Observation}
\theoremstyle{definition}
\newtheorem{definition}{Definition}
\newtheorem{remark}{Remark}
\newtheorem{example}{Example}

\title{Symmetric lexicographic symmetric-subset\\
  reverse search\\
  for the\\
  enumeration of circuits, cocircuits, and triangulations\\
  up to symmetry\\[2ex]
  \textbf{Preprint~2}
}

\date{\today\\\footnotesize\doclicenseThis}

\author{Jörg Rambau}
\affil{University of Bayreuth, Germany}

\begin{document}

\maketitle

\begin{abstract}
  This paper introduces, analyzes, and applies variants of the
  enumeration framework symmetric lexicographic symmetric-subset
  reverse search for the enumeration of symmetric feasible subsets of
  a finite set up to symmetry.  The framework is implemented in detail
  for three applications: cocircuits, circuits, and triangulations of
  point configurations.  There are two new methods presented and
  analyzed to check the lexicographic minimality of a subset in its
  orbit: the critical-element method and the modified switch-table
  method.  Moreover, new application-dependent methods to reduce the
  number of necessary enumeration nodes are introduced: rank-pruning
  for cocircuits and lex-pruning for triangulations.  With a
  C++-implementation of the ideas in the software package
  \texttt{TOPCOM}, in all three applications known benchmarks can be
  computed faster by a large margin.  The following new numbers could
  be computed for the first time (among others): the number of
  cocircuits of the 9-cube, the number of circuits of the 8-cube, and
  the number of all triangulations of the product of a 5- and a
  3-simplex, as well as the number of all triangulations of a point
  configuration in dimension six with 17~points with disconnected
  flip-graph (constructed by Santos).  Moreover, for Santos's
  triangulation it has computationally been checked that its
  flip-graph component is indeed purely non-regular. Furthermore, in
  another instance in dimension five with 26 points (also constructed
  by Santos), a flaw has been detected: Santos's triangulation can be
  heuristically flipped to a regular triangulation in the original
  point configuration.  In a mildly modified version of the point
  configuration, the heuristics cannot flip Santos's triangulation to
  a regular triangulation anymore.
\end{abstract}


\section*{Acknowledgement}

The author is very grateful to Michael Joswig and Lars Kastner who
provided the exact chirotope of the regular dodecahedron, computed by
\texttt{polymake}. Chirotope data makes possible the exact computation
of all triangulations of a point configurations even if there are no
rational coordinates for it.  Moreover, the author thanks Lukas Kühne
who brought the importance of central and centrally symmetric
triangulations of the centered full root polytope to the author's
attention.  The author also thanks Lisa Lamberti and Komei Fukuda for
communicating the (at that time open) question about the number of
circuits of the $6$-cube, which led to the simultaneous consideration
of triangulations and circuits, which was beneficial for both
applications.  A special thanks goes to Francisco Santos who
contributed to the resolution of the seemingly contradictory results
on his dimension-five 26-points configuration.

Some calculations were performed using the \texttt{festus-cluster} of
the Bayreuth Centre for High Performance Computing
(\url{https://www.bzhpc.uni-bayreuth.de}), funded by the Deutsche
Forschungsgemeinschaft (DFG, German Research Foundation) --
523317330. The author thanks for the ressources and, in particular,
René Meißner for his excellent user support.

\tableofcontents
\listoftables
\listofalgorithms

\section{Introduction}
\label{sec:introduction}

This paper systematically studies how to enumerate symmetric subsets
of a finite set up to symmetry based on a specialization of the
reverse-search paradigm to orbits of subsets.  In order to demonstrate
the benefits of this approach, three applications are presented: for a
point or vector configuration, enumerate up to symmetry all its
cocircuits, all its circuits, and all its triangulations,
respectively.  In all three applications the scales of problem
instances that can be handled are extended significantly.

Fast enumeration in general is important also in optimization: in
combination with a dual-bound scheme, it gives rise to a
branch-and-bound optimization algorithm.  In dynamic
column-generation, the pricing problem is sometimes attacked by
enumerative algorithms (see, e.g.,
\cite{KiesslingKurzRambau_exactcolumngeneration_2021}).  Whenever
symmetries are present, an enumeration up to symmetry is the
appropriate procedure.  This is particularly important for large
symmetry groups in branch-and-bound (like the $n!$ symmetries emerging
by re-indexing $n$ equivalent decision variables), since the $n!$
branches containing the equivalent optimal solutions can never be
pruned.

The reverse-search paradigm is appealing because it generates an
enumeration \emph{tree}.  An enumeration based on organizing the
objects to be enumerated in a tree has two main advantages: first, it
can be implemented in a memory efficient way by depth-first-search in
the enumeration tree; second, it can be parallelized easily in a
lock-free fashion.  For example, enumeration based on the
reverse-search paradigm can become faster after parallelization than
alternative algorithms that are faster without parallelization
\cite{AvisJordan_mtslightframework_2021}.

A general account for the complexity of a variety of enumeration
algorithms can be found in
\cite{MaryStrozecki_Efficientenumerationsolutions_2019}, where the
objects to be enumerated are considered as the results of certain
abstract closure operations.  Symmetries are not considered.

The top-level enumeration method in this paper with no prescribed
symmetries can be seen as a specialization of \emph{reverse search}
\cite{AvisFukuda_Reversesearchenumeration_1996} to subsets up to
symmetry.  Reverse search has been specialized to the enumeration of
orbits in
\cite{ImaiMasadaTakeuchiImai_Enumeratingtriangulationsgeneral_2002}
and parallelized in \cite{AvisJordan_mtslightframework_2021}.  The
specialization of reverse-search to feasible subsets of a finite set
is probably folklore.  A specialization exploiting the lexicographic
ordering of subsets to enumerate subset \emph{orbits} was first
formalized (in a different language) in
\cite{PechReichard_EnumeratingSetOrbits_2009} with applications for
very large symmetry groups, where not all group elements can be held
in memory.  Though examples show the potential of the resulting
algorithm, no theoretical analysis of the speed-up compared to naive
approaches was provided.  In this paper, variants of the same basic
idea are designed, analyzed, and exploited in three example
applications.  One extension is the enumeration of \emph{symmetric}
subsets up to symmetry, which, to the best of my knowledge, has not
been done before.

The first application in this paper is concerned with the enumeration
of \emph{signed cocircuits} of a point or vector configuration.
Cocircuits in, e.g., hypercubes are in connection to the various ways
binary data can be weakly linearly separated and therefore related to
topics in data science.  In
\cite{AichholzerAurenhammer_ClassifyingHyperplanesHypercubes_1996} all
hyperplanes spanned by vertices of the $d$-dimensional hypercube
$\hypercube{d}$ were classified and enumerated up to $d = 8$ (this
took 12 days of cpu time on a computer that was fast according to the
standards of that time).  The authors estimated a cpu time of 35 years
for $d = 9$ for their method based on the classification of normal
vectors.  There seems to be no documented algorithm or code
enumerating the hyperplanes in a general point or vector
configuration.  In this paper, all hyperplanes for, e.g., the $9$-cube
are enumerated up to symmetry for the first time.

The second application in this paper is concerned with the enumeration
of \emph{signed circuits} of a point or vector configuration. This is
the dual problem to the first one (in the oriented-matroid sense, see
\cite{BjoernerLasVergnasSturmfelsWhiteZiegler_OrientedMatroids_1999}).
Circuits are in close connection with minimal infeasible subsystems
and sparse kernel elements, which are related to compressed sensing.
There is an incremental-polynomial-time algorithm for the enumeration
of (unsigned) circuits of a matroid in
\cite{KhachiyanBorosElbassioniGurvichMakino_ComplexitySomeEnumeration_2005}
based on the exchange axiom for circuits, which improves on older
algorithms like the one in \cite{Minieka_Findingcircuitsmatroid_1976}
based on using bases for constructing circuits.  However, it is not
clear how symmetries could be exploited in this algorithm.  Moreover,
the nature of the exchange axiom shows that the worst-case time to add
one more circuit is at least quadratic in the number of already
computed circuits, which seems prohibitive at least for the larger
examples in this paper.  In this paper, all circuits for, e.g., the
$8$-cube are enumerated up to symmetry for the first time.  Moreover,
the listing of all circuits of the \emph{matroid base polytope for the
  Fano matroid} obtained by the methods in this paper were a crucial
building block in
\cite{BackmanCheungLasonLiuMichalek_GrobnerVersionWhites_2026} in
answering in the negative the (over 20 years-old) \emph{Herzog-Hibi
  question} \cite[p.~242
Question~a)]{HerzogHibi_DiscretePolymatroids_2002} about the existence
of a regular unimodular flag triangulation.  Curiously enough, this
result was achieved independently a few days earlier in
\cite{LoeraFerroniMoralesRambau_Therearematroid_2026} by related
methods, albeit without using circuits.

The third application in this paper is concerned with the enumeration
of \emph{triangulations} of a point or vector configuration.  See
\cite{DeLoeraRambauSantos_TriangulationsStructuresApplications_2010}
for background on triangulations and why their enumeration is
interesting.  Recently, a parallel enumeration up to symmetry of all
\emph{subregular} triangulations was suggested in
\cite{JordanJoswigKastner_Parallelenumerationtriangulations_2018}.  A
triangulation is \emph{subregular}, if it can be flipped to a
regular-triangulation by \emph{upflips}, i.e., flips that
lexicographically increase the GKZ-vector.  Hence, whether or not a
triangulation is subregular usually depends on the order of points in
the input. The presented computational results were achieved based on
the freely available package \texttt{mptopcom}.  It enumerates all
subregular triangulations up to symmetry. It significantly extends
the scales of instances that can be handled compared to the earlier
\texttt{TOPCOM} \cite{Rambau_TOPCOMTriangulationsPoint_2002}, which
itself is a large step forward from de Loera's pioneering
\texttt{maple}-code \texttt{PUNTOS} from his thesis
\cite{Loera_TriangulationsPolytopesComputational_1995}.  The code
\texttt{mptopcom} has later been specialized for cyclic polytopes to
generate some new numbers \cite{JoswigKastner_NewCountsNumber_2018}
extending the computational results in
\cite{Rambau_TOPCOMTriangulationsPoint_2002,RambauReiner_surveyhigherStasheff_2012}.
This was possible although for \texttt{mptopcom} symmetries must
conserve simplex volumes, so that the standard cyclic polytopes do not
have valid symmetries in that stronger sense.

All mentioned enumeration algorithms for triangulations are
\emph{flip-based}, i.e., they explore the flip graph of
triangulations, where two triangulations share an edge if they are
connected by a \emph{bistellar flip}.  This is a generalization of
swapping diagonals in a convex quadrilateral in dimension two (see
\cite{DeLoeraRambauSantos_TriangulationsStructuresApplications_2010}
for a definition in all dimensions).  Since Santos's triangulation
without flips in \cite{Santos_pointsetwhose_2000} it has been known
that flip-based algorithms may not find all triangulations in general.

\texttt{TOPCOM} \cite{Rambau_TOPCOMTriangulationsPoint_2002} was the
first published software to enumerate all triangulations by building
them simplex-by-simplex, which will be called an
\emph{extension-based} algorithm.  However, the instances that could
be handled were only toy-size examples.  In
\cite{ImaiMasadaTakeuchiImai_Enumeratingtriangulationsgeneral_2002}
the extension-based enumeration of triangulations was reduced to the
enumeration of maximal cliques in the \emph{proper-intersection graph}
of all simplices.  However, no computational results were given. And
the results in this paper show evidence for the fact that a pure
max-clique enumeration cannot be efficient: there are too many maximal
cliques that correspond to a \emph{maximal incomplete triangulation}.
This is a set of simplices with proper intersections that is not a
triangulation and that cannot be extended (for exact definitions see
Section~\ref{sec:application-triangulations}).

The enumeration of all triangulations corresponds to an enumeration of
all vertices of the \emph{universal polytope}
\cite{LoeraHostenSantosSturmfels_polytopealltriangulations_1996};
since no complete outer description of this polytope is available, the
vertex enumeration problem for it is not straight-forward.  In this
paper, all triangulations of, e.g., the regular dodecahedron or the
product of a 5- and a 3-simplex are enumerated up to symmetry for the
first time.  Moreover, for the first time, all triangulations of a
point configuration constructed by Santos in dimension six with 17
points and disconnected flip-graph are computed up to symmetry.

Besides the fact that extension-based algorithms reach all
triangulations, there is one other motivation for them: If the search
shall be restricted to triangulations using only \emph{special
  simplices} (like unimodular simplices or simplices not containing
any points of the configuration other than their vertices), then the
flip-based algorithms have to explore the whole flip-graph anyway and
filter ex-post by the wanted triangulations (since the subgraph of all
wanted triangulations may be disconnected even for easy examples),
whereas an extension-based algorithm can exclude the unwanted
simplices right from the beginning.  Moreover, the new method to
enumerate subsets with \emph{prescribed symmetry} is applied to
triangulations with results that have never been achieved before.  In
this paper, e.g., all central and centrally symmetric triangulations
of the full root polytope in ambient 6-space are enumerated for the
first time, while the enumeration of all its triangulations seems
currently out of reach by far.  For Santos's point configuration in
dimension six with 17 points it is shown computationally that Santos's
triangulation is essentially the unique triangulation (up to symmetry)
having the same symmetries.

The overall contributions of this paper are both incremental and
original.  The top-level method used in this paper consists of a
generic enumeration framework for the enumeration of all orbits of a
\emph{downset}, i.e., a set of subsets closed under taking subsets.
This framework is called \emph{Symmetric Lexicographic Subset Reverse
  Search} (\symLSRS) in this paper.  Three variants of \symLSRS are
studied: enumerate orbits of \emph{maximal} elements \emph{in} a
downset, enumerate orbits of \emph{minimal} elements \emph{not in} a
downset, and enumerate orbits of \emph{feasible} antichains \emph{in}
a downset.  The generic algorithm \symLSRS and the variant for
feasible subsets without prescribed symmetries in this paper -- though
developed independently -- are essentially identical to the proposed
algorithm in~\cite{PechReichard_EnumeratingSetOrbits_2009} for the
enumeration of set orbits.  The methods in this paper can be seen as
variants, refinements, and new specializations of it and its
subroutines.  Without these incremental and original achievements, the
new results would not have been possible. All applications have been
implemented in the \texttt{TOPCOM} package, which is freely available
under the Gnu Public Licence on the author's webpage.

The first original contribution of this paper is an extension of
\symLSRS to the enumeration of feasible subsets with prescribed
symmetry: \emph{symmetric lexicographic symmetric-subset reverse
  search.}  Further original constribution are: for the general
framework, two alternative checks of lexicographic minimality of a
subset in its orbit (called the \emph{lex-min check}) are proposed;
for the particular applications, new methods are presented for
recognizing that a subset cannot be extended to a feasible subset by
adding larger elements (called the \emph{lex-ext check}).

Concerning the \emph{lex-min check}, the first new alternative is the
\emph{critical-element method}.  It is based on new theory presented
in Section~\ref{sec:theory}.  This alternative is mainly interesting
for cases in which the symmetry group is of moderate order, i.e., is
given as a list of all permutations in it, and of a degree in the same
order of magnitude, i.e., there are at least as many elements as there
are permutations.  The enumeration of all triangulations usually fits
into this scheme.

The second alternative is the \emph{modified switch-table method}. It
combines the ideas in \cite{PechReichard_EnumeratingSetOrbits_2009}
with the switch-table method
in~\cite{JordanJoswigKastner_Parallelenumerationtriangulations_2018}.
This alternative works best for symmetry groups whose order is large
compared to the degree.  The enumeration of circuits and cocircuits is
an appropriate use-case.

In order to assess the efficiency of each method for the lex-min check
a-priori, a \emph{hyper-amortized} analysis on uniform inputs is
presented.  Roughly speaking: for small order and large degree, the
critical-element method is faster. For large order and small degree,
the modified switch-table method is faster.  This analysis sheds light
on why, in computational experiments, for triangulations of hypercubes
the critical-element method is faster, whereas for (co-)circuits of
hypercubes the modified switch-table method wins.  This constitutes
the first ever theoretical analysis of algorithms based on
switch-tables.

Concerning the \emph{lex-ext checks} for the applications, the new
rank-based \emph{rank-pruning} accelerates the enumeration of
cocircuits up to symmetry.  This allowed the first ever enumeration of
all hyperplanes in the $9$-cube $\hypercube{9}$ up to symmetry in less
than 14 hours.  It is unclear how fast the code from
\cite{AichholzerAurenhammer_ClassifyingHyperplanesHypercubes_1996}
would run on today's computers.  However, \texttt{TOPCOM}'s
enumeration algorithm goes beyond the method in
\cite{AichholzerAurenhammer_ClassifyingHyperplanesHypercubes_1996}
anyway, since it works for general configurations and does not exploit
any theory about cubes.

For the enumeration of circuits no effective lex-ext check was found
so far.  Still, the algorithm could compute some new numbers, among
them the numbers of circuits up to symmetry of the hypercubes
$\hypercube{6}$, $\hypercube{7}$, and $\hypercube{8}$.  Note that in
order to enumerate circuits one can also enumerate cocircuits in the
\emph{Gale-transform} (see
\cite{DeLoeraRambauSantos_TriangulationsStructuresApplications_2010}
for more background on this).  Whether or not this is faster or slower
usually depends on the rank and the corank of the configuration.
Having specialized algorithms for both means that one can pick the
respective faster strategy.

For the enumeration of triangulations up to symmetry, the new lex-ext
check \emph{lex-pruning} is the single most important progress. It is
based on the property of any triangulation that each interior facet of
a simplex is covered by another simplex
\cite[Cor.~4.1.32]{DeLoeraRambauSantos_TriangulationsStructuresApplications_2010}.
From this one can derive the rather tight lex-ext check
\emph{full-pruning}.  A weaker variant is \emph{strong pruning}, which
can be implemented more easily.  The new lex-ext check lex-pruning
heavily exploits the lexicographic ordering of all simplices and their
interior facets in all data structures involved.  While strong-pruning
has to perform many subset operations, lex-pruning only compares two
certain integers.  Still it is almost as effective as strong-pruning.
The numbers of all triangulations of, e.g., the dodecahedron and the
product of a 5- and a 3-simplex could be computed this way for the
first time in a couple of days.  These are numbers that are several
orders of magnitudes larger than the largest numbers of triangulations
computed so far.

For the enumeration of triangulations up to symmetry with a prescribed
automorphism group, the algorithm symmetric lexicographic
symmetric-subset reverse search for the first time allows to restrict
the search space for triangulations directly to the symmetric
triangulations, thereby avoiding the need to implicitly enumerate all
triangulations during the process.  This contribution opens the door
for the application of a variant of the \emph{Kramer-Mesner method}
\cite{KramerMesner_tdesignshypergraphs_1976} from design theory
(restrict the search to symmetric designs) to otherwise intractable
search problems in the set of all triangulations.  This way, all
central and centrally symmetric triangulations of the full root
polytope in ambient 6-space have been found in about an hour.
Moreover, for the first time the number of all triangulations of a
six-dimensional point configuration with $17$ points by Santos
\cite{Santos_Geometricbistellarflips_2006} with disconnected
flip-graph could be computed with high-performance computing.
Furthermore, it could be shown computationally that the triangulation
constructed by Santos is the unique (up to symmetry) triangulation
having the same symmetries.

Symmetric lexicographic symmetric-subset reverse search can be used to
count, enumerate, or to list objects up to symmetry (details below).
The proposed specializations to the three applications are in most
cases provably not output-polynomial.  This is shown by concrete
pathological examples.  Still, for the computational examples in these
three applications no implementations have been published so far that
are nearly as fast as the implementations in TOPCOM
(see~\cite{Rambau_TOPCOMTriangulationsPoint_2002} for the foundations
of the version prior to this work) of the methods in this paper.

The paper is organized as follows.  Section~\ref{sec:preliminaries}
reviews some important notions and algorithms.  Moreover, the
notational conventions are fixed and a short general problem statement
is given.  In Section~\ref{sec:theory} the theoretical foundations
for the new lex-min checks are proven.  Section~\ref{sec:algorithm} is
devoted to an algorithmic analysis of the variants of the top-level
algorithm and the lex-min checks valid for arbitrary applications.
The discussion of three special applications starts in
Section~\ref{sec:applications-preliminaries} with some common
preliminaries on point and vector configurations.
Section~\ref{sec:applications-environment} describes the computational
environment used for the numerical experiments.  Then,
Sections~\ref{sec:application-cocircuits} through
\ref{sec:application-triangulations} present the new results that are
relevant for each application individually.  Finally,
Section~\ref{sec:conclusions} contains a summary and some conclusions.

\section{Preliminaries}
\label{sec:preliminaries}

In this section, some basic notions and notation are introduced.
Moreover, some known algorithms are formulated in terms of the
\emph{reverse-search} framework in order to make it easier to compare
the known with the novel.

Let $\indexSet{n}$ denote the set of integers $\{1, 2, \dots, n\}$
with the convention $\indexSet{0} = \emptyset$.  For
$k \in \mathbb{Z}$ the set of all $k$-element subsets of
$\indexSet{n}$ is written as~$\kSubsets{n}{k}$.  The power set of
$\indexSet{n}$ is denoted by~$\powerSet{\indexSet{n}}$.

The elements of~$\powerSet{\indexSet{n}}$ and, thus, also
of~$\kSubsets{n}{k}$ are totally ordered by the
\emph{subset-lexicographic order} given by
$\emptyset \lexsmaller \setStyle{S}$ for all
$\setStyle{S} \neq \emptyset$ and
$\setStyle{S} \lexsmaller \setStyle{R}$ if and only if either
$\min \setStyle{S} < \min \setStyle{R}$ or
$\min \setStyle{S} = \min \setStyle{R}$ and
$\setStyle{S} \setminus \{\min \setStyle{S}\} \lexsmaller \setStyle{R}
\setminus \{\min \setStyle{R}\}$.

For subsets $\setStyle{S}$ and $\setStyle{S}'$ of $\indexSet{n}$ the
\emph{lex-inclusion partial order} is defined by
$\setStyle{S} \lexsubset \setStyle{S}'$ if
$\setStyle{S} \subseteq \setStyle{S}'$ and $s < s'$ for all
$s \in \setStyle{S}$ and $s' \in \setStyle{S}'\setminus\setStyle{S}$.
In that case $\setStyle{S}$ is a \emph{lex-subset} of~$\setStyle{S}'$
or $\setStyle{S}'$ \emph{lex-contains}~$\setStyle{S}$.  Note that the
Hasse-diagram of this partial order is a tree rooted at the empty set.
For $0 \le k \le \abs{\setStyle{S}}$ let
$\lexksubset{\setStyle{S}}{k}$ be the unique $k$-element subset
lex-contained in~$\setStyle{S}$.

For an undirected graph $\GG = (\VV, \EE)$ and a node $\vv \in \VV$,
the \emph{neighborhood~$\adjSet(\vv)$ of~$\vv$ in~$\GG$} is the set of
nodes connected to~$\vv$ by an edge in~$\EE$.  Its cardinality
$\abs{\adjSet(\vv)}$ is the \emph{degree} $\degree(\vv)$ of
$\vv \in \VV$, and the maximum of all degrees over all nodes is
denoted by~$\maxDegree(\GG)$ (or $\maxDegree$ if $\GG$ is clear from
the context).

The symmetric group on $n$ elements is considered as the set of
bijections from $\indexSet{n}$ to itself and is denoted
by~$\symGroup{n}$.  Its elements are called \emph{permutations}, and
$n$ is called the \emph{degree} of the permutations.  Subgroups
of~$\symGroup{n}$ are usually denoted by~$\GrG$.  The \emph{order}
of~$\GrG$ is the number of permutations in~$\GrG$.  For a subgroup
$\GrG$ of~$\symGroup{n}$ and a finite set~$\setStyle{\Omega}$, a map
$\phi: \GrG \times \setStyle{\Omega} \to \Omega$ is a \emph{(left)
  group action} of $\GrG$ on $\setStyle{\Omega}$ if
$\phi(\grg \cdot \grh, \omega) = \phi\bigl(\grg, \phi(\grh,
\omega)\bigr)$ for all $\grg, \grh \in \GrG$ and all
$\omega \in \setStyle{\Omega}$.  Most of the times, certain clearly
induced group actions $\phi$ are denoted by
$\grg(\omega) := \phi(\grg, \omega)$ for $\grg \in \GrG$ and
$\omega \in \setStyle{\Omega}$, leading to
$(\grg\cdot\grh)(\omega) = \grg\bigl(\grh(\omega)\bigr)$.  Given such
a group action, the \emph{stabilizer subgroup} of $\omega \in \Omega$
in $\GrG$ is
$\stabGroup{\GrG}{\omega} := \{ \grg \in \GrG : \grg(\omega) = \omega
\}$.  For $\Omega'\subseteq \Omega$, the \emph{point-wise stabilizer
  of~$\Omega'$} is
$\stabGroup{\GrG}{\Omega'} := \{ \grg \in \GrG : \grg(\omega) = \omega
\text{ for all $\omega \in \Omega'$} \}$. This must be distinguished
from the \emph{set-wise stabilizer of~$\Omega'$}, denoted by
$\stabGroup{\GrG}{\{\Omega'\}}$, which is defined as
$\stabGroup{\GrG}{\{\Omega'\}} := \{ \grg \in \GrG : \grg(\omega) \in
\Omega' \text{ for all $\omega \in \Omega'$} \}$.  The
\emph{$\GrG$-orbit} of $\omega$ is
$\orbitSet{\GrG}{\omega} := \{ \grg(\omega) : \grg \in \GrG \}$.  The
set of all $\GrG$-orbits is
$\setOrbits{S}{G} := \{ \orbitSet{\GrG}{\omega} : \omega \in \Omega
\}$.

For a permutation $\grg \in \symGroup{n}$ and a graph $\GG$ with node
set $\VV = \indexSet{n}$ and edge set $\EE$ let
$\grg\bigl(\{\vv, \ww\}\bigr) = \bigl\{\grg(\vv), \grg(\ww)\bigr\}$
denote the induced action of $\grg$ on edges, for all edges
$\{\vv, \ww\} \in \EE$.  The definition
$\grg(\EE) := \{ \grg(\ee) : \ee \in \EE \}$ induces a new graph
$\grg(\GG) = \bigl(\grg(\indexSet{n}) = \indexSet{n}, \grg(\EE)\bigr)$
on the same node set.  This defines an action of $\symGroup{n}$ on the
set of all graphs on the node set~$\indexSet{n}$.  The automorphism
group $\autGroup(\GG)$ of a graph $\GG = (\indexSet{n}, \EE)$ is the
set of all $\grg \in \symGroup{n}$ with $\grg(\GG) = \GG$.  The
elements of~$\autGroup(\GG)$ are the \emph{symmetries of~$\GG$}.

Similarly, for $\grg \in \symGroup{n}$ the induced action of $\grg$ on
any $k$-element subset
$\setStyle{S} = \{s_1, \dots, s_k\} \in \kSubsets{n}{k}$ is denoted by
$\grg(\setStyle{S}) = \bigl\{\grg(s_1), \dots, \grg(s_k)\bigr\} \in
\kSubsets{n}{k}$.  For a subset
$\setsysStyle{S} = \{\setStyle{S}_1, \dots, \setStyle{S}_m\}$
of~$\powerSet{\indexSet{n}}$ the induced action of $\grg$
on~$\setsysStyle{S}$ is denoted by
$\grg(\setsysStyle{S}) = \bigl\{\grg(\setStyle{S}_1), \dots,
\grg(\setStyle{S}_m)\bigr\}$.  The automorphism group
$\autGroup(\setsysStyle{S})$ is the set of all $\grg \in \symGroup{n}$
with $\grg(\setsysStyle{S}) = \setsysStyle{S}$.

The analogous concept can be used for sets of subsets
of~$\powerSet{\indexSet{n}}$.  For example, in one application there
is the induced action of a permutation $\grg \in \symGroup{n}$ on a
triangulation $\setsysStyle{T}$ of a point configuration with $n$
labeled points given by the label sets of its maximal simplices. That
action is denoted by~$\grg(\setsysStyle{T})$ as well.

Here and in the following, the size of the actual output, generated in
the function $\output$ in all the algorithms below, is significant for
the assessment of whether or not an algorithm can be called
output-polynomial.  If the output is just printing a single symbol for
each found object, then the output is equivalent to a unary encoding
of the number of found objects, and an output-polynomial algorithm can
take time polynomial in the input size and the number of found
objects.  This problem setup is called the \emph{enumeration problem}
in this paper; if the output is skipped altogether, then the remaining
output is only the number of objects to be counted, and an
output-polynomial algorithm must be polynomial in the input size and
the logarithm of the number of found objects.  This problem setup is
called the \emph{counting problem} in this paper.  Finally, if for
each enumerated object the object has to be listed, then the ouput is
a list of encodings of all enumerated objects.  This problem setup is
called the \emph{listing problem} in this paper.  If the output size
of a single object is exponential in the input size, then an algorithm
that is polynomial in the input and output sizes for the listing
problem can be exponential in the input and output sizes for the
enumeration problem.  For example, the output of a single
triangulation can be exponential in the input size and even in the
number of triangulations (see Section~\ref{sec:triangs:analysis}
Theorem~\ref{thm:triang-alg-efficiency}~\ref{itm:triangs-effcross}).
A similar distinction will be necessary concerning the encoding of the
input. If the symmetries are given as an explicit, complete set of
permutations in tupel notation, then a polynomial-time algorithm can
take time polynomial in the order and the degree of the symmetry
group.  If the symmetries are given in terms of generators in
cycle-notation, then a polynomial-time algorithm needs to be
polynomial in the actual input-size of the generators, which can be
substantially more restrictive.

In the remainder of this section, the general-purpose algorithm
\emph{Symmetric Lexicographic Subset Reverse Search} (\symLSRS) is
developed in the language of the reverse-search paradigm
from~\cite{AvisFukuda_Reversesearchenumeration_1996}.  Starting at the
basic Reverse Search, the extensions to orbits and the specializations
to subsets are introduced one-by-one.  The algorithms in this section
are not new, but a summary of all variants in a unified notational
environment is helpful to understand the extensions.

As a simplification, a recursive form is used for presentation, which
usually increases the memory consumption of the algorithms.  However,
the more memory-efficient original Reverse-Search framework (with
non-recursive backtracking by pivoting) can be applied to all
presented algorithms.  In this paper, recursive implementations have
been used throughout, since they were faster in the presented
applications.

First, the original Reverse Search is formulated
\cite{AvisFukuda_Reversesearchenumeration_1996}.  Reverse Search (\RS)
is an algorithm to enumerate the nodes of a graph $\GG = (\VV, \EE)$
with known maximal degree~$\maxDegree$. The graph is implicitly given
by an adjacent-nodes function
$\adjNode: \VV \times \indexSet{\maxDegree} \to \VV \cup \{\nil\}$
that returns for each node $\vv \in \VV$ and each integer
$i \in \indexSet{\maxDegree}$ the $i$th neighbor of~$\vv$ if
$i \le \degree(\vv)$ and $\nil$ if $\degree(\vv) < i \le \maxDegree$.
The enumeration is structured by a non-degenerate cost function
$\nodeCost: \VV \to \mathbb{R}$ and a pivot-function
$\impNode: \VV \to \adjSet(\VV) \cup \{\nil\}$, which is a rule to
specify a $\nodeCost$-improving neighbor in~$\GG$ in case there is
one.  A recursive representation of Reverse Search in pseudo-code is
shown in Algorithm~\ref{alg:RS}.

The run-time complexity of \RS is
$O(\maxDegree \timeComplexity(\adjNode) \abs{\VV} +
\timeComplexity(\impNode) \abs{\EE})$, where $\timeComplexity(f)$
denotes the maximal time to compute a function value of a
function~$f$.  Since $2\abs{\EE} \le \maxDegree\abs{\VV}$, this
run-time complexity is
in~$O(\maxDegree(\timeComplexity(\adjNode) +
\timeComplexity(\impNode)) \abs{\VV})$.  This is particularly
interesting when $\timeComplexity(\adjNode)$ and
$\timeComplexity(\impNode)$ do not depend on~$\abs{\VV}$.  In that
case, \RS is linear in the output size for the enumeration
problem~\cite{AvisFukuda_Reversesearchenumeration_1996}.

\begin{algorithm}[htbp]
  \TitleOfAlgo{\RS{\adjNode, \nodeCost, \impNode, \vv}}

  \KwIn{a connected graph $\GG = (\VV, \EE)$ with maximal
    degree~$\maxDegree$, implicitly given by an adjacent-nodes
    function
    $\adjNode: \VV \times \indexSet{\maxDegree} \to \VV \cup
    \{\nil\}$, a cost function $\nodeCost: \VV \to \mathbb{R}$ with
    $\nodeCost(\vv) \neq \nodeCost(\ww)$ for all $\vv \neq \ww$
    in~$\VV$ and
    $\nodeCost(\vopt) = \min_{\vv \in \VV} \nodeCost(\vv)$, a pivot
    function $\impNode: \VV \to \adjSet(\VV) \cup \{\nil\}$ with
    $\nodeCost\bigl(\impNode(\vv)\bigr) < \nodeCost(\vv)$ for all
    $\vv \in \VV \setminus \{\vopt\}$ and $\impNode(\vopt) = \nil$, a
    seed node $\vv$ in~$\VV$}

  \KwOut{the number of nodes
    in~$\VV_{\ge \vv} := \{\ww \in \VV : \nodeCost(\ww) \ge
    \nodeCost(\vv)\}$}

  \tcc{build a depth-first-search tree with root node~$\vv$:} \output
  $\vv$ and $c \gets 1$ \tcc*{count $\vv$} \For(\tcc*[f]{iterate over
    neighbors of~$\vv$}){ $j = 1, \dots, \maxDegree$ }{
    $\ww \gets \adjNode(\vv, j)$ \tcc*{get $j$th neighbor}
    \If(\tcc*[f]{if past the last neighbor}){ $\ww = \nil$ }{ break
      \tcc*{exit the loop} } \If(\tcc*[f]{if $\ww$ pivots to~$\vv$}){
      $\vv = \impNode(\ww)$ }{
      $c \gets c + \RS{\adjNode, \nodeCost, \impNode, \ww}$
      \tcc*{recurse} } } \Return c\;
  
  \caption[Reverse Search]{The generic reverse search algorithm
    (cf.~\cite{AvisFukuda_Reversesearchenumeration_1996})
  }
  \label{alg:RS}
\end{algorithm}

\begin{algorithm}[htbp]
  \TitleOfAlgo{\symRS{\adjNode, \nodeCost, \impNode, \GrG, \vcan}}

  \KwIn{a connected graph $\GG = (\VV, \EE)$ with maximal
    degree~$\maxDegree$, implicitly given by an adjacent-nodes
    function
    $\adjNode: \VV \times \indexSet{\maxDegree} \to \VV \cup
    \{\nil\}$, a cost function $\nodeCost: \VV \to \mathbb{R}$ with
    $\nodeCost(\vv) \neq \nodeCost(\ww)$ for all $\vv \neq \ww$
    in~$\VV$ and
    $\nodeCost(\vopt) = \min_{\vv \in \VV} \nodeCost(\vv)$ , a
    pivot-function $\impNode: \VV \to \adjSet(\VV) \cup \{\nil\}$ with
    $\nodeCost\bigl(\impNode(\vv)\bigr) < \nodeCost(\vv)$ for all
    $\vv \in \VV \setminus \{\vopt\}$ and $\impNode(\vopt) = \nil$, a
    subgroup $\GrG$ of $\autGroup(\GG)$ and a canonical-representative
    function $\Can{\vv} =: \vcan$ with
    $\nodeCost(\vcan) < \nodeCost(\ww)$ for all
    $\ww \in \orbitSet{\GrG}{\vv}$, a seed node $\vcan$ in~$\VV$ that
    is the canonical representative of $\orbitSet{\GrG}{\vcan}$}
  
  \KwOut{the number of canonical $\wcan$ with
    $\nodeCost(\wcan) \ge \nodeCost(\vcan)$}
  
  \tcc{build a depth-first-search tree with root node~$\vcan$:}
  \output $\vcan$ and $c \gets 1$ \tcc*{count $\vcan$}
  $\Wcan \gets \{\vcan\}$ \tcc*{collects already processed canonicals}
  \For(\tcc*[f]{iterate over neighbors of~$\vcan$}){
    $j = 1, \dots, \maxDegree$ }{ $\ww \gets \adjNode(\vcan, j)$
    \tcc*{get $j$th neighbor} \If(\tcc*[f]{if past the last
      neighbor}){ $\ww = \nil$ }{ break \tcc*{exit the loop} }
    $\wcan \gets \Can{\ww}$ \tcc*{compute canonical representative}
    \If(\tcc*[f]{if $\wcan$ is new}){ $\wcan \notin \Wcan$ }{
      $\uu \gets \impNode(\wcan)$ \tcc*{compute pivot}
      $\ucan \gets \Can{\uu}$ \tcc*{compute canonical representative}
      \If(\tcc*[f]{if $\orbitSet{\GrG}{\ww}$ pivots to
        $\orbitSet{\GrG}{\vcan}$}){ $\ucan = \vcan$ }{
        $\Wcan \gets \Wcan \cup \{\wcan\}$ \tcc*{add $\wcan$ to set of
          processed canonicals}
        $c \gets c + \RS{\adjNode, \nodeCost, \impNode, \GrG, \wcan}$
        \tcc*{recurse} } } } \Return c\;
  
  \caption[Reverse Search for Orbits]{The application of
    reverse-search to the graph of all orbits in~$\setOrbits{\VV}{\GrG}$
  }
  \label{alg:symRS}
\end{algorithm}

If one has a group $\GrG$ acting on $\GG$ one is usually only
interested in $\abs{\VV}$ up to symmetry.  In other words, the number
of $\GrG$-orbits $\abs{\setOrbits{\VV}{\GrG}}$ of~$\VV$ shall be
determined.  How the reverse search principle can be adapted to this
setting, was first presented in
\cite{ImaiMasadaTakeuchiImai_Enumeratingtriangulationsgeneral_2002}.
The idea is to extend the adjacent-nodes function $\adjNode$ to
$\GrG$-orbits in the obvious way and to modify the pivot-function to
consist of two steps: In the first step, the $\nodeCost$-minimal
element, the \emph{canonical element}, in the current orbit is taken;
in the second step, the $\nodeCost$-reducing pivot-function on $\GG$
is followed.  Thus, another orbit is reached in case the minimal
element of the current orbit was no global minimum already.  The
downside is that all canonical elements in the neighborhood of a node
must be stored intermediately to avoid duplicate counting of orbits.
This undermines the memory efficiency of the reverse search principle.
Algorithm~\ref{alg:symRS} shows a detailed pseudo-code representation.

\begin{algorithm}[htbp]
  \TitleOfAlgo{\SRS{n, \downsetStyle{D}, \setStyle{S}}}

  \KwIn{$n \in \mathbb{N}$, a downset $\downsetStyle{D}$ of
    $\powerSet{\indexSet{n}}$, a seed set
    $\setStyle{S} \in \downsetStyle{D}$}

  \KwOut{the number of elements $\setStyle{S}'$ in~$\downsetStyle{D}$
    with $\setStyle{S} \lexsubset \setStyle{S}'$}
  
  \tcc{build a depth-first-search tree with root node $\setStyle{S}$:}
  \output $\setStyle{S}$ and $c \gets 1$ \tcc*{count~$\setStyle{S}$}
  \For(\tcc*[f]{for extensions with new max}){
    $j \in \indexSet{n} \setminus \setStyle{S}$ with
    $j > \maxElem{\setStyle{S}}$ }{
    $\setStyle{S}' \gets \setStyle{S} \cup \{j\}$ \tcc*{build new set}
    \If(\tcc*[f]{if new set belongs to $\downsetStyle{D}$}){
      $\isInD{\setStyle{S'}, \downsetStyle{D}}$}{
      $c \gets c + \LSRS{n, \downsetStyle{D}, \setStyle{S'}}$
      \tcc*{recurse} } } \Return c\;
  
  \caption[Reverse Search for Subsets]{The standard application of
    reverse-search to the Hasse-diagram of a downset
    $\downsetStyle{D}$ of subsets of an $n$-element set that
    exploits the straight-forward set-valued inverse of the pivot
    function
  }
  \label{alg:SRS}
\end{algorithm}

\begin{algorithm}[htbp]
  \TitleOfAlgo{\LSRS{n, \downsetStyle{D}, \setStyle{S}}}

  \KwIn{$n \in \mathbb{N}$, a downset $\downsetStyle{D}$ of
    $\powerSet{\indexSet{n}}$, a seed set
    $\setStyle{S} \in \downsetStyle{D}$}

  \KwOut{the number of elements $\setStyle{S}'$ in~$\downsetStyle{D}$
    with $\setStyle{S} \lexsubset \setStyle{S}'$}
  
  \tcc{build a depth-first-search tree with root node $\setStyle{S}$:}
  \output $\setStyle{S}$ and $c \gets 1$ \tcc*{count~$\setStyle{S}$}
  \For(\tcc*[f]{ordered traversal of new elements}){
    $j = \maxElem{\setStyle{S}} + 1, \dots, n$ }{
    $\setStyle{S}' \gets \setStyle{S} \cup \{ j \}$ \tcc*{build new
      set} \If(\tcc*[f]{if new set belongs to $\downsetStyle{D}$}){
      $\isInD{\setStyle{S'}, \downsetStyle{D}}$}{
      $c \gets c + \LSRS{n, \downsetStyle{D}, \setStyle{S'}}$
      \tcc*{recurse} } } \Return c\;
  
  \caption[Lexicographic Reverse Search for Subsets]{A specialized
    implementation of $\SRS{n, \downsetStyle{D}, \setStyle{S}}$
    traversing the pivot-inverse according to the lexicographic order of
    subsets of~$\indexSet{n}$
  }
  \label{alg:LSRS}
\end{algorithm}

A specialized form of Reverse Search arises when a set of ``feasible''
subsets of a finite set shall be enumerated by adding elements,
checking feasibility one-by-one, and back-tracking.  This
automatically will touch many subsets of feasible sets.  Thus, one can
restrict to the enumerations of downsets, i.e., sets of subsets, where
each subset of a feasible set is feasible itself.
Algorithm~\ref{alg:SRS} shows the specialization of Reverse Search,
which is a folklore observation.\footnote{Strictly speaking, for the
  algorithms in this paper it is enough for~$\downsetStyle{D}$ to be a
  \emph{left-downset}, i.e., closed under taking lex-subsets.  General
  downsets are taken in order to keep the problem statements
  independent of the algorithms and additional structures used for
  their solutions.}  A notable further specialization can be seen in
Algorithm~\ref{alg:LSRS}: by using the natural order for extending
subsets by a new element one can guarantee that the subsets are found
in lexicographic order.

\begin{algorithm}[htbp]
  \TitleOfAlgo{\symLSRS{n, \downsetStyle{D}, \GrG, \setStyle{S}}}

  \KwIn{$n \in \mathbb{N}$, a downset $\downsetStyle{D}$ of
    $\powerSet{\indexSet{n}}$, a subgroup $\GrG$ of the automorphism
    group of $\downsetStyle{D}$, a seed set
    $\setStyle{S} \in \downsetStyle{D}$ with
    $\setStyle{S} = \lexmin\orbitSet{\GrG}{\setStyle{S}}$}

  \KwOut{the number of canonical $\setStyle{S}'$ in~$\downsetStyle{D}$
    with $\setStyle{S} \lexsubset \setStyle{S}'$}
  
  \tcc{build a depth-first-search tree with root node $\setStyle{S}$:}
  \output $\setStyle{S}$ and $c \gets 1$ \tcc*{count~$\setStyle{S}$}
  \For(\tcc*[f]{ordered traversal of new maximal elements}){
    $i = \maxElem{\setStyle{S}} + 1, \dots, n$ }{
    $\setStyle{S}' \gets \setStyle{S} \cup \{ i \}$ \tcc*{add a new
      element on the right}

    $\mathrm{answer} \gets \isLexMin{\setStyle{S'}, \GrG}$
    \tcc*{lex-min check} \If(\tcc*[f]{if new set is not lex-min in its
      orbit}){ $\mathrm{answer} = \false$}{ continue \tcc*{next loop
        element} }
    $\mathrm{answer} \gets \isInD{\setStyle{S'}, \downsetStyle{D}}$
    \tcc*{membership check} \If(\tcc*[f]{if new set is not in
      $\downsetStyle{D}$}){ $\mathrm{answer} = \false$}{ continue
      \tcc*{next loop element} } $c \gets c + \symLSRS{$n,
      \downsetStyle{D}, \GrG, \setStyle{S}'$}$ \tcc*{recurse} }
  \Return c\;
  
  \caption[Symmetric Lexicographic Reverse Search for Orbits of
  Subsets]{A specialized implementation of the combination of $\LSRS$
    and $\symRS$ for orbits of subsets in a downset $\downsetStyle{D}$
    of $\powerSet{\indexSet{n}}$
    (cf.~\cite{PechReichard_EnumeratingSetOrbits_2009})
  }
  \label{alg:symLSRS}
\end{algorithm}

If downsets shall be enumerated up to symmetry, the lexicographic
order on subsets simplifies affairs substantially. If the canonical
representative is defined to be the lex-min element in each orbit,
then there is no need for computing and comparing canonical
representatives for all found elements anymore.  The only thing needed
is to check whether or not the found element can be lex-decreased
\emph{at all} by the action of an element in~$\GrG$. Since
Lemma~\ref{thm:lexmin-lemma} guarantees that a canonical subset can
never lex-contain a non-canonical subset, the enumeration can simply
ignore any non-canonical element and backtrack.  Therefore, in the
following, a subset $\setStyle{S}$ is called \emph{canonical} if it is
lex-min in its orbit. The resulting algorithm
(Algorithm~\ref{alg:symLSRS}) is essentially identical to the
algorithm proposed in~\cite{PechReichard_EnumeratingSetOrbits_2009}.

With this, the processing of a node in the enumeration of subsets up
to symmetry is reduced to checking whether a subset is lex-min in its
orbit and to checking, whether a subset is feasible.  Thus, the
collection of canonicals in the neighborhood of a subset need not be
stored.

In \cite{PechReichard_EnumeratingSetOrbits_2009}, a
\emph{generator-based recursive method} is proposed to check whether a
subset is lex-min in its orbit.  The applications in that paper
indicate that the methods aim at groups of very large order.  For
groups of medium-sized order (like in the millions) the, also
generator-based, \emph{switch-table method} (originally for vectors,
which is more general) is presented
in~\cite{JordanJoswigKastner_Parallelenumerationtriangulations_2018}.
A new combination of the methods in
\cite{PechReichard_EnumeratingSetOrbits_2009}
and~\cite{JordanJoswigKastner_Parallelenumerationtriangulations_2018}
(for subsets) turned out to be the fastest method for the enumeration
of cocircuits and circuits in the test cases of this paper.  For even
smaller group orders (like in the thousands) the overhead of all the
generator-based methods may not pay off.  Hence, a new taylor-made
\emph{element-based} method is developed in Section~\ref{sec:theory}
exactly for those cases.  This turned out to be the fastest method for
the enumeration of triangulations.

Switch tables
\cite{JordanJoswigKastner_Parallelenumerationtriangulations_2018} are
non-standard and are therefore introduced in the following.  Given a
subgroup $\GrG$ of $\symGroup{n}$, a \emph{switch table} is a function
$\switchPerm{\cdot}{\cdot}: \indexSet{n} \times \indexSet{n} \to \GrG$
with either $\switchPerm{i}{j} = \grg$ so that $\grg(j) = i$ for some
$j > i$ and $\grg(k) = k$ for all $k < i$, if such a $\grg \in \GrG$
exists, or $\switchPerm{i}{j} = \id$ otherwise.  A switch table need
not be unique. The entries $\switchPerm{i}{j}$ are called
\emph{switches}.  An entry of a switch table is \emph{trivial} if it
is the identity.  A row of a switch table is \emph{effective} if it
contains at least one non-trivial switch.  The set $\effRowSet$ of all
row indices $i$ so that $\switchPerm{i}{\cdot}$ is effective is called
the \emph{effective row set} of~$\switchPerm{\cdot}{\cdot}$.  The
(non-empty) set of column indices of non-trivial switches in an
effective row $i \in \effRowSet$ is the \emph{effective column set}
of~$i$, denoted by $\effColSet(i)$.  For the trivial group, there are
no effective rows, for which we adopt the convention that
$\max \effRowSet = -\infty$.

One key property of a switch table is that each element of $\GrG$ is
an essentially unique (up to the insertion of trivial factors) product
of switches consisting of at most one non-trivial switch from each
row.  Not all switches can contribute to a lex-decreasing
switch-product, so that not all products need to be considered.  A
straight-forward version of the algorithm for subsets is presented in
Algorithm~\ref{alg:isLexMinSwitches}.  Details for general vectors
instead of subsets -- which can be seen as special, namely
characteristic, vectors -- can be found
in~\cite{JordanJoswigKastner_Parallelenumerationtriangulations_2018}.
The algorithm uses global auxiliary data, which is given by a switch
table $\switchTable = \switchPerm{\cdot}{\cdot}$ for~$\GrG$.

\begin{algorithm}[htbp]
  \TitleOfAlgo{\isLexMinSwitchTable{$i, \setStyle{S}, \setStyle{S'},
      \GrG, \switchTable$}}

  \KwIn{an integer $i \in [1, n]$, a subset $\setStyle{S}$ (the
    original subset), a subset $\setStyle{S'}$ (the subset mapped by a
    partial switch product), a subgroup $\GrG$ of~$\symGroup{n}$, and
    a switch table~$\switchTable$ of~$\GrG$}

  \KwOut{$\true$ if
    $\setStyle{S} = \lexmin
    \stabGroup{\GrG}{\indexSet{i-1}}(\setStyle{S})$ and $\false$
    otherwise}
  
  \tcc{check for empty set:} \If{$\setStyle{S} = \emptyset$}{ \Return
    \true\; } \tcc{beyond the effective row set?}
  \If{$i > \max \effRowSet(\switchTable)$}{ \Return
    ($\setStyle{S}' \notlexsmaller \setStyle{S}$); } \tcc{recursively
    check a switch product with leading identity:}
  \If{$\isLexMinSwitchTable{$i + 1, \setStyle{S}, \setStyle{S'}, \GrG,
      \switchTable$} = \false$}{ \Return \false\; }

  \tcc{collect "good" switches lex-decreasing~$\setStyle{S}'$:}
  $\goodSwitchesSet \gets \{ j \in \effColSet(i) \;\vert\; i \notin
  \setStyle{S'}, j \in \setStyle{S'}\}$\;

  \If{$\goodSwitchesSet = \emptyset$}{ \tcc{collect "good" switches
      not preventing future lex-decrease:}
    $\goodSwitchesSet \gets \bigl\{ j \in \effColSet(i) \;\big\vert\;
    \abs{\setStyle{S'} \cap \{i, j\}} \in \{0, 2\} \bigr\}$\; }

  \For{$j \in \goodSwitchesSet$}{
    $\setStyle{S''} \gets \switchTable[i][j](\setStyle{S'})$ \tcc*{map
      subset} \If{$\setStyle{S''} \lexsmaller \setStyle{S}$}{ \Return
      \false \tcc*{decreasing switch product found} } \tcc{recurse
      with mapped subset:} \If{$\isLexMinSwitchTable{$i + 1,
        \setStyle{S}, \setStyle{S''}, \GrG, \switchTable$} = \false$}{
      \Return \false\; } } \Return \true \tcc*{no decreasing switch
    product found}

  \caption[Switch-Table Method]{The switch-table method from
    \cite{JordanJoswigKastner_Parallelenumerationtriangulations_2018}
    specialized to subsets
  }
  \label{alg:isLexMinSwitches}
\end{algorithm}

The membership-test in $\isInD$ can be difficult.  In cases, where the
downset to be enumerated up to symmetry has only been implicitly
defined as a downset by a set of really interesting subsets together
with all its subsets, $\isInD$ in fact needs to answer the question
whether or not a subset is a subset of a really interesting set.  This
problem arises in all of the applications in this paper and has to be
solved for each application seperately.  In
section~\ref{sec:algorithm}, variants of $\symLSRS$ will be designed
where this membership test is relaxed in order to make faster progress
at the cost of accepting deadends in the enumeration.


\paragraph*{Problem Statement}
\label{sec:problem}

In this paper, the following questions concerning the algorithm
\symLSRS and some variants are studied in detail:
\begin{itemize}
\item How can the subroutine \isLexMin of \symLSRS be implemented
  efficiently \emph{in general}?
\item How can the subroutines of \symLSRS and its variants be
  implemented efficiently \emph{for each of the applications}.  More
  specifically:
  \begin{itemize}
  \item How can one effectively prune subsets of vectors that are not
    lex-contained in a circuit or cocircuit, respectively?
  \item How can one effectively prune subsets of simplices that are
    not lex-contained in a triangulation of the point configuration?
  \end{itemize}
\end{itemize}
The term \emph{efficiently} here does \emph{not} mean
\emph{polynomial-time}.  It rather indicates that compared to naive
approaches the number of steps is reduced as much as possible.  The
resulting algorithms are, in general, not
(counting/enumerating/listing for cocircuits and circuits as well as
counting/enumerating for triangulations) or not known to be (listing
for triangulations) output-polynomial.

\section{On Lexicographically Minimal Elements in Subset-Orbits}
\label{sec:theory}

This section provides mathematical structures that help to decide
whether or not a subset is lexicographically minimal in its orbit.
For some of the upcoming arguments the following characterization of
the lexicographic order on $k$-element subsets of~$\indexSet{n}$ is
used.

\begin{lemma}[Lexicographic order on subsets]
  Let $n \in \mathbb{N}$.  Moreover, let $\setStyle{S}$ and
  $\setStyle{R}$ be $k$-element subsets of~$\indexSet{n}$.  Then
  $\setStyle{S}$ is lexicographically smaller than~$\setStyle{R}$ if
  the minimal element of their symmetric difference is
  in~$\setStyle{S}$.  In formulae:
  \begin{equation}
    \setStyle{S} \lexsmaller \setStyle{R} \iff \min (\setStyle{S}
    \symdiff \setStyle{R}) \in \setStyle{S},
    \label{eq:lexsmaller}
  \end{equation}
  where $\min \emptyset = \infty$.
\end{lemma}

The following lemma states that if subsets are built
element-by-element with backtracking and only the lexicographically
minimal ones in their orbits are followed, then one will reach all
subsets that are lexicographically minimal in their orbits.  This
result was already presented
in~\cite{PechReichard_EnumeratingSetOrbits_2009}.  Here, some more
details for the proof are provided.
\begin{lemma}[cf.~\cite{PechReichard_EnumeratingSetOrbits_2009}] Let
  $\setStyle{S}$ be a non-empty subset of~$\indexSet{n}$ and
  $\setStyle{S}^- := \setStyle{S} \setminus \{\max
  \setStyle{S}\}$. Then, for all subgroups $\GrG$ of~$\symGroup{n}$ we
  have:
  \begin{equation}
    \setStyle{S} = \lexmin \orbitSet{\GrG}{\setStyle{S}}
    \Rightarrow
    \setStyle{S}^- = \lexmin \orbitSet{\GrG}{\setStyle{S}^-}
    \label{thm:lexmin-lemma}
  \end{equation}
\end{lemma}

\begin{proof}
  Let $\setStyle{S}$ be lex-min in its $\GrG$-orbit.  Assume, for the
  sake of contradiction, that $\setStyle{S}^-$ is not lex-min in its
  $\GrG$-orbit.  Then, there is a set $\setStyle{R}$ which is
  lex-smaller than~$\setStyle{S}^-$ and a permutation $\pi \in \GrG$
  with $\pi(\setStyle{S}^-) = \setStyle{R}$.  In particular,
  $\setStyle{S}^-$ and $\setStyle{R}$ are non-empty and non-identical.
  That means,
  $\min (\setStyle{S}^- \symdiff \setStyle{R}) =: r \in \setStyle{R}
  \setminus \setStyle{S}^-$. Moreover,
  $\pi(\max \setStyle{S}) \notin \setStyle{R}$, by the bijectivity of
  any permutation. Consider
  $\setStyle{R}^+ := \setStyle{R} \cup \{\pi(\max \setStyle{S})\}$.
  
  Since
  $\max \setStyle{S} > \min (\setStyle{S}^- \setminus \setStyle{R}) >
  \min (\setStyle{R} \setminus \setStyle{S}^-) = \min (\setStyle{S}^-
  \symdiff \setStyle{R}) = r$, we know that
  $\min (\setStyle{S} \setminus \setStyle{R}) > r$.  Moreover,
  $\setStyle{R} \subset \setStyle{R}^+$ implies
  $\min (\setStyle{S} \setminus \setStyle{R}^+) \ge \min (\setStyle{S}
  \setminus \setStyle{R}) > r$.  Since
  $r \in \setStyle{R} \setminus \setStyle{S}^-$ and
  $r < \max \setStyle{S}$, it follows that
  $r \in \setStyle{R} \setminus \setStyle{S} \subseteq \setStyle{R}^+
  \setminus \setStyle{S}$.  Thus,
  $\min (\setStyle{R}^+ \setminus \setStyle{S}) \le r < \min
  (\setStyle{S} \setminus \setStyle{R}^+)$, and, hence,
  $\min (\setStyle{S} \symdiff \setStyle{R}^+) \in \setStyle{R}^+$.
  Thus, by definition, $\setStyle{R}^+$ is lex-smaller
  than~$\setStyle{S}$ with $\pi(\setStyle{S}) = \setStyle{R}^+$:
  contradiction.
\end{proof}

Next, two methods for the lex-min check are presented: one for
``small'' symmetry groups (order small compared to degree) based on
the new concept of \emph{critical-element tables} and one for
``large'' symmetry groups (order large compared to degree) based on a
modified use of switch tables.

First, the method for small groups is presented.  It is based on the
new structure of \emph{critical-element tables}. The third application
in this paper concerns most often smaller symmetry groups with
relatively high degree that can easily be enumerated first.  How can
one reduce the number of evaluations of actions on subsets in that
check?  Remember that the orbits' lexicographically-minimal subsets
are built element-by-element.  That is, for the check of whether or
not a given $m$-subset $\setStyle{S}$ is lexicographically minimal in
its orbit, one can exploit that its predecessor-subset
$\setStyle{S}^-$ with $m-1$ elements is already known to be
lexicographically minimal in its orbit.  The new question is: has the
addition of the new maximal element led to the existence of a
permutation that lexicographically decreases (\emph{lex-decreases},
for short) the new subset $\setStyle{S}$ in its orbit?

The new idea in this paper is to keep the information about \emph{why}
the predecessor subset $\setStyle{S}^-$ is lexicographically minimal
in its orbit.  The reason is that for each permutation $\pi \in \GrG$
one has
\begin{equation}
  \min \bigl(\setStyle{S}^- \symdiff \pi(\setStyle{S}^-)\bigr) \notin \pi(\setStyle{S}^-).
  \label{eq:lexsmaller-in-orbit}
\end{equation}
These minimal elements of the symmetric differences of subsets and
their images are critical for the question at hand, which motivates
the following definition:
\begin{definition}
  Let $n \in \mathbb{N}$ and $\GrG$ be a subgroup of $\symGroup{n}$.
  For a given subset $\setStyle{S}$ the \emph{critical-element
    table with respect to~$\setStyle{S}$} is defined as follows:
  \begin{equation}
    \critelem{\setStyle{S}} \colon
    \left\{
      \begin{array}{rcl}
        \GrG     & \to      & \indexSet{n} \cup \{\infty\},\\
        \pi      & \mapsto  & \min \bigl(\setStyle{S} \symdiff \pi(\setStyle{S})\bigr),
      \end{array}
    \right.
    \label{eq:critical-element-table}
  \end{equation}
  where $\min \emptyset := \infty$.  The function value
  $\critelem{\setStyle{S}} (\pi)$ is called the \emph{critical element
    of~$\pi$ with respect to~$\setStyle{S}$}.
\end{definition}

From the critical-element table, some important of properties of the
action of a permutation~$\pi$ on the given subset~$\setStyle{S}$ can
be derived easily:
\begin{lemma}
  \label{thm:char-lex-min-cet}
  Let $n \in \mathbb{N}$ and $\GrG$ be a subgroup of $\symGroup{n}$.
  Moreover, let $\setStyle{S}$ be a non-empty subset
  of~$\indexSet{n}$.  Then:
  \begin{enumerate}[label=(\roman*)]
  \item The stabilizer of $\setStyle{S}$ in~$\GrG$ is the set of
    permutations with critical element~$\infty$.
  \item $\setStyle{S}$ is lexicographically minimal in its
    $\GrG$-orbit if and only if
    $\critelem{\setStyle{S}} (\pi) \notin \pi(\setStyle{S})$ for all
    $\pi \in \GrG$.
  \item Let
    $\setStyle{S}^- := \setStyle{S} \setminus \{\max\setStyle{S}\}$.
    Moreover, assume that $\setStyle{S}^-$ is lexicographically
    minimal in its $\GrG$-orbit.  Then a permutation $\pi \in \GrG$
    lex-decreases~$\setStyle{S}$ if and only if one of the following
    cases occurs:
    \begin{enumerate}[label=\Roman*.]
    \item $\critelem{\setStyle{S}^-} (\pi) = \infty$ and
      $\pi(\max\setStyle{S}) < \max\setStyle{S}$
    \item $\critelem{\setStyle{S}^-} (\pi) \in \setStyle{S}^-$ and
      $\pi(\max\setStyle{S}) < \critelem{\setStyle{S}^-} (\pi)$ or
    \item $\critelem{\setStyle{S}^-} (\pi) \in \setStyle{S}^-$,
      $\pi(\max\setStyle{S}) = \critelem{\setStyle{S}^-} (\pi)$ and
      $\critelem{\setStyle{S}} (\pi) \in \pi(\setStyle{S})$\qed
    \end{enumerate}
  \end{enumerate}
\end{lemma}
Call the application of this the \emph{critical-element method} for
checking lexicographic minimality of a subset in its orbit.

The crucial gain of this lemma is the following: given the
critical-element table with respect to~$\setStyle{S}^-$, one can check
lexicographic minimality of $\setStyle{S}$ in its orbit without
actually computing~$\pi(\setStyle{S})$, with the only exception when
$\pi$ maps the new element of~$\setStyle{S}$ exactly to the critical
element of~$\setStyle{S}^-$.  And this exception roughly happens for a
$\frac{1}{n}$-fraction of the permutations, on average.

There are now two ways to implement the critical-element method.
\begin{enumerate}
\item iterate over all permutations and apply
  Lemma~\ref{thm:char-lex-min-cet} to each of them (the
  \emph{iteration-based critical-element method}; see
  Algorithm~\ref{alg:isLexMinCritelemIter} in
  Section~\ref{sec:algorithm} for a possible implementation);
\item first, from certain fixed, preprocessed subsets of permutations,
  compute the subsets of permutations that
  \begin{enumerate}
  \item certainly lexicographically decrease the given subset, and if
    empty,
  \item possibly lexicographically decrease the given subset.
  \end{enumerate}
  If the subset of certainly lexicographically decreasing permutations
  is non-empty, the subset is not lexicographically minimal in its
  orbit; if it is empty, iterate over the possibly lexicographically
  decreasing permutations and apply~Lemma~\ref{thm:char-lex-min-cet}
  to only those (the \emph{set-based critical-element method}; see
  Algorithm~\ref{alg:isLexMinCritelemSet} in
  Section~\ref{sec:algorithm} for a possible implementation).
\end{enumerate}

In the sequel, some details for the set-based method are
explained.\footnote{The approach is motivated by the fact that with a
  set structure based on dynamic bitstrings it is possible to compute
  unions, intersections, differences, and symmetric differences of
  sets fast in practice.  Nota bene: From a complexity standpoint,
  bitstrings only gain something if their length is uniformly bounded.
  However, in practice the reduction of the runtime by a constant
  factor by using dynamic bitstrings for set operations is not
  irrelevant.  Moreover, operations on dynamic bitstrings residing
  consecutively in memory are more cache coherent than data structures
  that are scattered in main memory.}  The following structures of the
possibly lexicographically decreasing permutations for a specific
subset.  The first three structures can be preprocessed prior to the
enumeration.  The fourth structure has to be updated for each
enumeration node.
\begin{definition}
  For $n \in \mathbb{N}$ and a subgroup $\GrG$ of~$\symGroup{n}$, the
  \emph{hit-element classification of~$\GrG$} is defined as
  \begin{equation}
    \hitclass \colon
    \left\{
      \begin{array}{rcl}
        \indexSet{n} \times \indexSet{n} & \to     & \powerSet{\GrG},\\
        (i, j)                           & \mapsto & \{ \pi \in \GrG :
                                                     \pi(i) = j \}.
      \end{array}
    \right.
    \label{eq:hitting-classification}    
  \end{equation}
  The \emph{increasing-element classification of~$\GrG$} is defined as
  \begin{equation}
    \incclass \colon
    \left\{
      \begin{array}{rcl}
        \indexSet{n} & \to     & \powerSet{\GrG},\\
        i            & \mapsto & \{ \pi \in \GrG : \pi(i) > i \} =
                                 \bigcup_{k=i + 1}^n \hitclass(i, k).
      \end{array}
    \right.
    \label{eq:incerasing-classification}    
  \end{equation}
  The \emph{decreasing-element classification of~$\GrG$} is defined as
  \begin{equation}
    \decclass \colon
    \left\{
      \begin{array}{rcl}
        \indexSet{n} \times \indexSet{n} & \to     & \powerSet{\GrG},\\
        (i, j)                           & \mapsto & \{ \pi \in \GrG :
                                                     \pi(i) < j \} =
                                                     \bigcup_{k=1}^{j - 1} \hitclass(i, k).
      \end{array}
    \right.
    \label{eq:passing-classification}    
  \end{equation}
  Moreover, for a subset $\setStyle{S}$ of~$\indexSet{n}$, the
  \emph{critical-element classification of~$\GrG$ with respect
    to~$\setStyle{S}$} is defined as
  \begin{equation}
    \critclass{\setStyle{S}} \colon
    \left\{
      \begin{array}{rcl}
        \indexSet{n} \cup \{\infty\} & \to     & \powerSet{\GrG},\\
        i                            & \mapsto & \{ \pi \in \GrG :
                                                 \critelem{\setStyle{S}}
                                                 (\pi) = i \}.
      \end{array}
    \right.    
    \label{eq:critical-element-classification}
  \end{equation}
\end{definition}
The next lemma characterizes lexicographical minimality in an orbit by
intersections of certain decreasing-element and critical-element
classifications.
\begin{lemma}
  \label{thm:char-lex-min-cet-set}
  Let $n \in \mathbb{N}$ and $\GrG$ a subgroup of~$\symGroup{n}$.
  Moreover, let $\setStyle{S}$ be a non-empty subset
  of~$\indexSet{n}$, and let
  $\setStyle{S}^- = \setStyle{S} \setminus \{ \max \setStyle{S} \}$ be
  lexicographically minimal in its $\GrG$-orbit.  Then $\setStyle{S}$
  is not lexicographically minimal in its $\GrG$-orbit if and only if
  at least one of the following cases occurs:
  \begin{enumerate}[label=\Roman*.]
  \item The set $\decclass(\max \setStyle{S}, \max \setStyle{S})
    \cap \critclass{\setStyle{S}^-}(\infty)$ is non-empty.
  \item There is an element $i \in \setStyle{S}^-$
    such that the set
    $\decclass(\max \setStyle{S}, i) \cap
    \critclass{\setStyle{S}^-}(i)$ is non-empty.
  \item There is an $i \in \setStyle{S}^-$ and a
    permutation $\pi$ in the set
    $\hitclass(\max \setStyle{S}, i) \cap
    \critclass{\setStyle{S}^-}(i)$ such that
    $\critelem{\setStyle{S}} (\pi) \in \pi(\setStyle{S})$.\qed
  \end{enumerate}
\end{lemma}
The set-based method has the advantage that the checks that
potentially lead to an immediate answer can be done prior to the
complicated cases, whereas in the iteration method it depends on the
order of the permutations when the complicated cases have to be
handled.  And whether or not an order of the permutations is
advantageous heavily depends on the subset at hand. The disadvantage
of the set-based method is that for large~$n$ especially the
decreasing-element classification can grow large.  Intermediate
preprocessing structures are possible trading speed for
size. Corresponding details, however, are not discussed any further in
this work, since this is rather a topic of software-engineering.

For the applications presented in this paper the following
observations can be made: For enumerating triangulations (one
prominent example $\simplexproduct{6}{2}$ has
$\abs{\GrG} = 30{,}240$ and $n = 35{,}721$), the iteration method is
mostly faster, whereas for the large cases in the enumeration of
circuits and cocircuits (our largest example, the cocircuits of the
$9$-cube~$\hypercube{9}$, has $\abs{\GrG} = 185{,}794{,}560$ and
$n = 512$) the set-based method is significantly quicker.

However, for the latter application with a large group order compared
to the degree the second new method based on switch tables is even
faster than the set-based critical-element method.  This method is
explained in the following.  It works by combining switch tables with
the recursive algorithm
in~\cite{PechReichard_EnumeratingSetOrbits_2009}.  The recursive
algorithm in \cite{PechReichard_EnumeratingSetOrbits_2009} answers a
slightly more general question: for two given subsets $\Simg$
and~$\Sorg$ with the same number of elements, does there exist a
permutation $\grg \in \GrG$ with $\grg(\Simg) \lexsmaller \Sorg$?  The
answer is certainly ``no'' if the subsets or their complements are
empty.  The answer is ``yes'' if the some element of $\Simg$ can be
mapped to something smaller than the minimal element $\sMin$
of~$\Sorg$.  The answer is ``no'' if all elements of $\Simg$ are
mapped by $\GrG$ to something strictly larger than $\sMin$.  And if
some $\grg \in \GrG$ maps some element $k \in \Simg$ exactly
to~$\sMin$, the answer is given by answering the same question
recursively for $\grg(\Simg) \setminus \{\sMin\}$,
$\Sorg \setminus \{\sMin\}$, and the group~$\pointStab{\GrG}{\sMin}$.
Switch tables allow to run this recursive algorithm without the
administration of new stabilizer groups.  The new adaption is based on
the following.

\begin{lemma}
  \label{thm:char-lex-min-st}
  Consider a switch table $\switchPerm{\cdot}{\cdot}$ for a subgroup
  $\GrG$ of~$\symGroup{n}$. Let $1 \le i \le n$, and let
  $\pointStab{\GrG}{i-1}$ be the point-wise stabilizer of
  $\indexSet{i-1}$ in $\GrG$. Moreover, let $\Sorg$ and $\Simg$ be
  subsets of~$\indexSet{n} \setminus \indexSet{i - 1}$ with identical
  cardinality.  Then:
  \begin{enumerate}[leftmargin=*, label=(\roman*)]
  \item\label{itm:char-lex-min-st:empty} If $\Sorg = \emptyset$ or
    $\Sorg = \indexSet{n} \setminus \indexSet{i - 1}$, then there is
    no permutation $\grg \in \pointStab{\GrG}{i-1}$ with
    $\grg(\Simg) \lexsmaller \Sorg$.
  \item\label{itm:char-lex-min-st:trivial-group} If
    $\GrG$ is the trivial group or $i > \maxElem{\effRowSet}$, then
    there is a permutation $\grg \in \pointStab{\GrG}{i-1}$ with
    $\grg(\Simg) \lexsmaller \Sorg$ if and only if $\Simg \lexsmaller
    \Sorg$. 
  \item\label{itm:char-lex-min-st:member-row} If $i \in \Sorg$, then:
    \begin{enumerate}[label=(\alph*)]
    \item\label{itm:char-lex-min-st:member-row:id-recursive} If $i \in \Simg$
      and there is a permutation $\grg' \in \pointStab{\GrG}{i}$
      with
      $\grg' \bigl(\Simg \setminus \{ i \})\bigr)
      \lexsmaller \Sorg \setminus \{ i \}$, or
    \item\label{itm:char-lex-min-st:member-row:recursive} if there is
      a non-trivial switch $\switchPerm{i}{j_i}$ with $j_i \in \Simg$
      and a permutation $\grg' \in \pointStab{\GrG}{i}$ with
      $\grg' \bigl(\switchPerm{i}{j_i}(\Simg \setminus \{j_i\})\bigr)
      \lexsmaller \Sorg \setminus \{ i \}$,
    \end{enumerate}
    then there is a permutation $\grg \in \pointStab{\GrG}{i-1}$ with
    $\grg(\Simg) \lexsmaller \Sorg$.  If none of these two cases occurs,
    then there is no permutation $\grg \in \pointStab{\GrG}{i-1}$ with
    $\grg(\Simg) \lexsmaller \Sorg$.
  \item\label{itm:char-lex-min-st:non-member-row} If $i \notin \Sorg$, then:
    \begin{enumerate}[label=(\alph*)]
    \item\label{itm:char-lex-min-st:non-member-row:id} If $i \in \Simg$,
      or
    \item\label{itm:char-lex-min-st:non-member-row:one-switch}
      $\effColSet(i) \cap \Simg \neq \emptyset$, or
    \item\label{itm:char-lex-min-st:non-member-row:id-recursive} there
      is a permutation $\grg' \in \pointStab{\GrG}{i}$ with
      $\grg' \bigl(\Simg\bigr) \lexsmaller \Sorg$, or
    \item\label{itm:char-lex-min-st:non-member-row:recursive} there is
      a switch $\switchPerm{i}{j_i}$ with
      $j_i \notin \Simg$ and a permutation
      $\grg' \in \pointStab{\GrG}{i}$ with
      $\grg' \bigl(\switchPerm{i}{j_i}(\Simg)\bigr) \lexsmaller
      \Sorg$,
    \end{enumerate}
    then there is a permutation $\grg \in \pointStab{\GrG}{i-1}$ with
    $\grg(\Simg) \lexsmaller \Sorg$.  If none of these four cases
    occurs, then there is no permutation
    $\grg \in \pointStab{\GrG}{i-1}$ with
    $\grg(\Simg) \lexsmaller \Sorg$.
  \end{enumerate}

  In particular, for $i = 1$ and $\Simg = \Sorg$ these conditions
  characterize whether there is a permutation $\grg \in \GrG$ with
  $\grg(\Sorg) \lexsmaller \Sorg$.
\end{lemma}

\begin{proof}
  There is a permutation $\grg \in \pointStab{\GrG}{i-1}$ with
  $\grg(\Simg) \lexsmaller \Sorg$ if and only if there is a switch
  product $\switchPerm{n}{j_n} \cdots \switchPerm{i}{j_i}$ with
  $\bigl(\switchPerm{n}{j_n} \cdots \switchPerm{i}{j_i}\bigr)(\Simg)
  \lexsmaller \Sorg$,
  by \cite{JordanJoswigKastner_Parallelenumerationtriangulations_2018}.

  Case~\ref{itm:char-lex-min-st:empty}
  and~\ref{itm:char-lex-min-st:trivial-group} are straight-forward.

  Case~\ref{itm:char-lex-min-st:member-row}\ref{itm:char-lex-min-st:member-row:id-recursive}
  is formally required to account for the identity switch in row~$i$;
  the argument is then the same as for the next case.  In
  case~\ref{itm:char-lex-min-st:member-row}\ref{itm:char-lex-min-st:member-row:recursive},
  consider the case when there is a non-trivial switch
  $\switchPerm{i}{j_i}$ with $j_i \in \Simg$ and a permutation
  $\grg' \in \pointStab{\GrG}{i}$ with
  $\grg' \bigl(\switchPerm{i}{j_i}(\Simg \setminus \{j_i\})\bigr)
  \lexsmaller \Sorg \setminus \{ i \}$. Consider
  $\grg := \grg' \cdot \switchPerm{i}{j_i}$.  Then,
  $\grg(j_i) = i = \minElem{\Sorg}$, since
  $\grg' \in \pointStab{\GrG}{i}$ stabilizes~$i$.  Because
  $j_i \in \Simg$, it follows that $\grg(\Simg) \lexsmaller \Sorg$ if
  and only if
  $\grg'(\Simg \setminus \{j_i\}) \lexsmaller \Sorg \setminus \{i\}$,
  which is the case by the choice of~$\grg'$.  If there is no
  non-trivial switch $\switchPerm{i}{j_i}$ with $j_i \in \Simg$, then
  all switch products $\switchPerm{n}{j_n} \cdots \switchPerm{i}{j_i}$
  map all elements of~$\Simg$ to elements strictly larger than
  $i = \minElem{\Sorg}$, and $\Sorg$ is therefore strictly lex-smaller
  than any subset in the $\pointStab{\GrG}{i-1}$-orbit of~$\Simg$.
  If, moreover, for all switches $\switchPerm{i}{j_i}$ with
  $j_i \in \Simg$ the $\pointStab{\GrG}{i}$-orbit of
  $\switchPerm{i}{j_i}(\Simg \setminus \{j_i\})$ contains no
  lex-smaller element than $\Sorg \setminus \{i\}$, then no switch
  product $\switchPerm{n}{j_n} \cdots \switchPerm{i}{j_i}$ with
  $j_i \in \Simg$ can lex-decrease~$\Simg$ below~$\Sorg$.

  In
  case~\ref{itm:char-lex-min-st:non-member-row}\ref{itm:char-lex-min-st:non-member-row:id}
  the minimal element of $\Sorg$ is strictly larger than the minimal
  element~$i$ of~$\Simg$, thus $\Simg \lexsmaller \Sorg$ so that
  $\grg = \id$ will do the job.  In
  case~\ref{itm:char-lex-min-st:non-member-row}\ref{itm:char-lex-min-st:non-member-row:one-switch}
  there is a switch $\switchPerm{i}{j_i}$ mapping an element
  of~$\Simg$ to $i < \minElem{\Sorg}$, so
  $\grg = \switchPerm{i}{j_i} = \id \cdots \id \cdot
  \switchPerm{i}{j_i}$ maps $\Simg$ to a lex-smaller subset
  than~$\Sorg$.  The
  cases~\ref{itm:char-lex-min-st:non-member-row}\ref{itm:char-lex-min-st:non-member-row:id-recursive}
  and \ref{itm:char-lex-min-st:non-member-row:recursive} are
  essentially analogous to
  case~\ref{itm:char-lex-min-st:member-row}\ref{itm:char-lex-min-st:member-row:id-recursive}
  and \ref{itm:char-lex-min-st:member-row:recursive}.

  The final assertion follows from backward induction on~$i$, rooted
  at the case $i = n$ with the trivial group~$\pointStab{\GrG}{n}$.
\end{proof}
Call the application of Lemma~\ref{thm:char-lex-min-st} the
\emph{modified switch-table method} for checking lexicographic
minimality of a subset in its orbit; see
Algorithm~\ref{alg:isLexMinModifiedSwitches} in
Section~\ref{sec:algorithm} for a possible implementation.

\section{Algorithms}
\label{sec:algorithm}

In this section, the following are introduced and analyzed:
\begin{itemize}
\item Variants of $\symLSRS$ (Algorithm~\ref{alg:symLSRS} in
  Section~\ref{sec:preliminaries}) supporting global and local
  auxiliary data, counting only maximal elements, minimal
  non-elements, and feasible symmetric subsets, resp., and utilizing
  \emph{semi-deciding} algorithms to prune the enumeration, i.e.,
  roughly speaking, algorithms that correctly answer ``\true'' (i.e.,
  ``prune'') only if a subset is not a subset of an interesting set
  and ``\false'' (i.e., ``continue'') if a subset might or might not
  be a subset of an interesting set.
\item Three new ways to implement $\isLexMin$ in those
  variants. Recall that this subroutine checks whether or not a subset
  is lex-min in its orbit.
\end{itemize}

\subsection{Variants of Symmetric Lexicographic Subset Reverse Search}
\label{sec:vari-symm-lexic}

First, the algorithm $\symLSRSwithData$ is presented, which is a
slightly modified form of~$\symLSRS$.  The reason is that for most
non-trivial problems auxiliary data has to be computed and to be kept
in memory.  The are two principal ways auxiliary data can be made
available.  Data that depends on the current subset in the
reverse-search tree is stored in a node of the reverse-search tree
together with the subset.  Data independent of the subset can be
stored globally, e.g., together with the problem data.

\begin{algorithm}[t]
  \TitleOfAlgo{\symLSRSwithData{n, \downsetStyle{D}, \GrG,
      \globalData, \RSnode}}

  \KwIn{$n \in \mathbb{N}$, a downset $\downsetStyle{D}$ of
    $\powerSet{\indexSet{n}}$, a subgroup $\GrG$ of the automorphism
    group of $\downsetStyle{D}$, global data $\globalData$, a node
    $\RSnode = (\setStyle{S}, \localData)$ with canonical
    $\setStyle{S} \in \downsetStyle{D}$ and $\localData$ local data}

  \KwOut{the number of canonical $\setStyle{S}'$ in~$\downsetStyle{D}$
    with $\setStyle{S} \lexsubset \lexmin \setStyle{S}'$}
  
  \output $\setStyle{S}$ and $c \gets 1$
  \tcc*{count~$\setStyle{S}$}
  \For(\tcc*[f]{ordered traversal of new maximal elements}){
    $i = \maxElem{\setStyle{S}} + 1, \dots, n$
  }{
    $\setStyle{S}' \gets \setStyle{S} \cup \{ i \}$
    \tcc*{add a new element on the right}
    default initialize $\localData'$
    \tcc*{prepare a local data structure}
    $(\mathrm{answer}, \localData') \gets \isLexMin{$\setStyle{S}', \GrG, \globalData, \localData, \localData'$}$
    \tcc*{lex-min check}
    \If(\tcc*[f]{if new set is not lex-min in its orbit}){
      $\mathrm{answer} = \false$}{
      continue
      \tcc*{next loop element}
    }
    $(\mathrm{answer}, \localData') \gets \isInD{$\setStyle{S}', \downsetStyle{D}, \globalData, \localData, \localData'$}$
    \tcc*{membership check}
    \If(\tcc*[f]{if new set is not in $\downsetStyle{D}$}){
      $\mathrm{answer} = \false$}{
      continue
      \tcc*{next loop element}
    }      
    $\nextRSnode \gets (\setStyle{S}', \localData')$
    \tcc*{build new node}
    $c \gets c + \symLSRSwithData{n, \downsetStyle{D}, \GrG,
      \globalData, \nextRSnode}$
    \tcc*{recurse}
  }
  \Return c\;
  
  \caption[Symmetric Lexicographic Reverse Search for Orbits of
  Subsets with Data]{A variant of symmetric lexicographic subset
    reverse search with explicit use of global and local auxiliary
    data}
  \label{alg:symLSRSwithData}  
\end{algorithm}

To keep track of this in the following, specify by
$\RSnode = (\setStyle{S}_{\RSnode}, \localData_{\RSnode})$ a node in the
reverse-search tree consisting of a subset
$\setStyle{S}_{\RSnode} \in \powerSet{\indexSet{n}}$ and a subset-specific
collection of data~$\localData_{\RSnode}$, which formally is just a tupel
of mathematical objects.  The global collection of data is denoted
by~$\globalData$.
Any implementation of $\symLSRSwithData$ must specify the exact
structure of $\localData$ and~$\globalData$.  It is desirable that the
local data for a new node can be computed during one or both of the
rather expensive subroutines $\isLexMin$ and $\isInD$. A possible such
layout can be seen in Algorithm~\ref{alg:symLSRSwithData}.  The
run-time complexity depends on the run-time complexities of the
subroutines and can be derived readily from the underlying reverse
search structure -- it is in
$O\bigl(n(\timeComplexity(\isLexMin) + \timeComplexity(\isInD))
\abs{\setOrbits{\downsetStyle{D}}{\GrG}}\bigr)$, assuming initializing
local data and building a new node are dominated by the other
subroutines.

At times, not the cardinality of the downset $\downsetStyle{D}$ up to
symmetry is of interest but only the number of its maximal elements.
This is, e.g., the case for the enumeration of cocircuits up to
symmetry.  In such situations, the downset $\downsetStyle{D}$ is only
implicitly defined as the downset of all subsets of the subsets one is
really interested in.  This models the process building interesting
objects from scratch element-by-element.  In the ordered
reverse-search tree of subsets considered in the algorithms so far,
the added elements are always larger than the current maximal element.
Therefore, the following notions are defined:
\begin{definition}
  \label{def:expansion}
  For a subset $\setStyle{S} \in \downsetStyle{D}$ an \emph{expansion
    of $\setStyle{S}$} is an element
  $i \in \indexSet{n} \setminus \setStyle{S}$ so that
  $\setStyle{S} \cup \{i\} \in \downsetStyle{D}$.  A
  \emph{right-expansion} is an expansion $i$ with
  $i > \max \setStyle{S}$.  The subset $\setStyle{S}$ is
  \emph{maximal} if there is no expansion for it, and it is
  \emph{right-maximal} if there is no right-expansion for it.  The
  subset~$\setStyle{S}$ is \emph{right-completable} if there is a set
  $\setStyle{S}'$ that is maximal in~$\downsetStyle{D}$ so that
  $\setStyle{S}'$ lex-contains~$\setStyle{S}$.
\end{definition}

By definition, each right-expansion of a set lex-contains that set.
Algorithm~\ref{alg:symLSRSMaxwithData} shows the adapted algorithm
enumerating the maximal elements in a downset. It might appear that
the leaves of the enumeration tree automatically correspond to maximal
subsets in~$\downsetStyle{D}$.  This, however, is unfortunately not
the case -- the leaves are, by construction, only right-maximal:
Consider an arbitrary maximal non-empty subset $\setStyle{S}$
in~$\downsetStyle{D}$.  Then, e.g., the enumeration branch starting
with the second smallest element in~$\setStyle{S}$ will have
$\setStyle{S} \setminus \min \setStyle{S}$ as one of its leaves, which
is obviously not maximal in~$\setStyle{D}$. Thus, a maximality check
has to be performed on the leaves, and non-maximal leaves are simply
ignored for the count.  Sometimes the reverse-search tree can be
pruned by semi-deciding whether or not a subset can be
right-completed.  The subroutine $\semiIsNotRightComp$ in
Algorithm~\ref{alg:symLSRSMaxwithData} takes care of that.  Whenever
it can be detected locally that an expansion is unavoidable on the
path to a maximal subset, one can skip the traversal of supersets not
containing it.  This idea is supported by the subroutine
$\isInEachMax$.  The run-time complexity is in
$O\bigl((n(\timeComplexity(\isLexMin) + \timeComplexity(\isInD) +
\timeComplexity(\semiIsNotRightComp) + \timeComplexity(\isInEachMax))
+ \timeComplexity(\isMaxInD))
\abs{\setOrbits{\nonPrunables}{\GrG}}\bigr)$, where $\nonPrunables$ is
the subset of~$\downsetStyle{D}$ containing subsets for which
$\semiIsNotRightComp$ returns ``\false''.

\begin{algorithm}[t]
  \TitleOfAlgo{\symLSRSMaxwithData{n, \downsetStyle{D}, \GrG,
      \globalData, \RSnode}}

  \KwIn{$n \in \mathbb{N}$, a downset $\downsetStyle{D}$ of
    $\powerSet{\indexSet{n}}$, a subgroup $\GrG$ of the automorphism
    group of $\downsetStyle{D}$, global data $\globalData$, a node
    $\RSnode = (\setStyle{S}, \localData)$ with a canonical
    $\setStyle{S} \in \downsetStyle{D}$ and
    $\localData$ local data}

  \KwOut{the number of canonical $\setStyle{S}'$ maximal
  in~$\downsetStyle{D}$ with $\setStyle{S} \lexsubset \setStyle{S}'$}
  
  $c \gets 0$
  \tcc*{do not count~$\setStyle{S}$ yet}
  \For(\tcc*[f]{ordered traversal of new maximal elements}){
    $i = \maxElem{\setStyle{S}} + 1, \dots, n$
  }{
    $\setStyle{S}' \gets \setStyle{S} \cup \{ i \}$
    \tcc*{add a new element on the right}
    default initialize $\localData'$
    \tcc*{prepare a local data structure}
    $(\mathrm{answer}, \localData') \gets \isLexMin{$\setStyle{S}', \GrG, \globalData, \localData, \localData'$}$
    \tcc*{lex-min check}
    \If(\tcc*[f]{if new set is not lex-min in its orbit}){
      $\mathrm{answer} = \false$}{
      continue
      \tcc*{next loop element}
    }
    $(\mathrm{answer}, \localData') \gets \isInD{$\setStyle{S}', \downsetStyle{D}, \globalData, \localData, \localData'$}$
    \tcc*{membership check}
    \If(\tcc*[f]{if new set is not in $\downsetStyle{D}$}){
      $\mathrm{answer} = \false$}{
      continue
      \tcc*{next loop element}
    }      
    $(\mathrm{answer}, \localData') \gets \semiIsNotRightComp{$\setStyle{S}', \downsetStyle{D}, \globalData, \localData, \localData'$}$
    \tcc*{completability}
    \If(\tcc*[f]{if new set is not right-completable}){
      $\mathrm{answer} = \true$}{
      continue
      \tcc*{next loop element}
    }      
    $\nextRSnode \gets (\setStyle{S}', \localData')$
    \tcc*{build new node}
    $c \gets c + \symLSRSMaxwithData{n, \downsetStyle{D}, \GrG,
      \globalData, \nextRSnode}$
    \tcc*{recurse}
    $(\mathrm{answer}, \localData') \gets \isInEachMax{$\setStyle{S}', \downsetStyle{D}, \globalData, \localData, \localData'$}$
    \tcc*{check unavoidability}
    \If(\tcc*[f]{if new set is in each right-completion}){
      $\mathrm{answer} = \true$}{
      break
      \tcc*{exit the loop}
    }    
  }
  \If(\tcc*[f]{if supersets in~$\downsetStyle{D}$ found}){
    $c > 0$
  }{
    \Return $c$
    \tcc*{return no.~of max.~supersets of~$\setStyle{S}$ in~$\downsetStyle{D}$}
  }
  \If(\tcc*[f]{if $\setStyle{S}$ is maximal in~$\downsetStyle{D}$}){
    $\isMaxInD{\setStyle{S}, \globalData, \localData}$
  }{
    \output $\setStyle{S}$ and \Return $1$
    \tcc*{return the count for~$\setStyle{S}$}
  }
  \caption[Symmetric Lexicographic Reverse Search for Orbits of
  Maximal Subsets]{A variant of symmetric lexicographic
    subset reverse search for maximal subsets with explicit use of
    global and local auxiliary data}
  \label{alg:symLSRSMaxwithData}  
\end{algorithm}

Other applications are rather interested in objects that can be better
represented as the minimal subsets \emph{not} contained in a
downset. The enumeration of all circuits up to symmetry is an example.
\begin{definition}
  For a subset
  $\setStyle{R} \in \powerSet{\indexSet{n}} \setminus
  \downsetStyle{D}$ a \emph{reduction of~$\setStyle{R}$} is an element
  $i \in \setStyle{R}$ so that
  $\setStyle{R} \setminus \{i\} \in \powerSet{\indexSet{n}} \setminus
  \downsetStyle{D}$. The subset $\setStyle{R}$ is \emph{co-minimal
    w.r.t.~$\downsetStyle{D}$} if it contains no reduction.  It is
  \emph{right-co-minimal} if its maximal element is not a reduction.
  A subset $\setStyle{S}$ in~$\downsetStyle{D}$ is \emph{right-exitable
    w.r.t.~$\downsetStyle{D}$} if there is a set~$\setStyle{R}$ that
  is co-minimal w.r.t.~$\downsetStyle{D}$ so that $\setStyle{R}$
  lex-contains~$\setStyle{S}$.
\end{definition}
The idea for a reverse-search algorithm is to adapt
$\symLSRSMaxwithData$: add elements until the first non-member is met
and check afterwards if the resulting set is a minimal
non-member. Algorithm~\ref{alg:symLSRSCominwithData} shows a possible
layout for this.  The subroutine $\semiIsNotRightExit$ can help to
prune the tree.  For the enumeration of circuits as in
Section~\ref{sec:application-circuits} there is no sensible such
option known, but for other applications there may be one. The run-time
complexity is in
$O\bigl(n(\timeComplexity(\isLexMin) + \timeComplexity(\isInD) +
\timeComplexity(\semiIsNotRightExit) + \timeComplexity(\isCominOfD))
\abs{\setOrbits{\nonPrunables}{\GrG}}\bigr)$, where $\nonPrunables$ is
the subset of~$\downsetStyle{D}$ containing subsets for which
$\semiIsNotRightExit$ returns ``\false''.

\begin{algorithm}[t]
  \TitleOfAlgo{\symLSRSCominwithData{n, \downsetStyle{D}, \GrG,
      \globalData, \RSnode}}

  \KwIn{$n \in \mathbb{N}$, a downset $\downsetStyle{D}$ of
    $\powerSet{\indexSet{n}}$, a subgroup $\GrG$ of the automorphism
    group of $\downsetStyle{D}$, global data $\globalData$, a node
    $\RSnode = (\setStyle{S}, \localData)$ with canonical
    $\setStyle{S} \in \downsetStyle{D}$ and $\localData$ local data}

  \KwOut{the number of canonical $\setStyle{S}'$ cominimal
    w.r.t.~$\downsetStyle{D}$ with
    $\setStyle{S} \lexsubset \setStyle{S}'$}

  $c \gets 0$
  \tcc*{as a member, $\setStyle{S}$ cannot be co-minimal}
  \For(\tcc*[f]{ordered traversal of new maximal elements}){
    $i = \maxElem{\setStyle{S}} + 1, \dots, n$
  }{
    $\setStyle{S}' \gets \setStyle{S} \cup \{ i \}$
    \tcc*{add a new element on the right}
    default initialize $\localData'$
    \tcc*{prepare a local data structure}
    $(\mathrm{answer}, \localData') \gets \isLexMin{$\setStyle{S}', \GrG, \globalData, \localData, \localData'$}$
    \tcc*{lex-min check}
    \If(\tcc*[f]{if new set is not lex-min in its orbit}){
      $\mathrm{answer} = \false$}{
      continue
      \tcc*{next loop element}
    }
    $(\mathrm{answer}, \localData') \gets \isInD{$\setStyle{S}', \downsetStyle{D}, \globalData, \localData, \localData'$}$
    \tcc*{membership check}
    \If(\tcc*[f]{if new set is right-co-minimal}){
      $\mathrm{answer} = \false$}{
      \If(\tcc*[f]{if set is co-minimal}){
        $\isCominOfD{$\setStyle{S}', \globalData, \localData, \localData'$}$
      }{
        \output $\setStyle{S}'$ and $c \gets c + 1$ 
        \tcc*{count~$\setStyle{S}'$}
        continue
        \tcc*{next loop element}
      }
    }      
    $(\mathrm{answer}, \localData') \gets \semiIsNotRightExit{$\setStyle{S}', \downsetStyle{D}, \globalData, \localData, \localData'$}$
    \tcc*{exitability}
    \If(\tcc*[f]{if new set is not right-exitable}){
      $\mathrm{answer} = \true$}{
      continue
      \tcc*{next loop element}
    }      
    $\nextRSnode \gets (\setStyle{S}', \localData')$
    \tcc*{build new node}
    $c \gets c + \symLSRSCominwithData{n, \downsetStyle{D}, \GrG,
      \globalData, \nextRSnode}$
    \tcc*{recurse}
  }
  \Return $c$\;
  \caption[Symmetric Lexicographic Reverse Search for Orbits of
  Co-Minimal Subsets]{A variant of symmetric lexicographic
    subset reverse search for co-minimal subsets with explicit use of
    global and local auxiliary data}
  \label{alg:symLSRSCominwithData}  
\end{algorithm}

Occasionally, the representation of interesting objects as the maximal
or co-minimal elements of a downset leads to a very difficult
membership test.  For example, it is NP-complete to decide whether or
not a set of simplices is a subset of any triangulation. See
Section~\ref{sec:application-triangulations} for more details.  Then,
it can be preferable to travers an auxiliary downset with simpler
membership test that contains all feasible subsets and check for
extendability to a feasible subset separately.

In the following, distinguish for a subset $\setStyle{S}$ of a
feasible subset between an \emph{expansion} (see
Definition~\ref{def:expansion}) of~$\setStyle{S}$ by an element, which
remains in~$\downsetStyle{D}$, and an \emph{extension}
of~$\setStyle{S}$ by an element, which remains a subset of a feasible
subset.  Most interesting is the case where the feasible subsets
form an \emph{antichain} in~$\indexSet{n}$, i.e., no feasible subset
contains any other feasible subset.
\begin{definition}
  \label{def:extension}
  Let $\setsysStyle{F}$ be an antichain in
  $\powerSet{\indexSet{n}}$. Members of~$\setsysStyle{F}$ are called
  \emph{feasible subsets}. Each $\setStyle{S} \in \setsysStyle{F}$ is
  said to \emph{represent a solution}.  \emph{Extendable subsets} are
  subsets of feasible subsets.  A subset $\setStyle{S}$ is
  \emph{right-extendable} if it is extendable to a feasible subset
  $\setStyle{S}'$ so that $\setStyle{S}'$ lex-contains~$\setStyle{S}$.
  In this case, $\setStyle{S}'$ is called a \emph{feasible
    right-completion} of~$\setStyle{S}$.
\end{definition}

\begin{algorithm}[t]
  \TitleOfAlgo{\symLSRSFeaswithData{n, \downsetStyle{D},
      \setsysStyle{F}, \GrG, \globalData, \RSnode}}
  
  \KwIn{$n \in \mathbb{N}$, a downset $\downsetStyle{D}$ of
    $\powerSet{\indexSet{n}}$, a subset
    $\setsysStyle{F} \subseteq \downsetStyle{D}$ of feasible subsets,
    a subgroup $\GrG$ of the automorphism group of $\downsetStyle{D}$,
    global data $\globalData$, a node
    $\RSnode = (\setStyle{S}, \localData)$ with canonical
    $\setStyle{S} \in \downsetStyle{D}$ and $\localData$ local data
    encompassing the right-expansion sequence
    $\setStyle{E}(\setStyle{S})$}
  
  \KwOut{the number of canonical $\setStyle{S}'$ feasible
    in~$\downsetStyle{D}$ with
    $\setStyle{S} \lexsubset \setStyle{S}'$}

  \If(\tcc*[f]{if $\setStyle{S}$ is feasible}){
    $\isFeasible{$\setStyle{S}, \setsysStyle{F}, \globalData, \localData$}$
  }{
    \output $\setStyle{S}$ and \Return $1$
    \tcc*{count $\setStyle{S}$}
  }
  $c \gets 0$
  \tcc*{do not count~$\setStyle{S}$}
  \For(\tcc*[f]{ordered traversal of right-expansions}){
    $j = 1, \dots, \abs{\setStyle{E}(\setStyle{S})}$
  }{
    default initialize $\localData'$
    \tcc*{prepare a local data structure}
    $(\mathrm{break}, \setStyle{S}', \localData') \gets \expand{$\setStyle{S}, \setStyle{E}(\setStyle{S})_j, \globalData, \localData, \localData'$}$
    \tcc*{expand to the right}
    \If(\tcc*[f]{if expansion returns to stop}){
      $\mathrm{break} = \true$}{
      break
      \tcc*{exit loop}
    }    
    $(\mathrm{answer}, \localData') \gets \isLexMin{$\setStyle{S}', \GrG, \globalData, \localData, \localData'$}$
    \tcc*{lex-min check}
    \If(\tcc*[f]{if new set is not lex-min in its orbit}){
      $\mathrm{answer} = \false$}{
      continue
      \tcc*{next loop element}
    }
    \tcc{incomplete right-extendability check:}
    $(\mathrm{answer}, \localData') \gets
    \semiIsNotRightExt{$\setStyle{S}', \downsetStyle{D},
    \setsysStyle{F}, \globalData, \localData, \localData'$}$ \;
    \If(\tcc*[f]{if new set is certainly not right-ext.}){
      $\mathrm{answer} = \true$}{
      continue
      \tcc*{next loop element}
    }      
    $\nextRSnode \gets (\setStyle{S}', \localData')$
    \tcc*{build new node}
    $c \gets c + \symLSRSFeaswithData{n, \downsetStyle{D}, \GrG,
      \globalData, \nextRSnode}$
    \tcc*{recurse}
  }
  \Return $c$ \;
  \caption[Symmetric Lexicographic Reverse Search for Orbits of
  Feasible Subsets]{A variant of symmetric lexicographic
    subset reverse search for feasible subsets with pruning by an
    incomplete extendability check}
  \label{alg:symLSRSFeaswithPruning}  
\end{algorithm}

If only a small fraction of the non-elements in a subset lead to an
element of the downset, then the iteration over all new maximal
elements followed by membership tests can lead to an unnecessarily
large number of loop traversals.  In the enumeration of triangulations
this case occurs, since, in general, out of the many simplices not in
a partial triangulation only a few can be added.  In such cases it may
be possible to compute all possible right-expansions of a subset
before the loop and iterate only over those. The following notion
supports this idea.
\begin{definition}
  \label{def:expansion-sequence}
  The \emph{right-expansion sequence} $\setStyle{E}(\setStyle{S})$
  of~$\setStyle{S}$ is the sequence
  $i_1 < \dots < i_{\abs{\setStyle{E}(\setStyle{S})}}$ of all $i_k$,
  $k = 1, \dots, \abs{\setStyle{E}(\setStyle{S})}$, with
  $i_1 > \max S$ and $\setStyle{S} \cup \{i_k\} \in \downsetStyle{D}$.
\end{definition}

For the enumeration of feasible subsets with expensive exact
extendability check it is desirable to prune the enumeration tree as
soon as one knows that a subset is not right-extendable. This paper
suggests an incomplete check that returns $\true$ if a subset is not
right-extendable and $\false$ if the subset might be right-extendable.
Algorithm~\ref{alg:symLSRSFeaswithPruning} shows a possible
pseudo-code for this method that will be used in the enumeration of
all triangulations of a point configuration in
Section~\ref{sec:application-triangulations}.  If $\maxexpcard$ is the
largest cardinality of any expansion sequence, then the run-time
complexity is in
$O\bigl((\maxexpcard(\timeComplexity(\isLexMin) +
\timeComplexity(\expand) + \timeComplexity{\semiIsNotRightExt}) +
\timeComplexity(\isFeasible))
\abs{\setOrbits{\nonPrunables}{\GrG}}\bigr)$, where $\nonPrunables$ is
the subset of~$\downsetStyle{D}$ containing subsets for which
$\semiIsNotRightExt$ returns ``\false''.

If membership in the downset~$\downsetStyle{D}$ is governed by a
\emph{compatibility graph} $\GG = (\indexSet{n}, \EE)$ with
$\GrG(\GG) = \GG$, then more can be done as follows.
\begin{definition}
  Let $\GG = (\indexSet{n}, \EE)$ be a simple graph.  The
  \emph{downset $\downsetStyle{D}$ induced by the compatibility
    graph~$\GG$} is defined as the downset of all cliques in~$\GG$.
  That is:
  \begin{equation}
    \downsetStyle{D}
    :=
    \bigl\{
    \setStyle{S} \subseteq \indexSet{n}
    :
    \text{$\{i, j\} \in \EE$ for all $i,j \in \setStyle{S}$}
    \bigr\}.
  \end{equation}
\end{definition}
For a downset induced by a compatibility graph, the expansion sequence
is easy to compute from~$\GG$.
\begin{observation}
  An expansion sequence for a downset induced by a compatibility
  graph~$\GG$ can be generated as follows, where $\adjSet(i)$ is the
  neighborhood of~$i$ in~$\GG$.
  \begin{equation}
    \label{eq:expansion-via-neighborhood}
    \setStyle{E}(\setStyle{S})
    =
    \bigcap_{i \in \setStyle{S}} \adjSet(i)
    =
    \setStyle{E}(\setStyle{S}\setminus\max\setStyle{S})
    \cap \adjSet(\max \setStyle{S}).
  \end{equation}
\end{observation}
Consequently, if the expansion sequence is maintained in the local
auxiliary data, then the next expansion sequence can be computed by a
single set intersection.  This essentially encodes how cliques in a
graph can be enumerated up to symmetry by extending them node-by-node.
However, more is possible.  In the following, the focus is on feasible
subsets that have prescribed symmetries.
\begin{definition}
  \label{def:symLsymSRS-notions}
  Let $\downsetStyle{D}$ be a downset in $\powerSet{\indexSet{n}}$
  induced by a compatibility graph $\GG = (\indexSet{n}, \EE)$ and
  $\GrG$ be a subgroup of the automorphism group of~$\GG$.  Elements
  $i$ and $j$ with $\{i, j\} \in \EE$ are \emph{compatible}.  For a
  subgroup $\groupStyle{H}$ of~$\symGroup{n}$, a subset $\setStyle{S}$
  in $\downsetStyle{D}$ is \emph{$\groupStyle{H}$-invariant} if for
  each $\grh \in \groupStyle{H}$ one has
  $\grh(\setStyle{S}) = \setStyle{S}$.  An element
  $i \in \indexSet{n}$ is \emph{$\groupStyle{H}$-feasible} if its
  $\groupStyle{H}$-orbit is a clique in $\GG$, i.e., for each
  $\grh \in \groupStyle{H}$, either $\grh(i) = i$ or
  $\{\grh(i), i\} \in \EE$.  Some $\groupStyle{H}$-feasible elements
  $i, j \in \indexSet{n}$ are \emph{$\groupStyle{H}$-compatible} if
  the union of their $\groupStyle{H}$-orbits is a clique in~$\GG$,
  i.e., for each $\grh \in \groupStyle{H}$ and each
  $k, \ell \in \orbitSet{H}{i} \cup \orbitSet{H}{j}$ either
  $\grh(k) = \ell$ or $\{\grh(k), \ell\} \in \EE$.  The
  \emph{$\groupStyle{H}$-compatibility graph of $\GG$} is the graph
  $\GG_{\groupStyle{H}} = \bigl(\VV_{\groupStyle{H}},
  \EE_{\groupStyle{H}} \bigr)$ with
  $\VV_{\groupStyle{H}} := \{ i \in \indexSet{n} : \text{$i$ is
    $\groupStyle{H}$-feasible} \}$ and
  $\EE_{\groupStyle{H}} := \bigl\{ \{i, j\} \in \EE : \text{$i$ and
    $j$ are $\groupStyle{H}$-compatible} \bigr\}$.  A permutation
  $\grg \in \GrG$ is \emph{$\groupStyle{H}$-feasible}, if it is in the
  set-wise stabilizer of the $\groupStyle{H}$-feasible elements, i.e.,
  if for all $i \in \indexSet{n}$ one has that $\grg(i)$ is
  $\groupStyle{H}$-feasible if and only if $i$ is
  $\groupStyle{H}$-feasible.
\end{definition}
Assume next that each feasible $\setStyle{S} \in \downsetStyle{D}$ is
a maximal clique in~$\GG$ (not necessarily vice versa).  Then,
feasible subsets with prescribed symmetries can be enumerated as
follows:
\begin{theorem}
  \label{thm:symLsymSRS-main}
  Let $\downsetStyle{D}$ be a downset in $\powerSet{\indexSet{n}}$
  induced by a compatibility graph~$\GG = (\indexSet{n}, \EE)$ and
  $\GrG$ be a subgroup of the automorphism group of~$\GG$.  Assume
  that each feasible $\setStyle{S} \in \downsetStyle{D}$ is a maximal
  clique in~$\GG$.  Moreover, let $\groupStyle{H}$ be a subgroup
  of~$\symGroup{n}$, and let $\GG_{\groupStyle{H}}$ be the
  $\groupStyle{H}$-compatibility graph of~$\GG$.  Let $\setStyle{S}$
  be a feasible subset of pairwise $\groupStyle{H}$-compatible
  elements.  Then $\setStyle{S}$ is an $\groupStyle{H}$-invariant
  feasible subset.
  
  In particular, $\symLSRSFeaswithData$ applied to the downset
  $\downsetStyle{D}_{\groupStyle{H}}$ induced by the
  $\groupStyle{H}$-compatibility graph with unaltered feasibility
  check enumerates all $\groupStyle{H}$-invariant feasible subsets
  of~$\downsetStyle{D}$ up to $\groupStyle{H}$-feasible symmetries.
\end{theorem}
\begin{proof}
  By the assumptions, $\setStyle{S}$ is feasible.  It remains to show
  that $\setStyle{S}$ is $\groupStyle{H}$-invariant.  Let $\grh$ be a
  permutation in~$\groupStyle{H}$ and $i$ an element
  in~$\setStyle{S}$.  Since $i$ is $\groupStyle{H}$-feasible,
  $\grh(i)$ is identical to or compatible with~$i$.  Since all
  elements of~$\setStyle{S}$ are pairwise $\groupStyle{H}$-compatible,
  $\grh(i)$ is identical to or compatible with all elements
  in~$\setStyle{S}$.  Since $\setStyle{S}$ is feasible, it is a
  maximal clique in the compatibility graph.  Thus,
  $\grh(i) \in \setStyle{S}$.  Since this holds for all
  $i \in \setStyle{S}$, one has
  $\grh(\setStyle{S}) \subseteq \setStyle{S}$.  This implies
  $\grh(\setStyle{S}) = \setStyle{S}$ because each permutation $\grh$
  is a bijection.  Since $\grh$ was chosen as an arbitrary
  permutation in~$\groupStyle{H}$, the subset $\setStyle{S}$ is
  $\groupStyle{H}$-invariant.
\end{proof}
The resulting method is finally called \emph{symmetric lexicographic
  symmetric-subset reverse search}.  This will allow for the
enumeration of symmetric triangulations in
Section~\ref{sec:triangs:enhancements}.

\subsection{New Checks for the Lexicographic Minimality of Subsets}
\label{sec:new-checks-lexic}

Next, the findings from Section~\ref{sec:theory} are used to specify
suitable local data to accelerate the lex-min check in orbits. In
particular, the critical-element-table $\critelem{\setStyle{S}}$ of a
subset will be used as local data stored together with $\setStyle{S}$
in a node.  It can be seen that one can update the critical-element
table for the next node during its usage in the lex-min check.

First, the iterative method $\isLexMinIter{$\setStyle{S'}, \GrG,
  \globalData, \localData, \localData'$}$ is described, where
$\localData$ contains $\critelemTable$, which is a function-value
table of~$\critelem{\setStyle{S}}$.  It neither needs global data
$\globalData$ nor the initialized local data $\localData'$
for~$\setStyle{S}'$. In case the answer is ``\true'', the complete
updated critical-element table of~$\setStyle{S}'$ is stored
in~$\critelemTable'$ representing a function-value table of
$\critelem{\setStyle{S}'}$, which will be part of~$\localData'$.  In
case the answer is ``\false'', $\critelemTable'$ will only partially
be updated.  However, this is irrelevant, since this answer leads to
immediate pruning.  The function iterates over the whole symmetry
group but is quite lean inside the loop.
Algorithm~\ref{alg:isLexMinCritelemIter} shows a possible listing in
pseudo-code.  Its run-time complexity amortized over all possible
instances is in $O(\abs{\GrG})$ under the assumption that on average
for only a $\frac{1}{n}$-fraction of the permutations the added new
maximal element of the new subset is mapped exactly to the critical
element of the previous subset.

\begin{algorithm}[t]
  \TitleOfAlgo{\isLexMinIter{\setStyle{S'}, \GrG, \critelemTable}}

  \KwIn{a set $\setStyle{S}'$, a subgroup $\GrG$ of~$\symGroup{n}$,
    and the critical-element table
    $\critelemTable = \critelem{\setStyle{S}}$
    of~$\setStyle{S} = \setStyle{S}' \setminus \max
      \setStyle{S}'$}
  
  \KwOut{$(\true, \critelemTable' = \critelem{\setStyle{S}'})$ if
    $\setStyle{S}' = \lexmin \GrG(\setStyle{S}')$ and $(\false, -)$
    otherwise}
  
  $\critelemTable' \gets \critelemTable$
  \tcc*{copy the old critical-element table}
  \For(\tcc*[f]{iterate over $\GrG$}){
    $\grg \in \GrG$
  }{
    $j \gets \critelemTable[\grg]$
    \tcc*{retrieve critical element}
    \If(\tcc*[f]{if $\grg$ stabilizes~$\setStyle{S}$}){
      $j = \infty$
    }{
      \If(\tcc*[f]{if $\grg$ decreases $\max \setStyle{S}'$}){
        $\grg(\max \setStyle{S}') < \max \setStyle{S}'$
      }{
        \Return $(\false, -)$
        \tcc*{$\grg$ is lex-decreasing}
      }
      \ElseIf(\tcc*[f]{if $\grg$ increases $\max \setStyle{S}'$}){
        $\grg(\max \setStyle{S}') > \max \setStyle{S}'$
      }{
        $\critelemTable'[\grg] \gets \max \setStyle{S}'$
        \tcc*{update critical element}
      }
    }
    \Else(\tcc*[f]{$\grg$ strictly lex-increases $\setStyle{S}$}){
      \If(\tcc*[f]{if $\grg$ decreases $\max \setStyle{S}'$ beyond
        $j$}){
        $\grg(\max \setStyle{S}') < j$
      }{
        \Return $(\false, -)$
        \tcc*{$\grg$ is lex-decreasing}
      }
      \ElseIf(\tcc*[f]{if $\grg$ maps $\max \setStyle{S}'$ to
        crit.~element}){
        $\grg(\max \setStyle{S}') = j$
      }{
        $j' \gets \min \bigl(\setStyle{S}' \symdiff \grg(\setStyle{S}')\bigr)$
        \tcc*{compute critical element of~$\setStyle{S}'$}
        \If(\tcc*[f]{if new critical element is in~$\grg(\setStyle{S}')$}){
          $j' \in \grg(\setStyle{S}')$
        }{
          \Return $(\false, -)$
          \tcc*{$\grg$ is lex-decreasing}
        }
        \Else(\tcc*[f]{if new critical element is not in~$\grg(\setStyle{S}')$}){
          $\critelemTable'[\grg] \gets j'$
          \tcc*{update critical-element table}
        }
      }
    }
  }
  \Return $(\true, \critelemTable')$\;
  
  \caption[Iteration-Based Critical-Element Method]{The
    iteration-based critical-element method}
  \label{alg:isLexMinCritelemIter}
\end{algorithm}

Next, the set-based variant $\isLexMinSets{$\setStyle{S}', \GrG,
  \globalData, \localData, \localData'$}$ is described, where
$\localData$ contains $\critelemClass$, which is a function-value
table of $\critclass{\setStyle{S}}$.  Global data $\globalData$
containing $(\hitelemClass, \incelemClass, \decelemClass)$ are needed,
which are function-value tables for $\hitclass$, $\incclass$, and
$\decclass$, respectively. It does not need the initialized local data
$\localData'$
for~$\setStyle{S}'$. Algorithm~\ref{alg:isLexMinCritelemSet} shows the
pseudo-code for this.

\begin{algorithm}[t]
  \TitleOfAlgo{\isLexMinSets{$\setStyle{S}', \GrG, (\hitelemClass,
      \incelemClass, \decelemClass), \critelemClass$}}

  \KwIn{a set $\setStyle{S}'$, a subgroup $\GrG$ of~$\symGroup{n}$
    represented by its hit-element, increasing-element, and
    decreasing-element classifications~$\hitelemClass$,
    $\incelemClass$, and $\decelemClass$, the critical-element
    classification $\critelemClass = \critclass{\setStyle{S}}$
    of~$\setStyle{S} = \setStyle{S}' \setminus \max \setStyle{S}'$}
  
  \KwOut{$(\true, \critelemClass' = \critclass{\setStyle{S}'})$ if
    $\setStyle{S}' = \lexmin \GrG(\setStyle{S}')$ and $(\false, -)$
    otherwise}

  \If(\tcc*[f]{check for case~I}){
    $\decelemClass[\max \setStyle{S}'][\max \setStyle{S}']
    \cap \critelemClass[\infty] \neq \emptyset$}{
    \Return $(\false, -)$
    \tcc*{$\setStyle{S}'$ is not lex-min}
  }
  \If(\tcc*[f]{check for case~II}){
    $\exists i \in \setStyle{S}: \decelemClass[\max \setStyle{S}'][i] \cap
    \critelemClass[i] \neq \emptyset$}{
    \Return $(\false, -)$
    \tcc*{$\setStyle{S}'$ is not lex-min}
  }
  $\critelemClass' \gets \critelemClass$
  \tcc*{copy the old critical-element classification}
  \For(\tcc*[f]{iterate over $\setStyle{S}$}){
    $j \in \setStyle{S}$
  }{
    \For(\tcc*[f]{check $j$ for case~III}){
      $\grg \in \critelemClass[j] \cap \hitelemClass[\max \setStyle{S}'][j]$
    }{
      $j' \gets \min \bigl(\setStyle{S}' \symdiff \grg(\setStyle{S}')\bigr)$
      \tcc*{compute critical element of~$\setStyle{S}'$}
      \If(\tcc*[f]{if new critical element is in~$\grg(\setStyle{S}')$}){
        $j' \in \grg(\setStyle{S}')$
      }{
        \Return $(\false, -)$
        \tcc*{$\grg$ is lex-decreasing}
      }
      \Else(\tcc*[f]{if new critical element is not in~$\grg(\setStyle{S}')$}){
        $\critelemClass'[j] \gets \critelemClass'[j] \setminus \{\grg\}$
        \tcc*{update classification}
        $\critelemClass'[j'] \gets \critelemClass'[j'] \cup \{\grg\}$
        \tcc*{update classification}
      }
    }
  }
  \For(\tcc*[f]{permutations no longer stabilizing}){
    $\grg \in \incelemClass[\max \setStyle{S}'] \cap \critelemClass[\infty]$
  }{   
    $\critelemClass'[\infty] \gets \critelemClass'[\infty] \setminus \{\grg\}$
    \tcc*{update classification}
    $\critelemClass'[\max \setStyle{S}'] \gets \critelemClass'[\max \setStyle{S}'] \cup \{\grg\}$
    \tcc*{update classification}
  }
  \Return $(\true, \critelemClass')$\;
  
  \caption[Set-Based Critical-Element Method]{The set-based
    critical-element method}
  \label{alg:isLexMinCritelemSet}  
\end{algorithm}

Finally, the modified switch-table method for larger symmetry groups is
implemented.  Algorithm~\ref{alg:isLexMinModifiedSwitches} is a
utilization of Lemma~\ref{thm:char-lex-min-st} from
Section~\ref{sec:theory}.  The advantage compared to
Algorithm~\ref{alg:isLexMinSwitches} is that it recursively removes
elements from the subsets that play no role in the lex-comparison.
The advantage compared to the algorithm in
\cite{PechReichard_EnumeratingSetOrbits_2009} is that one only uses a
single group representation, namely a switch table, which avoids the
creation and deletion of non-trivial local data structures, namely
stabilizers, during the recursion.

\begin{algorithm}[p]
  \TitleOfAlgo{\isLexMinModSwitchTable{$i, \setStyle{S}, \setStyle{S'}, \GrG, \switchTable$}}

  \KwIn{an integer $i \in [1, n]$, a subset $\setStyle{S}$ of
    $\{i, \dots, n\}$ (the original subset), a subset $\setStyle{S'}$
    of $\{i, \dots, n\}$ (the image of~$\setStyle{S}$ under a partial
    switch product), a subgroup $\GrG$ of~$\symGroup{n}$, and a switch
    table~$\switchTable$ of~$\GrG$}

  \KwOut{$\true$ if
    $\lexmin \stabGroup{\GrG}{\indexSet{i-1}}(\setStyle{S'})
    \notlexsmaller \setStyle{S}$ and $\false$ otherwise}
  
  \If(\tcc*[f]{no or all elements?}){
    $\abs{\setStyle{S}} = 0$ or $\abs{\setStyle{S}} = n - i + 1$}{
    \Return \true
    \tcc*{subsets identical and invariant} 
  }
  \If(\tcc*[f]{no non-trivial permutation in group?}){
    $i > \max \effRowSet(\switchTable)$}{
    \Return ($\setStyle{S}' \notlexsmaller \setStyle{S}$)
    \tcc*{subset invariant}
  }

  \tcc{case distinction whether $i \in \setStyle{S}$:}

  \If{
    $i \in \setStyle{S}$
  }{
    \If{
      $i \in \setStyle{S}'$
    }{
      \tcc{recursively check a switch product with leading identity:}
      \If{$\isLexMinModSwitchTable{$i + 1, \setStyle{S} \setminus \{i\},
          \setStyle{S'} \setminus \{i\}, 
          \GrG, \switchTable$} = \false$}{
        \Return \false\;
      }
    }
    \tcc{travers all switches mapping an element of~$\setStyle{S}'$
      to~$i$:}
    \For{
      $j \in \effColSet(i) \cap \setStyle{S}'$
    }{
      \If{$\isLexMinModSwitchTable{$i + 1, \setStyle{S} \setminus \{i\},
          \switchTable[i][j]\bigl(\setStyle{S'} \setminus \{j\}\bigr),
          \GrG, \switchTable$} = \false$}{
        \Return \false\;
      }
    }      
    \Return \true
    \tcc*{no decreasing switch product found}
  }
  \Else{
    \tcc{compare minimal elements using $\minElem{\setStyle{S}} > i$:}
    \If{
      $i \in \setStyle{S}'$
    }{
      \Return \false\;
    }

    \tcc{check for a switch mapping an element of~$\setStyle{S}'$ to~$i$:}
    \If{
      $\effColSet(i) \cap \setStyle{S}' \neq \emptyset$
    }{
      \Return \false\;
    }
    
    \tcc{recursively check a switch product with leading identity:}
    \If{$\isLexMinModSwitchTable{$i + 1, \setStyle{S}, \setStyle{S'},
        \GrG, \switchTable$} = \false$}{
      \Return \false\;
    }
  
    \tcc{travers all switches mapping a non-element of~$\setStyle{S}'$
      to~$i$:}
    \For{
      $j \in \effColSet(i) \setminus \setStyle{S}'$
    }{
      \If{$\isLexMinModSwitchTable{$i + 1, \setStyle{S},
          \switchTable[i][j](\setStyle{S'}), \GrG, \switchTable$} = \false$}{
        \Return \false\;
      }
    }
    \Return \true
    \tcc*{no decreasing switch product found}
  }

  \caption[Modified Switch-Table Method]{The modified switch-table
    method}
  \label{alg:isLexMinModifiedSwitches}  
\end{algorithm}

It is clear that for the critical-element method a very large order of
the symmetry group is prohibitive, so that in those cases the modified
switch table method is mandatory.  For groups of moderate order that
can be stored completely in main memory, it is difficult to predict
precisely which method is faster.  This is mainly due to the fact that
in an analysis of the recursion for the modified switch-table method
it is difficult to exploit the group structure.

Therefore, as a simplification, a \emph{hyper-amortized} analysis is
presented indicating the parameter values for which the
critical-element method is typically preferable.  The idea is based on
a drastically simplified probability model for the objects involved.
In hyper-amortized analysis, the input subsets are considered
uniformly and independently distributed among \emph{all} possible
subsets, not considering their extendability to feasible subsets or
their dependence on the earlier subsets.  Finally, instead of a
probability model for the possible permutation \emph{groups}, a
probability model for permutation \emph{subsets} with uniform
probabilities is considered.  This probability model ignores that the
subsets occurring as inputs actually are feasible subsets in a given
downset and that they have been reached by the enumeration algorithm
somehow.  Moreover, it ignores that the actual sets of permutations
form groups.  Thus, the respective analysis cannot give the amortized
effort for a particular problem instance.  It can, however, be roughly
interpreted as a higher-aggregated information about the average
effort when the respective methods are applied to all conceivable
problem instances with similar parameters, where all subsets are
equally likely to occur as feasible subsets in \emph{some} instance
and where all permutations are individually equally likely to occur in
\emph{some} group.  To stress the reference to the uniform probability
model all expected values with respect to this model are called
\emph{typical values}.

For the analysis it is important to identify from which operations the
main effort results.  Profiling of actual code gives the hint that
three types of operations cause the main part of the computation
times: applying a permutation to a subset or a singleton, passing a
copy of a subset as an argument to a recursive call, and lex-comparing
a subset or a singleton with its image under a permutation.  Applying
a permutation $\pi$ to a subset $\setStyle{S}$ means generating a new
subset consisting of all $\pi(i)$ with $i \in \setStyle{S}$.  Passing
a subset is essentially a special case with $\pi(i) = i$.
Lex-comparing a subset or a singleton with its image requires the
comparison of some the elements with image elements, which will not be
counted extra, since for each comparison there is at least one image
computation.  The only substantial effort that does not lead to
single-element evaluations is the lex-comparison of two subsets, which
happens in exactly one spot in the modified switch-table method.  In
order to unify the effort estimation, we count $m$ single-element
evaluations for the lex-comparison of two $m$-subsets.  In this
generalized sense, the typical number of \emph{single-element
  evaluations} $\pi(i)$ is investigated for
permutations~$\pi \in \GrG$ and elements $i \in \setStyle{S}$ during
the execution of the algorithms.

In the following, the effort of $\isLexMin$ for either method is
estimated.  In ``\false''-instances the typical run-time is difficult to
tell because it is not clear how early a lex-decreasing permutation
will be found.  Therefore, it is assumed for the analysis that any
premature answer does not lead to an immediate stop; rather the
algorithm is carried out until all relevant (depending on the method)
permutations have been processed.  This is equivalent to the typical
effort for \emph{counting} the lex-decreasing permutations.

The analysis of the naive method (check each permutation whether
or not it lex-decreases the subset) and the critical-element method is
straight-forward.
\begin{theorem}
  \label{thm:nve-cet:analysis}
  Consider subsets $\GrG$ of $\symGroup{n}$ with order~$k$ and
  degree~$n$.  Then:
  \begin{enumerate}[label=(\roman*)]
  \item The typical number $\seenve(m)$ of single-element evaluations
    of the naive method on instances with random $m$-element subsets
    $\Sorg$ of $\isLexMin$ is $\seenve(m) = k m$.
  \item The typical number $\seecet(m)$ of single-element evaluations
    of the iteration-based critical-element method $\isLexMinIter$ on
    instances with random $m$-element subsets $\Sorg$ is
    $\seecet(m) = k\frac{n + m - 1}{n} = \frac{n + m - 1}{nm} km$.
  \end{enumerate}
\end{theorem}
\begin{proof}
  For the naive method, $k$ permutations need to be applied to the
  $m$-element subset.  For the critical-element method, typically
  $\frac{1}{n}k$ permutations map the new maximal element of~$\Sorg$
  to the critical element, so that, in the worst case, $m$
  single-element evaluations are necessary to compute the new critical
  element from scratch.  Moreover, typically $\frac{n - 1}{n} k$
  permutations do not map the maximal element of~$\Sorg$ to the
  critical element, incurring only one single-element evaluation each.
\end{proof}

Recall that the modified switch-table method takes two subsets of
identical cardinality as input and decides whether or not the second
can be lex-decreased below the first.  The analysis for the modified
switch-table now works in two steps.  First, for a \emph{given}
switch-table the typical number of single-element evaluations is
analyzed on two subsets chosen uniformly and independently at
random.\footnote{Note that this means already a simplification because
  for the applications of the modified switch-table in this paper,
  $\Sorg$ and $\Simg$ are not stochastically independent. On the
  contrary: in the top-level call, they are identical.  The hope is
  that over all instances that might ever be dealt with in any
  recursion level, they are sufficiently independent.}  Second, for a
\emph{typical} switch table of a random subset of permutations the
typical number of non-trivial switches in each row and the typical
maximal effective row are analyzed.

It is often more instructive to estimate the factor by which a
(non-trivial) method for $\isLexMin$ is faster than the naive method.
\begin{definition}
  Let $\method$ be a method for $\isLexMin$. Then \emph{relative
    effort $\releffort{\method}{\ell}$} of~$\method$ is defined as the
  typical number of single-element evaluations $\seemth(\ell)$ of
  $\method$ divided by the typical number $\seenve(\ell)$ of
  single-element evaluations of the naive method on random
  $\ell$-element subsets.
\end{definition}

\begin{example}
  By Theorem~\ref{thm:nve-cet:analysis}, the relative effort of the
  critical-element method $\releffort{cet}{\ell}$ is
  $\frac{n + m - 1}{nm} \le \frac{2}{m}$, which is a speed-up linear
  in~$m < n$.\qed
\end{example}

It will turn out that the following parameters are crucial for the
efficiency of a switch table.  They all contribute to measures for how
many times the various branches in the recursions are typically
traversed.
\begin{definition}
  \label{def:mst-parameters}
  Consider a switch table $\switchTable$ for a group of permutations
  of degree~$n$ with maximal effective row~$r$ and $c_i$ non-trivial
  switches in row~$i$, $i = 1, 2, \dots, r$.  Moreover, consider
  subsets of cardinality~$\ell$ with $0 \le \ell \le m < n$.

  The \emph{level-$i$ order} $\iorder{i}$ of $\switchTable$ is defined
  as the number of switch-products in $\switchTable$ from rows at
  least~$i$, that is
  \begin{equation}
    \label{eqn:mst-parameters:order}
    \iorder{i} := \prod_{j = i}^r (c_j + 1).
  \end{equation}
  With this, the order $k$ of the group equals~$\iorder{1}$.

  The \emph{level-$i$ subset density and the level-$i$ subset
    codensity of $\switchTable$} are defined for
  $0 < \ell < n - i + 1$ as
  \begin{equation}
    \label{eqn:mst-parameters:density}
    \density{i}{\ell} := \frac{\ell}{n - i + 1}
    \text{ and }
    \codensity{i}{\ell} := \frac{n - i + 1 - \ell}{n - i + 1} =
    1 - \density{i}{\ell} \text{, respectively.}
  \end{equation}
  For $\ell \le 0$ and $\ell \ge n - i + 1$ define the exceptional
  values
  \begin{equation}
    \label{eqn:mst-parameters:density-exception}
    \density{i}{\ell} := \codensity{i}{\ell} := 0.
  \end{equation}
  The \emph{level-$i$ transitivity gap of $\switchTable$} is defined
  for $0 \le \ell \le n - i - c_i$ as
  \begin{equation}
    \label{eqn:mst-parameters:transitivitygap}
    \transitivitygap{i}{\ell} := \frac{\bigl(n - i + 1 - (c_i + 1)\bigr)_{\ell}}{(n - i + 1)_{\ell}},
  \end{equation}
  where $(x)_{\ell} := x(x - 1) \cdots (x - \ell + 1)$, as usual.  For
  $\ell > n - i + 1 - (c_i + 1)$ the transitivity gap is defined to be
  zero.

  In algorithm $\isLexMinModSwitchTable$, call the branch at recursion
  level~$i$ with $i \in \Sorg$ the \emph{level-$i$ element branch} and
  the branch with $i \notin \Sorg$ the \emph{level-$i$ non-element
    branch}.  The \emph{level-$i$ reduction factor of the element
    branch of $\switchTable$} is defined as
  \begin{equation}
    \label{eqn:mst-parameters:elm-reduction-factor}
    \elmreductionfactor{i}{\ell} := \density{i}{\ell}^2,
  \end{equation}
  and the \emph{level-$i$ reduction factor of the non-element branch
    of $\switchTable$} is defined as
  \begin{equation}
    \label{eqn:mst-parameters:nel-reduction-factor}
    \nelreductionfactor{i}{\ell} := \codensity{i}{\ell} \transitivitygap{\ell}.
  \end{equation}
  For $p \in \mathbb{N}$, a \emph{branch-selection vector of
    length~$p$} is a binary vector
  $\branchvec = (\bb{1}, \bb{2}, \dots, \bb{p}) \in \{0, 1\}^p$, where
  $()$ is the only branch selection vector for $p = 0$.  Let
  $\branchVecs{p}$ be the set of all branch-selection vectors of
  length~$p$ with $\branchVecs{p} = \emptyset$ for all $p < 0$.  Let
  $\ones{\branchvec}$ be the number of ones in~$\branchvec$, and
  denote by $\onesupto{j}{\branchvec}$ the number of ones
  in~$\branchvec$ among the first $j-1$ coordinates of~$\branchvec$,
  $1 \le j \le p$.

  For $i \in \indexSet{r}$ and a branch-selection vector of
  length~$p \le r - i + 1$ let its \emph{level-$i$ branch weight}
  $\branchweight{i}{\ell}{\branchvec}$ be the following product of $p$
  reduction factors, where an empty product ($p = 0$) is defined to be
  one:
  \begin{equation}
    \label{eqn:mst-parameters:branch-weight}
    \branchweight{i}{\ell}{\branchvec}
    :=
    \prod_{t = 1}^{p}
    \reductionfactor{i + t - 1}{\ell - \onesupto{t}{\branchvec}}{\bb{t}}.
  \end{equation}
  Finally, let its \emph{node weight} be
  \begin{equation}
    \label{eqn:mst-parameters:node-weight}
    \nodeweight{\ell}{\branchvec}
    :=
    (\ell - \ones{\branchvec})^+
    =
    \max(\ell - \ones{\branchvec}, 0).
  \end{equation}
\end{definition}
These notions are motivated by the following: The subset density and
codensity determine the probabilities with which the two branches in
the modified switch-table method are entered.  The exceptional
zero-values for both model the cases where the modified switch-table
method returns $\true$ without entering any branch. The transitivity
gaps indicate the fraction of $\ell$-subsets for which no image of a
switch in row~$i$ of $\switchTable$ contains~$i$.  For example, in a
switch table of a transitive group, the level-$1$-transitivity gap is
zero even for $\ell = 1$ (one-element subsets), because any element
can be mapped to~$1$.  The reduction factors compute the probabilities
that one of the switches in a row has to be processed during the
execution of the algorithm at the corresponding recursion level.  A
branch selection vector of length~$p$ encodes for recursion
levels~$i, i + 1, \dots, i + p - 1$ and a fixed~$\Sorg$ which of the
branches is entered by the algorithm: A ``$1$'' selects the
element-branch, a ``$0$'' selects the non-element branch.  In this
sense, all paths in the recursion tree are indexed by branch-selection
vectors.  The branch weights are products of reduction factors, where
the first index (indicating the row of $\switchTable$) is increased
with each new factor, and the second index (indicating the cardinality
of a processed subset) is decreased directly after an element branch
has been entered.

With these parameters, all of which can easily be derived directly
from the switch table and the subsets under consideration, the
following can be said about its efficiency:
\begin{theorem}
  \label{thm:mst:analysis}
  Let $\switchTable$ be a switch table for permutations in
  $\symGroup{n}$ with maximal effective row~$r$ and $c_i$ non-trivial
  switches in row~$i$ for $i = 1, 2, \dots, r$.  Let $\branchVecs{i}$
  be the set of all branch-selection vectors of length~$i$,
  $1 \le i \le r$.
  
  Then, the typical number $\seemst(m)$ of single-element evaluations
  for two independently random $m$-subsets $\Sorg, \Simg$ is
  \begin{equation}
    \label{eqn:mst:see-exact}
    \seemst(m)
    =
    \biggl(
      \sum_{\branchvec \in \branchVecs{r}}
      \branchweight{1}{m}{\branchvec} \nodeweight{m}{\branchvec}
    \biggr)
    k
    +
    \sum_{j = 1}^r
    \Biggl(
    \biggl(
      \sum_{\branchvec \in \branchVecs{j}}
      \branchweight{1}{m}{\branchvec} \nodeweight{m}{\branchvec}
    \biggr)
    \frac{c_j}{c_j + 1} \cdot \frac{k}{\iorder{j + 1}}
    \Biggr)
  \end{equation}  
\end{theorem}

\begin{corollary}
  \label{thm:mst:analysis:alternative}
  Since the $r$th summand of the second sum is actually
  $\frac{c_r}{c_r + 1}$ times the first summand, an alternative
  expression for $\seemst(m)$ is:
  \begin{align}
    \label{eqn:mst:see-exact-alt}
    \seemst(m)
    &=
    \Bigl(
    1 + \frac{c_r}{c_r + 1}
    \Bigr)
    \biggl(
      \sum_{\branchvec \in \branchVecs{r}}
      \branchweight{1}{m}{\branchvec} \nodeweight{m}{\branchvec}
    \biggr)
    k\\
    &\quad
      {}
      +
      \sum_{j = 1}^{r - 1}
      \Biggl(
      \biggl(
      \sum_{\branchvec \in \branchVecs{j}}
      \branchweight{1}{m}{\branchvec} \nodeweight{m}{\branchvec}
      \biggr)
      \frac{c_j}{c_j + 1} \cdot \frac{k}{\iorder{j + 1}}
      \Biggr).\qed
  \end{align}
\end{corollary}
\begin{proof}
  Let $\switchTable$ be as in the assumptions of the theorem. For
  $1 \le i \le r$ and $1 \le \ell \le m < n$ call $f_i(\ell)$ the
  number of single-element evaluations the modified switch-table
  method with start row~$i$ typically needs on two random
  $\ell$-element subsets~$\Sorg$ and $\Simg$ of
  $\indexSet{n} \setminus \indexSet{i - 1}$. For $i = r + 1$, no
  non-trivial switches are available, and after one lex-comparison of
  the given $\ell$-element subsets incurring $\ell$ single-element
  evaluations the answer is immediate. Consequently,
  $f_{r+1}(\ell) = \ell$. Moreover, if $\ell = 0$ and
  $\ell = n - i + 1$ both sets are either the empty or the complete
  subset, thus fixed under any action of a switch in row~$i$ and
  higher, and the recursion subtree needs no further single-element
  evaluations, leading to $f_i(0) = f_i(n - i + 1) = 0$.

  The element-branch in the modified switch-table method has
  $i \in \Sorg$. Given a random~$\Sorg$, this branch is entered with
  probability $\frac{\ell}{n - i + 1} = \density{i}{\ell}$.
 
  In the element branch, the identity is considered only if
  $i \in \Simg$, and non-trivial switches $\switchTable[i][j]$ are
  considered only if $j \in \Simg$.  For fixed $j \ge i$ one has
  $j \in \Simg$ with probability
  $\frac{\ell}{n - i + 1} = \density{i}{\ell}$.\footnote{At this point
    a difference occurs for $\Simg = \Sorg$, which occurs in the
    lex-min check for the case $i = 1$ and $\ell = m$. Then,
    $1 \in \Sorg$ implies (with probability one) $1 \in \Simg$, and,
    therefore, the probability that the identity has to be processed
    is $1$ instead of $\density{1}{m}$, changing the expression
    $\density{1}{m}^2 (c_1 + 1)$ to
    $\density{1}{m} (\density{1}{m}c_1 + 1)$ in the results below.
    This increases the typical number of single-element evaluations in
    the first recursion level.  The difference is the larger the
    smaller $c_1$ is.  Since this leads to more case-distinctions, in
    this paper only the analysis for two independently random subset
    is carried out.}  Thus, only a fraction of $\density{i}{\ell}$ of
  the $(c_i + 1)$ many switches including one identity in row~$i$ are
  processed at all.  Processing a switch incurs $\ell - 1$
  single-element evaluations for the non-trivial switches plus a
  recursive call with two $\ell - 1$-element subsets for row~$i + 1$
  of the switch table for the $c_i$ non-trivial switches plus the
  identity.  The typical number of single-element evaluations in the
  element branch is, therefore,
  $\density{i}{\ell} \bigl(c_i (\ell - 1) + (c_i + 1) f_{i+1}(\ell -
  1)\bigr)$.

  In the non-element branch with $i \notin \Sorg$, the answer is
  immediately $\false$ in this branch whenever $i \in \Simg$ or there
  exists a non-trivial switch $\switchTable[i][j]$ with $j \in \Simg$;
  in this case there are no subsequent single-element evaluations in
  this branch.  If all $\ell$ elements
  of~$\Simg \subseteq \{i, i + 1, \dots, n\}$ are among the trivial
  columns $j > i$, then the remaining switches have to be
  processed. Therefore, with
  probability~$\frac{(n - i + 1 - (c_i + 1))_{\ell}}{(n - i +
    1)_{\ell}} = \transitivitygap{i}{\ell}$ neither the $c_i$
  non-trivial switches nor the identity lead to an immediate $\false$
  in this branch.  This event incurs $\ell$ single-element evaluations
  for all non-trivial switches plus a recursive call with two
  $\ell$-element subsets with row~$i + 1$ for all non-trivial switches
  and the identity.  Hence, the typical number of single-element
  evaluations in the non-element branch is
  $\transitivitygap{i}{\ell} \bigl(c_i \ell + (c_i + 1)
  f_{i+1}(\ell)\bigr)$.


  Summarized, the following recursion for $f_i(\ell)$ is obtained with
  $f_1(m) = \seemst(m)$:
  \begin{align}
    f_i(\ell) &= \elmreductionfactor{i}{\ell} 
                \bigl(
                c_i (\ell - 1) + (c_i + 1)
                f_{i+1}(\ell - 1)
                \bigr)
                \notag \\
              &\quad{} + \nelreductionfactor{i}{\ell} 
                \bigl
                (c_i \ell + (c_i + 1)
                f_{i+1}(\ell)
                \bigr).
                \label{eqn:mst:analysis:recursion}
  \end{align}
  Recall that the final comparison of $\ell$-subsets at maximal
  recursion level~$r + 1$ as well as the empty-set case $\ell = 0$ and
  the full-set case $\ell = n - i + 1$ lead to the boundary conditions
  \begin{equation}
    \label{eqn:mst:analysis:basecase}
    f_{r + 1}(\ell) = \ell \text{ for all $1 \le \ell \le m$}, \quad
    f_i(0) = f_i(n - i + 1) = 0 \text{ for all $1 \le i \le r$}.
  \end{equation}
  In order to prove the theorem, it is sufficient show that the
  following function $g$ is a solution to the recursion
  for~$f$, since the claim of the theorem can then be expressed as
  $\seemst(m) = g_1(m)$:
  \begin{equation}
    \label{eqn:mst:analysis:solution}
    g_i(\ell) :=
    \biggl(
      \sum_{\branchvec \in \branchVecs{r - i + 1}}
      \branchweight{i}{\ell}{\branchvec} \nodeweight{\ell}{\branchvec}
    \biggr)
    \iorder{i}
    +
    \sum_{j = i}^{r}
    \Biggl(
    \biggl(
      \sum_{\branchvec \in \branchVecs{j - i + 1}}
      \branchweight{i}{\ell}{\branchvec} \nodeweight{\ell}{\branchvec}
    \biggr)
    \frac{c_j}{c_j + 1} \cdot \frac{\iorder{i}}{\iorder{j + 1}}
    \Biggr).
  \end{equation}

  To this end, note that the following recursion holds for all
  $1 \le i \le r$ and $1 \le \ell \le m$ by definition for the
  products of branch weights and node weights:
  \begin{equation}
    \label{eqn:mst:analysis:weight-recursion}
    \sum_{\branchvec \in \branchVecs{i}}
    \branchweight{i}{\ell}{\branchvec}
    \nodeweight{\ell}{\branchvec}
    =
    \elmreductionfactor{i}{\ell}
    \sum_{\branchvec \in \branchVecs{i-1}}
    \branchweight{i+1}{\ell-1}{\branchvec}
    \nodeweight{\ell-1}{\branchvec}
    +
    \nelreductionfactor{i}{\ell}
    \sum_{\branchvec \in \branchVecs{i-1}}
    \branchweight{i+1}{\ell}{\branchvec}
    \nodeweight{\ell}{\branchvec}.
  \end{equation}
  Moreover, the following recursion holds for the level-$i$ orders:
  \begin{equation}
    \label{eqn:mst:analysis:k-recursion}
    \iorder{i} = (c_i + 1) \iorder{i + 1}.
  \end{equation}
  The boundary condition $g_{r+1}(\ell) = \ell$ holds true, because
  the boundary cases $\iorder{r + 1} = 1$, $\branchweight{i}{\ell}{()}
  = 1$, and $\nodeweight{\ell}{()} = \ell$ imply
  \begin{align}
    \lefteqn{g_{r + 1}(\ell)}\\
    &=
    \biggl(
      \sum_{\branchvec \in \branchVecs{0}}
      \branchweight{r + 1}{\ell}{\branchvec} \nodeweight{\ell}{\branchvec}
    \biggr)
    \iorder{r + 1}
    +
    \sum_{j = r + 1}^{r}
    \Biggl(
    \biggl(
      \sum_{\branchvec \in \branchVecs{j - r}}
      \branchweight{r + 1}{\ell}{\branchvec} \nodeweight{\ell}{\branchvec}
    \biggr)
    \frac{c_j}{c_j + 1} \cdot \frac{\iorder{r + 1}}{\iorder{j + 1}}
      \Biggr)\\
    &=
      \branchweight{r + 1}{\ell}{()} \nodeweight{\ell}{()}\\
    &=
      \ell.
  \end{align}
  The boundary conditions $g_i(0) = 0$ and $g_i(n - i + 1) = 0$ hold
  true, because
  $\density{i}{n - i + 1} = \codensity{i}{n - i + 1} = 0$ leading to
  zero reduction factors throughout.  Moreover, $g_i(\ell)$ satisfies
  the recursion formula for $1 \le \ell \le m$ and $1 \le i \le r$,
  because \allowdisplaybreaks
  \begin{align}
    \label{eqn:mst:analysis:solutioncheck}
    g_i(\ell)
    &=
      \biggl(
      \sum_{\branchvec \in \branchVecs{r - i + 1}}
      \branchweight{i}{\ell}{\branchvec} \nodeweight{\ell}{\branchvec}
      \biggr)
      \iorder{i}
      +
      \sum_{j = i}^{r}
      \biggl(
      \sum_{\branchvec \in \branchVecs{j - i + 1}}
      \branchweight{i}{\ell}{\branchvec} \nodeweight{\ell}{\branchvec}
      \biggr)
      \frac{c_j}{c_j + 1} \cdot \frac{\iorder{i}}{\iorder{j + 1}}\\
    &\stackrel{\eqref{eqn:mst:analysis:weight-recursion},\eqref{eqn:mst:analysis:k-recursion}}{=}
      \biggl(
      \elmreductionfactor{i}{\ell}
      \sum_{\branchvec \in \branchVecs{r - i}}
      \branchweight{i+1}{\ell-1}{\branchvec} \nodeweight{\ell - 1}{\branchvec}
      +
      \nelreductionfactor{i}{\ell}
      \sum_{\branchvec \in \branchVecs{r - i}}
      \branchweight{i+1}{\ell}{\branchvec} \nodeweight{\ell}{\branchvec}
      \biggr)
      (c_i + 1)\iorder{i + 1}
    \\
    &\quad
      {} +
      \sum_{j = i + 1}^{r}
      \Biggl(
      \biggl(
      \elmreductionfactor{i}{\ell}
      \sum_{\branchvec \in \branchVecs{j - i}}
      \branchweight{i+1}{\ell-1}{\branchvec} \nodeweight{\ell - 1}{\branchvec}
      +
      \nelreductionfactor{i}{\ell}
      \sum_{\branchvec \in \branchVecs{j - i}}
      \branchweight{i+1}{\ell}{\branchvec} \nodeweight{\ell}{\branchvec}
      \biggr)
    \\
    &\quad
      \hphantom{{} + \sum_{j = i + 1}^{r} \Biggl(}
      {} \cdot
      \frac{c_j}{c_j + 1} \cdot \frac{(c_i + 1)\iorder{i + 1}}{\iorder{j + 1}}
      \Biggr)
    \\
    &\quad
      {} +      
      \sum_{\branchvec \in \branchVecs{1}}
      \branchweight{i}{\ell}{\branchvec} \nodeweight{\ell}{\branchvec}
      \frac{c_i}{c_i + 1} \cdot \frac{(c_i + 1)\iorder{i + 1}}{\iorder{i + 1}}
    \\
    &=
      \elmreductionfactor{i}{\ell}
      \biggl(
      \sum_{\branchvec \in \branchVecs{r - i}}
      \branchweight{i+1}{\ell-1}{\branchvec} \nodeweight{\ell - 1}{\branchvec}
      \biggr)
      (c_i + 1) \iorder{i + 1}
    \\
    &\quad
      {} +
      \nelreductionfactor{i}{\ell}
      \biggl(
      \sum_{\branchvec \in \branchVecs{r - i}}
      \branchweight{i+1}{\ell}{\branchvec} \nodeweight{\ell}{\branchvec}
      \biggr)
      (c_i + 1) \iorder{i + 1}
    \\
    &\quad
      +
      \elmreductionfactor{i}{\ell}
      \sum_{j = i + 1}^{r}
      \biggl(
      \sum_{\branchvec \in \branchVecs{j - i}}
      \branchweight{i+1}{\ell-1}{\branchvec} \nodeweight{\ell - 1}{\branchvec}
      \biggr)
      \frac{c_j}{c_j + 1} \cdot \frac{(c_i + 1) \iorder{i + 1}}{\iorder{j + 1}}    
    \\
    &\quad
      {} +
      \nelreductionfactor{i}{\ell}
      \sum_{j = i + 1}^{r}
      \biggl(
      \sum_{\branchvec \in \branchVecs{j - i}}
      \branchweight{i+1}{\ell}{\branchvec} \nodeweight{\ell}{\branchvec}
      \biggr)
      \frac{c_j}{c_j + 1} \cdot \frac{(c_i + 1) \iorder{i + 1}}{\iorder{j + 1}}    
    \\
    &\quad
      {} +
      c_i
      \underbrace{\branchweight{i}{\ell}{(1)}}_{= \elmreductionfactor{i}{\ell}}
      \underbrace{\nodeweight{\ell}{(1)}}_{= \ell - 1}
      {}
      +
      {}
      c_i
      \underbrace{\branchweight{i}{\ell}{(0)}}_{= \nelreductionfactor{i}{\ell}}
      \underbrace{\nodeweight{\ell}{(0)}}_{= \ell}
    \\
    &=
      \elmreductionfactor{i}{\ell} (c_i + 1)
      g_{i + 1}(\ell - 1)
      +
      \nelreductionfactor{i}{\ell} (c_i + 1)
      g_{i + 1}(\ell)
      +
      \elmreductionfactor{i}{\ell} c_i (\ell - 1)
      +
      \nelreductionfactor{i}{\ell} c_i (\ell).
  \end{align}
  Thus, $g$ satisfies recursion \eqref{eqn:mst:analysis:recursion},
  which completes the proof.
\end{proof}
Theorem~\ref{thm:mst:analysis} can be used to derive simpler upper and
lower bounds for the application of the modified switch-table method.
In order to prepare this, some monotonicity properties of the
quantities involved in Theorem~\ref{thm:mst:analysis} and
Corollary~\ref{thm:mst:analysis:alternative}, respectively, are
presented.  Note, that the monotonicity properties below do not imply
any kind of monotonicity for the function~$f_i(\ell)$ from the
previous proof. In fact, it is not monotone at all, as plotting
experiments easily reveal. Thus, the exact evaluation of the recursion
in Theorem~\ref{thm:mst:analysis} is the key to simpler bounds.
\begin{lemma}
  \label{thm:mst:basic-facts}
  For all $1 \le j \le r \le m$, all $1 \le i \le j$, all
  $m - j + 1 \le \ell \le m$, and all $\branchvec \in \branchVecs{j}$
  the following hold:
  \begin{enumerate}[label=(\roman*)]
  \item $\density{1}{m - j + 1} \le \density{i}{\ell} \le
    \density{j}{m}$.
  \item $\codensity{j}{m} \le \codensity{i}{\ell} \le \codensity{1}{m
      - j + 1}$.
  \item $\transitivitygap{j}{m} \le \transitivitygap{i}{\ell} \le
    \transitivitygap{1}{m - j + 1}$.
  \item $\elmreductionfactor{1}{m - j + 1} \le \elmreductionfactor{i}{\ell} \le \elmreductionfactor{m}{j}$.
  \item $\nelreductionfactor{j}{m} \le
    \nelreductionfactor{i}{\ell} \le \nelreductionfactor{1}{m - j + 1}$.
  \item \label{itm:mst:basic-facts:branch-weights} $\left(
      \elmreductionfactor{1}{m - j + 1}
    \right)^{\ones{\branchvec}}
    \left(
      \nelreductionfactor{j}{m}\right)^{j - \ones{\branchvec}}
    \le
    \branchweight{1}{m}{\branchvec}
    \le
    \left(
      \elmreductionfactor{j}{m}
    \right)^{\ones{\branchvec}}
    \left(
      \nelreductionfactor{1}{m - j + 1}
    \right)^{j - \ones{\branchvec}}$.
  \item \label{itm:mst:basic-facts:node-weights} $m - j
    \le
    \nodeweight{m}{\branchvec}
    \le
    m$.
  \end{enumerate}
\end{lemma}
\begin{proof}
  All claims follow directly from the definitions of the respective
  quantities.
\end{proof}
From this, the following conclusion can be drawn.
\begin{corollary}
  \label{thm:mst:analysis:effort-bounds}
  Let $\switchTable$ be a switch table for permutations in
  $\symGroup{n}$ with maximal effective row~$r$ and $c_i$ non-trivial
  switches in row~$i$ for $i = 1, 2, \dots, r$.  Let $\branchVecs{i}$
  be the set of all branch-selection vectors of length~$i$,
  $1 \le i \le r$.
  
  Then, the typical number $\seemst(m)$ of single-element evaluations
  for two independently random $m$-subsets $\Sorg, \Simg$ with $m \ge
  r$ is bounded as follows:
  \begin{align}
    \label{eqn:mst:analysis:effort-bounds}    
    \lefteqn{
    \Bigl(
    1 + \frac{c_r}{c_r + 1}
    \Bigr)
    \biggl(
    \elmreductionfactor{1}{m - r + 1} + \nelreductionfactor{r}{m}
    \biggr)^r
    k(m - r)
    }\\
    &\quad
      {}
      +
      \sum_{j = 1}^{r - 1}
      \biggl(
      \elmreductionfactor{1}{m - j + 1} + \nelreductionfactor{j}{m}
      \biggr)^j
      \frac{c_j}{c_j + 1} \cdot \frac{k(m - j)}{\iorder{j + 1}}\\
    &\le
      \seemst(m)\\
    &\le
      \Bigl(
      1 + \frac{c_r}{c_r + 1}
      \Bigr)
      \biggl(
      \elmreductionfactor{r}{m} + \nelreductionfactor{1}{m - r + 1}
      \biggr)^r
      km\\
    &\quad
      {}
      +
      \sum_{j = 1}^{r - 1}
      \biggl(
      \elmreductionfactor{j}{m} + \nelreductionfactor{1}{m - j + 1}
      \biggr)^j
      \frac{c_j}{c_j + 1} \cdot \frac{km}{\iorder{j + 1}}.
  \end{align}
\end{corollary}
\begin{proof}
  Substituting the monotonicity properties
  \ref{thm:mst:basic-facts}~\ref{itm:mst:basic-facts:branch-weights}
  and~\ref{itm:mst:basic-facts:node-weights} for branch and node
  weights into the formula in
  Corollary~\ref{thm:mst:analysis:alternative} yields the following:
  \begin{align}
    \lefteqn{
    \Bigl(
    1 + \frac{c_r}{c_r + 1}
    \Bigr)
    \Biggl(
    \sum_{\lambda = 0}^r
    \binom{\lambda}{r}
    \biggl(
    \elmreductionfactor{1}{m - r + 1}
    \biggr)^{\lambda}
    \biggl(
    \nelreductionfactor{r}{m}
    \biggr)^{r - \lambda}
    \Biggr)
    (m - r)
    k
    }\\
    &\quad
      {}
      +
      \sum_{j = 1}^{r - 1}
      \Biggl(
      \sum_{\lambda = 0}^j
      \binom{\lambda}{r}
      \biggl(
      \elmreductionfactor{1}{m - j + 1}
      \biggr)^{\lambda}
      \biggl(
      \nelreductionfactor{j}{m}
      \biggr)^{j - \lambda}
      \Biggr)
      (m - j)
      \cdot
      \frac{c_j}{c_j + 1} \cdot \frac{k}{\iorder{j + 1}}\\
    &\le
        \seemst(m)\\
    &\le
      \Bigl(
      1 + \frac{c_r}{c_r + 1}
      \Bigr)
      \Biggl(
      \sum_{\lambda = 0}^r
      \binom{\lambda}{r}
      \biggl(
      \elmreductionfactor{r}{m}
      \biggr)^{\lambda}
      \biggl(
      \nelreductionfactor{1}{m - r + 1}
      \biggr)^{r - \lambda}
      \Biggr)
      m
      k\\
    &\quad
      {}
      +
      \sum_{j = 1}^{r - 1}
      \Biggl(
      \sum_{\lambda = 0}^j
      \binom{\lambda}{r}
      \Bigl(
      \elmreductionfactor{j}{m}
      \Bigr)^{\lambda}
      \Bigl(
      \nelreductionfactor{1}{m - j + 1}
      \Bigr)^{j - \lambda}
      \Biggr)
      m
      \cdot
      \frac{c_j}{c_j + 1} \cdot \frac{k}{\iorder{j + 1}}.
  \end{align}
  The assertion now follows directly from an application of the
  binomial theorem.
\end{proof}
One key learning can be derived right away: Whenever the sum of the
two respective largest possible reduction factors in the element and
non-element branches is strictly smaller than one (which is clear for
identical parameters and is the common case otherwise), then the gain
of the switch table is exponential in $r$ at the cost of an overhead
that is essentially linear in~$r$.  For a small~$r$ the overhead may
outweigh the gain.  For a large~$r$ the gain will in most cases
outweigh the overhead.  Note that conclusions based on asymptotics
in~$r$ alone may be misleading, since $r$ is not an unbounded, free
input parameter but bounded by~$n$ and determined by both the
structure of the symmetry group and the order of the elements
in~$\indexSet{n}$.  In the third application of this paper
(triangulations, see Section~\ref{sec:application-triangulations}),
$r$ is often small, as will be shown later in this section.

The following result shows that if the reduction factors are uniformly
close to $\frac{1}{4}$, then the gain of the switch table is
substantial.
\begin{corollary}
  \label{thm:mst:bounds:upper}
  Let $\switchTable$ be a switch table for a subgroup~$\GrG$
  of~$\symGroup{n}$ with maximal effective row~$r$, exactly $s$
  effective rows in total, and $c_i$ non-trivial switches in row~$i$,
  $i = 1, 2, \dots, r$.

  Assume there is a $\mu \in [1, 2)$ with
  $\elmreductionfactor{i}{\ell} \le \frac{\mu}{4}$ and
  $\nelreductionfactor{i}{\ell} \le \frac{\mu}{4}$ for all effective rows
  $i \in \indexSet{r}$ and $\ell = m - r + 1, m - r + 2, \dots, m$.  
  Then, the relative effort $\releffort{mst}{m}$ of~$\switchTable$ is
  at most $(s + 1) \left(\frac{\mu}{2}\right)^s$, i.e.,
  \begin{equation}
    \label{eqn:mst:bounds:upper}
    \seemst(m) \le
    (s + 1)
    \left(\frac{\mu}{2}\right)^s
    km
    =
    (s + 1)
    \left(\frac{\mu}{2}\right)^s
    \seenve(m).
  \end{equation}

  Moreover, the assumptions are, for example, satisfied if
  $\density{r}{m} = \frac{m}{n - r + 1} \le \frac{\sqrt{2 - \mu}}{2}$
  and
  \begin{equation}
    \label{eqn:mst:bounds:upper:sufficient}
    c_i + 1
    \ge
    \left(
    1 - \sqrt[\uproot{\theuprootcounter}\leftroot{\theleftrootcounter}{m - i + 1}]{\frac{\mu(n - i + 1)}{4(n - m)}}
    \right)
    (n - i + 1)
    \text{ for all $i = 1, 2, \dots, r$}.
  \end{equation}
\end{corollary}
\begin{proof}
  The starting point is
  Corollary~\ref{thm:mst:analysis:effort-bounds}. Denote by
  $r(j) \in \indexSet{r - j}$ the number of effective rows with index
  strictly larger than~$j$.  Then, for an effective row~$j$ one has
  $r(j) = r(j-1) - 1$ and $0 < \tfrac{c_j}{c_j + 1} < 1$; for a
  trivial row~$j$ one has $r(j) = r(j-1)$ and
  $\tfrac{c_j}{c_j + 1} = 0$.  Moreover, the $j$th effective row has
  row index at least~$j$.  For each row~$j$ one has
  $\iorder{j + 1} \ge 2^{r(j)}$.  This implies
  $\frac{k}{\iorder{j + 1}} \le \bigl(\frac{1}{2}\bigr)^{r(j)}k$.


  Therefore:
  \begin{align}
    \seemst(m)
    &\le
      \Bigl(
      1 + \frac{c_r}{c_r + 1}
      \Bigr)
      \biggl(
      \elmreductionfactor{r}{m} + \nelreductionfactor{1}{m - r + 1}
      \biggr)^r
      km\\
    &\quad
      {}
      +
      \sum_{j = 1}^{r - 1}
      \biggl(
      \elmreductionfactor{j}{m} + \nelreductionfactor{1}{m - j + 1}
      \biggr)^j
      \frac{c_j}{c_j + 1} \cdot \frac{km}{\iorder{j + 1}}\\
    &\stackrel{\mathllap{\text{row~$r$ eff.}}}{\le}
      \left(
      2 \biggl(\frac{\mu}{2}\biggr)^r
      +
      \sum_{\substack{j \in \indexSet{r-1}\\\text{$j$ eff.}}}
    \biggl(\frac{\mu}{2}\biggr)^j
    \cdot
    \biggl(\frac{1}{2}\biggr)^{r(j)}
    \right)
    km\\
    &\stackrel{\mathllap{\text{$\mu < 2$}}}{\le}
      \left(
      2\biggl(\frac{\mu}{2}\biggr)^r
      +
      \sum_{j = 1}^{s - 1}
      \biggl(\frac{\mu}{2}\biggr)^j 
      \biggl(\frac{1}{2}\biggr)^{s - j}
      \right)
      km\\
    &\stackrel{\mathllap{\mu \ge 1}}{\le}
      \left(
      2\biggl(\frac{\mu}{2}\biggr)^r
      +
      \sum_{j = 1}^{s - 1}
      \biggl(\frac{\mu}{2}\biggr)^j 
      \biggl(\frac{\mu}{2}\biggr)^{s - j}
      \right)
      km\\
    &=
      \left(
      2\biggl(\frac{\mu}{2}\biggr)^r
      +
      (s - 1)
      \biggl(\frac{\mu}{2}\biggr)^s
      \right)
      km\\
    &\stackrel{\mathllap{s \le r \text{ and } \mu < 2}}{\le}
      (s + 1)
      \biggl(\frac{\mu}{2}\biggr)^s
      km. 
  \end{align}
  If the modified switch-table method is called on $m$-element
  subsets, then the minimal $\ell$ in a reduction factor of level~$i$
  is $m - i + 1$ and the maximal $\ell$ is~$m$.  Therefore,
  $\density{i}{\ell} \le \density{i}{m} \le \frac{\sqrt{2 - \mu}}{2}$
  implies
  $\elmreductionfactor{i}{\ell} = \density{i}{\ell}^2 \le \frac{2 -
    \mu}{4} \le \frac{1}{4} \le \frac{\mu}{4}$.  Moreover,
  $\frac{\mu}{4} < \frac{1}{2} \le 1 - \frac{\sqrt{2 - \mu}}{2} \le 1 -
  \density{i}{\ell} = \codensity{i}{\ell} \le \codensity{i}{m - i + 1}
  = \frac{n - i + 1 -(m - i + 1)}{n - i + 1} = \frac{n - m}{n - i +
    1}$.  This implies $\frac{\mu(n - i + 1)}{4(n - m)} < 1$, and, hence,
  $\sqrt[\ell]{\frac{\mu(n - i + 1)}{4(n - m)}}$ is monotonically
  increasing in~$\ell$.

  With this, one can conclude for all
  $\ell = m, m - 1, \dots, m - i + 1$ and all $i = 1, 2, \dots, r$:
  \allowdisplaybreaks
  \begin{align}
    c_i + 1
    &\ge
      \left(
      1 - \sqrt[\uproot{\theuprootcounter}\leftroot{\theleftrootcounter}{m - i + 1}]{\frac{\mu(n - i + 1)}{4(n - m)}}
      \right)
      (n - i + 1)\\
    \Rightarrow
    c_i + 1
    &\ge
      \left(
      1 - \sqrt[\uproot{\theuprootcounter}\leftroot{\theleftrootcounter}{\ell}]{\frac{\mu}{4\codensity{i}{\ell}}}
      \right)
      (n - i + 1)\\
    \Rightarrow
    c_i + 1
    &\ge
      (n - i + 1)
      -
      (n - i + 1)
      \sqrt[\uproot{\theuprootcounter}\leftroot{\theleftrootcounter}{\ell}]{\frac{\mu}{4\codensity{i}{\ell}}}\\
    \Rightarrow
    \left(\frac{n - i + 1 - (c_i + 1)}{n - i + 1}\right)^{\ell}
    &\le
      \frac{\mu}{4 \codensity{i}{\ell}}\\
    \Rightarrow
    \codensity{i}{\ell}
    \left(
    \frac{\bigl(n - i + 1 - (c_i + 1)\bigr)_{\ell}}{(n - i +
    1)_{\ell}}
    \right)
    &\le
      \frac{\mu}{4}\\
    \Rightarrow
    \nelreductionfactor{i}{\ell} = \codensity{i}{\ell}
    \transitivitygap{i}{\ell}
    &\le
      \frac{\mu}{4}.
  \end{align}
\end{proof}
The following result shows that if the reduction factors are not too
much smaller than~$\frac{1}{2}$ and $r$ is sufficiently smaller
than~$m$, then the switch table may even be slower than the naive
method.
\begin{corollary}
  \label{thm:mst:bound:lower}
  Let $\switchTable$ be a switch table for a subgroup~$\GrG$
  of~$\symGroup{n}$ with maximal effective row~$r$ and $c_i$
  non-trivial switches in row~$i$, $i = 1, 2, \dots, r$.

  Assume that there is a $\mu \in (0, 1]$ with
  $\elmreductionfactor{i}{\ell} \ge \frac{\mu}{2}$ and
  $\nelreductionfactor{i}{\ell} \ge \frac{\mu}{2}$ for all
  $i = 1, 2, \dots, r$ and $\ell = m - r, m - r + 1, \dots, m$.  Then,
  the relative effort $\releffort{mst}{m}$ of~$\switchTable$ is at
  least $\frac{3}{2} \mu^r \left(1 - \frac{r}{m}\right)$, i.e.,
  \begin{equation}
    \label{eqn:mst:bounds:lower}
    \seemst(m) \ge \frac{3}{2} \mu^r \left(1 - \frac{r}{m}\right) k m
    = \frac{3}{2} \mu^r \left(1 - \frac{r}{m}\right) \seenve(m).
  \end{equation}
  In particular, if additionally $\mu > \sqrt[\uproot{0}\leftroot{\theleftrootcounter}{r}]{\frac{2}{3}}$ and
  $r < m\left(1 - \frac{2}{3\mu^r}\right)$, then the modified
  switch-table method typically needs more single-element evaluations
  on $m$-element subsets than the naive method.
\end{corollary}
\begin{proof}
  The starting point is again
  Corollary~\ref{thm:mst:analysis:effort-bounds}.  The fact that
  $c_r > 0$ implies $\frac{c_r}{c_r + 1} \ge \frac{1}{2}$.  Therefore:
  \begin{align}
    \seemst(m)
    &\ge
      \Bigl(
      1 + \frac{c_r}{c_r + 1}
      \Bigr)
      \biggl(
      \elmreductionfactor{1}{m - r + 1} + \nelreductionfactor{r}{m}
      \biggr)^r
      k(m - r)\\
    &\quad
      {}
      +
      \sum_{j = 1}^{r - 1}
      \biggl(
      \elmreductionfactor{1}{m - j + 1} + \nelreductionfactor{j}{m}
      \biggr)^j
      \frac{c_j}{c_j + 1} \cdot \frac{k(m - j + 1)}{\iorder{j + 1}}\\
    &\ge
      \frac{3}{2} \cdot \mu^r \cdot 2^r \cdot \left(\frac{1}{2}\right)^r k (m - r)
    \\
    &\ge
      \frac{3}{2}
      \mu^r
      k(m - r)\\
    &=
      \frac{3}{2} \mu^r \biggl(1 - \frac{r}{m}\biggr) km.
  \end{align}
  The additional claim follows from solving
  $\frac{3}{2} \mu^r \bigl(1 - \frac{r}{m}\bigr) > 1$.
\end{proof}

Since the assumptions in both Corollaries~\ref{thm:mst:bounds:upper}
and~\ref{thm:mst:bound:lower} are somewhat unrealistic (uniformly
bounded reduction factors are rare in practice), the focus of the
following is on bounds that actually explain the difference in
performance for important use cases, namely the enumeration of
circuits and triangulations of hypercubes.

A closer inspection of the non-element reduction factors reveals that
the transitivity of a symmetry group has a positive influence on the
performance of the modified switch-table method.  An important example
for this are the symmetry groups of hypercubes: for each pair of
vertices of the hypercube there is a permutation in its symmetry group
that maps one vertex to the other.  Switch tables allow for a somewhat
continuous measure for the transitivity of a symmetry group: the
level-$i$ transitivity gaps.  A switch table encodes a transitive
group if and only if $n - (c_1 + 1) = 0$. If an $m$-element subset
shall be checked for lexicographic minimality in its orbit, then it is
more relevant whether the level-$1$ transitivity gap is zero.  Such a
zero-transitivity gap is already guaranteed whenever
$n - (c_1 + 1) < m$.  This immediately annihilates the 50\,\% branch
weights of all branch-selection vectors starting with a ``$0$'' and
guarantees that the first reduction factor in the element branch is
active for all remaining branches.  Similarly, if also the level-$2$
transitivity gap is zero, then another 25\,\% of branch weights
disappears with an active second reduction factor of the element
branch, etc.  The reduction factors in the element branch are
particularly small if $m$ is small compared to~$n$. The following
corollary presents an analysis of the resulting relative effort for
level-$1$ transitivity gaps of zero.
\begin{corollary}
  \label{thm:mst:bounds-transitive:upper}
  Let $\switchTable$ be a switch table for a subgroup~$\GrG$
  of~$\symGroup{n}$ with maximal effective row~$r$ and $c_i$
  non-trivial switches in row~$i$, $i = 1, 2, \dots, r$ so that
  $\transitivitygap{1}{m} = 0$. This is, e.g., the case whenever
  $\GrG$ is transitive.  Denote by $r(j)$ the number of effective rows
  with index strictly larger than~$j$.

  Define $\mu := \frac{n-1}{n}$ and assume that
  $\density{m}{r}^2 + \mu^2 \le 1$. Then:
  \begin{equation}
    \releffort{mst}{m}
    \le
    \biggl(
    1  + \sum_{j = 1}^{r} \frac{1}{2^{r(j)}} \cdot \frac{c_j}{c_j + 1}
    \biggr)
    \density{1}{m}^2
    \le
    3 \density{1}{m}^2.    
  \end{equation}
\end{corollary}

\begin{proof}
  First note that both $\codensity{i}{\ell} \le \mu$ and
  $\transitivitygap{i}{\ell} \le \mu$, whence
  $\nelreductionfactor{i}{\ell} \le \mu^2$ for all
  $i = 1, 2, \dots, r$ and $\ell = 1, 2, \dots, m$.  The assumptions
  and Lemma~\ref{thm:mst:basic-facts} guarantee that for a
  branch-selection vector $\branchvec = \tbinom{1}{\branchvec'}$ with
  $\branchvec' \in \branchVecs{j - 1}$ for $j = 1, 2, \dots, r - 1$
  its weight is bounded from above as follows:
  \begin{align}
    \branchweight{1}{m}{\tbinom{1}{\branchvec'}} \le
    \density{1}{m}^2
    \cdot
    \bigl(\density{r}{m}^{2}\bigr)^{\ones{\branchvec'}}
    \cdot
    \bigl(\mu^2\bigr)^{r - 1 - \ones{\branchvec'}}.
  \end{align}
  All branch-selection vectors with $b_1 = 0$ have weight zero.  Thus,
  the first sum in Corollary~\ref{thm:mst:analysis:alternative} can be
  bounded as follows:
  \allowdisplaybreaks
  \begin{align}
    \lefteqn{\Bigl(
    1 + \frac{c_r}{c_r + 1}
    \Bigr)
    \sum_{\branchvec \in \branchVecs{r}}
    \branchweight{1}{m}{\branchvec} \nodeweight{m}{\branchvec} k}\\
    &\le
      k
      \Bigl(
      1 + \frac{c_r}{c_r + 1}
      \Bigr)
      \elmreductionfactor{1}{m} \sum_{\branchvec' \in \branchVecs{r - 1}}
      \branchweight{1}{m}{\branchvec'} \nodeweight{m}{\branchvec'}\\
    &\le
      km
      \Bigl(
      1 + \frac{c_r}{c_r + 1}
      \Bigr)
      \density{1}{m}^2 \sum_{\lambda = 0}^{r - 1} \binom{r - 1}{\lambda}
      \bigl(\density{r}{m}^2\bigr)^{\lambda}
      \bigl(\mu^2 \bigr)^{r - 1 - \lambda}\\
    &=
      km
      \left[
      \Bigl(
      1 + \frac{c_r}{c_r + 1}
      \Bigr)
      \density{1}{m}^2
      \bigl(\density{r}{m}^2 + \mu^2\bigr)^{r - 1}
      \right]\\
    &\le
      km\left[
      \Bigl(
      1 + \frac{c_r}{c_r + 1}
      \Bigr)
      \density{1}{m}^2
      \right].
  \end{align}
  For the second sum,
  $\frac{k}{\iorder{j + 1}} \le \frac{k}{2^{r(j)}}$ for all
  $j = 1, 2, \dots, r-1$, since $c_j + 1 \ge 2$ for all effective
  rows~$j$.  Therefore:
  \allowdisplaybreaks
  \begin{align}
    \lefteqn{\sum_{j = 1}^{r - 1}
    \sum_{\branchvec \in \branchVecs{j}}
    \branchweight{1}{m}{\branchvec} \nodeweight{m}{\branchvec}
    \frac{c_j}{c_j + 1} \cdot \frac{k}{\iorder{j + 1}}}\\
    &\le
      k
      \elmreductionfactor{1}{m}
      \sum_{j = 1}^{r - 1}
      \frac{1}{2^{r(j)}} \cdot \frac{c_j}{c_j + 1}
      \sum_{\branchvec' \in \branchVecs{j - 1}}
      \branchweight{1}{m}{\branchvec'} \nodeweight{m}{\branchvec'}\\
    &\le
      k
      \density{1}{m}^2
      \sum_{j = 1}^{r - 1}
      \frac{1}{2^{r(j)}} \cdot \frac{c_j}{c_j + 1}
      \sum_{\branchvec' \in \branchVecs{j - 1}}
      \branchweight{1}{m}{\branchvec'} \nodeweight{m}{\branchvec'}\\
    &\le
      km
      \density{1}{m}^2
      \sum_{j = 1}^{r - 1}
      \frac{1}{2^{r(j)}} \cdot \frac{c_j}{c_j + 1}
      \sum_{\lambda = 0}^{j  - 1} \binom{j - 1}{\lambda}
      \bigl(\density{r}{m}^2\bigr)^{\lambda}
      \bigl(\mu^2\bigr)^{j - 1 - \lambda}\\
    &=
      km
      \density{1}{m}^2
      \sum_{j = 1}^{r - 1}
      \frac{1}{2^{r(j)}} \cdot \frac{c_j}{c_j + 1}
      \bigl(\density{r}{m}^2
      +
      \mu^2\bigr)^{j - 1}\\
    &=
      km
      \density{1}{m}^2
      \sum_{j = 1}^{r - 1}
      \frac{1}{2^{r(j)}} \cdot \frac{c_j}{c_j + 1}.
  \end{align}
  In total:
  \begin{equation}
    \seemst(m) \le
    km\left[
      \Bigl(
      1 + 
      \sum_{j = 1}^{r}
      \frac{1}{2^{r(j)}} \cdot \frac{c_j}{c_j + 1}
      \Bigr)
      \density{1}{m}^2
    \right].
  \end{equation}
  This implies the assertion.  For the simplified, weaker upper bound,
  note that $\frac{c_j}{c_j + 1} \le 1$ for each effective row~$j$ and
  $\frac{c_j}{c_j + 1} = 0$ for each trivial row~$j$.  Therefore, for
  $s$ effective rows a restriction of the sum to effective row
  indices, reversing the summation, and extension to the infinite
  series yields the claim as follows:
  \begin{equation}
    \sum_{j = 1}^{r}
    \frac{1}{2^{r(j)}} \cdot \frac{c_j}{c_j + 1}
    \le
    \sum_{j = 1}^{s}
      \frac{1}{2^{s - j}}
    =
      \sum_{j = 0}^{s - 1}
      \frac{1}{2^j}
    \le
      \sum_{j = 0}^{\infty}
      \frac{1}{2^j}
    =       
      2.
  \end{equation}
\end{proof}

\begin{example}
  \label{exa:mst:bounds-transitive:upper:cube}
  Consider the point configuration consisting of the set of vertices
  of the $\dimension$-dimensional hypercube~$\hypercube{\dimension}$.
  It has $n = 2^{\dimension}$ points and a transitive symmetry group
  of order~$\dimension! 2^{\dimension}$, generated by all permutations
  of coordinates and a 0/1-flip of an arbitrary coordinate.  If one is
  interested, e.g., in the circuits (see
  Section~\ref{sec:application-circuits}) of~$\hypercube{\dimension}$,
  then all subsets to be considered contain at least $4$ and at most
  $m = \dimension + 2$ points.  Consider an order of the points so
  that, recursively,
  $\hypercube{\dimension} = \begin{tmatrix}{\hypercube{\dimension -
        1}&\hypercube{\dimension - 1}\\0&1}\end{tmatrix}$ in
  non-homogeneous coordinates or with a row of ones at the top in
  homogeneous coordinates.  Then
  \begin{equation}
    \label{eq:1}
    \hypercube{\dimension} =
    \begin{pmatrix}
      \hypercube{\dimension - 2} & \hypercube{\dimension - 2} &
      \hypercube{\dimension - 2} & \hypercube{\dimension - 2}\\
      0 & 1 & 0 & 1\\
      0 & 0 & 1 & 1
    \end{pmatrix}
  \end{equation}
  There is exactly one symmetry that point-wise stabilizes the first
  $2^{\dimension - 2}$ columns, namely flipping the last two
  coordinates.  Similarly, flipping either of the last two coordinates
  with the third-to-last coordinate yields exactly two symmetries that
  stabilize the first $2^{\dimension - 3}$ columns from the left, and
  so on.  For $\dimension = 6$ this yields $1$, $2$, $3$, $4$, $5$
  non-trivial switches in rows $17$, $9$, $5$, $3$, $2$, respectively,
  and $63$ non-trivial switches in row~$1$ (because of transitivity).
  Since $k = 46{,}080 = 2 \cdot 3 \cdot 4 \cdot 5 \cdot 6 \cdot 64$,
  there are no further non-trivial switches.  Therefore, $r = 17$.
  Moreover, for $m = 8$, one obtains
  $\density{1}{m} = \frac{8}{64} = \frac{1}{8}$,
  $\density{r}{m} = \frac{8}{48} = \frac{1}{6}$, and
  $\mu = \frac{63}{64}$.  Therefore,
  $\density{r}{m}^2 + \mu^2 = \frac{1}{6^2} + \frac{63^2}{64^2} \le
  1$.  With this, Corollary~\ref{thm:mst:bounds-transitive:upper}
  yields:
  \begin{align}
    \label{eqn:mst:bounds-transitive:upper:cube}
    \releffort{mst}{8}
    &\le
      \Biggl(
      1
      + \frac{1}{1} \cdot \frac{1}{2}
      + \frac{1}{2} \cdot \frac{2}{3}
      + \frac{1}{4} \cdot \frac{3}{4}
      + \frac{1}{8} \cdot \frac{4}{5}
      + \frac{1}{16} \cdot \frac{5}{6}
      + \frac{1}{32} \cdot \frac{63}{64}
      \biggr)
      \Biggr)
      \cdot \frac{1}{64}\\
    &<
      0.035.
  \end{align}
  Hence, the modified switch-table method for the lex-min checks for
  $8$-element subsets in the $6$-cube (relevant for circuits)
  typically needs less than $3.5\,\%$ of the effort of the naive
  method.  The simplified upper bound (which does not need the
  detailed data of the switch table) still yields
  $\releffort{mst}{8} \le \frac{3}{64} < 0.047$.  Compare this to the
  relative effort of the iteration-based critical-element method,
  which, by Theorem~\ref{thm:nve-cet:analysis}, is always
  $\frac{n + m - 1}{nm}$, in this case $\frac{71}{512} > 13.8\,\%$.
  So, it can be assumed a-priori that for $8$-element subsets in the
  $6$-cube the modified switch-table method will be typically at least
  almost four times as fast (in terms of single-element evaluations)
  as the critical-element method.  The superiority in terms of
  single-element evaluations of the modified switch-table method in
  the \emph{enumeration} of circuits in the $6$-cube (see
  Table~\ref{tab:lexmin-methods}) is much more pronounced: First,
  Corollary~\ref{thm:mst:bounds-transitive:upper} is not very tight in
  order to arrive at a simpler formula.  Second, during the
  enumeration most checks run on subsets with fewer than~$8$ elements
  (some circuits have fewer elements, and the enumeration has to check
  at least all lex-subsets of circuits), which reduces the speed-up of
  the critical-element method.  Third, in this example $m < r$, so
  that all branch-selection vectors with more than $m$ one-components
  lead to summands that are zero (which was ignored in the estimates
  above, since taking this into account leads to ugly binomial tails).
  The superiority of the modified switch-table method in terms of CPU
  times is much less apparent.  This can be attributed to the fact
  that the modified switch-table method incurs a larger effort for the
  adminstration of local data structures and the recursion compared to
  the structurally simple, purely iterative critical-element method.
  In particular, in the critical-element method many single-element
  evaluations are only read-outs from a consecutive array while in the
  modified switch-table method many single-element evaluations are
  used to generate new subsets, incurring a memory management
  overhead. \qed
\end{example}
It may seem that the modified switch-table method should be the
superior method in most cases.  However, in the following it is shown
that for large degrees and small orders so that the non-trivial
elements in the switch table are extremely sparse the critical-element
method can be substantially better.
\begin{corollary}
  \label{thm:mst:bounds:lower-mu}
  Let $\switchTable$ be a switch table for a subgroup~$\GrG$
  of~$\symGroup{n}$ with maximal effective row~$r$ and $c_i$
  non-trivial switches in row~$i$, $i = 1, 2, \dots, r$.

  Assume that there is a $\mu \in (0,1)$ such that
  \begin{itemize}
  \item $m \le \mu(n - r + 1)$
  \item $c_i + 1 \le \mu(n - i + 1 - m + 1)$ for all $i = 1, 2, \dots, r$.
  \end{itemize}
  Then, the relative effort $\releffort{mst}{m}$ of~$\switchTable$ is
  at least $\frac{3}{2} (1 - \mu)^{r(m + 1)}$, i.e.,
  \begin{equation}
    \label{eqn:mst:bound:lower-mu}
    \seemst(m) \ge \frac{3}{2} (1 - \mu)^{r(m + 1)} km.
  \end{equation}
\end{corollary}

\begin{proof}
  First, note that $m \le \mu(n - r + 1)$ implies
  $\codensity{i}{\ell} \ge \codensity{r}{m} \ge (1 - \mu)$.  Moreover,
  $c_i + 1 \le \mu(n - i + 1 - m + 1)$ implies
  $\transitivitygap{i}{\ell} \ge \transitivitygap{i}{m} \ge (1 - \mu)^m$.

  The starting point for the lower bound is now again
  Corollary~\ref{thm:mst:analysis:alternative}, where all
  branch-selection vectors are ignored except the zero vector.
  Moreover, since $c_i$ might be zero for $i = 2, 3, \dots, r - 1$
  only the expressions with branch-selection vectors of length~$r$ are
  taken into account.  With this one obtains:
  \begin{align}
    \seemst(m)
    &\ge
      \Bigl(1 + \frac{c_r}{c_r + 1}\Bigr)
      \branchweight{1}{m}{\vectorStyle{0}} \nodeweight{m}{\vectorStyle{0}}\\
    &\ge
      km\Bigl(1 + \frac{c_r}{c_r + 1}\Bigr)
      \prod_{t = 1}^r \nelreductionfactor{t}{m}\\
    &\ge
      km\Bigl(1 + \frac{c_r}{c_r + 1}\Bigr)
      \prod_{t = 1}^r (\codensity{t}{m} \transitivitygap{t}{m})\\
    &\ge
      km\Bigl(1 + \frac{c_r}{c_r + 1}\Bigr)
      \Bigl((1 - \mu) (1 - \mu)^m\Bigr)^r\\
    &\ge
      km\frac{3}{2}
      \Bigl((1 - \mu) (1 - \mu)^m\Bigr)^r\\
    &\ge
      \frac{3}{2} (1 - \mu)^{r(m + 1)} km.
  \end{align}
\end{proof}
Thus, for a very small $\mu$ it may happen that the exponential
speed-up in~$r$ of the modified switch-table method is outperformed by
the linear speed-up in~$m$ of the critical-element method.  For the
enumeration of triangulations this is the case in all instances
computed in this paper.  More specifically: The result explains why
for triangulations of the $4$-cube, the critical-element method is
faster, as is shown in the following example.
\begin{example}
  Consider the point configuration consisting of the set of vertices
  of the $\dimension$-dimensional hypercube~$\hypercube{\dimension}$,
  in particular for $\dimension = 4$.  If one is interested, e.g., in
  triangulations (see Section~\ref{sec:application-triangulations}),
  then the relevant action of its symmetry group of order~$k = 384$ is
  the action on the full-dimensional simplices.  Concrete computer
  calculations show the following. There are $3008$ such simplices.
  Thus, the degree here is~$n = 3008$.  Note that $384 < 3008$ already
  implies that the symmetry group is not acting transitively on
  simplices.  Moreover, for the order of points as in
  Example~\ref{exa:mst:bounds-transitive:upper:cube} and the
  lexicographic order on simplices a switch table has $r = 3$ with
  $c_1 = 15$, $c_2 = 11$, and $c_3 = c_r = 1$.  Finally,
  triangulations of~$\hypercube{4}$ have between $16$ and $24$
  simplices.  Corollary~\ref{thm:mst:bounds:lower-mu} can, therefore,
  be applied to this particular switch table, e.g., with
  $\mu = \frac{1}{100}$.  Then, in the best-case $m = 24$ for the
  modified switch-table method one obtains:
  \begin{equation}
    \releffort{mst}{24} \ge \frac{3}{2}\biggl(\frac{99}{100}\biggr)^{3 \cdot 25} > 70\,\%.
  \end{equation}
  In contrast to this, in the worst-case $m = 16$ for the
  iteration-based critical-element method one obtains:
  \begin{equation}
    \releffort{cet}{16} \le \frac{3008 + 16 - 1}{3008 \cdot 16} < 7\,\%.
  \end{equation}
  Thus, for triangulations one can a-priori guarantee that for the
  lex-min check on $16$- to $24$-element subsets of simplices in the
  $4$-cube the critical-element method will typically be at least $10$
  times as fast (in terms of single-element evaluations) as the
  modified switch-table method. Again, in
  Table~\ref{tab:lexmin-methods} it can be seen that inside the
  \emph{enumeration} of triangulations of the $4$-cube the
  critical-element method is only slightly more than $3$ times as fast
  (in terms of single-element evaluations) as the modified
  switch-table method, because of the many checks of smaller subsets
  of triangulations.  As before, in terms of CPU times the superiority
  of the critical-element method is more pronounced, probably because
  of its overall simpler implementation structure.\qed
\end{example}

So far, a switch table was considered given, and the input subsets
were considered random.  Now, a switch table on a random $k$-subset
$\GrG$ of permutations is discussed.  The suitable probability space
for this is $(\Omega, 2^{\Omega}, \probability{\cdot})$, where
$\Omega$ is the set of all $k$-subsets
$\setStyle{P} \in \tbinom{\symGroup{n}}{k}$ of permutations
$\pi \in \symGroup{n}$ with $\abs{\Omega} = \tbinom{n!}{k}$ and
$\probability{\setStyle{P}} = \frac{1}{\abs{\Omega}}$ for all
$\setStyle{P} \in \tbinom{\symGroup{n}}{k}$.

Motivated by the applications in this paper (see the examples in
Sections~\ref{sec:application-cocircuits}
through~\ref{sec:application-triangulations}), the focus in the
following is on orders~$k$ that are much smaller than the number $n!$
of all permutations in~$\symGroup{n}$.  For these cases, a binomial
approximation is added based on sampling $k$ permutations uniformly
and independently at random with replacement to generate a sequence
of $k$ (not necessarily distinct) permutations, which simplifies
formulas.
\begin{theorem}
  \label{thm:mst:exp-nontrivial-switches}
  Let $\GrG$ be a $k$-subset of permutations in~$\symGroup{n}$ that
  has been chosen uniformly at random from all such $k$-subsets.
  Then, the typical number $c_i$ of non-trivial switches in row~$i$
  for $i = 1, 2, \dots, n$ is given by
  \begin{equation}
    \label{eqn:mst:exp-nontrivial-switches}
    \expectation{c_i}
    =
    (n - i)
    \Biggl(
    1
    -
    \prod_{t = 0}^{k - 1}
    \biggl(1 - \frac{(n - i)!}{n! - t}\biggr)
    \Biggr)
  \end{equation}

  If $k$ is sufficiently smaller than~$n!$, then a binomial
  approximation of this is given as follows: Let $\GrG$ be a random
  sequence of $k$ permutations in~$\symGroup{n}$, where each permutation
  has been chosen uniformly and independently at random with
  replacement.  Then, the typical number $c_i$ of non-trivial switches
  in row~$i$ for $i = 1, 2, \dots, n$ is given by
  \begin{equation}
    \label{eqn:mst:exp-nontrivial-switches-binomial}
    \expectation{c_i}
    =
    (n - i)
    \Biggl(1 - \biggl(1 - \frac{1}{(n)_i}\biggr)^k\Biggr).
  \end{equation}
\end{theorem}

\begin{proof}
  Call $A_i(j)$ the set of all permutations $\pi \in \symGroup{n}$ that
  stabilize $1, 2, \dots, i - 1$ and map a given $j > i$ to~$i$.  Since
  for permutations in~$A_i(j)$ exactly $i$ images are fixed and the
  others are arbitrary, there are $(n - i)!$ permutations in~$A_i(j)$.
  The number of $k$-subsets of permutations that contain no permutation
  in~$A_i(j)$ is, therefore, $\tbinom{n! - (n - i)!}{k}$.  Let
  $\setsysStyle{C}_i(j)$ be the event containing those $k$-subsets of
  permutations for which row~$i$ and column~$j$ in a switch table
  of~$\GrG$ is non-trivial.  The cardinality of~$\setsysStyle{C}_i(j)$
  is $\tbinom{n!}{k} - \tbinom{n! - (n - i)!}{k}$.  Hence, the
  probability $\probability{\setsysStyle{C}_i(j)}$ that row~$i$ and
  column~$j$ of a switch table for a random $k$-subset of permutations
  is non-trivial can be computed as follows.
  \allowdisplaybreaks
  \begin{align}
    \probability{\setsysStyle{C}_i(j)}
    &=
      \frac{\binom{n!}{k} - \binom{n! - (n - i)!}{k}}{\binom{n!}{k}}\\
    &=
      1
      - 
      \frac{\binom{n! - (n - i)!}{k}}{\binom{n!}{k}}\\
    &=
      1
      -
      \frac{\bigl(n! - (n - i)!\bigr)_k}{(n!)_k}\\
    &=
      1
      -
      \prod_{t = 0}^{k - 1}
      \frac{n! - (n - i)! - t}{n! - t}\\
    &=
      1
      -
      \prod_{t = 0}^{k - 1}
      \frac{(n! - t) - (n - i)!}{n! - t}\\
    &=
      1
      -
      \prod_{t = 0}^{k - 1}
      \biggl(1 - \frac{(n - i)!}{n! - t}\biggr)\\
      \label{eqn:mst:prob-non-trivial-switches}
  \end{align}
  Using the indicator variable $\indVar{\setsysStyle{C}_i(j)}$ of
  $\setsysStyle{C}_i(j)$ and the linearity of expectation, the typical
  number of non-trivial switches $\expectation{c_i}$ in row~$i$ of a
  switch-table for~$\GrG$ is given by
  \begin{equation}
    \label{eqn:mst:expectation:c}
    \expectation{c_i}
    =
    \sum_{j = i + 1}^n \expectation{\indVar{\setsysStyle{C}_i(j)}}
    =
    \sum_{j = i + 1}^n \probability{\setsysStyle{C}_i(j)}
    =
    (n - i)
    \Biggl(1 - \prod_{t = 0}^{k - 1} \biggl(1 - \frac{(n - i)!}{n! - t}\biggr)\Biggr).
  \end{equation}
  This equals the first assertion.  If $k$ is sufficiently smaller
  than~$n!$, then the $k$ factors in the product are all approximately
  equal to $1 - \frac{(n - i)!}{n!} = 1 - \frac{1}{(n)_i}$, which
  yields the assertation about the binomial approximation.
\end{proof}
The formula shows that for small orders~$k$ and large degrees~$n$
there are typically only very few non-trivial switches in rows
$i > 1$, which leads to large level-$i$ transitivity gaps and,
therefore, a smaller performance gain of a switch table.

For the following bound on the maximal effective row the
considerations are based on the binomial approximation.  It shows that
in the absence of special structure the maximal effective row $r$ is
usually small when the degree $n$ is large compared to the order $k$
of the symmetry group.  Of course, in such cases $k$ is, in
particular, much smaller than $n!$, so that a binomial approximation
is justified.  Recall that for a small $r$ the linear overhead of a
switch table may outweigh the exponential gain; therefore, this is an
important finding.
\begin{theorem}
  \label{thm:mst:typical-r}
  The typical maximal effective row $r$ of a switch table for a
  sequence of $k$ permutations of degree~$n$ chosen uniformly and
  independently at random with replacement is at most
  $1 + \frac{2k}{n}$. In particular, whenever the degree is at least
  twice the order, then the typical maximal effective row~$r$ is at
  most two.
\end{theorem}
Note that a group may have special structure leading to more effective
rows in a switch table.  The claim is about the hyper-amortized
average for $k \ll n!$ only and provides a hint what to expect in the
long run over many instances.
\begin{proof}
  Call a random permutation \emph{$i$-non-trivial} for
  $i = 2, \dots, n$ if it could be chosen as a non-trivial switch in
  row $i$ or larger of a switch table.  Note that the probability
  $\probability{\pi \text{ $i$-non-trivial}}$ is given by the
  probability that $\pi$ stabilizes $1, 2, \ldots, i - 1$, which
  yields
  \begin{equation}
    \probability{\pi \text{ $i$-non-trivial}} = \frac{1}{n (n - 1) \cdots (n - i + 2)}.
  \end{equation}
  Moreover, the number $b_i$ of $i$-non-trivial permutations is
  typically
  $\expectation{b_i} = k \probability{\pi \text{ $i$-non-trivial}}$
  (approximation by the expectation in a binomial distribution).  The
  maximal effective row $r$ is at least~$i$ if and only if there is at
  least one $i$-non-trivial permutation.  Consequently:
  \allowdisplaybreaks
  \begin{equation}
    \expectation{b_i}
    =
    \sum_{j = 0}^k j \probability{b_i = j}
    \ge
    \sum_{j = 1}^k \probability{b_i = j}
    =
    \probability{b_i \ge 1}
    =
    \probability{r \ge i}.
  \end{equation}
  This leads to the following estimate:
  \begin{align}
    \label{eq:estimate-effrow-number}
    \expectation{r}
    &= \sum_{i = 1}^n i \probability{r = i}\\
    &= \sum_{i = 1}^n \probability{r \ge i}\\
    &= 1 + \sum_{i = 2}^n \probability {r \ge i}\\
    &\le 1 + \sum_{i = 2}^n \expectation{b_i}\\
    &= 1 + \sum_{i = 2}^n k \probability{\pi \text{ $i$-non-trivial}}\\
    &= 1 + k \sum_{i = 2}^n \frac{1}{n (n - 1) \cdots (n - i + 2)}\\
    &= 1 + \frac{k}{n} + \frac{k}{n} \sum_{i = 3}^n \frac{1}{(n - 1) \cdots (n - i + 2)}\\
    & \le 1 + \frac{k}{n} \sum_{i = 2}^n \left(\frac{1}{2}\right)^{i - 2}\\
    & = 1 + \frac{k}{n} \sum_{i = 0}^{n - 2} \left(\frac{1}{2}\right)^i\\
    & \le 1 + \frac{2k}{n}.
  \end{align}
  For the special case $n \ge 2k$, consequently,
  $\expectation{r} \le 2$.
\end{proof}



The key learning of the previous theorems beyond the complicated
formulas is that the typical efficiency of the modified switch-table
method becomes better (ceteris paribus) as the order increases
compared to the degree, and vice versa.

The analysis is qualitatively confirmed by the computational
experiments for the applications in this paper.
Table~\ref{tab:lexmin-methods} shows for the circuits of the $6$-cube
$\hypercube{6}$ and the triangulations of the $4$-cube $\hypercube{4}$
the qualitative consistency of
Corollaries~\ref{thm:mst:bounds-transitive:upper} and
\ref{thm:mst:bounds:lower-mu} with computational experience
single-threaded on an M1Max machine (for details on the computational
environment see Section~\ref{sec:applications-environment}).
\begin{table}[htbp]
  \centering
  {%
    \sffamily\footnotesize
    \begin{tabular}{l*{4}{r@{\:\:\:\:}}*{6}{r@{\:\:\:\:}}}
      \toprule
      comput.
      & $n$
      & $k$
      & $r$
      & $m$
      & \multicolumn{3}{c}{CPU time [s]}
      & \multicolumn{3}{c}{\# realized single-element evaluations}\\
      &
      &
      &
      &
      & \multicolumn{1}{c}{\codeStyle{cei}} & \multicolumn{1}{c}{\codeStyle{ces}} & \multicolumn{1}{c}{\codeStyle{mst}}
      & \multicolumn{1}{c}{\codeStyle{cei}} & \multicolumn{1}{c}{\codeStyle{ces}} & \multicolumn{1}{c}{\codeStyle{mst}}\\
      \midrule
      $\hypercube{6}$ circuits
      &   64 & 46{,}080 & 17 &   4--8 & 14 & 7 &  3 & 4{,}259{,}939{,}467 & 498{,}266{,}262 &      50{,}735{,}629\\
      $\hypercube{4}$ triang's
      & 3008 &      384 &  3 & 16--24 &  5 & - & 32 &     796{,}884{,}547 &               - & 2{,}673{,}110{,}644\\
      \bottomrule
    \end{tabular}
  }
  \caption[Comparison of lex-min checks]{Spot-light comparison of
    methods for lex-min checks \codeStyle{cei} =
    ``critical-element method via iteration'', \codeStyle{ces} =
    ``critical-element method via sets'',  and \codeStyle{mst} = ``modified switch-table
    method'' (single-threaded)
  }
  \label{tab:lexmin-methods}
\end{table}
\begin{remark}
  The analysis in this section neglects the necessary
  memory-management overhead for non-trivial temporary local
  data-structures for subsets.  In \texttt{TOPCOM}, e.g., this
  overhead -- which in naive implementations can dominate anything
  else -- is more prevalent in the modified switch-table method than
  in the critical-element method.
\end{remark}

\subsection{Parallelization}
\label{sec:parallelization}

The enumeration of trees is much easier to parallelize than general
enumeration, since the enumerations of distinct subtrees do not
influence each other.

How to explore reverse-search trees in parallel using multiple
processes with only little inter-process communication was proposed in
\cite{AvisJordan_mtslightframework_2021}.  The core idea is
\emph{budgeted load balancing}.  That means, a coordinating process
assigns to each worker process the root of a subtree and a limited
budget of nodes to enumerate in that subtree.  The enumeration results
in that subtree roots are stored with the worker process.  After the
completion of its budget, the worker process returns the control to
the coordinating process, which in turn merges the results from the
worker process, including any unprocessed nodes in the subtree, to the
results obtained so far.  The coordinating process then assigns to
each idle worker process a new subtree root and a new budget.  Note
that, because of independent enumerations in each subtree, there can
be no guarantee anymore for the order in which the objects of interest
are output.  However, the ``is-canonical'' check implemented via the
lex-min check in \symLSRS remains correct since lex-minimality and,
thus, canonicalicity does not depend on when a subset orbit was found
for the first time.

The big advantage of tree-based parallelization is that such a
parallelization scheme can carry out the actual enumeration and the
merge of each worker's results completely lock-free.  (The experience
with various load-balancing mechanisms shows that a lock-free
implementation is the single most important success factor for fast
parallelizations of enumeration algorithms like the ones in this
paper.)

\texttt{TOPCOM} uses a modified version of this paradigm using a lean shared
memory for multiple threads in one process instead of multiple
processes: the coordinating thread updates a count of currently
unassigned subtree roots.  Each worker thread reads out this count
after each discovery of a new node.  If the count falls below the
number of worker threads (a threshold that can be configured), the
worker thread stops and returns its results -- including the
unprocessed nodes in the subtree -- as well as the control to the
coordinating thread.  The read-out of the count can be done lock-free
because it does not matter at which node exactly the worker thread
interrupts.

In \texttt{\texttt{TOPCOM}} the new method is called \emph{workbuffer control}.  It was
used throughout the computational experiments.  The advantage of
workbuffer control compared to budgeted load balancing is that the
number of unprocessed subtree roots (which must be stored in the
coordinating process) turns out to be much smaller for large problems.
Moreover, the worker processes usually stop less often unnecessarily
this way.  The cpu times of budgeted load balancing versus workbuffer
control did not show any significant differences in the examples of
this paper.

\section{Applications: Common Preliminaries}
\label{sec:applications-preliminaries}

In this section, some basic notions and notation are summarized for
the common setup of the applications presented in the upcoming
sections.

Consider a point or vector configuration of rank~$\rank$ with $\no$
points or vectors.  It is represented by an $\rank \times \no$-matrix
$\Conf$ containing the coordinates
$\col_1, \dots, \col_{\no} \in \mathbb{R}^{\rank}$ of the points (as
homogeneous coordinates) or vectors as columns.  The number
$\corank := \no - \rank$ is the \emph{corank} of~$\Conf$.  For a
subset $\setStyle{S}$ of column indices of~$\Conf$ denote by
$\subConf{S}$ the submatrix consisting of the columns with increasing
indices in~$\setStyle{S}$.  Moreover, for a matrix $\matrixStyle{M}$
in column-echelon form denote by $\matrixStyle{M}_{\mathrm{NZ}}$ its
submatrix of non-zero columns.

The usual language of linear algebra can be adapted to subsets of
indices as follows.
\begin{definition}
  \label{def:subset-properties}
  A subset $\setStyle{B} \in \indexSet{\no}$ is \emph{spanning} if
  $\rankOf(\subConf{B}) = \rank$, otherwise it is \emph{not-spanning}
  or \emph{coplanar}.  A coplanar subset $\setStyle{B}$ with
  $\rankOf{\subConf{B}} = \rank - 1$ is \emph{hyperplanar}. It is
  \emph{independent} if $\rankOf(\subConf{B}) = \abs{\setStyle{B}}$,
  otherwise it is \emph{dependent}. A \emph{basis} is an independent
  spanning subset.  A \emph{simplex} (in the context of
  triangulations) is a basis $\setStyle{S}$ such that
  $\cone(\subConf{S})$ is a pointed polyhedral cone, i.e., the origin
  is a vertex of it.  An $\rank - 1$-subset of a simplex is a
  \emph{simplex facet} or simply a \emph{facet} whenever no confusion
  with facets of~$\Conf$ can arise.
\end{definition}

The first application deals with cocircuits, which can be seen as
hyperplanes spanned by the configuration.
\begin{definition}
  \label{def:cocircuits}
  Any inclusion-maximal coplanar subset~$\setStyle{C}^*_0$ of column
  indices of the configuration~$\Conf$ is called \emph{(the zero part
    of) a cocircuit}.  A \emph{cocircuit signature}
  of~$\setStyle{C}^*_0$ is a map
  $\signature^*\colon \indexSet{\no} \to \{-, 0, +\}$ with
  $\setStyle{C}^*_0 = \signature^{-1}(\{0\})$ so that there is a
  $\vectorStyle{c} \in \mathbb{R}^{\rank}$ with
  $\vectorStyle{c}^{\top}\col_i = 0$ for all
  $i \in (\signature^*)^{-1}(\{0\})$, $\vectorStyle{c}^{\top}\col_i > 0$
  for all $i \in (\signature^*)^{-1}(\{+\})$, and
  $\vectorStyle{c}^{\top}\col_i < 0$ for all
  $i \in (\signature^*)^{-1}(\{-\})$. Here,
  $\setStyle{C}^*_+ := (\signature^*)^{-1}(\{+\})$ is called the
  \emph{positive part}, and
  $\setStyle{C}^*_- := (\signature^*)^{-1}(\{-\})$ is called the
  \emph{negative part} of the cocircuit signature~$\signature^*$.  By
  elementary linear algebra, there are exactly two \emph{opposite}
  cocircuit signatures of~$\setStyle{C}^*_0$.  Moreover, these two
  signatures are uniquely determined by any hyperplanar subset
  of~$\setStyle{C}^*_0$.  The pair
  $(\setStyle{C}^*_+, \setStyle{C}^*_-)$ is a \emph{signed cocircuit}.
\end{definition}
Intuitively, a cocircuit is the subset of \emph{all} elements lying on
a hyperplane spanned by \emph{some} of the elements in~$\Conf$.
Counting all cocircuits is, therefore, the same as counting all
hyperplanes spanned by elements of the configuration.

Another application is concerned with circuits, which can be seen as
unique intersection points of subpolytopes spanned by the
configuration.
\begin{definition}
  \label{def:circuits}  
  Any inclusion-minimal dependent subset $\setStyle{C}$ of the column
  indices of~$\Conf$ is \emph{(the support of) a circuit}.  A
  \emph{circuit signature} of a circuit $\setStyle{C}$ is a map
  $\signature\colon \indexSet{\no} \to \{-,0,+\}$ so that
  $\setStyle{C} = \signature^{-1}(\{-, +\})$ and
  $\sum_{i \in \signature^{-1}(\{+\})} \lambda_i \col_i = \sum_{i \in
    \signature^{-1}(\{-\})} \lambda_i \col_i$ for suitable
  $\lambda_i > 0$, $i = 1, 2, \dots, n$.  Here,
  $\setStyle{C}_+ := \signature^{-1}(\{+\})$ is called the
  \emph{positive part}, $\setStyle{C}_- := \signature^{-1}(\{-\})$ the
  \emph{negative part}, and
  $\setStyle{C}_0 := \signature^{-1}(\{ 0\})$ the \emph{zero-part} of
  the circuit signature~$\signature$.  By elementary linear algebra,
  there are exactly two \emph{opposite} circuit signatures
  of~$\setStyle{C}$.  The pair $(\setStyle{C}_+, \setStyle{C}_-)$ is a
  \emph{signed circuit}.
\end{definition}

Symmetries of a configuration are the (co-)circuit-maintaining
permutations of elements.
\begin{definition}
  \label{def::symmetry-group}
  A permutation $\grg \in \symGroup{n}$ is a \emph{(combinatorial)
    symmetry of $\Conf$}, if the following holds for each pair
  $(\setStyle{S}, \setStyle{R})$ of subsets of $\indexSet{n}$:
  $(\setStyle{S}, \setStyle{R})$ is a signed cocircuit of~$\Conf$ if
  and only if $\bigl(\grg(\setStyle{S}), \grg(\setStyle{R})\bigr)$ is
  a signed cocircuit of~$\Conf$, or, equivalently,
  $(\setStyle{S}, \setStyle{R})$ is a signed circuit of~$\Conf$ if and
  only if $\bigl(\grg(\setStyle{S}), \grg(\setStyle{R})\bigr)$ is a
  signed circuit of~$\Conf$.  A symmetry is \emph{affine} if it is
  induced by an affine isomorphism $f: \aff(\Conf) \to \aff(\Conf)$
  via $\grg(i) = j$ if and only if $f(\col_i) = \col_j$.  The
  \emph{(combinatorial) symmetry group
    $\autGroup(\Conf) = \autGroupComb(\Conf)$ of~$\Conf$} is the group
  of all (combinatorial) symmetries of~$\Conf$.  The \emph{affine
    symmetry group $\autGroupAff(\Conf)$} is the group of all affine
  symmetries.
\end{definition}
By this definition, it makes sense to enumerate all (co-)circuits
of~$\Conf$ up to (combinatorial) symmetry.

The third application in this paper deals with triangulations
of~$\Conf$.  It is advantageous to use a well-known characterization
of triangulations of point configurations as the definition
\cite[Cor.~4.1.32]{DeLoeraRambauSantos_TriangulationsStructuresApplications_2010}.
\begin{definition}
  \label{def:triangulations}
  Two simplices $\setStyle{S}_1$ and $\setStyle{S}_2$ are
  \emph{intersecting properly} if there is no signed circuit
  $\setStyle{C}$ with $\setStyle{C}_+ \subseteq \setStyle{S}_1$ and
  $\setStyle{C}_- \subseteq \setStyle{S}_2$.  A simplex facet
  $\setStyle{F}$ is \emph{interior} if there is a cocircuit
  $\setStyle{C}^*_0$ with $\setStyle{F} \subseteq \setStyle{C}^*_0$ so
  that $\setStyle{C}^*_+$ and $\setStyle{C}^*_-$ are both non-empty.
  A non-empty subset $\setsysStyle{T}$ of simplices is \emph{covering}
  if for each interior facet $\setStyle{F}$ of some simplex
  $\setStyle{S} \in \setsysStyle{T}$ there is another simplex
  $\setStyle{S}' \in \setsysStyle{T}$ containing~$\setStyle{F}$.  A
  \emph{triangulation} is a non-empty covering subset
  $\setsysStyle{T}$ of pairwise properly intersecting simplices.
\end{definition}
Proper intersection of two simplices $\setStyle{S}_1$ and
$\setStyle{S}_2$ roughly means geometrically that the convex hulls
$\conv \subConf{S_1}$ and $\conv \subConf{S_2}$ intersect in a common
(possibly empty) face of both.  The covering property together with
proper intersection roughly means geometrically that all facets of a
simplex in $\setsysStyle{T}$ interior in~$\Conf$ are covered by
simplices from boths sides in $\setsysStyle{T}$.

Whether or not a subset of simplices of~$\Conf$ is a triangulation
depends only on the combinatorics of~$\Conf$ as given by its set of
(co-)circuits
\cite[Thm.~4.1.31]{DeLoeraRambauSantos_TriangulationsStructuresApplications_2010}.
Therefore:
\begin{lemma}
  \label{thm:triangulation-symmetries}
  For any combinatorial symmetry $\grg$ of~$\Conf$ and any subset
  $\setStyle{T}$ of~$\indexSet{n}$ one has that $\setStyle{T}$ indexes
  a triangulation of~$\Conf$ if and only if the set
  $\grg(\setStyle{T})$ indexes a triangulation of~$\Conf$.\qed
\end{lemma}
By this lemma, it makes sense to enumerate all triangulations
of~$\Conf$ up to (combinatorial) symmetry.  Note that additional
restrictions on the triangulations to be enumerated (like regularity
or unimodularity) can break this property.  In that case, the
symmetries have to be restricted to such symmetries that maintain the
required structures imposed on the triangulations to be enumerated.
To enumerate regular triangulations up to symmetry the symmetries have
to be restricted to affine symmetries (maintaining the convexity of
liftings), and for the enumeration of unimodular triangulations the
symmetries have to be restricted to isometric symmetries (maintaining
the volumes of simplices).  In Section~\ref{sec:triangs:enhancements}
further examples for interesting restrictions can be found that
require a restriction of the combinatorial symmetries of~$\Conf$ to a
smaller subgroup.

\begin{table}[htbp]
  \sffamily\footnotesize
  \centering
  \begin{tabular}{lrrr}
    \toprule
    $\Conf$ & dimension & \# points & \# symmetries \\
    \midrule
    $\hypercube{5}$          &      5 &  32 &       38,040\\
    $\hypercube{6}$          &      6 &  64 &       46,080\\
    $\hypercube{7}$          &      7 & 128 &      645,120\\
    $\hypercube{8}$          &      8 & 256 &   10,321,920\\
    $\hypercube{9}$          &      9 & 512 &  185,794,560\\
    \midrule
    $\hypersimplex{8}{3}$    &      7 &  56 &       40,320\\
    $\hypersimplex{8}{4}$    &      7 &  70 &       40,320\\
    $\hypersimplex{9}{3}$    &      8 &  84 &      362,880\\
    $\hypersimplex{9}{4}$    &      8 & 126 &      362,880\\
    $\hypersimplex{10}{3}$   &      9 & 120 &    3,628,800\\
    $\hypersimplex{10}{4}$   &      9 & 210 &    3,628,800\\
    $\hypersimplex{10}{5}$   &      9 & 252 &    3,628,800\\
    \bottomrule
  \end{tabular}
  \caption[Parameters of large configurations]{Parameters of the
    configurations for the enumeration of (co-)circuits}
  \label{tab:conf-large-parameters}
\end{table}

\begin{table}[htbp]
  \sffamily\footnotesize
  \centering
  \begin{tabular}{lrrr}
    \toprule
    $\Conf$ & dimension & \# points & \# symmetries \\
    \midrule
    $\hypercube{4}$          &      4 &  16 &          384\\
    \midrule
    $\cyclic{n}{2k}$         &   $2k$ & $n$ &         $2n$\\
    $\cyclic{n}{2k-1}$       & $2k-1$ & $n$ &          $2$\\
    \midrule
    $\hypersimplex{6}{3}$    &      5 &  20 &         1440\\
    $\hypersimplex{7}{2}$    &      6 &  21 &         5040\\
    \midrule
    $\simplexproduct{6}{2}$  &      8 &  21 &       30,240\\
    $\simplexproduct{4}{3}$  &      7 &  20 &         2880\\
    $\simplexproduct{5}{3}$  &      8 &  24 &       17,280\\
    $\simplexproduct{4}{4}$  &      8 &  25 &       28,800\\
    \midrule
    $3\standardsimplex{3}$   &      3 &  20 &           24\\
    $4\standardsimplex{3}$   &      3 &  35 &           24\\
    $3\standardsimplex{4}$   &      4 &  35 &          120\\
    \midrule
    icosahedron              &      3 &  12 &          120\\
    pseudo-icosahedron       &      3 &  12 &           24\\
    dodecahedron             &      3 &  20 &          120\\
    pyritohedron             &      3 &  20 &           24\\
    \midrule
    $\rootpoly{2}$           &      2 &   7 &           12\\
    $\rootpoly{3}$           &      3 &  13 &           48\\
    $\rootpoly{4}$           &      4 &  21 &          240\\
    $\rootpoly{5}$           &      5 &  31 &         1440\\
    \midrule
    $\santospointconfseventeen{0}$ &      6 &  17 &    128\\
    $\santospointconftwentysix$    &      5 &  26 &     48\\
    $\santospointconftwentysixmod$ &      5 &  26 &     48\\
    \bottomrule
  \end{tabular}
  \caption[Parameters of small configurations]{Parameters of the
    configurations for the enumeration of triangulations}
  \label{tab:conf-small-parameters}
\end{table}

In the computational results, the following point configurations are
mentioned.  Since some results depend on the order of points, we also
mention the order used in this paper.  The homogenization of
coordinates is not explicitly mentioned in the following:
\begin{itemize}
\item The $d$-dimensional point configuration $\hypercube{d}$ is the
  $d$-dimensional \emph{hypercube}.  The $2^d$ points are ordered in
  this paper such that
  $\hypercube{d} = \bigl(\begin{smallmatrix}\hypercube{d-1} &
    \hypercube{d-1}\\ 0 & 1\end{smallmatrix}\bigr)$ with
  $\hypercube{0} := ()$ (which can be interpreted as the only point in
  $\mathbb{R}^0$).
\item The $d$-dimensional point configuration $\cyclic{n}{d}$ is the
  $d$-dimensional \emph{cyclic polytope} with~$n$ vertices.  The
  points are ordered in this paper with increasing first coordinate.
\item The $n+n'$-dimensional point configuration
  $\simplexproduct{n}{n'}$ is the \emph{product of simplices} of
  dimensions $n$ and~$n'$.  The $(n+1)(n'+1)$ points are ordered in
  this paper such that
  $\simplexproduct{n}{n'} = \bigl( \begin{smallmatrix} I_{n+1} &
    I_{n+1} & \dots & I_{n+1}\\e_1 & e_2 & \dots &
    e_{n'+1} \end{smallmatrix}\bigr)$.
\item The $d$-dimensional point configuration $\hypersimplex{d}{k}$ is
  the $d$-dimensional \emph{hypersimplex} with coordinate sum~$k$.
  The points are ordered lexicographically in this paper.
\item The $d$-dimensional point configuration $k\standardsimplex{d}$
  is the \emph{dilated simplex} in dimension~$d$ with \emph{dilation
    factor}~$k$; it may have interior points.  The points are ordered
  lexicographically in this paper.  The number of
  subregular/regular/unimodular triangulations of
  $3\standardsimplex{3}$ were computed for the first time
  in~\cite{JordanJoswigKastner_Parallelenumerationtriangulations_2018}
  with the help of high-performance computing with \texttt{mptopcom}
  -- the largest number of triangulations computed at that time.
\item The $3$-dimensional point configurations \emph{icosahedron} and
  \emph{dodecahedron} consist of the vertices of the regular polytopes
  corresponding to the respective Platonic solids with $12$ and $20$
  vertices, respectively.  There are no rational coordinates for them,
  but all their (co-)circuits can still be computed exactly using
  algebraic field extensions.  There are no results in this paper that
  depend on the order of points.
\item The $3$-dimensional point configuration
  \emph{pseudo-icosahedron} consists of the vertices of an
  approximation of the regular icosahedron by rational coordinates
  generated by ratios of Fibonacci-numbers approximating the golden
  ratio.  In this paper, the exact coordinates are chosen and ordered
  as follows:
  \begin{equation*}
    \begin{psmallmatrix*}[r]
      0 & 0 & 21 & -21 & 13 & -13 & 13 & -13 & 21 & -21 & 0 & 0\\
      21 & 21 & 13 & 13 & 0 & 0 & 0 & 0 & -13 & -13 & -21 & -21\\ 
      13 & -13 & 0 & 0 & 21 & 21 & -21 & -21 & 0 & 0 & 13 & -13
    \end{psmallmatrix*}
  \end{equation*}
\item The $3$-dimensional point configuration \emph{pyritohedron}
  consists of the vertices of a \emph{pseudo-dodecahedron} in the same
  sense known from cristallography.  In this paper, the exact
  coordinates are chosen and ordered as follows:
  \begin{equation*}
    \begin{psmallmatrix*}[r]
      -1 & 1 & -1 & -1 & 0 & -\sfrac{3}{4} & -\sfrac{3}{2} & 0 & -\sfrac{3}{4} & \sfrac{3}{2} & -\sfrac{3}{2} & \sfrac{3}{4} & 0 & \sfrac{3}{2} & \sfrac{3}{4} & 0 & 1 & 1 & -1 & 1 \\
      -1 & -1 & -1 & 1 & -\sfrac{3}{4} & -\sfrac{3}{2} & 0 & -\sfrac{3}{4} & \sfrac{3}{2} & 0 & 0 & -\sfrac{3}{2} & \sfrac{3}{4} & 0 & \sfrac{3}{2} & \sfrac{3}{4} & -1 & 1 & 1 & 1 \\
      -1 & -1 & 1 & -1 & -\sfrac{3}{2} & 0 & -\sfrac{3}{4} & \sfrac{3}{2} & 0 & -\sfrac{3}{4} & \sfrac{3}{4} & 0 & -\sfrac{3}{2} & \sfrac{3}{4} & 0 & \sfrac{3}{2} & 1 & -1 & 1 & 1
    \end{psmallmatrix*}
  \end{equation*}
  Its symmetry group of order $24$ contains certain rotations by $120$
  degrees (later referred to as $\mathbb{Z}_3$-symmetry), certain
  rotations by $180$ degrees, and the central reflection at the origin
  (later referred to as $\mathbb{Z}_2$-symmetry).
\item The \emph{$n$-dimensional full root polytope} consists of the
  vertices of the $n$-polytope
  $\rootpoly{n} = \conv\{e_j - e_i: 1 \le i, j \le n+1\} \subset
  \mathbb{R}^{n+1}$ in ambient $n+1$-space together with the origin.
  The points are ordered in this paper as $(e_j - e_i , e_i - e_j)$
  for $(i, j)$ ordered lexicographically and transformed to full rank
  by Gaussian elimination.  Of particular interest are the numbers of
  \emph{central} triangulations of full root polytopes, i.e., the ones
  in which each simplex contains the origin.  For details about the
  enumeration of triangulations, see
  Section~\ref{sec:application-triangulations}. In
  \cite{JoswigSchroeter_ParametricShortestPath_2022} these numbers are
  computed up to~$\rootpoly{4}$ to classify the combinatorial types of
  \emph{polytropes}.  This classification was achieved earlier in
  \cite{Tran_Enumeratingpolytropes_2017} via Gröbner fans. Among the
  central triangulations of full root polytopes, the centrally
  symmetric ones have attracted recent attention as well.  No results
  have been published so far about~$\rootpoly{5}$.  The new TOPCOM
  enumeration results for the number of central and centrally
  symmetric triangulations of~$\rootpoly{5}$ presented in
  Section~\ref{sec:triangs:enhancements} have meanwhile been utilized
  for the classification of certain finite metrics in
  \cite{DelucchiKuehneMuehlherr_Combinatorialinvariantsfinite_2024},
  where also the relevance of these numbers is explained.
\item The \emph{Santos-dim6-n17 configuration
    $\santospointconfseventeen{t}$} is a point configuration
  parametrized by $t \in \mathbb{R}$ that consists of $17$ points in
  dimension six.  The configuraion is interesting because Santos has
  shown that it has at least nine flip-graph components
  \cite{Santos_Geometricbistellarflips_2006} (see
  \cite[Sec.~7.4.2]{DeLoeraRambauSantos_TriangulationsStructuresApplications_2010}
  for a version with more explanations) by constructing a
  triangulation and eight equivalent versions of it that cannot be
  flipped to each other, and, thus, not to any regular triangulation.
  That is, an enumeration of its triangulations based on flip-graph
  exploration is impossible.  In this paper, only the configuration
  for parameter value $t = 0$ is investigated.  The resulting
  coordinates are chosen as rational numbers and ordered as follows:
  \begin{equation*}
    \begin{psmallmatrix*}[r]
      1 & 1 & 1 & 1 & 1 & 1 & 1 & 1 & 1 & 1 & 1 & 1 & 1 & 1 & 1 & 1 & 1 \\
      1 & 0 & 0 & 0 & 1 & 0 & 0 & 0 & -1 & 0 & 0 & 0 & -1 & 0 & 0 & 0 & 0 \\ 
      0 & 1 & 0 & 0 & 0 & 1 & 0 & 0 & 0 & -1 & 0 & 0 & 0 & -1 & 0 & 0 & 0 \\
      0 & 0 & 1 & 0 & 0 & 0 & 1 & 0 & 0 & 0 & -1 & 0 & 0 & 0 & -1 & 0 & 0 \\
      0 & 0 & 0 & 1 & 0 & 0 & 0 & 1 & 0 & 0 & 0 & -1 & 0 & 0 & 0 & -1 & 0 \\ 
      17 & 7 & -7 & -17 & -17 & -7 & 7 & 17 & 17 & 7 & -7 & -17 & -17 & -7 & 7 & 17 & 0 \\ 
      7 & 17 & 17 & 7 & -7 & -17 & -17 & -7 & 7 & 17 & 17 & 7 & -7 & -17 & -17 & -7 & 0
    \end{psmallmatrix*}
  \end{equation*}
  This configuration has $128$ symmetries. The automorphism group of
  Santos's triangulation of order $32$ is generated by the two
  permutations
  \begin{align*}
    (5,4,11,2,1,0,15,6,13,12,3,10,9,8,7,14,16)\\
    (14,7,0,1,10,3,4,5,6,15,8,9,2,11,12,13,16)
  \end{align*}
  and is denoted by~$\santossymgroupseventeen$ for the purpose of this
  paper.  In this paper, the first computational confirmation of
  Santos's result about the disconnected flip-graph is provided.
  Moreover, its number of triangulations is computed using the new
  extension-based method (which for disconnected flip-graphs is the
  only method as of today).
\item The \emph{Santos-dim5-n26 configuration
    $\santospointconftwentysix$} is a point configuration that
  consists of $26$ points in dimension five.  The configuration is
  interesting because Santos has claimed in
  \cite{Santos_Nonconnectedtoric_2005} that it has at least 17
  flip-graph components by constructing a non-regular triangulation an
  16 equivalent versions of it that cannot be flipped to each other,
  and, thus, not to any regular triangulation.  This led to the first
  known non-connected toric Hilbert-schemes.  The coordinates are the
  following:
  \begin{equation*}
    \begin{psmallmatrix*}[r]
      1 & 1 & 1 & 1 & 1 & 1 & 1 & 1 & 1 & 1 & 1 & 1 & 1 & 1 & 1 & 1 & 1 & 1 & 1 & 1 & 1 & 1 & 1 & 1 & 1 & 1 \\
      1 & 0 & 0 & 0 & -1 & 0 & 0 & 0 & 1 & 0 & 0 & 0 & -1 & 0 & 0 & 0 & 0 & 0 & \sfrac{1}{2} & -\sfrac{1}{2} & \sfrac{1}{2} & 0 & -\sfrac{1}{2} & -\sfrac{1}{2} & 0 & \sfrac{1}{2} \\
      0 & 1 & 0 & 0 & 0 & -1 & 0 & 0 & 0 & 1 & 0 & 0 & 0 & -1 & 0 & 0 & 0 & 0 & \sfrac{1}{2} & -\sfrac{1}{2} & -\sfrac{1}{2} & \sfrac{1}{2} & 0 & \sfrac{1}{2} & -\sfrac{1}{2} & 0 \\
      0 & 0 & 1 & 0 & 0 & 0 & -1 & 0 & 0 & 0 & 1 & 0 & 0 & 0 & -1 & 0 & 0 & 0 & \sfrac{1}{2} & -\sfrac{1}{2} & 0 & -\sfrac{1}{2} & \sfrac{1}{2} & 0 & \sfrac{1}{2} & -\sfrac{1}{2} \\
      0 & 0 & 0 & 1 & 0 & 0 & 0 & -1 & 0 & 0 & 0 & 1 & 0 & 0 & 0 & -1 & 0 & 0 & 0 & 0 & -\sfrac{1}{2} & -\sfrac{1}{2} & -\sfrac{1}{2} & \sfrac{1}{2} & \sfrac{1}{2} & \sfrac{1}{2} \\ 
      0 & 0 & 0 & 0 & 0 & 0 & 0 & 0 & 1 & 1 & 1 & 1 & 1 & 1 & 1 & 1 & 0 & 1 & \sfrac{1}{2} & \sfrac{1}{2} & \sfrac{1}{2} & \sfrac{1}{2} & \sfrac{1}{2} & \sfrac{1}{2} & \sfrac{1}{2} & \sfrac{1}{2}
    \end{psmallmatrix*}
  \end{equation*}
  This configuration has $48$ symmetries generated by cyclic
  permutation of coordinates $x_2$, $x_3$, and $x_4$, by the
  simultaneous 90-degree-rotations in the $x_2$-$x_3$-plane and the
  $x_4$-$x_5$-planes, as well as the central symmetry. The
  automorphism group of Santos's triangulation of order $24$ is
  generated by all these symmetries except the central symmetry and is
  denoted by~$\santossymgrouptwentysix$ for the purpose of this paper.
  In this paper, however, the computations show that all
  triangulations of $\santospointconftwentysix$ with prescribed
  symmetries in~$\santossymgrouptwentysix$ can be flipped to a regular
  triangulation.  Further experiments showed that changing some of the
  $\sfrac{1}{2}$-coordinates to, e.g., $\sfrac{1}{3}$ destroys the computed flip paths.
  Call the modified point configuration~$\santospointconftwentysixmod$.
  Meanwhile, Santos (personal communication) found that there is a
  flaw in only one spot in the proof in
  \cite{Santos_Nonconnectedtoric_2005} that can be rectified exactly
  by these $\sfrac{1}{2}\to\sfrac{1}{3}$-tweaks.  The draw-back of the corrected
  construction is that -- in contrast to the original one -- there is
  no obvious unimodular version anymore.  The modified coordinates of
  $\santospointconftwentysixmod$ are the following (symmetries are
  unchanged):
  \begin{equation*}
    \begin{psmallmatrix*}[r]
      1 & 1 & 1 & 1 & 1 & 1 & 1 & 1 & 1 & 1 & 1 & 1 & 1 & 1 & 1 & 1 & 1 & 1 & 1 & 1 & 1 & 1 & 1 & 1 & 1 & 1 \\
      1 & 0 & 0 & 0 & -1 & 0 & 0 & 0 & 1 & 0 & 0 & 0 & -1 & 0 & 0 & 0 & 0 & 0 & \sfrac{1}{3} & -\sfrac{1}{3} & \sfrac{1}{3} & 0 & -\sfrac{1}{3} & -\sfrac{1}{3} & 0 & \sfrac{1}{3} \\
      0 & 1 & 0 & 0 & 0 & -1 & 0 & 0 & 0 & 1 & 0 & 0 & 0 & -1 & 0 & 0 & 0 & 0 & \sfrac{1}{3} & -\sfrac{1}{3} & -\sfrac{1}{3} & \sfrac{1}{3} & 0 & \sfrac{1}{3} & -\sfrac{1}{3} & 0 \\
      0 & 0 & 1 & 0 & 0 & 0 & -1 & 0 & 0 & 0 & 1 & 0 & 0 & 0 & -1 & 0 & 0 & 0 & \sfrac{1}{3} & -\sfrac{1}{3} & 0 & -\sfrac{1}{3} & \sfrac{1}{3} & 0 & \sfrac{1}{3} & -\sfrac{1}{3} \\
      0 & 0 & 0 & 1 & 0 & 0 & 0 & -1 & 0 & 0 & 0 & 1 & 0 & 0 & 0 & -1 & 0 & 0 & 0 & 0 & -\sfrac{1}{3} & -\sfrac{1}{3} & -\sfrac{1}{3} & \sfrac{1}{3} & \sfrac{1}{3} & \sfrac{1}{3} \\ 
      0 & 0 & 0 & 0 & 0 & 0 & 0 & 0 & 1 & 1 & 1 & 1 & 1 & 1 & 1 & 1 & 0 & 1 & \sfrac{1}{2} & \sfrac{1}{2} & \sfrac{1}{2} & \sfrac{1}{2} & \sfrac{1}{2} & \sfrac{1}{2} & \sfrac{1}{2} & \sfrac{1}{2}
    \end{psmallmatrix*}
  \end{equation*}
\end{itemize}

\section{Applications: Computational Environment}
\label{sec:applications-environment}

For all computational tests a C++-implementation in \texttt{TOPCOM}
preview-versions \texttt{1.2.0b}--\texttt{1.2.0f} were used.  See
\cite{Rambau_TOPCOMTriangulationsPoint_2002} for a paper on an earlier
version of \texttt{TOPCOM}.

The computer is an Intel(R) Xeon(R) CPU E5-2690, 2.90GHz (384 GB RAM)
with 2 sockets of 8 cores each and two virtual threads per core.  The
operating system was \texttt{Ubuntu Linux 5.4.0-172-generic} with the
C++compiler from \texttt{gcc version 9.4.0 (Ubuntu
  9.4.0-1ubuntu1~20.04.2)}.  Unless state otherwise, for each run 16
threads were used on an otherwise idle machine since hyper-threading
has no advantages for computationally demanding tasks.

For some computations, also an Apple MacBookPro (2021) was used with
M1Max (64 GB RAM) with 8 performance cores and 2 efficiency cores.
The operating systems ranged over the time of the experiments from
\texttt{MacOSX Sonoma 14.5 (23F79)} with the C++-compiler
\texttt{Apple clang version 15.0.0 (clang-1500.3.9.4)} through
\texttt{MacOSX Sequoia 15.6.1} with the C++-compiler \texttt{Apple
  clang version 17.0.0 (clang-1700.4.4.1)}.  Unless stated otherwise,
this computer was run with 8 threads.  Computational results achieved
with this computer are marked with ``M1Max''.

The largest computations have been carried out in the \emph{Bayreuth
  Center for High-Performance Computing} on up to 192 threads; those
are marked with ``[HPC$^{192}$]''.  Likewise, ``[HPC$^{24}$]'' denotes
a HPC with up to 24 threads.

No multiple runs were executed for the experiments in this paper,
since for the larger instances this simply would have taken too much
time.  Therefore, the cpu times could vary a little.  Since the cpu
times in this paper are several times faster than for any earlier
effort, this does not affect the relevance of the results.

\section{Application I: Cocircuits}
\label{sec:application-cocircuits}

In this section it is shown how all cocircuits of a point
configuration~$\Conf$ can be enumerated up to symmetry by exploring
the downset $\downsetStyle{D}$ of all coplanar subsets using
$\symLSRSMaxwithData$ (Algorithm~\ref{alg:symLSRSMaxwithData}).  The
method to enumerate subsets with prescribed symmetry cannot be applied
straight-forwardly to cocircuits because the downset of all coplanar
subsets can not be represented as the set of cliques in a
compatibility graph.

\subsection{Specializations}
\label{sec:cocircuits:specializations}

In order to make the application of $\symLSRSMaxwithData$ explicit,
its subroutines $\isInD$, $\semiIsNotRightComp$, $\isInEachMax$, and
$\isMaxInD$ must be formally specified. Recall that in these
subroutines one can utilize global data (preprocessed prior to the
enumeration) and local data (updated in each enumeration node) to
speed up the computations and avoid duplicate work.

In this particular case, the local data $\localData$ consists of a
matrix $\colechMatrix$ that is a column-echelon form
of~$\Conf_{\setStyle{S}}$.  It can be computed by Gaussian elimination
on columns from left to right.  Storing it allows us to avoid the
repetition of identical eliminations in matrices with identical
initial segments of columns.  The local data of the node of a subset
$\setStyle{S}'$ with
$\setStyle{S} = \setStyle{S}' \setminus \max \setStyle{S'}$ is
initialized as the augmented matrix
$(\colechMatrix, \subConf{\max \setStyle{S'}})$. The global data used
is the full matrix~$\Conf$.

\begin{algorithm}[htbp]
  \TitleOfAlgo{\isInD{$\setStyle{S}', \downsetStyle{D}, \Conf,
      \colechMatrix, \colechMatrix'$}}

  \KwIn{A subset $\setStyle{S}'$ of column indices of~$\Conf$, the
    downset $\downsetStyle{D}$ of coplanar subsets, the
    configuration~$\Conf$, a column-echelon form $\colechMatrix$
    of~$\subConf{S' \setminus \max S'}$, and the
    augmented matrix
    $\colechMatrix' = (\colechMatrix, \subConf{\max \setStyle{S}'})$
    (not yet in column-echelon form, in general)}

  \KwOut{$(\true, \colechMatrix')$ with $\colechMatrix'$ a
    column-echelon form of $\subConf{S'}$ if
    $\rankOf(\colechMatrix') \le \rank - 1$ and
    $(\false, \colechMatrix')$ otherwise}

  $\colechMatrix' \gets \colechForm{$\colechMatrix'$}$
  \tcc*{compute column-echelon form}
  \If(\tcc*[f]{if at most $\rank - 1$ non-zero columns}){
    $\abs{\setStyle{S'}} < \rank$ or $\colechMatrix'_{*,\rank} = \vectorStyle{0}$
  }{
    \Return $(\true, \colechMatrix')$
    \tcc*{$\setStyle{S}$ is coplanar}
  }
  \Else{
    \Return $(\false, \colechMatrix')$ 
    \tcc*{$\setStyle{S}$ is spanning}
  }
  
  \caption[Coplanarity Check]{Check whether a subset of columns is
    coplanar}
  \label{alg:isInD_cocircuits}  
\end{algorithm}

The implementation of~$\semiIsNotRightComp$ in
Algorithm~\ref{alg:semiIsNotRightComp_cocircuits} is based on the
following theorem.
\begin{theorem}\label{thm:semiIsNotRightComp_cocircuits}
  Let $\setStyle{S}$ be right-completable to (the zero-set of) a
  cocircuit~$\setStyle{C}_0^*$ of $\Conf$.  Then:
  \begin{enumerate}[label=(\roman*)]
  \item\label{itm:cocircuits-not-right-comp:right} For the set
    $\setStyle{R} := \{ i \in \indexSet{n} : i > \max \setStyle{S} \}$
    the rank bound $\rankOf(\subConf{S \cup R}) \ge \rank - 1$ holds.
  \item\label{itm:cocircuits-not-right-comp:left} For all
    $i \in \indexSet{n} \setminus \setStyle{S}$ with
    $i < \max \setStyle{S}$ the rank bound
    $\rankOf(\subConf{S}) < \rankOf(\subConf{S \cup \{i\}})$ holds.
\end{enumerate}
\end{theorem}

\begin{proof}
  For item~\ref{itm:cocircuits-not-right-comp:right} assume that
  $\setStyle{S}$ is right-completable to a maximal and, thus,
  maximally hyperplanar subset~$\setStyle{S'}$ with
  $\rankOf(\subConf{S \cup R}) < \rank - 1$.  Then, since
  $\setStyle{S'} \setminus \setStyle{S} \subseteq \setStyle{R}$, we
  have $\rankOf(\subConf{S'}) < \rank - 1$, contradicting the fact
  that $\setStyle{S'}$ is hyperplanar.

  For item~\ref{itm:cocircuits-not-right-comp:left} assume there is an
  $i \in \indexSet{n} \setminus \setStyle{S}$ with
  $i < \max \setStyle{S}$ so that
  $\rankOf(\subConf{S}) = \rankOf(\subConf{S \cup \{i\}})$, and assume
  there is a maximal hyperplanar subset $\setStyle{S'}$ that is a
  right-completion of~$\setStyle{S}$.  Then, $i$ is not
  in~$\setStyle{S'}$, and
  $\rankOf(\subConf{S' \cup \{i\}})
  = \rankOf\bigl(\subConf{(S' \setminus S) \cup (S \cup \{i\})}\bigr)
  = \rankOf\bigl(\subConf{(S' \setminus S) \cup S}\bigr)
  = \rankOf(\subConf{S'})$,
  contradicting the maximality of~$\setStyle{S'}$.
\end{proof}

\begin{algorithm}[htbp]
  \TitleOfAlgo{\semiIsNotRightComp{$\setStyle{S}', \downsetStyle{D}, \Conf,
      \colechMatrix, \colechMatrix'$}}

  \KwIn{A coplanar subset $\setStyle{S}'$ of column indices
    of~$\Conf$, the downset $\downsetStyle{D}$ of coplanar subsets,
    the configuration~$\Conf$, a column-echelon form $\colechMatrix$
    of~$\subConf{S' \setminus \max \setStyle{S'}}$, and a
    column-echelon form $\colechMatrix'$ of the augmented matrix
    $\subConf{S'}$}
  
  \KwOut{$(\true, \colechMatrix')$ if $\setStyle{S}'$ cannot be
    right-completed and $(\false, \colechMatrix')$ otherwise}

  $\rank' \gets \rankOf(\colechMatrix')$
  \tcc*{get current rank}
    \For(\tcc*[f]{traverse left non-elements}) {
      $i = 1, \dots, \max \setStyle{S'} - 1$ with $i \notin \setStyle{S'}$
    }{
      $\colechMatrix'' \gets \colechForm{$\colechMatrix', \col_i$}$
      \tcc*{compute column-echelon form}
      \If(\tcc*[f]{if column~$i$ is in the current span}){
        $\colechMatrix''_{*,\rank' + 1} = \vectorStyle{0}$
      }{
        \Return $(\true, \colechMatrix')$
        \tcc*{$\setStyle{S'}$ not right-completable}
      }
    }
    \If(\tcc*[f]{if rank not yet sufficient}){
      $\rank' < \rank - 1$
    }{
      $\colechMatrix'' \gets \colechMatrix'$
      \tcc*{prepare a rank-increase checker matrix}
      $\rank'' \gets \rank'$
      \tcc*{keep track of rank increase}
      \For(\tcc*[f]{traverse right non-elements}) {
        $i = \max \setStyle{S'} + 1, \dots, n$
      }{
        $\colechMatrix'' \gets \colechForm{$\colechMatrix'', \col_i$}$
        \tcc*{compute column-echelon form}
        \If(\tcc*[f]{if $i$ increases rank}){
          $\colechMatrix''_{*,\rank'' + 1} \neq \vectorStyle{0}$
        }{
          $\rank'' \gets \rank'' + 1$
          \tcc*{update rank increase}
          \If(\tcc*[f]{if rank increase sufficient}){
            $\rank'' = \rank - 1$
          }{
            \Return $(\false, \colechMatrix')$
            \tcc*{$\setStyle{S'}$ might be right-completable}
          }
        }
        \ElseIf(\tcc*[f]{if target rank unreachable}){
          $\rank'' + n - i < \rank - 1$
        }{
          \Return $(\true, \colechMatrix')$
          \tcc*{$\setStyle{S'}$ not right-completable}
        }
      }
    }  
    \Return $(\false, \colechMatrix')$
    \tcc*{$\setStyle{S'}$ might be right-completable}
    \caption[Rank Pruning]{The rank-pruning semi-check whether a
      coplanar subset of columns is right-completable based on a
      direct application of
      Theorem~\ref{thm:semiIsNotRightComp_cocircuits}}
  \label{alg:semiIsNotRightComp_cocircuits}  
\end{algorithm}
Call the use of the semi-check based on this theorem
\emph{rank-pruning}; skipping this semi-check is called
\emph{no-pruning} for later reference.  Moreover, call the set
$\nonPrunables$ of subsets in~$\downsetStyle{D}$ for which
rank-pruning returns ``\false'' the \emph{non-prunable} subsets.  In
Table~\ref{tab:cocircuits-results-nodesavings-pruning} the node counts
for no-pruning and rank-pruning are compared for tiny to small
examples.
\begin{table}[htbp]
  \centering
  {%
    \sffamily\footnotesize
    \begin{tabular}{l*{5}{r}}
      \toprule
      $\Conf$
      & \multicolumn{1}{c}{\# cocircuits}
      & \multicolumn{2}{c}{\# nodes by pruning method}
      & \multicolumn{2}{c}{CPU time [s] by pruning method}\\
      & \multicolumn{1}{c}{in total}
      & no & rank
      & no & rank\\
      \midrule
      $\hypercube{5}$         & 3254 & 1,026,636 &  78,050 & 13.62 & 3.37 \\
      $\simplexproduct{4}{4}$ &  460 & 2,018,570 &  87,440 & 37.52 & 5.36 \\
      $\hypersimplex{8}{2}$   & 1661 & 3,133,114 & 131,007 & 50.88 & 6.60 \\
      \bottomrule
    \end{tabular}
  }
  \caption[Comparison of pruning methods for cocircuits]{Comparison of no-pruning versus rank-pruning on tiny to 
    small examples (single-threaded, symmetries ignored)}
  \label{tab:cocircuits-results-nodesavings-pruning}
\end{table}

In order to find out whether an expansion is in each maximal superset,
one can use a similar observation: any element that does not increase
the rank of the current subset can be added to each rank-$(\rank - 1)$
superset thereof without increasing the rank. Thus, it is in each
maximal superset.  The pseudo-code is listed in
Algorithm~\ref{alg:isInEachMax_cocircuits}.  All data necessary for
this check have been computed before. Thus, this check is very fast,
and there is no point in skipping it.
\begin{algorithm}[htbp]
  \TitleOfAlgo{\isInEachMax{$\setStyle{S}', \downsetStyle{D}, \Conf,
      \colechMatrix, \colechMatrix'$}}

  \KwIn{A subset $\setStyle{S}'$ of column indices of~$\Conf$, the
    downset $\downsetStyle{D}$ of coplanar subsets, the
    configuration~$\Conf$, a column-echelon form $\colechMatrix$
    of~$\subConf{S' \setminus \max S'}$, and a column-echelon form
    $\colechMatrix'$ of the augmented
    matrix~$\bigl(M, \subConf{\max S'}\bigr)$}
  
  \KwOut{$(\true, \colechMatrix')$ if each maximal superset
    of~$\setStyle{S}$ contains $\setStyle{S}'$ and
    $(\false, \colechMatrix')$ otherwise}

  \If(\tcc*[f]{$\setStyle{S}'$ did not increase rank}){
    $\rankOf(\colechMatrix') = \rankOf(\colechMatrix)$
  }{
    \Return $(\true, \colechMatrix')$ \;
  }
  \Else{
    \Return $(\false, \colechMatrix')$ \;
  }
  
  \caption[Unavoidability Check]{Check whether the expansion
    from~$\setStyle{S}$ to $\setStyle{S}'$ is in each maximal superset
    of~$\setStyle{S}$; the local data $\colechMatrix'$ remains
    unchanged}
  \label{alg:isInEachMax_cocircuits}  
\end{algorithm}

In order to implement the maximality check one can utilize a signature
corresponding to a hyperplanar subset, which must be computed anyway.
This can be done as follows.
\begin{lemma}
  Let $\setStyle{S}$ be a hyperplanar subset and let $\matrixStyle{B}$
  be the $\rank \times (\rank - 1)$-matrix of non-zero-columns of a
  column-echelon form of~$\subConf{S}$. Then, there is a unique
  cocircuit~$\setStyle{C}^*_0$ containing~$\setStyle{S}$. Moreover,
  one of the two opposite signatures of $\setStyle{C}^*_0$ is given by
  \begin{equation}
    \label{eq:cocircuit-signature}
    \signature^*_{\setStyle{S}} (i)
    =
    \signOf \bigl(\det (\matrixStyle{B}, \col_i)\bigr).
  \end{equation}

  In particular, $\setStyle{S}$ is maximal and, thus, a cocircuit if
  and only if
  $\det (\matrixStyle{B}, \col_i) \neq 0$ for all
  $i \in \indexSet{n} \setminus \setStyle{S}$.
\end{lemma}
\begin{proof}
  The assertion follows from the fact that the right-hand side is the
  sign of a linear form in $\col_1, \dots, \col_{\rank}$ that vanishes
  on all points in~$\subConf{S}$, which was assumed to be hyperplanar.
\end{proof}

That is, whenever during the enumeration a right-maximal hyperplanar
subset is reached, the first $\rank - 1$ non-zero columns
$\colechMatrixNZ$ of the corresponding matrix in the local data
determine the signature of the unique cocircuit containing it.
\begin{algorithm}[htbp]
  \TitleOfAlgo{\isMaxInD{$\setStyle{S}, \downsetStyle{D},
      \Conf, \colechMatrix$}}

  \KwIn{A subset $\setStyle{S}$ of column indices of~$\Conf$, the
    downset $\downsetStyle{D}$ of coplanar subsets, the
    configuration~$\Conf$, and a column-echelon form~$\colechMatrix$
    of~$\subConf{S}$}

  \KwOut{$\true$ if $\setStyle{S}$ is maximal in~$\downsetStyle{D}$
    and $\false$ otherwise}

  \For(\tcc*[f]{for all elements outside~$\setStyle{S}$}){
    $i \in \indexSet{n} \setminus \setStyle{S}$
  }{
    \If(\tcc*[f]{if the signature of $i$ is zero}){
      $\det(\colechMatrixNZ, \col_i) = 0$
    }{
      \Return $\false$
      \tcc*{$\setStyle{S}$ is not maximal}
    }
  }
  \Return \true \;

  \caption[Maximal Coplanarity Check]{Check whether a subset of columns is
    maximally coplanar}
  \label{alg:isMaxInD_cocircuits}  
\end{algorithm}

\subsection{Analysis}
\label{sec:cocircuits:analysis}

The following can be said about the efficiency of this algorithm.
\begin{theorem}
  \label{thm:cocircuit-alg-efficiency}
  Assume that the critical-element method is used for~$\isLexMin$ in
  Algorithm~\ref{alg:symLSRSMaxwithData}.  Then, the run-time
  complexity of the resulting specialized algorithm is in
  $O\bigl(\no(\abs{\GrG} +
  \no\rank)\abs{\setOrbits{\nonPrunables}{\GrG}}\bigr)$.

  Moreover, the effectivity of rank-pruning depends on the instance.
  More specifically, with symmetries ignored:
  \begin{enumerate}[label=(\roman*)]
  \item\label{itm:cocircuits-UB} There is a point configuration
    $\standardsimplex{\rank - 1}^{\partial\mathrm{c}}$ with
    $\no = 2\rank$ points of rank~$\rank$ with
    $\tfrac{\rank(\rank + 3)}{2}$ cocircuits for which the total
    number of not right-completable subsets is in $\Omega(2^{\rank})$,
    whereas none of the not right-completable subsets is non-prunable.
  \item\label{itm:cocircuits-LB} There is a point configuration
    $\standardsimplex{\rank - 1}^{\mathrm{dup}}$ with $\no = 2\rank$
    points of rank~$\rank$ with $\rank$ cocircuits for which the
    number of not right-completable subsets that are non-prunable is
    in $\Omega(2^{\rank})$.
  \end{enumerate}
\end{theorem}
The first item shows that, for the enumeration and listing problems,
rank-pruning can lead to a spead-up from exponential to polynomial in
the input and output sizes, whereas the second item shows that, in
general, rank-pruning cannot guarantee a run-time polynomial in the
input and output sizes.
\begin{proof}
  The number of recursive calls equals the number of orbits of
  not-prunable subsets.  For each such subset there are at most $\no$
  traversals of the main loop.  In each main-loop the worst-case
  run-time complexity is dominated by $\semiIsNotRightComp$, which
  computes column-echelon forms for each non-element of the current
  subset.  This amounts to at most $\no$ column-echelon-form
  computations, taking at most $\rank$ operations each.  This proves
  the run-time bound.

  For item~\ref{itm:cocircuits-UB}, consider for $\rank \ge 3$ the
  point configuration
  $\standardsimplex{\rank - 1}^{\partial\mathrm{c}}$ consisting of the
  $(\rank-1)$-dimensional standard simplex
  $\standardsimplex{\rank - 1}$ together with $\rank$ copies of the
  barycenter~$\vectorStyle{c}$ of its lex-min facet
  $\{1, 2, \dots, \rank - 1\}$, forming the elements indexed by
  $\rank + 1, \rank + 2, \dots, 2\rank$.  There are $\rank$ cocircuits
  correponding to the $r$ facets of $\standardsimplex{\rank - 1}$:
  $\rank - 1$ of them not containing any~$\vectorStyle{c}$ and one
  containing all copies of~$\vectorStyle{c}$.  Moreover, there are
  $\tbinom{\rank}{\rank - 2} = \tbinom{\rank}{2} = \frac{\rank(\rank +
    1)}{2}$ cocircuits containing all copies of~$\vectorStyle{c}$ but
  not the lex-min facet.  In any non-empty subset $\setStyle{S}$ call
  an element $i \in \indexSet{n} \setminus \setStyle{S}$ with
  $i < \max \setStyle{S}$ a \emph{gap} in~$\setStyle{S}$.  The empty
  subset, by definition, has no gaps.

  A coplanar subset $\setStyle{S}$ is not right-completable if and
  only if each cocircuit containing it contains a gap
  in~$\setStyle{S}$.  Define for a subset~$\setStyle{S}$ the part
  $\setStyle{R} := \setStyle{S} \cap \{1, 2, \dots, \rank\}$
  corresponding to $\standardsimplex{\rank - 1}$ and the part
  $\setStyle{C} := \setStyle{S} \cap \{ \rank + 1, \rank + 2, \dots,
  2\rank \}$ corresponding to the copies of~$\vectorStyle{c}$.  Call a
  subset~$\setStyle{S} = \setStyle{R} \cup \setStyle{C}$ a
  \emph{deadend} if one of the following cases occurs:
  \begin{itemize}
  \item $\setStyle{C}$ has at least one gap.
  \item $\setStyle{C}$ has no gap, is non-empty, and
    $\abs{\setStyle{R}} \le \rank - 3$.
  \item $\setStyle{C}$ is empty, and $\setStyle{R}$ has at least three
    gaps.
  \end{itemize}
  In the following it is shown that a coplanar $\setStyle{S}$ is not
  right-completable if and only if it is a deadend.  Assume
  $\setStyle{S}$ is a deadend. Then, in the first case $\setStyle{S}$
  violates the first rank bound, in the second and third case
  $\setStyle{S}$ violates the second rank bound in
  Theorem~\ref{thm:semiIsNotRightComp_cocircuits}. In particular,
  $\setStyle{S}$ is prunable.  Hence, it is not right-completable.

  If $\setStyle{S}$ is not a deadend, then, in particular,
  $\setStyle{C}$ has no gaps.  Moreover, if $\setStyle{C}$ is
  non-empty, then $\setStyle{R}$ has $\rank - 1$ or $\rank - 2$
  elements.  If $\setStyle{R}$ has exactly $\rank - 1$ elements, then
  $\setStyle{R}$ is a facet of~$\standardsimplex{\rank - 1}$.  Since
  $\setStyle{S}$ is coplanar and $\setStyle{C}$ is non-empty, this
  facet can only be the lex-min facet $\{ 1, 2, \dots, \rank - 1\}$.
  Then, because $\setStyle{C}$ has no gaps, a right-expansion with all
  remaining copies
  $\{\max \setStyle{C} + 1, \max \setStyle{C} + 2 \dots, 2\rank \}$
  of~$\vectorStyle{c}$ leads to the cocircuit
  $\{ 1, 2, \dots, \rank - 1, \rank + 1, \rank + 2, \dots, 2\rank \}$.
  If $\setStyle{R}$ has exactly $\rank - 2$ elements, then
  $\vectorStyle{c}$ is affinely independent of~$\setStyle{R}$ in any
  case.  Thus, the same right-expansion as above leads to the
  cocircuit
  $\setStyle{R} \cup \{ \rank + 1, \rank + 2, \dots, 2\rank \}$.  If
  $\setStyle{C}$ is empty, then $\setStyle{R} = \setStyle{S}$ has at
  most two gaps.  If $\setStyle{R}$ has exactly one gap, then
  $\setStyle{S}$ is already a facet of~$\standardsimplex{\rank}$.  If
  it is the lex-min facet, a right-expansion with all copies
  $\{ \rank + 1, \rank + 2 , \dots, 2\rank \}$ of~$\vectorStyle{c}$
  leads to a cocircuit, and otherwise $\setStyle{S}$ is a cocircuit
  already.  If $\setStyle{R}$ has exactly two gaps, then, a
  right-expansion with
  $\{ \max \setStyle{R} + 1, \max \setStyle{R} + 2, \dots, \rank \}$
  (leading to a rank-$\rank - 2$-subset) followed by a right-expansion
  with all copies $\{ \rank + 1, \rank + 2 , \dots, 2\rank \}$
  of~$\vectorStyle{c}$ (all affinely independent of the 
  elements in~$R$) leads to a cocircuit.

  Consequently, being a deadend is equivalent to being not
  right-completable.  Since all deadends are prunable, so are all not
  right-completable subsets.  Roughly estimated by counting the
  possible $\setStyle{C}$'s in deadends, the number of not
  right-completable subsets is in $\Omega(2^{\rank})$, and all of them
  can be rank-pruned.

  For item~\ref{itm:cocircuits-LB}, consider for $\rank \ge 3$ the
  point configuration $\standardsimplex{\rank - 1}^{\mathrm{dup}}$
  consisting of the $(\rank-1)$-dimensional standard simplex
  $\standardsimplex{\rank - 1}$ together with a copy of each of its
  elements, forming the elements indexed by
  $\rank + 1, \rank + 2, \dots, 2\rank$.  There are $\rank$ cocircuits
  corresponding to the $\rank$ facets of the standard simplex.  No
  subset $\setStyle{S}$ with
  $\emptyset \neq \setStyle{S} \subseteq \{3, 4, \dots, \rank\}$ is
  right-completable, since any cocircuit containing it must also
  contain one of the points $1$ or~$2$, both gaps in~$\setStyle{S}$.
  Moreover, such a subset $\setStyle{S}$ is non-prunable since, first,
  no non-element to the left is in the span of it (the first $\rank$
  points are independent) and, second, the non-elements to the right
  can sufficiently increase the rank of any such $\setStyle{S}$ (its
  possible right-expansions contain a complete copy of the standard
  $\rank - 1$-simplex).  The number of these subsets is
  $2^{\rank - 2} - 1$, which is in~$\Omega(2^{\rank})$.
\end{proof}

If $\GrG$ is given as an explicit set of permutations, the run-time is
polynomial in input size and the number of all non-prunable subsets up
to symmetry.  However, as the example in
Theorem~\ref{thm:cocircuit-alg-efficiency}\ref{itm:cocircuits-LB}
shows, Algorithm~\ref{alg:isMaxInD_cocircuits} for
counting/enumeration/listing is, in general, not polynomial in the
input and output sizes.

\subsection{Results}
\label{sec:cocircuits:results}

The modified switch-table method in
Algorithm~\ref{alg:isLexMinModifiedSwitches} turned out to be the
fastest variant of $\isLexMin$ for all the larger instances.  (See the
end of Section~\ref{sec:applications-preliminaries} for explanations
concerning the point configurations.) Particularly interesting is the
enumeration of hyperplanes spanned by the vertices of the
$d$-dimensional hypercube $\hypercube{d}$.  This problem has already
been studied a long time ago by Aichholzer and Aurenhammer
\cite{AichholzerAurenhammer_ClassifyingHyperplanesHypercubes_1996}.
Using cleverly a lot of structural properties of hypercubes in
particular, they were able to enumerate all hyperplanes spanned by the
vertices of the $8$-cube, while the $9$-cube's hyperplanes remained
out-of-reach.  In contrast to their efforts, the general-purpose
algorithm of this paper could compute their numbers without using any
specific knowledge about cubes.  And it was able to compute the number
of cocircuits (total and up to symmetry) of the $9$-cube.  For
$\hypercube{9}$ compare the total number of its cocircuits to the
number of all its $\rank - 1$-subsets, which is
$\tbinom{512}{9} = 6{,}208{,}116{,}950{,}265{,}950{,}720$ (four orders
of magnitude larger) -- direct signature computations for all these
subsets, given today's computation power, would have been out of
reach.  Table~\ref{tab:results-cocircuits} shows all the results.

\begin{table}[htbp]
  \centering
  {%
    \sffamily\footnotesize
    \begin{tabular}{lrrrrr}
      \toprule
      $\Conf$
      & \# symmetries
      & \# cocircuits
      & \# cocircuits
      & \# nodes
      & CPU time\\
      && up to symmetry & in total && [hh:mm:ss]\\
      \midrule
      $\hypercube{2}$         &           8 &         2 &                   6 &           9 &  0:00:00\\
      $\hypercube{3}$         &          48 &         3 &                  20 &          31 &  0:00:00\\
      $\hypercube{4}$         &         384 &         6 &                 140 &         126 &  0:00:00\\
      $\hypercube{5}$         &        3840 &        15 &                3254 &         609 &  0:00:00\\
      $\hypercube{6}$         &      46,080 &        63 &             252,434 &        4149 &  0:00:01\\
      $\hypercube{7}$         &     645,120 &       623 &          71,343,208 &      55,540 &  0:00:04\\
      $\hypercube{8}$         &  10,321,920 &    22,432 &      86,246,755,608 &   2,403,058 &  0:03:02\\
      *$\hypercube{9}$        & 185,794,560 & 3,899,720 & 448,691,419,804,586 & 530,623,381 & 13:30:12\\
      *$\hypersimplex{8}{3}$  &      40,320 &        56 &             166,420 &       4,644 &  0:00:00\\
      *$\hypersimplex{8}{4}$  &      80,640 &        83 &           1,105,575 &       9,081 &  0:00:00\\
      *$\hypersimplex{9}{3}$  &     362,880 &       231 &          10,004,154 &      21,034 &  0:00:02\\
      *$\hypersimplex{9}{4}$  &     362,880 &     2,522 &         359,022,180 &     226,077 &  0:00:07\\
      *$\hypersimplex{10}{3}$ &   3,628,800 &     1,337 &         889,205,792 &     113,814 &  0:00:21\\
      *$\hypersimplex{10}{4}$ &   3,628,800 &    87,254 &     178,227,172,388 &   6,889,144 &  0:07:38\\
      *$\hypersimplex{10}{5}$ &   7,257,600 &   194,489 &     939,079,703,204 &  21,423,661 &  1:02:09\\
      \bottomrule
    \end{tabular}
  }
  \caption[Computational results for cocircuits]{Computational results for the enumeration of cocircuits in
    hypercubes and hypersimplices using 16 threads (numbers with a
    ``*'' are new)}
  \label{tab:results-cocircuits}
\end{table}

\section{Application II: Circuits}
\label{sec:application-circuits}

In this section it is shown how all circuits of a point
configuration~$\Conf$ can be enumerated up to symmetry by exploring
the downset $\downsetStyle{D}$ of all independent subsets using
$\symLSRSCominwithData$ (Algorithm~\ref{alg:symLSRSCominwithData}).
Again, the method to enumerate subsets with prescribed symmetry cannot
be applied straight-forwardly to circuits.

\subsection{Specializations}
\label{sec:circuits:specializations}

In order to use $\symLSRSCominwithData$, one has to specify how its
problem-specific subroutines $\isInD$, $\semiIsNotRightExit$, and
$\isCominOfD$ work. Again, in these subroutines global data
(preprocessed prior to the enumeration) and local data (updated in
each enumeration node) can be utilized to speed up the computations
and avoid duplicate work.  This time, the local data will consist of a
matrix in column-echelon form stacked on top of another matrix.  The
columns of the additional matrix yield additional information useful
for circuits.  For a subset~$\setStyle{S}$ the columns of the top
matrix are the columns of a column-echelon form of the sub
configuration~$\subConf{S}$.  The column in the bottom matrix contains
the coefficients of a linear combination of all original columns
weakly to the left that yields the column on top.  If the column on
top is the zero-column, then the corresponding original column can be
combined linearly from original columns strictly to the left, and the
signs of the coefficients in the bottom column yield the corresponding
circuit signature.

Formally, the stacked matrix can be defined as follows:
\begin{definition}
  Let $\subConf{S}$ be a subset of columns of a rank-$\rank$
  configuration~$\Conf$.  For an integer $k \in [1, \rank + 1]$ let
  $\identityMatrix{k}$ denote the $k \times k$-identity matrix.
  
  A $(\rank + \abs{\setStyle{S}}) \times \abs{\setStyle{S}}$-matrix
  $\matrixStyle{R} =
  \tmatrix{\matrixStyle{B}\\\matrixStyle{C}}$
  is a \emph{column-representation matrix of~$\setStyle{S}$} if it is
  a column-echelon form of the matrix
  \begin{equation}
    \label{eq:columnd-representation-matrix}
    \begin{pmatrix}
      \subConf{S}\\
      \identityMatrix{\abs{\setStyle{S}}}
    \end{pmatrix}
  \end{equation}
  Here, $\matrixStyle{B}$ is the \emph{configuration part}, whereas
  $\matrixStyle{C}$ is the \emph{coefficient part}.
\end{definition}

\begin{theorem}\label{thm:circuits:isComin}
  Let $\matrixStyle{R} = \tmatrix{\matrixStyle{B}\\\matrixStyle{C}}$
  be a column-representation matrix of a subset~$\setStyle{S}$ of
  column indices of~$\Conf$.  Then $\setStyle{S}$ is dependent if and
  only if $\matrixStyle{B}_{*, \abs{\setStyle{S}}} =
  \vectorStyle{0}$. Moreover, $\setStyle{S}$ is a circuit if and only
  if $\matrixStyle{B}_{*, \abs{\setStyle{S}}}$ is the first
  zero-column in $\matrixStyle{B}$ and
  $\matrixStyle{C}_{*, \abs{\setStyle{S}}}$ contains no zero entry.  In
  that case, the signs in~$\matrixStyle{C}_{*, \abs{\setStyle{S}}}$
  specify one of the two possible circuits signatures
  of~$\setStyle{S}$ restricted to its support~$\setStyle{S}$.
\end{theorem}

\begin{proof}
  Note that
  \begin{equation}
    \label{eq:circuit-equation-untransformed}
    \begin{pmatrix}
      \subConf{S}\\
      \identityMatrix{\abs{\setStyle{S}}}
    \end{pmatrix}
    \cdot
    \identityMatrix{\abs{\setStyle{S}}}
    =
    \begin{pmatrix}
      \subConf{S}\\
      \identityMatrix{\abs{\setStyle{S}}}
    \end{pmatrix}.
  \end{equation}
  If this is transformed by admissible column operations, represented
  by the multiplication of a matrix $\matrixStyle{C}$ from the right,
  into column-echelon form, one has:
  \begin{equation}
    \label{eq:circuit-equation-transformed}
    \begin{pmatrix}
      \subConf{S}\\
      \identityMatrix{\abs{\setStyle{S}}}
    \end{pmatrix}
    \cdot
    \matrixStyle{C}
    =
    \begin{pmatrix}
      \matrixStyle{B}\\
      \matrixStyle{C}
    \end{pmatrix}.
  \end{equation}
  with a column-representation matrix
  $\tmatrix{\matrixStyle{B}\\\matrixStyle{C}}$ of~$\setStyle{S}$.
  In particular, one has for the last column:
  \begin{equation}
    \label{eq:circuit-equation-lastcol}
    \subConf{S} \cdot \matrixStyle{C}_{*, \abs{\setStyle{S}}} = \vectorStyle{B}_{*, \abs{\setStyle{S}}}.
  \end{equation}
  $\setStyle{S}$ is dependent if and only if the last column is zero,
  by the properties of a column-echelon form.  Moreover, the entries
  of $\matrixStyle{C}_{*, \abs{\setStyle{S}}}$ constitute a linear
  dependence.  Since the first $\rankOf{\matrixStyle{B}}$ columns of
  $\matrixStyle{B}$ are linearly independent and
  $\matrixStyle{B}_{*, \abs{\setStyle{S}}}$ is the first zero-column
  in~$\matrixStyle{B}$, one has that
  $\rankOf(\subConf{S}) = \abs{\setStyle{S}} - 1$.  Thus, the kernel
  of~$\subConf{S}$ is one-dimensional, and a non-zero vector in it is
  unique up to a non-zero scalar multiple.  Therefore, the signs
  of~$\matrixStyle{C}_{*, \abs{\setStyle{S}}}$ are unique up to
  sign-reversal.  Thus, the non-zero components of
  $\matrixStyle{C}_{*, \abs{\setStyle{S}}}$ constitute the inclusion
  minimal support of a linear dependence among the columns
  in~$\subConf{S}$.  Hence, the signs of the components of the
  complete vector~$\matrixStyle{C}_{*, \abs{\setStyle{S}}}$ are a
  signature on the support of a circuit~$\setStyle{S}$ if and only if
  there are no zero-components in it.
\end{proof}

One can derive the information to process a node from the
column-representation matrix of its subset.  For the membership in the
downset of independent sets this is straight-forward.
Algorithm~\ref{alg:isInD_circuits} shows the procedure.
\begin{algorithm}[htbp]
  \TitleOfAlgo{\isInD{$\setStyle{S}', \downsetStyle{D}, \Conf,
      \colrepMatrix, \colrepMatrix'$}}

  \KwIn{A subset $\setStyle{S}'$ of column indices of~$\Conf$, the
    downset $\downsetStyle{D}$ of independent subsets, the
    configuration~$\Conf$, a column-representation matrix
    $\tmatrix{\colrepMatrixConf\\\colrepMatrixCoef}$
    of~$\setStyle{S}'$, and the augmented matrix
    $\hat{\colrepMatrix} =
    \left(
    \begin{smallmatrix}
      \colrepMatrixConf & \subConf{\max S'}\\
      \colrepMatrixCoef & \vectorStyle{0}\\
      \vectorStyle{0}^{\top} & 1
    \end{smallmatrix}
    \right)$
    (not yet in
    column-echelon form, in general)}

  \KwOut{$(\true, \colrepMatrix')$ with
    $\colrepMatrix' =
    \tmatrix{\colrepMatrixConf'\\\colrepMatrixCoef'}$ a
    column-representation matrix of~$\setStyle{S}'$ if
    $\rankOf(\colrepMatrix') = \abs{\setStyle{S}'}$ and
    $(\false, \colrepMatrix')$ otherwise}

  $\colrepMatrix' \gets \colechForm{$\hat{\colrepMatrix}$}$
  \tcc*{compute column-echelon form}
  \If(\tcc*[f]{if right-most conf-column non-zero}){
    $\colrepMatrixConf'_{*,\abs{\setStyle{S}'}} \neq \vectorStyle{0}$
  }{
    \Return $(\true, \colrepMatrix')$
    \tcc*{$\setStyle{S}'$ is independent}
  }
  \Else{
    \Return $(\false, \colrepMatrix')$ 
    \tcc*{$\setStyle{S}'$ is dependent}
  }
  
  \caption[Independence Check]{Check whether a subset of columns is
    independent
  }
  \label{alg:isInD_circuits}  
\end{algorithm}

Given Lemma~\ref{thm:circuits:isComin} it is easy to check a subset
for co-minimality.  The corresponding algorithm is listed in
Algorithm~\ref{alg:isCominOfD_circuits}.
\begin{algorithm}[htbp]
  \TitleOfAlgo{\isCominOfD{$\setStyle{S}', \downsetStyle{D}, \Conf, \colrepMatrix,
      \colrepMatrix'$}}
  
  \KwIn{A subset $\setStyle{S}'$ of column indices of~$\Conf$, the
    downset $\downsetStyle{D}$ of independent subsets, the
    configuration~$\Conf$, and local data given by a
    column-representation
    matrix~$\colrepMatrix' =
    \tmatrix{\colrepMatrixConf'\\\colrepMatrixCoef'}$ of~$\setStyle{S}'$}
  
  \KwOut{$\true$ if $\setStyle{S}'$ is co-minimal
    in~$\downsetStyle{D}$, $\false$ otherwise}

  \If(\tcc*[f]{zero-coefficients?}){
    $\colrepMatrixCoef'_{*,\abs{\setStyle{S}'}}$
    has zero-entries
  }{
    \Return \false
    \tcc*{subset is not co-min}
  }
  \Else{
    \Return \true
    \tcc*{subset is co-min}
  }
  
  \caption[Comin-Check For Dependent Sets]{Check whether a subset of
    columns is contained in a circuit utilizing the local data given
    by a column representation matrix}
  \label{alg:isCominOfD_circuits}  
\end{algorithm}

\subsection{Analysis}
\label{sec:circuits:analysis}

Since each maximal independent subset is touched by the algorithm and
each maximal independent subset is built on a unique path of length at
most the rank in the enumeration tree, the efficiency of the algorithm
solely depends on how many (maximal) independent sets one can find
compared to the number of circuits.
\begin{theorem}
  \label{thm:circuit-alg-efficiency}
  Assume that the critical-element method is used for~$\isLexMin$ in
  Algorithm~\ref{alg:symLSRSCominwithData}.  Then, the run-time
  complexity of the resulting specialized algorithm is in
  $O(\no(\abs{\GrG} + \rank)
  \abs{\setOrbits{\downsetStyle{D}}{\GrG}})$.

  Moreover:
  \begin{enumerate}[label=(\roman*)]
  \item\label{itm:circuits-UB} There is a point configuration
    $\standardsimplex{\rank - 1}^{\mathrm{c}}$ with $\no = 2\rank$
    points of rank~$\rank$ with $\frac{\rank(\rank+3)}{2}$ circuits
    for which the total number of maximal independent subsets is
    $1 + \rank^2$.
  \item\label{itm:circuits-LB} There is a point configuration
    $\standardsimplex{\rank - 1}^{\mathrm{dup}}$ with $\no = 2\rank$
    points of rank~$\rank$ with $\rank$ circuits for which the number
    of maximal independent subsets is $2^{\rank}$.
  \end{enumerate}
\end{theorem}
Because the number of $(\rank + 1)$-subsets in the example
configurations is $\tbinom{2\rank}{\rank + 1}$, the algorithm achieves
a speed-up in both cases compared to the naive algorithm that computes
the signature of each $(\rank + 1)$-subset.  In the former case, the
speed-up is substantial from exponential to polynomial, in the latter
case it ``only'' reduces the exponential runtime of the naive
algorithm by an exponential factor to a shorter exponential runtime.

\begin{proof}
  The number of recursive calls equals the number of orbits of
  independent subsets.  For each such subset there are at most $\no$
  traversals of the main loop.  In that loop, the lex-min check takes
  time in $O(\abs{\GrG})$.  The worst-case run-time complexity of the
  remaining subroutines in the main loop is in $O(\rank)$.  This
  proves the run-time bound.

  For item~\ref{itm:circuits-UB} consider the point configuration
  $\standardsimplex{\rank - 1}^{\mathrm{c}}$ consisting of the
  $(\rank-1)$-dimensional standard simplex
  $\standardsimplex{\rank - 1}$ together with $\rank$ copies of its
  barycenter~$\vectorStyle{c}$, forming the elements indexed by
  $\rank + 1, \rank + 2, \dots, 2\rank$.  There are $\rank$ circuits
  containing the standard simplex with an arbitrary copy of
  ~$\vectorStyle{c}$ and $\tbinom{\rank}{2}$ circuits consisting of
  two copies of~$\vectorStyle{c}$, resulting in
  $\frac{\rank(\rank+3)}{2}$ circuits in total.  The maximal
  independent subsets are the standard simplex and an arbitrary copy
  of~$\vectorStyle{c}$ with an arbitrary facet of the standard
  simplex, leading to a total of $1 + \rank^2$ many maximal
  independent subsets.

  For item~\ref{itm:circuits-LB} consider the point configuration
  $\standardsimplex{\rank - 1}^{\mathrm{dup}}$ consisting of the
  $(\rank-1)$-dimensional standard simplex
  $\standardsimplex{\rank - 1}$ together with a copy of each of its
  elements, forming the elements indexed by
  $\rank + 1, \rank + 2, \dots, 2\rank$.  There are $\rank$ circuits
  corresponding to the $\rank$ pairs of identical points.  In contrast
  to this, any choice of a copy for the $\rank$ many vertices of the
  standard simplex is a maximal independent subset, resulting in
  $2^{\rank}$ many of them.
\end{proof}

If $\GrG$ is given as an explicit set of permutations, the runtime is
polynomial in the input size and the number of all independent sets up
to symmetry.  However, the example in
Theorem~\ref{thm:circuit-alg-efficiency}\ref{itm:circuits-LB} shows
that for counting/enumeration/listing it is, in general, not
polynomial in the input and output sizes.

Note that even for graphic matroids it is NP-hard to decide whether a
given subset of elements is contained in a circuit
\cite[Section~4]{KhachiyanBorosElbassioniGurvichMakino_ComplexitySomeEnumeration_2005}.
Call this NP-hard decision problem the ``Extension-to-a-Circuit
Problem (ECP)''.  There is no order involved in this result.  However,
given an instance to the ECP for a vector matroid given by a matrix,
one can reorder the columns of the matrix to start with the subset in
question.  Then, all potential circuits containing the subset
lex-contain it.  Thus, an answer to the right-exitability problem for
circuit enumeration would answer the ECP instance as well.  Thus, it
cannot be expected that all deadends can be avoided efficiently in our
enumeration algorithm.

Even a useful semi-check for right-exitability of a subset for circuit
enumeration is still unknown.  Thus, for circuit enumeration,
$\semiIsNotRightExit$ is defined to simply return~$\false$ for each
subset (i.e., the subset might be a left segment of a co-minimal
subset of~$\downsetStyle{D}$).

While this seems unsatisfactory at first glance, there are examples
where the difficulty of showing right-non-exitability becomes
plausible.  Extend the standard simplex to the augmented standard
simplex $(\standardsimplex{n}, \vectorStyle{p})$ with $n+1$ points,
where $\vectorStyle{p}$ is an additional point at the barycenter of
some geometric $k$-face of~$\standardsimplex{n}$, $k = 0, \dots, n$,
the \emph{centered face}.  There is exactly one circuit, namely the
new point~$\vectorStyle{p}$ together with the centered face. Whether
or not any subset of the points in $(\standardsimplex{n})$ is
contained in a circuit now depends on lex-containment in the centered
face.  Any useful semi-decision algorithm for right-non-exitability
would have to find out the centered face (without actually knowing in
advance that the centered face is the key object).

\subsection{Results}
\label{sec:circuits:results}

Table~\ref{tab:results-circuits} shows some results that could be
obtained using the resulting algorithm, where again the modified
switch-table method in Algorithm~\ref{alg:isLexMinModifiedSwitches}
turned out to be the fastest variant of $\isLexMin$ for the larger
instances (see the end of Section~\ref{sec:applications-preliminaries}
for explanations concerning the point configurations).

\begin{table}[htbp]
  \centering{%
    \sffamily\footnotesize
    \begin{tabular}{lrrrrr}
      \toprule
      $\Conf$
      & \# symmetries
      & \# circuits
      & \# circuits
      & \# nodes
      & CPU time\\
      && up to symmetry & in total && [hh:mm:ss]\\
      \midrule
      $\hypercube{2}$         &          1 &              1 &                      1 &              11 &   0:00:00\\
      $\hypercube{3}$         &         48 &              3 &                     20 &              40 &   0:00:00\\
      $\hypercube{4}$         &        384 &             15 &                   1348 &             219 &   0:00:00\\
      $\hypercube{5}$         &       3840 &            186 &                353,616 &            2616 &   0:00:00\\
      *$\hypercube{6}$        &     46,080 &         12,628 &            446,148,992 &         119,638 &   0:00:01\\
      *$\hypercube{7}$        &    645,120 &      3,591,868 &      2,118,502,178,496 &      25,274,904 &   0:04:49\\
      *$\hypercube{8}$        & 10,321,920 &  3,858,105,362 & 38,636,185,528,212,416 &  21,028,416,821 & 163:37:00\\
      *$\hypersimplex{8}{3}$  &     40,320 &          7,240 &            251,651,820 &         153,429 &   0:00:01\\
      *$\hypersimplex{8}{4}$  &     80,640 &         41,875 &         3,134,451,775  &         670,792 &   0:00:04\\
      *$\hypersimplex{9}{3}$  &    362,880 &        228,432 &         75,267,509,940 &       4,298,974 &   0:00:28\\
      *$\hypersimplex{9}{4}$  &    362,880 &     31,671,609 &     11,259,090,122,490 &     346,869,278 &   1:05:40\\
      *$\hypersimplex{10}{3}$ &  3,628,800 &      7,494,056 &     25,290,095,161,170 &     140,451,528 &   0:20:42\\
      *$\hypersimplex{10}{4}$ &  3,628,800 & 12,609,824,635 & 45,270,853,845,998,550 & 110,076,768,816 & 523:25:52\\
      \bottomrule
    \end{tabular}
  }
  \caption[Computational results for circuits]{Computational results for the enumeration of circuits in
    hypercubes and hypersimplices using 16 threads (numbers with a
    ``*'' are new)
  }
  \label{tab:results-circuits}
\end{table}

\section{Application III: Triangulations}
\label{sec:application-triangulations}

In this section it is shown how all triangulations of a point
configuration~$\Conf$ can be enumerated up to symmetry by exploring
the downset $\downsetStyle{D}$ of all subsets of mutually properly
intersecting simplices using $\symLSRSFeaswithData$
(Algorithm~\ref{alg:symLSRSFeaswithPruning}).  In contrast to the
first two applications, the downset of all subsets of mutually
properly intersecting simplices is naturally the set of cliques in a
compatibility graph: to simplices are compatible if and only if they
are properly intersecting.  How symmetries for enumerated
triangulations can be prescribed is explained in more detail later.

Most general-purpose enumeration algorithms for triangulations rely on
an algorithm based on the flip graph of triangulations
\cite{Rambau_TOPCOMTriangulationsPoint_2002,DeLoeraRambauSantos_TriangulationsStructuresApplications_2010,JordanJoswigKastner_Parallelenumerationtriangulations_2018}.
Because Santos has found a triangulation without flips
\cite{Santos_pointsetwhose_2000} it is known that one might not find
all triangulations this way.  There have always been hints in the
literature on how to enumerate all triangulations of a point
configuration by enumerating maximal cliques in the
proper-intersection graph of all simplices.  However, to date no
implementation of this idea could ever compete with flip-based
algorithms.

Here, an all new attempt is presented based on symmetric lexicographic
subset reverse search for feasible subsets with pruning
(Algorithm~\ref{alg:symLSRSFeaswithPruning}).  In the following, a
subset of pairwise properly intersecting simplices is called a
\emph{partial triangulation}.  Note that, according to this
definition, all triangulations are partial triangulations as well.

\subsection{Specializations}
\label{sec:triangs:specializations}

As a first step towards the application of $\symLSRSFeaswithData$
(Algorithm~\ref{alg:symLSRSFeaswithPruning}), one has to specify how
triangulations are represented as subsets of some finite set.  To this
end, let $\Conf$ be a rank-$\rank$ configuration with $n$ elements.
Let $\simpSet$ be the set of feasible simplices of~$\Conf$ (where
``feasible'' may depend on the exact enumeration task, see
Section~\ref{sec:triangs:enhancements} for some examples), and let
$\noOfSimplices$ be its cardinality.  Similarly, let $\facetSet$ be
the set of all simplex-facets not contained in any facet of~$\Conf$
(to be distinguished from the facets of~$\Conf$) and $\noOfFacets$ be
its cardinality.

By fixing an arbitrary bijection
$\simpSet \to \indexSet{\noOfSimplices}$ to encode simplices, any
triangulation $\setsysStyle{T}$ given by its set of maximal simplices
can be represented as a subset $\setStyle{T}$
of~$\indexSet{\noOfSimplices}$.  While any bijection will allow to use
Algorithm~\ref{alg:symLSRSFeaswithPruning}, there is a special
bijection that helps to accelerate the semi-check for the
non-right-extendability of a subset.  For the rest of this section,
let
$\lexSimpIndex\colon (\simpSet, \lexsmaller) \to
(\indexSet{\noOfSimplices}, <)$ and
$\lexFacetIndex\colon (\facetSet, \lexsmaller) \to
(\indexSet{\noOfFacets}, <)$ be the respective order-preserving
bijections.  For convenience, denote the inverse functions by
$\lexIndexSimp = \lexSimpIndex^{-1}$
and~$\lexIndexFacet = \lexFacetIndex^{-1}$, respectively.  With this
notation, the subset $\setStyle{T}$ corresponding to a triangulation
$\setsysStyle{T}$ is given by
$\setStyle{T} = \{ s \in \indexSet{\noOfSimplices} : \lexIndexSimp(s)
\in \setsysStyle{T} \}$, and the triangulation corresponding to a
subset~$\setStyle{T}$ is given by
$\setsysStyle{T} = \{ \setStyle{S} \in \simpSet :
\lexSimpIndex(\setStyle{S}) \in \setStyle{T} \}$.  This motivates to
say that \emph{$\setStyle{T}$ indexes~$\setsysStyle{T}$}.

A brief inspection of Algorithm~\ref{alg:symLSRSFeaswithPruning} shows
that this results in the following: not only are the triangulations
enumerated in the lexicographic order
of~$\powerSet{\indexSet{\noOfSimplices}}$, but also each triangulation
itself is built by adding simplices one-by-one in lexicographic order.

The downset $\downsetStyle{D}$ used in the method is the set of
subsets indexing simplex subsets that are pairwise properly
intersecting.  The set $\setsysStyle{F}$ of feasible subsets is the
set of all subsets indexing a triangulation.  Since no triangulation
strictly contains any other triangulation, one can apply
Algorithm~\ref{alg:symLSRSFeaswithPruning} by specializing its
subroutines $\isFeasible$, $\expand$, and $\semiIsNotRightExt$.

Next, some notions are introduced that help to set up appropriate
global and local auxiliary data, which will speed up the
implementation.  The global auxiliary data is defined first. It is
desirable to access quickly the incidence information of spanning
simplices and their facets.  Given the characterization of a
triangulation in Definition~\ref{def:triangulations}, the information
about which simplices contain which interior simplex facets is
important.

\begin{definition}
  \label{def:incidences-table}
  For a rank-$\rank$-configuration $\Conf$ with $\no$ points,
  $\noOfSimplices$ many simplices, and $\noOfFacets$ many
  simplex-facets, define the \emph{interior-facets table} as
  \begin{equation}
    \intfacets\colon
    \left\{
      \begin{array}{rcl} 
        \indexSet{\noOfSimplices} & \to     & \powerSet{\indexSet{\noOfFacets}};\\
        s                         & \mapsto & \bigl\{ f \in \indexSet{\noOfFacets} :
                                              \text{$\lexIndexFacet(f)$
                                              is an interior facet
                                              of~$\lexIndexSimp(s)$} \bigr\},
      \end{array}
    \right.
  \end{equation}
  and the reverse \emph{convering-simplices table} as 
  \begin{equation}
    \covsimps\colon
    \left\{
      \begin{array}{rcl} 
        \indexSet{\noOfFacets} & \to     & \powerSet{\indexSet{\noOfSimplices}};\\
        f                      & \mapsto & \bigl\{ s \in \indexSet{\noOfSimplices} :
                                              \text{$\lexIndexFacet(f)$
                                              is an interior facet
                                              of~$\lexIndexSimp(s)$} \bigr\}.
      \end{array}
    \right.
  \end{equation}
\end{definition}

In Algorithm~\ref{alg:symLSRSFeaswithPruning} an expansion sequence is
used.  In the current application such a sequence corresponds to a
sequence of lexicographically greater simplices that can be added to a
partial triangulation without violating proper intersection.  Such a
sequence can be updated more easily when for all simplices 
the set of all simplices that have a proper intersection
with it can be accessed quickly.

\begin{definition}
  \label{def:admissibles-table}
  For a rank-$\rank$-configuration $\Conf$ with $\no$ points and
  $\noOfSimplices$ many simplices, define the \emph{admissibles table}
  as
  \begin{equation}
    \admsimps\colon
    \left\{
      \begin{array}{rcl} 
        \indexSet{\noOfSimplices} & \to     & \powerSet{\indexSet{\noOfSimplices}};\\
        s                         & \mapsto & \bigl\{ s' \in \indexSet{\noOfSimplices} :
                                              \text{$\lexIndexSimp(s')$
                                              intersects properly
                                              with~$\lexIndexSimp(s)$}
                                              \bigr\}.
      \end{array}
    \right.
  \end{equation}
\end{definition}

Next, some local auxiliary data is defined: with each subset indexing
a partial triangulation store
\begin{itemize}
\item the index set of all lex-greater simplices intersecting properly with it;
\item the index set of all uncovered interior facets.
\end{itemize}
To this end, define:
\begin{definition}
  Let $\setStyle{T}$ index a partial triangulation.  Then the
  \emph{admissibles of~$\setStyle{T}$} are defined as
  \begin{equation}
    \admissibles(\setStyle{T}) :=
    \bigl\{
    s \in \indexSet{\noOfSimplices}
    :
    \text{$s > s'$, $s \in \admsimps(s')$ for all $s' \in \setStyle{T}$}
    \bigr\}.
  \end{equation}
  Moreover, define the \emph{free interior facets of~$\setStyle{T}$
    as}
  \begin{equation}
    \freefacets(\setStyle{T}) :=
    \bigl\{
    f \in \indexSet{\noOfFacets}
    :
    \text{$f \in \intfacets(s)$ for exactly one $s \in \setStyle{T}$}
    \bigr\}.
  \end{equation}
\end{definition}
From these data, some useful information can directly be derived.
\begin{lemma}
  \label{thm:triangulations:facts}
  Let $\setStyle{T}$ index a non-empty partial triangulation.
  Then:
  \begin{enumerate}[label=(\roman*)]
  \item $\setStyle{T}$ indexes a triangulation if and only if
    $\freefacets(\setStyle{T}) = \emptyset$.
  \item The admissibles of~$\setStyle{T}$ constitute the expansion
    sequence of~$\setStyle{T}$.
  \item If $\admissibles(\setStyle{T}) = \emptyset$ and
    $\freefacets(\setStyle{T}) \neq \emptyset$, then $\setStyle{T}$ is
    not right-extendable.
    \newcounter{intermediateenumi}
    \setcounter{intermediateenumi}{\value{enumi} + 1}
  \end{enumerate}
  Moreover, for $\setStyle{T}' = \setStyle{T} \cup \{s\}$ with
  $s \in \admissibles(\setStyle{T})$:
  \begin{enumerate}[label=(\roman*), start=\value{intermediateenumi}]
  \item The admissibles of~$\setStyle{T}'$ can be computed as the
    intersection
    \begin{equation}
      \admissibles(\setStyle{T}') =
      \admissibles(\setStyle{T}) \cap \admsimps(s).
    \end{equation}
  \item The free interior facets of~$\setStyle{T}'$ can be computed as
    the symmetric difference
    \begin{equation}
      \freefacets(\setStyle{T}') = \freefacets(\setStyle{T})
      \symdiff \intfacets(s).
    \end{equation}
  \end{enumerate}
\end{lemma}
\begin{proof}
  The assertions follow from straight-forward checks of definitions.
\end{proof}
This is essentially what was used to-date in any attempt to enumerate
all triangulations of a configuration,
compare~\cite{Rambau_TOPCOMTriangulationsPoint_2002,DeLoeraRambauSantos_TriangulationsStructuresApplications_2010}.
These methods did not scale well because without any additional
pruning they would process a very large number of nodes.  For later
reference, call this the \emph{no-pruning} method.

\begin{algorithm}[t]
  \TitleOfAlgo{\isFeasible{$\setStyle{T}, \setsysStyle{F}, \fifSet$}}
  
  \KwIn{A subset $\setStyle{T}$ of simplex indices
    in~$\indexSet{\noOfSimplices}$, a set of feasible subsets
    $\setsysStyle{F}$ (implicit), the set~$\fifSet$ of free interior
    facets of~$\setStyle{T}$}
  
  \KwOut{$\true$ if $\setStyle{T}$ is a triangulation, $\false$ otherwise}
  
  \If(\tcc*[f]{no free interior facets?}){
    $\fifSet = \emptyset$
  }{
    \Return \true
    \tcc*{partial triangulation is covering}
  }
  \Else{
    \Return \false
    \tcc*{partial triangulation is not covering}
  }
  
  \caption[Cover-Check for Partial Triangulations]{Check whether a
    partial triangulation is a triangulation by checking whether all
    free interior facets are covered}
  \label{alg:isFeasible_triangulations}  
\end{algorithm}

One significant improvement over this are certain necessary conditions
for expansions being extensions.  The discussion starts with a
condition that, if false, allows to break out of the loop over all
expansions immediately.
\begin{theorem}
  \label{thm:triangulations:lexBreakExpansion}
  Let $\setStyle{T}$ index a non-empty partial triangulation with
  $\freefacets(\setStyle{T})$ and
  $\admissibles(\setStyle{T})$ both non-empty.

  Then, if $\setStyle{T} \cup \{s'\}$ is right-extendable for some
  $s' \ge s$ with $s, s' \in \admissibles(\setStyle{T})$, then
  \begin{equation}
    \min\bigl(\freefacets(\setStyle{T})\bigr)
    \ge
    \min\bigl(\intfacets(s)\bigr).
  \end{equation}
\end{theorem}

\begin{algorithm}[t]
  \TitleOfAlgo{\expand{$\setStyle{T}, s, \globalData, \localData$}}
  
  \KwIn{A subset $\setStyle{T}$ of simplex indices
    in~$\indexSet{\noOfSimplices}$, a new simplex
    index~$s \in \admSet$, global data $\globalData$ encompassing the
    interior-facets table $\iftTable$ and the admissibles table
    $\admTable$ of~$\Conf$, local data $\localData$ of~$\setStyle{T}$
    encompassing the free interior facets $\fifSet$ and the
    admissibles $\admSet$ of~$\setStyle{T}$}
    
  \KwOut{$(\mathrm{break}, \setStyle{T}', \localData')$, where
    $\mathrm{break}$ is true if this and all future expansions cannot
    be extensions, $(\setStyle{T}', \localData')$ is the new node with
    $\setStyle{T}' = \setStyle{T} \cup \{s\}$, and $\localData'$ is
    the correct local Data for~$\setStyle{T}'$}

  \If(\tcc*[f]{can $s$ cover minimal free facet?}){
    $\minElem{\fifSet} < \min\bigl(\iftTable(s)\bigr)$
  }{
    \Return $(\true, -, -)$
    \tcc*{min.~free facet not coverable by~$s' \ge s$}
  }  
  $\setStyle{T}' \gets \setStyle{T} \cup \{s\}$
  \tcc*{add $s$ to the subset}
  $\fifSet' \gets \fifSet \symdiff \iftTable[s]$
  \tcc*{free interior facets by symmetric difference}
  $\admSet' \gets \admSet \cap \admTable[s]$
  \tcc*{admissibles by intersection}
  $\localData' \gets (\fifSet', \admSet')$
  \tcc*{put together local data}  
  \Return $(\false, \setStyle{T}', \localData')$
  \tcc*{return the node}
    
  \caption[Expansion of a Partial Triangulation]{Expand a subset
    indexing a partial triangulation by a new simplex index
  }
  \label{alg:expand_triangulations}  
\end{algorithm}

\begin{proof}
  If the minimal interior facet index of an admissible $s$ is too
  large to cover the minimal free facet of~$\setStyle{T}$, then the
  same holds for each admissible $s' > s$ as well, since
  \begin{equation}
    \min\bigl(\intfacets(s')\bigr) \ge \min\bigl(\intfacets(s)\bigr)\text{ for each $s' > s$}.
  \end{equation}
   The theorem is a formal version of this
  observation.
\end{proof}
Since all partial triangulations are built in lexicographic order,
Theorem~\ref{thm:triangulations:lexBreakExpansion} means the
following: the loop over the expansion sequence in
Algorithm~\ref{alg:symLSRSFeaswithPruning} can be left as soon as the
minimal interior facet index of the new simplex index is larger than
the minimal free interior facet index of~$\setStyle{T}$, in which case
$\setStyle{T} \cup \{s\}$ cannot be right-extended.  The lexicographic
building order guarantees that the minimal free interior facet cannot
be covered by any simplex added in the future either.  The application
of this necessary condition is done inside the $\expand$ subroutine
and is called \emph{lex-breaking}.

In the sequel, some necessary conditions for right-extendability are
discussed.  Before going into the details, note that it is unlikely to
find a complete and efficient exact characterization of
right-extendability.  The reason is the following. The decision
problem of whether or not a partial triangulation can be extended to a
triangulation is NP-hard in general.  This can be derived from the
NP-hardness of triangulating non-convex
$3$-polytopes~\cite{RuppertSeidel_difficultytriangulatingthree_1992},
because for a general partial triangulation the uncovered part yet to
be triangulated is, in general, non-convex.  Here, only extensions to
the right are interesting w.r.t.\ a certain ordering of simplices.
However, since for the purpose of this decision problem one can
reorder the simplices in such a way that the simplices of the given
partial triangulation come first, the right-extendability problem is
NP-hard as well. Thus, it is justified to resort to necessary
conditions, resulting in a semi-check for pruning.  First, the
strongest necessary condition for right-extendability is established.
\begin{theorem}
  \label{thm:triangulations:semiIsNotRightExtFull}
  Let $\setStyle{T}$ index a non-empty partial triangulation.  Then:
  If $\setStyle{T}$ is right-extendable, then there is a covering set
  of simplex indices
  $\setStyle{C} \subseteq \admissibles(\setStyle{T})$ such that
  \begin{enumerate}[label=(\roman*)]
  \item the covering set is pairwise properly intersecting: For each
    $s \in \setStyle{C}$ one has for all
    $s' \in \setStyle{C} \setminus \{s\}$ that $s \in \admsimps(s')$,
  \item the covering set covers all free interior facets: For each
    $f \in \freefacets(\setStyle{T})$ there is an $s \in \setStyle{C}$
    with $f \in \intfacets(s)$.
  \end{enumerate}
\end{theorem}

\begin{proof}
  Note that the set of new simplices in any feasible right-completion
  of $\setStyle{T}$
  \begin{itemize}
  \item stems from the current admissibles of~$\setStyle{T}$,
  \item is itself properly intersecting,
  \item contains for each free interior facet a simplex containing that
    facet.
  \end{itemize}
  Thus, any feasible right-completion is a special case of a covering
  set $\setStyle{C}$ as in the theorem.
\end{proof}
Call the application of
Theorem~\ref{thm:triangulations:semiIsNotRightExtFull}
\emph{full-pruning}.  Note that the existence of a covering set of
simplices as in Theorem~\ref{thm:triangulations:semiIsNotRightExtFull}
does not guarantee the right-extendability, since its elements may
lead to new free interior facets that have to be covered by even more
simplices, which may fail at some point.  It is not straight-forward
how to implement full-pruning without branching-out the potential
covering sets of simplices -- after all, the idea of pruning is
essentially to avoid branching in the first place.

Thus, the following weaker test was developped.  It does not demand a
covering set of simplices which is pairwise intersecting properly.
Instead, for each free interior facet it is checked whether there is a
covering simplex that is intersecting properly with at least one such
covering simplex for each other free interior facet.
\begin{theorem}
  \label{thm:triangulations:semiIsNotRightExtStrong}
  Let $\setStyle{T}$ index a non-empty partial triangulation.  Then:
  If $\setStyle{T}$ is right-extendable, then for
  each~$f \in \freefacets(\setStyle{T})$ there is a multi-covering set
  of simplex indices
  $\setStyle{C}(f) \subseteq \admissibles(\setStyle{T})$, called the
  \emph{$f$-covering simplices}, such that
  \begin{enumerate}[label=(\roman*)]
  \item there is an $s_f \in \setStyle{C}(f)$ that is identical to or
    intersects properly with at least one
    $s_{f'} \in \setStyle{C}(f')$ for all
    $f' \in \freefacets(\setStyle{T})$,
  \item for all $f \in \freefacets(\setStyle{T})$ and all
    $s_f \in \setStyle{C}(f)$ one has $f \in \intfacets(s_f)$.
  \end{enumerate}
\end{theorem}
\begin{proof}
  Given a covering set of simplex indices as
  in~Theorem~\ref{thm:triangulations:semiIsNotRightExtFull},
  $\setStyle{C}(f)$ can be set to the unique simplex index
  in~$\setStyle{C}$ containing~$f$.
\end{proof}
Call the application of the following Theorem
\ref{thm:triangulations:semiIsNotRightExtStrong} \emph{strong
  pruning}.  It was this necessary condition that, for the first time
ever, allowed the enumeration of all symmetry classes of
triangulations for instances like the $4$-cube (a standard benchmark
that has 247,451 symmetry classes of triangulations) in a CPU time
comparable to the CPU times of flip-based algorithms.

\begin{algorithm}[t!]
  \TitleOfAlgo{\semiIsNotRightExtStrong{$\setStyle{T}', \downsetStyle{D},
      \setsysStyle{F}, \globalData, \localData, \localData'$}}
  
  \KwIn{A subset $\setStyle{T}'$ of simplex indices
    in~$\indexSet{\noOfSimplices}$, the downset $\downsetStyle{D}$ of
    subsets indexing partial triangulations, the set $\setsysStyle{F}$
    of triangulations (given implicitly by $\isFeasible$), global data
    $\globalData$ encompassing the covering-simplices table
    $\cstTable$ and the admissibles table $\admTable$ of~$\Conf$,
    local data $\localData$ of~$\setStyle{T}$ encompassing the free
    interior facets $\fifSet$ and the admissibles $\admSet$
    of~$\setStyle{T}'$}
    
  \KwOut{$\true$ if $\setStyle{T}'$ is not right-extendable, $\false$
    if $\setStyle{T}'$ may be right-extendable}

  \For(\tcc*[f]{for all free interior facets}){
      $f \in \fifSet$
  }{
    $\covSet[f] \gets \cstTable[f] \cap \admSet[\setStyle{T}']$
    \tcc*{update the $f$-covering simplices}
    \If(\tcc*[f]{no admissible $f$-covering simplex?}){
      $\covSet[f] = \emptyset$
    }{      
      \Return $\true$
      \tcc*{$f$ not coverable by admissibles}
    }
    $\covSetNew[f] \gets \true$
    \tcc*{$f$-covering simplices are new}
  }
  $\covAnyNew \gets \true$
  \tcc*{something is new}
  \While(\tcc*[f]{while something is new}){
    $\covAnyNew$
  }{
    \For(\tcc*[f]{for all free interior facets}){
      $f \in \fifSet$
    }{
      \If(\tcc*[f]{$f$-covering simplices not new?}){
        $\covSetNew[f] = \false$
      }{
        continue
        \tcc*{next loop element}        
      }
      \tcc{collect $f$-admissible simplices:}
      $\covAdm[f] \gets \bigcup_{s \in \covSet[f]} \admTable[s]$\;
    }
    $\covAnyNew \gets \false$
    \tcc*{nothing new so far}
    \For(\tcc*[f]{for all free interior facets}){
        $f \in \fifSet$
    }{
      \tcc{collect $f'$-admissibles among the $f$-covering simp's:}
      $\covSet' \gets \bigcap_{f' \in \fifSet \setminus \{f\}} \bigl(\covSet[f'] \cup \covAdm[f']\bigr)$\;
      \If(\tcc*[f]{no $f$-covering simplex admissible?}){
        $\covSet' = \emptyset$
      }{      
        \Return \true
        \tcc*{$f$ not coverable by admissibles}
      }
      \If(\tcc*[f]{$\covSet[f]$ need restriction?}){
        $\covSet' \not\supseteq \covSet[f]$
      }{
        \tcc{restrict to $f'$-admissible simplices:}
        $\covSet[f] \gets \covSet[f] \cap \covSet'$\;
        $\covSetNew[f] \gets \true$
        \tcc*{$f$-covering simplices are new}
        $\covAnyNew \gets \true$
        \tcc*{something has changed}
      }
    }
  }
  \Return \false
  \tcc*{$\covSet[f]$ is now a multi-covering set}
  
  \caption[Strong-Pruning]{The strong-pruning semi-check whether a
    partial triangulation can certainly not be right-extended to a
    triangulation based on a direct application of
    Theorem~\ref{thm:triangulations:semiIsNotRightExtStrong}}
  \label{alg:semiIsNotRightExt_triangulations_strong}  
\end{algorithm}

\begin{table}[t]
  \centering{%
    \sffamily\footnotesize
    \begin{tabular}{l*{7}{r}}
      \toprule
      $\Conf$
      & \multicolumn{1}{c}{\# triangulations}
      & \multicolumn{3}{c}{\# nodes}
      & \multicolumn{3}{c}{CPU time [hh:mm:ss]}\\
      & \multicolumn{1}{c}{in total}
      & \multicolumn{3}{c}{by pruning method}
      & \multicolumn{3}{c}{by pruning method}\\
      && no & strong & lex & no & strong & lex\\
      \midrule
      $\hypercube{3}$         &     74 &             2915 &       486 &       497 &           0.02 &  0.01 & 0.01 \\
      $\simplexproduct{3}{2}$ &   4488 &        2,385,961 &    29,423 &    29,577 &           1.31 &  0.11 & 0.06 \\
      $\cyclic{9}{4}$         &    357 &        8,627,257 &      4861 &      4926 &          12.87 &  0.08 & 0.02 \\
      $\cyclic{10}{4}$        &   4824 & $>$5,180,000,000 &    73,085 &    73,259 & $>$29:54:40.79 &  0.83 & 0.11 \\
      $\cyclic{11}{4}$        & 96,426 &               -- & 1,597,366 & 1,597,784 &             -- & 20.16 & 1.72 \\
      \bottomrule
    \end{tabular}
  }
  \caption[Comparison of pruning methods for triangulations]{Comparison of the three pruning methods no-pruning,
    strong-pruning, and lex-pruning for a $3$-cube, a product of a
    tetrahedron and a triangle, and some cyclic polytopes (tiny to
    small instances, single-thread with symmetries ignored)
  }
  \label{tab:triangulations-results-nodesavings-pruning}
\end{table}

Given the lex-orders of simplices and facets, one can prove another
necessary condition for right-extendability that can be evaluated much
faster and has -- surprisingly -- been almost as effective for the
experiments in this paper.  It is based on the following theorem that
is very similar to Theorem~\ref{thm:triangulations:lexBreakExpansion}.
\begin{theorem}
  \label{thm:triangulations:semiIsNotRightExt}
  Let $\setStyle{T}$ index a non-empty partial triangulation with
  $\freefacets(\setStyle{T})$ and
  $\admissibles(\setStyle{T})$ both non-empty.

  Then, if $\setStyle{T}$ is right-extendable, then the following holds:
  \begin{equation}
    \min\bigl(\freefacets(\setStyle{T})\bigr)
    \ge
    \min\Bigl(\intfacets\bigl(\min(\admissibles(\setStyle{T}))\bigr)\Bigr).
  \end{equation}
\end{theorem}

\begin{proof}
  If $\setStyle{T}$ is right-extendable, then for each free facet
  indexed in $\freefacets(\setStyle{T})$ there must be an admissible
  simplex indexed in the admissibles of~$\setStyle{T}$ covering it,
  since
  $\admissibles(\setStyle{T}') \subseteq \admissibles (\setStyle{T})$
  for all $\setStyle{T}' \supseteq \setStyle{T}$.  In particular, for
  the lex-minimal free facet there must be such an admissible simplex.
  The lex-minimal facet that can be covered by some admissible simplex
  is the lex-minimal facet of the lex-minimal admissible simplex.
  Thus, if the lex-minimal free facet is lex-smaller than this, then
  it cannot be covered by \emph{any} admissible simplex of any
  superset $\setStyle{T}' \supseteq \setStyle{T}$, and $\setStyle{T}$
  cannot be right-extendable. The assertion is just a translation of
  this into a formula.
\end{proof}

Call the application of
Theorem~\ref{thm:triangulations:semiIsNotRightExt} \emph{lex-pruning}.
In Table~\ref{tab:triangulations-results-nodesavings-pruning} a
comparison of the node counts is presented for no-pruning and no
lex-breaking, strong-pruning with lex-breaking, and lex-pruning with
lex-breaking for tiny to small examples.  No-pruning is not
competitive, which explains the limited success of all implementations
previously available. Moreover, the examples (and others not in the
table) exhibit that strong-pruning is slightly more effective than
lex-pruning.  However, in all examples computed so far the nodes
pruned extra do not justify the substantially larger effort (see
Section~\ref{sec:triangs:analysis} for an analysis of worst-case
run-times).

\begin{algorithm}[t]
  \TitleOfAlgo{\semiIsNotRightExtLex{$\setStyle{T}', \downsetStyle{D},
      \setsysStyle{F}, \globalData, \localData, \localData'$}}
  
  \KwIn{A subset $\setStyle{T}'$ of simplex indices
    in~$\indexSet{\noOfSimplices}$, the downset $\downsetStyle{D}$ of
    subsets indexing partial triangulations, the set $\setsysStyle{F}$
    of triangulations (given implicitly by $\isFeasible$), global data
    $\globalData$ encompassing the interior-facets table $\iftTable$
    and the admissibles table $\admTable$ of~$\Conf$, local data
    $\localData$ of~$\setStyle{T}$ encompassing the free interior
    facets $\fifSet$ and the admissibles $\admSet$ of~$\setStyle{T}'$}
    
  \KwOut{$\true$ if $\setStyle{T}'$ is not right-extendable, $\false$
    if $\setStyle{T}'$ may be right-extendable}

  \If(\tcc*[f]{min.~free facet too small?}){
    $\minElem{\fifSet} < \min\Bigl(\iftTable\bigl[\min(\admSet)\bigr]\Bigr)$
  }{
    
    \Return $\true$
    \tcc*{min.~free facet not coverable by admissibles}
  }
  \Else{
    \Return $\false$
    \tcc*{no contradiction to right-extendability}
  }
  
  \caption[Lex-Pruning]{The lex-pruning semi-check whether a partial
    triangulation can certainly not be right-extended to a
    triangulation based on a direct application of
    Theorem~\ref{thm:triangulations:semiIsNotRightExt}}
  \label{alg:semiIsNotRightExt_triangulations_lex}  
\end{algorithm}

Now, the specialized subroutines can be presented for application of
Algorithm~\ref{alg:symLSRSFeaswithPruning} to the enumeration of
\emph{all} triangulations.  For the instances considered, the fastest
variant of $\isLexMin$ was $\isLexMinIter$ from
Algorithm~\ref{alg:isLexMinCritelemIter}.  This comes as no surprise
(after the analysis in Section~\ref{sec:new-checks-lexic}), since the
degree $\noOfSimplices$ is usually large compared to the order~$k$ of
the symmetry group.  This time, the global and local auxiliary data
are more voluminous.  The global data $\globalData$ consists of the
configuration~$\Conf$ with hash maps $\admTable$ for its admissibles
and $\iftTable$ for its interior facets.  The local data $\localData$
stored with each subset $\setStyle{T}$ of simplex indices consists,
besides the local data necessary for the lex-min check (in this case a
critical-element table $\critelemTable$ of~$\setStyle{T}$), a set
$\admSet$ representing the admissible simplices of~$\setStyle{T}$ and
a set $\fifSet$ representing the free interior facets
of~$\setStyle{T}$.  The feasibility check just needs part of the local
data and is straight-forward (see
Algorithm~\ref{alg:isFeasible_triangulations}).  When a node is
processed, the corresponding partial triangulation is expanded by a
simplex admissible for it.  Thus, the set $\admSet$ in the local data
of a node and ordered lexicographically yields an expansion sequence.
When a new simplex is added, one can generate the new local data from
the local data of $\setStyle{T}$, as is described in
Algorithm~\ref{alg:expand_triangulations}.  Note that lex-breaking
according to Theorem~\ref{thm:triangulations:lexBreakExpansion} might
be possible.

Finally, Algorithms \ref{alg:semiIsNotRightExt_triangulations_strong}
and~\ref{alg:semiIsNotRightExt_triangulations_lex} list
implementations of strong-pruning and lex-pruning, respectively, that
reveal in many cases when a partial triangulation can certainly not be
right-extended to a triangulation.
Algorithm~\ref{alg:semiIsNotRightExt_triangulations_strong} requires
some explanation.  In that algorithm, for some partial
triangulation~$\setStyle{T}'$, first, for each free interior facet
$f \in \fifSet$ construct a set $\covSet[f]$ containing the
$f$-covering simplices that
\begin{enumerate}[label=(\alph*)]
\item contain the free interior facet~$f$ and
\item intersect properly with~$\setStyle{T}'$.
\end{enumerate}
Second, for each free interior facet $f \in \fifSet$ construct the set
$\covAdm[f]$ of all so-called \emph{$f$-admissible simplices} that are
identical to or intersect properly with at least one of the
$f$-covering simplices.  The goal is to eliminate in rounds over all
$f \in \fifSet$ all those simplices from the sets of $f$-covering
simplices that are not $f'$-admissible for at least one
other~$f' \in \fifSet$.  As soon as no further $f$-covering simplex is
eliminated in a complete round over all $f \in \fifSet$, the
$f$-covering simplices contain a multi-covering set as in
Theorem~\ref{thm:triangulations:semiIsNotRightExtStrong}.  Since in
every while-loop at least one $f$-covering simplex is removed, the
number of while-loop traversals and, hence, the algorithm is finite.

Finally, we specialize the method to enumerate only subsets with
prescribed symmetries (symmetric lexicographic symmetric-subset
reverse search) for triangulations.  Note that a triangulation can
have symmetries that are not in the automorphism group of the point
configuration, since, a-priori, a triangulation need not use all the
points.  All symmetries are in the automorphism group
$\autGroup\bigl(\vertexSet{\Conf}\bigr) \le \symGroup{n}$ of the
vertices $\vertexSet{\Conf}$ of~$\Conf$, though.  Since each
triangulation is a maximal subset of pairwise properly intersecting
simplices (though not vice versa), it is possible to utilize symmetric
lexicographic symmetric-subset reverse search from
Section~\ref{sec:vari-symm-lexic}, in particular,
Definition~\ref{def:symLsymSRS-notions} and
Theorem~\ref{thm:symLsymSRS-main}.  A specialization of
Definition~\ref{def:symLsymSRS-notions} to triangulations reads as
follows:
\begin{definition}
  Let $\Conf$ be a point configuration with $n$ points and
  $\groupStyle{H}$ be a subgroup
  of~$\autGroup\bigl(\vertexSet{\Conf}\bigr)$.  An element $s$ indexes
  an \emph{$\groupStyle{H}$-feasible simplex} if any two elements from
  the $\groupStyle{H}$-orbit of $s$ index identical or properly
  intersecting simplices.  Two elements $s$ and~$s'$ \emph{intersect
    $\groupStyle{H}$-properly} if any two elements from their
  respective $\groupStyle{H}$-orbits index identical or properly
  intersecting simplices.  A triangulation of~$\Conf$ indexed by
  $\setStyle{T}$ is \emph{$\groupStyle{H}$-invariant}, if
  $\sigma(T) = T$ for all $\sigma \in \groupStyle{H}$.  A symmetry
  $\grg$ of~$\Conf$ is \emph{$\groupStyle{H}$-feasible} whenever for
  all $s \in \indexSet{n}$ one has that $s$ indexes an
  $\groupStyle{H}$-feasible simplex if and only if $\grg(s)$ indexes
  an $\groupStyle{H}$-feasible simplex.
\end{definition}
An application of Theorem~\ref{thm:symLsymSRS-main} readily yields the
following specialization to triangulations.
\begin{theorem}
  \label{thm:symLsymSRS-main-triang}
  Let $\Conf$ be a point configuration with $n$ points, and let
  $\groupStyle{H}$ be a subgroup
  of $\autGroup\bigl(\vertexSet{\Conf}\bigr)$.  Let $\setStyle{T}$ be
  a non-empty subset of pairwise $\groupStyle{H}$-properly
  intersecting $\groupStyle{H}$-feasible simplices whose set of free
  interior facets is empty.  Then $\setStyle{T}$ indexes an
  $\groupStyle{H}$-invariant triangulation of~$\Conf$.
\end{theorem}
\begin{proof}
  Let $\setStyle{T}$ be as in the assumption.  Since proper
  $\groupStyle{H}$-intersection implies in particular proper
  intersection, $\setStyle{T}$ indexes a partial triangulation.  Since
  it has no interior free interior facets, it indexes a triangulation.
  The $\groupStyle{H}$-invariance now follows from
  Theorem~\ref{thm:symLsymSRS-main}.
\end{proof}
That means: Once the bijection indexing simplices has been restricted
to $\groupStyle{H}$-feasible elements, the symmetries have been
restricted to $\groupStyle{H}$-feasible symmetries, and the
admissibles table of the point configuration has been modified to
represent $\groupStyle{H}$-proper intersection, then
$\symLSRSFeaswithData$ will automatically enumerate only
$\groupStyle{H}$-invariant triangulations up to
$\groupStyle{H}$-feasible symmetries.

\subsection{Analysis}
\label{sec:triangs:analysis}

Let $\nonPrunables \subseteq \downsetStyle{D}$ be the set of subsets
satisfying the necessary conditions of
Theorem~\ref{thm:triangulations:lexBreakExpansion} and
Theorem~\ref{thm:triangulations:semiIsNotRightExtFull}
or~\ref{thm:triangulations:semiIsNotRightExt}, depending on which
pruning method is chosen.  Then, the following can be said about the
run-time complexity of the resulting algorithm.
\begin{theorem}
  \label{thm:triang-alg-efficiency}
  Assume that $\isLexMin$ is implemented via the critical-element
  method.  Let $\noOfSimplices$ be the number of simplices in~$\Conf$.
  Moreover, let $m$ be the maximal number of simplices in a
  triangulation of~$\Conf$.  Then, the run-time complexity of
  $\symLSRSFeaswithData$ with strong-pruning is in
  $O\bigl(\noOfSimplices (\abs{\GrG} + (\rank^3 m^3 \corank + \rank^2
  m^2 \corank^2)
  \noOfSimplices))\abs{\setOrbits{\nonPrunables}{\GrG}}\bigr)$; with
  lex-pruning the run-time is in
  $O\bigl(\noOfSimplices (\abs{\GrG} + \noOfSimplices + \rank
  m)\abs{\setOrbits{\nonPrunables}{\GrG}}\bigr)$.
  
  Moreover:
  \begin{enumerate}[label=(\roman*)]
  \item\label{itm:triangs-effdoublesimp} For any
    $\rank \in \mathbb{N}$, there is a point
    configuration~$\standardsimplex{\rank - 1}^{\mathrm{dup}}$ with
    $2\rank$ points of rank~$\rank$ whose number of triangulations up
    to symmetry is one and whose number of simplices $\noOfSimplices$
    is~$2^{\rank}$, for which $\symLSRSFeaswithData$ indeed traverses
    its main loop $2^{\rank}$ times.
  \item\label{itm:triangs-effline} For any $\no \in \mathbb{N}$, there
    is a point configuration~$\matrixStyle{L}_{\no}$ with $\no + 2$
    points of rank~$2$ whose number of triangulations is
    in~$\Theta\bigl(2^{\no}\bigr)$, whose number of right-extendable
    partial triangulations is in~$\Theta\bigl(2^{\no}\bigr)$, and
    whose number of all partial triangulations is in
    $\Theta\bigl((\tfrac{3 + \sqrt{5}}{2})^{\no}\bigr)$.  In
    particular, asymptotically the number of all partial
    triangulations is
    $\frac{1}{2}\bigl(\tfrac{3 + \sqrt{5}}{4}\bigr)^{\no}$ times as
    large as the number of right-extendable partial triangulations.
    Moreover, all expandable but not right-extendable partial
    triangulations can be pruned by lex-pruning or strong-pruning.
  \item\label{itm:triangs-effcross} For any
    $\dimension \in \mathbb{N}$, there is a point
    Configuration~$\hypercross{\dimension}$ with $\no = 2\dimension$
    points of rank~$\dimension + 1$ whose number of triangulations
    is~$\dimension$, whose number of right-extendable partial
    triangulations is $1 + \dimension 2^{\dimension - 1}$, and whose
    number of all partial triangulations is
    $1 + \dimension \bigl(2^{(2^{\dimension - 1})} - 1\bigr)$.  In
    particular, the number of all partial triangulations is
    asymptotically $2^{(2^{\dimension - 1} - \dimension)}$ times as
    large as the number of right-extendable partial triangulations.
    Moreover, all expandable but not right-extendable partial
    triangulations can be pruned by strong-pruning, but not
    necessarily by lex-pruning.
  \item\label{itm:triangs-UB} For any $k \in \mathbb{N}$, there is a
    point configuration~$\matrixStyle{P}_k$ with $\no = 4(k+3)$ points
    of rank~$4$ for which, independent of the order of points, there
    are at least $2^{4k}$ not right-extendable partial triangulations
    all of which can be pruned by strong-pruning.
  \item\label{itm:triangs-LB} For any $k \in \mathbb{N}$, there is a
    point configuration~$\matrixStyle{P}_k$ with $\no = 5(k+3)$ points
    of rank~$4$ for which, independent of the order of points, there
    are at least $2^{5k}$ not right-extendable partial triangulations
    none of which can be pruned by full-pruning, strong-pruning or
    lex-pruning.
  \end{enumerate}
\end{theorem}
The first item shows that the number of simplices, and, therefore, the
number of traversals of the main loop in $\symLSRSFeaswithData$, can
be exponential in the number of triangulation \emph{orbits}.  The
second and third items show to what extent (symmetries ignored) the
number $\abs{\downsetStyle{D}}$ of all partial triangulations can grow
compared to the number of right-extendable partial triangulations and
to what extent $\nonPrunables$ can be smaller than~$\downsetStyle{D}$.
The fourth item shows that, symmetries ignored and independent of the
order of points, strong-pruning can lead to a pruning of an
exponential number of partial triangulations, whereas the fifth item
shows that, even with full-pruning, an exponential number of
deadends in the enumeration tree is possible.
\begin{proof}
  Preparations: First note that unions and symmetric differences of
  subsets can be implemented in (amortized) time linear in the sum of
  the elements of all the operands, whereas intersections can be
  implemented in (amortized) time linear in the number of elements of
  the smaller set whenever it is known which set this is.  The
  cardinality of the set of free interior facets~$\fifSet$
  in~$\setStyle{T}'$ is at most $\rank m$.  The best straight-forward
  worst-case estimate for the number of simplices in~$\admTable[s]$
  admissible to a given simplex~$s$ is the number of all
  simplices~$\noOfSimplices$.  The best straight-forward worst-case
  estimate for the number of $f$-covering simplices $f \cup \{i\}$
  in~$\covSet[f]$ for a given~$f$ is at most the
  number~$\corank + 1 = \no - (\rank - 1) $ of points not in~$f$,
  which is in~$O(\corank)$.
  
  The number of elements in an expansion sequence of a non-empty
  subset is no larger than the number of admissibles of one of its
  members, which, by the preparations, is no more than the
  number~$\noOfSimplices$ of simplices.  This is the maximal number of
  traversals of the main loop in $\symLSRSFeaswithData$.  The
  dominating remaining subroutine in the main loop besides $\isLexMin$
  and $\semiIsNotRightExt$ is $\expand$.  By the preparations, the
  updates of $\fifSet$ and $\admSet$ take no more than
  $O(\rank m + \noOfSimplices)$ operations. The remaining effort
  inside the main loop differs according to the pruning method.

  For strong-pruning: Since in each while-loop traversal of
  strong-pruning at least one $f$-covering simplex is removed, the
  while-loop traversals are bounded by the number of simplices
  in~$\covSet[f]$ summed over all~$f \in \fifSet$.  Thus, the number
  of while-loop traversals is in $O(\rank m \corank)$.  The number of
  for-loop traversals is in $O(\rank m)$ for each of the three
  for-loops.
  
  The effort inside of the loops consists of
  \begin{enumerate}[label=(\alph*)]
  \item \label{itm:strong-pruning:covset} the computation for
    $\covSet[f]$ for all $f \in \fifSet$ prior to the while-loop;
  \item \label{itm:strong-pruning:covadm} the computation of
    $\covAdm[f]$ for all $f \in \fifSet$ inside the while-loop;
  \item \label{itm:strong-pruning:update} the computation of
    $\covSet'$ to update $\covSet[f]$ for all $f \in \fifSet$ inside
    the while-loop.
  \end{enumerate}   
  Concerning item~\ref{itm:strong-pruning:covset}, note that the
  asymptotically smaller subset in the set-intersection computation of
  $\covSet[f]$ is $\cstTable[f]$, whose number of elements is in
  $O(\corank)$.  Thus, the time to compute $\covSet[f]$ for all
  $f \in \fifSet$ is in~$O(\rank m \corank)$.

  Concerning item~\ref{itm:strong-pruning:covadm}, note that the
  number of elements involved in the set-union computation of
  $\covAdm[f]$ is in $O(\corank \noOfSimplices)$.  Thus, the
  computation of $\covAdm[f]$ for all $f \in \fifSet$ takes time
  in~$O(\rank m \corank \noOfSimplices)$.  Since this happens inside
  the while-loop, the total effort for this is in
  $O\bigl((\rank m \corank) (\rank m \corank \noOfSimplices)\bigr) =
  O(\rank^2 m^2 \corank^2 \noOfSimplices)$.

  Finally, concerning~\ref{itm:strong-pruning:update}, note that the
  total number of elements involved in the mixed set-union and
  set-intersection computation of $\covSet'$ is in
  $O(\rank m \noOfSimplices)$ (for the set-intersections the
  respective smaller sets are unclear, whence one cannot take
  advantage of this), which dominates the third for-loop.  Hence, the
  potential restriction of $\covSet[f]$ for all $f \in \fifSet$ takes
  time in $O(\rank^2 m^2 \noOfSimplices)$.  Since this happens inside
  the while-loop, the total effort for this is in
  $O\bigl((\rank m \corank) (\rank^2 m^2 \noOfSimplices)\bigr) =
  O(\rank^3 m^3 \corank \noOfSimplices)$.

  Putting all this together, strong-pruning takes time in
  $O\bigl((\rank^3 m^3 \corank + \rank^2 m^2 \corank^2)
  \noOfSimplices)\bigr)$, which dominates the effort for $\expand$.
  Consequently, the run-time complexity of \symLSRSFeaswithData is in
  $O\bigl(\noOfSimplices (\abs{\GrG} + (\rank^3 m^3 \corank + \rank^2
  m^2 \corank^2)
  \noOfSimplices))\abs{\setOrbits{\nonPrunables}{\GrG}}\bigr)$ for
  triangulations with strong-pruning.
  
  For lex-pruning: Whenever for all sets sorted structures are used or
  their minima are cached during each modification, the run-time
  complexity of lex-pruning is constant, which is dominated by the
  effort for $\expand$. Summarized, for lex-pruning the run-time of
  the resulting algorithm $\symLSRSFeaswithData$ is in
  $O\bigl(\noOfSimplices(\abs{\GrG} + \rank m +
  \noOfSimplices)\abs{\setOrbits{\nonPrunables}{\GrG}}\bigr)$ -- an
  improvement so significant compared to strong-pruning that now
  $\semiIsNotRightExt$ is not even the dominating subroutine anymore!

  For item~\ref{itm:triangs-effdoublesimp}, define
  $\standardsimplex{\rank - 1}^{\mathrm{dup}}$ to be the
  $\rank - 1$-dimensional standard simplex
  $\standardsimplex{\rank - 1}$ with all points duplicated.  Up to
  symmetry, there is only one triangulation.  The number of simplices,
  however, is $2^{\rank}$.  In $\symLSRSFeaswithData$, each of these
  simplices will first be added to the empty partial triangulation
  before the check for lexicographic minimality can dismiss these
  redundant branches.
  
  For item~\ref{itm:triangs-effline}, consider the point
  configuration~$\matrixStyle{L}_{\no}$ of rank~$2$ consisting of
  $\no + 2$ distinct points on a line labelled consecutively from $0$
  through~$\no + 1$.  It has $2^{\no}$ triangulations, each
  corresponding to the set of used interior points.  Any partial
  triangulation indexed by~$\setStyle{T}$ induces a \emph{segment
    pattern}
  $\mathcal{P}(\setStyle{T}) :=
  \epsilon_0\epsilon_1\dots\epsilon_{\no} \in \{-1, 0, 1\}^{\no + 1}$
  with the following meaning: $\epsilon_i = -1$ if $[i, i+1]$ is not
  covered by a segment in~$\setStyle{T}$, $\epsilon_i = 1$ if
  $[i, i+1]$ is covered by a segment in~$\setStyle{T}$ starting
  at~$i$, and $\epsilon_i = 0$ if $[i, i+1]$ is covered by a segment
  in~$\setStyle{T}$ starting at some~$k < i$, where
  $i = 0, 1, \dots, \no$.  Call a pattern in $\{-1, 0, 1\}^{\no + 1}$
  \emph{consistent} if it does not start with a ``$0$'' and if no
  ``$0$'' ever follows a ``$-1$''.  Then, $\mathcal{P}(\setStyle{T})$
  is consistent for all partial triangulations~$\setStyle{T}$.  Given
  any consistent pattern~$\mathcal{P} \in \{-1, 0, 1\}^{\no + 1}$, one
  can construct (from left to right) a partial triangulation
  $\setStyle{T}$ for which $\mathcal{P} = \mathcal{P}(\setStyle{T})$:
  A ``$1$'' starts a new segment, a ``$0$'' continues the same
  segment, and a ``$-1$'' starts an uncovered interval.  Therefore,
  $\mathcal{P}$ is a bijection between consistent patterns in
  $\{-1, 0, 1\}^{\no + 1}$ and partial triangulations.  The count
  $C(j)$ of all consistent patterns of length~$j$ is surprisingly
  interesting: Call $P(j)$ the number of patterns of length~$j$ not
  ending with a ``$-1$'' (\emph{positive patterns}), and call $N(j)$
  the number of patterns of length~$j$ ending with a ``$-1$''
  (\emph{negative patterns}).  Then the consistency implies that
  $P(1) = 1$, $N(1) = 1$.  In order to obtain a positive pattern,
  negative patterns can only be extended consistently by a ``$1$'',
  whereas positive patterns can be extended by a ``$0$'' or a ``$1$''.
  In order to obtain a negative pattern, any pattern can and must be
  extended by a ``$-1$''. Therefore:
  \begin{equation}
    \label{eq:consistent_pattern_recursion}
    P(j + 1) = 2 P(j) + N(j), \quad N(j + 1) = P(j) + N(j).
  \end{equation}
  Let now $F_1, F_2, F_3, F_4, \dots$ be the Fibonacci series with
  $F_1 = F_2 = 1$ and $F_j = F_{j-1} + F_{j-2}$ for $j > 2$.  Define
  $P'(j) := F_{2j}$ and $N'(j) := F_{2j-1}$ with $P'(1) = F_2 = 1$ and
  $N'(1) = F_1 = 1$. Moreover:
  \begin{align}
    \label{eq:consistent_pattern_fibonacci}
    P'(j + 1) &= F_{2j + 2}
    = F_{2j + 1} + F_{2j}
    = F_{2j} + F_{2j - 1} + F_{2j}
    = 2P'(j) + N'(j), \\
    N'(j + 1) &= F_{2j + 1}
    = F_{2j} + F_{2j - 1}
    = P'(j) + N'(j).
  \end{align}
  For all $j = 1, 2, \dots$ this proves that $P(j) = P'(j) = F_{2j}$
  and $N(j) = N'(j) = F_{2j - 1}$. Consequently,
  $C(j) = P(j) + N(j) = N(j + 1) = F_{2j + 1}$.  By a straight-forward
  application of the Eigenvalue-method to the linear system of
  difference equations
  \begin{equation}
    \binom{P(j+1)}{N(j+1)} = 
    \begin{pmatrix}
      2 & 1\\
      1 & 1
    \end{pmatrix}
    \binom{P(j)}{N(j)}
  \end{equation}
  for $P(j)$ and~$N(j)$, the asymptotic growth of
  $C(\no) = P(\no) + N(\no) = N(\no + 1)$ is exactly
  $(\tfrac{3 + \sqrt{5}}{2})^{\no}$.

  A non-empty partial triangulation $\setStyle{T}$ is right-extendable
  if and only if its segment pattern $\mathcal{P}(\setStyle{T})$ is
  consistent and has no subpattern ``$(-1)1$''.  That is,
  right-extendable partial triangulations cover the line completely up
  to some right-most covered point.  Hence, each subset of the points
  $1, 2, \dots, \no + 1$ gives rise to a right-extendable partial
  triangulation by using the maximal point as the right-most point
  covered by the partial triangulation ($0$ for the empty subset) and
  the remaining points to specify a complete triangulation of the line
  to the left of it.  This induces a bijection between the set of
  right-extendable partial triangulations and the set of subsets of
  $\{1, 2, \dots, \no + 1\}$, which has cardinality~$2^{\no + 1}$,
  which is in $\Theta\bigl(2^n\bigr)$.  Hence, the number of partial
  triangulations is, asymptotically, by an exponential factor
  $\bigl(\frac{3 + \sqrt{5}}{2}\bigr)^{\no}/2^{\no + 1} =
  \frac{1}{2}\bigl(\tfrac{3 + \sqrt{5}}{4}\bigr)^{\no} >
  \frac{1}{2}\bigl(\frac{5}{4}\bigr)^{\no}$ larger than the number of
  right-extendable partial triangulations of~$\matrixStyle{L}_{\no}$.

  Because of the given order of points from left to right, each
  ``$(-1)1$'' subpattern is detected immediately by lex-pruning:
  assume that the segment-pattern of an expandable but
  not-right-extendable partial triangulation $\setStyle{T}$ has the
  ``$-1$'' is in position $i-1$ and the ``$1$'' in position $i$ for
  some $i \in \{1, \dots, \no\}$.  This means, the interval $[i-1, i]$
  is uncovered, and the interval $[i, i+1]$ is covered by some simplex
  in~$\setStyle{T}$. Hence, $\setStyle{T}$ indexes a triangulation
  $\setsysStyle{T}$ that has $\{i\}$ as a free interior facet of one
  of its simplices $\{i, j\}$, $j > i$.  The admissibles of
  $\setStyle{T}$ index, by definition, simplices that are all
  lex-larger than~$\{i, j\}$ and intersecting properly with
  $\{i, j\}$.  The lex-minimal admissible simplex
  for~$\setsysStyle{T}$ is therefore lex-at-least $\{j, j+1\}$.  The
  minimal facet of the minimal admissible simplex
  for~$\setsysStyle{T}$ is lex-at-least~$\{j\}$, which is lex-larger
  than the free interior facet~$\{i\}$ of~$\setsysStyle{T}$, which is
  equal to or lex-larger than the minimal free interior facet
  of~$\setsysStyle{T}$.  Consequently, $\setStyle{T}$ will be
  lex-pruned.  Therefore, all the expandable but not right-extendable
  partial triangulations will be pruned by lex-pruning and, therefore,
  also by strong-pruning.
  
  For item~\ref{itm:triangs-effcross}, consider the
  $\dimension$-dimensional regular cross
  polytope~$\hypercross{\dimension + 1} = \bigl(\begin{smallmatrix}
    \hypercross{\dimension} & \vectorStyle{O}_{\dimension} & \vectorStyle{O}_{\dimension}\\
    0 & 1 & -1
  \end{smallmatrix}\bigr)$ with $\hypercross{1} =
  (1, -1)$ and $\vectorStyle{O}_{\dimension}$ the origin in
  dimension~$\dimension$ (coordinates may be homogenized by adding a
  row of ones throughout).  From the fact that $\hypercross{1}$ is a
  segment and $\hypercross{\dimension + 1}$ is the one-point
  suspension of~$\hypercross{\dimension}$ at the origin for
  $\dimension > 0$ (see
  \cite[Chp.~4]{DeLoeraRambauSantos_TriangulationsStructuresApplications_2010})
  the following can be easily derived by induction:
  $\hypercross{\dimension}$ has $2\dimension$ points in
  rank~$\dimension + 1$ and $2^{\dimension}$ facets.  For each
  \emph{antipodal diagonal} $\setStyle{D}$ between antipodal points
  $2k - 1$ and $2k$, $k = 1, 2, \dots, \dimension$, there is exactly
  one triangulation $\setStyle{T}(\setStyle{D})$ using it.  Each
  triangulation consists of~$2^{\dimension - 1}$ many simplices.  Each
  simplex in $\setStyle{T}(\setStyle{D})$ has exactly two boundary
  facets (those not containing~$\setStyle{D}$) and $\dimension - 1$
  interior facets (those containing~$\setStyle{D}$).  The crucial
  property of any triangulation of a regular cross polytope is that
  each triangulation $\setStyle{T}(\setStyle{D})$ is uniquely
  determined by any of its simplices, since all simplices in
  $\setStyle{T}(\setStyle{D})$ contain~$\setStyle{D}$.  Furthermore,
  only the simplices inside $\setStyle{T}(\setStyle{D})$ intersect
  properly with each other, since all distinct antipodal diagonals
  intersect improperly by themselves (the origin is in the interior of
  each antipodal diagonal).  Therefore, only the empty set (which is a
  partial triangulation right-extendable to any triangulation) and the
  $\dimension 2^{\dimension - 1}$ non-empty lex-subsets of one of its
  $\dimension$ many triangulations~$\setStyle{T}$ can be
  right-extended, namely to~$\setStyle{T}$.  In contrast to this, in
  total there are
  $\dimension \bigl(2^{(2^{\dimension - 1})} - 1\bigr)$ non-empty
  partial triangulations so that the total number of partial
  triangulations including the empty set is as claimed.  Since any
  free interior facet~$f$ of a partial triangulation~$\setStyle{T}$
  contains some antipodal diagonal~$\setStyle{D}$, only simplices
  from~$\setStyle{T}(\setStyle{D})$ contain~$f$.  Because in
  $\setStyle{T}(\setStyle{D})$ the interior facet~$f$ is contained in
  a unique other simplex~$s$ not in~$\setStyle{T}$, any multi-covering
  set of simplices $\setStyle{C}(f)$ like in
  Theorem~\ref{thm:triangulations:semiIsNotRightExtStrong} must
  contain~$s$.  Therefore, such a multi-covering set exists if and
  only if $\setStyle{T}$ is a lex-subset
  of~$\setStyle{T}(\setStyle{D})$, i.e., if and only if $\setStyle{T}$
  is right-extendable to~$\setStyle{T}(\setStyle{D})$.  In other
  words, all not right-extendable partial triangulations can be pruned
  by strong-pruning.  Note that there may be partial triangulations
  where the lex-min free interior facet is coverable by an admissible
  simplex, but not all free interior facets are.  Thus, lex-pruning
  may miss some not right-extendable partial triangulations.  (In
  experiments it can be seen that the number of misses is very small,
  though.)

  For the constructions of the~$\matrixStyle{P}_{k}$ in items
  \ref{itm:triangs-UB} and~\ref{itm:triangs-LB} proceed as follows.
  The order of points will not be relevant here, since in this
  particular case it can be proved that the partial triangulations
  involved cannot be extended at all, even if all so-far unused
  simplices could be used as expansions. For both items consider for a
  fixed $\ell \ge 3$ the point
  configuration~$\matrixStyle{P} := \matrixStyle{P}(\ell)$ forming a
  prism over a regular $\ell$-gon with $2\ell$ points.  For
  $i = 0, 1, \dots, \ell - 1$ (all $i$-indices below are considered
  modulo~$\ell$), call the top points $\vectorStyle{v}_i$ (labelled by
  $v_i$), and call the bottom points $\vectorStyle{w}_i$ (labeled by
  $w_i$).  Consider the cyclic set of diagonals $v_iw_{i+1}$ inside
  the $\ell$ quadrilateral facets of~$\matrixStyle{P}$ labelled
  by~$F_i$.  These diagonals induce unique triangulations
  $T_i := \{v_iw_iw_{i+1}, v_iw_{i+1}v_{i+1}\}$ of all the~$F_i$.
  From \cite{Rambau_generalizationSchonhardtspolyhedron_2005} it is
  known that~$\matrixStyle{P}$ has no triangulation that
  uses~$\bigcup_{i=0}^{\ell-1} T_i$.

  Beyond each of the $\ell$ facets~$F_i$, add $k + 1$ distinct points
  labeled by $U_i := \{u_{i1}, \dots, u_{ik}\}$ and $u_i$ in the order
  of increasing distance to~$F_i$ on a line through the barycenter
  of~$F_i$ perpendicular to~$\aff(F_i)$.  This results in the point
  configuration~$\matrixStyle{P}_k$ with $\ell(k + 3)$ points in
  rank~$4$.  Its convex hull consists of the prism over the regular
  $\ell$-gon and $\ell$ pyramids over the $\ell$ quadrilateral facets
  with apices~$u_i$.
  
  Construct now partial triangulations triangulating the
  $\ell$~pyramids.  To triangulate pyramid $F_i * u_{i}$, choose an
  arbitrary subset $V_i \subseteq U_i$ and place the points in order
  of increasing distance to~$F_i$ to generate a \emph{placing
    triangulation} (see
  \cite[Chp.~4]{DeLoeraRambauSantos_TriangulationsStructuresApplications_2010}).
  This results in the triangulation, where the first point in~$V_i$ or
  $u_i$ (if $V_i$ is empty) is joined to~$T_i$, and all the remaining
  points (including~$u_i$) are joined to the visible boundary facets
  of the previous partial triangulation.  For each pyramid, no two
  distinct subsets $V_i$ lead to identical triangulations, since a
  point is a vertex of some simplex in the placing triangulation if
  and only if it is in~$V_i$.  Moreover, the various triangulations of
  the pyramids can be combined arbitrarily to a partial triangulation
  of all the
  pyramids~$\bigcup_{i=1}^{\ell} \bigl( T_i * u_{i} \bigr)$.  This
  results in $2^{\ell k}$ distinct partial triangulations.  All these
  partial triangulations use the cyclic set of diagonals of the prism,
  i.e., none of them is extendable to a triangulation
  of~$\matrixStyle{P}_k$.  In particular, none of them is
  right-extendable.

  Furthermore, all of them have $\bigcup_{i=0}^{\ell-1} T_i$ as their
  free interior facets.  With the help of the Cayley-trick (see
  \cite{HuberRambauSantos_Cayleytricklifting_2000} for the general
  theory and \cite{Rambau_generalizationSchonhardtspolyhedron_2005}
  for the application to prims over $\ell$-gons) the following can be
  eye-balled: for $\ell \le 4$ not all free interior facets~$f$ can be
  covered by a covering set of simplices $\setStyle{C}(f)$ as in
  Theorem~\ref{thm:triangulations:semiIsNotRightExtStrong}, and for
  $\ell \ge 5$ the set of simplices
  $\bigcup_{i=0}^{k-1} \{v_iw_{i+1}v_{i+1}w_{i+2}\}$ is pairwise
  intersecting properly and covers all free interior facets.  Thus,
  for $\ell = 4$ we obtain item~\ref{itm:triangs-UB}, and for
  $\ell = 5$ we obtain item~\ref{itm:triangs-LB}.
\end{proof}
If $\GrG$ is given as an explicit set of permutations,
Theorem~\ref{thm:triang-alg-efficiency} shows that enumeration of
triangulations with $\symLSRSFeaswithData$ is polynomial in input size
and the number of orbits of non-prunable subsets.  This is because, in
this case, the number of simplices is at most the number of
triangulations, which is at most the number of triangulation orbits
times the group order, which is at most the number of non-prunable
subsets times the input size.  The first example shows that this may
be exponential in the input and output sizes in case $\GrG$ is given
by a set of generators, even if $\isLexMin$ could be implemented in
polynomial time. The second and third examples show that
$\symLSRSFeaswithData$ without pruning is for
counting/enumeration/listing, in general, not polynomial in the input
and output sizes. The third example, moreover, shows that
$\symLSRSFeaswithData$ with lex-breaking and strong- or lex-pruning
for counting/enumeration cannot be polynomial in the input and output
sizes because the exponentially many elements per found object all
have to be touched by the algorithm.  Whether $\symLSRSFeaswithData$
with lex-breaking and strong- or lex-pruning is polynomial in the
input and output sizes for listing when $\GrG$ is given as an explicit
set of permutations remains an open problem.

\subsection{Results}
\label{sec:triangs:results}

\begin{table}[htbp]
  \centering{%
    \sffamily\footnotesize
    \begin{tabular}{l@{}*{5}{r@{\ \ }}}
      \toprule
      $\Conf$
      & \# sym's
      & \# triangulations
      & \# triangulations
      & \# nodes
      & CPU time\\
      && up to symmetry & in total && [hh:mm:ss]\\
      \midrule
      $\hypercube{4}$                    &        384 &           247,451 &          92,487,256 &          3,446,659 &    0:00:01\\
      $\simplexproduct{6}{2}$            &     30,240 &           533,242 &      16,119,956,160 &          6,325,472 &    0:01:46\\
      $\simplexproduct{4}{3}$            &       2880 &         7,402,421 &      21,316,106,880 &        116,083,390 &    0:02:31\\
      *$\simplexproduct{5}{3}$           &      17280 &    25,606,173,722 & 442,472,050,753,920 &    429,725,338,124 & 1332:19:55\\
      *$\hypersimplex{7}{2}$             &       5040 &        37,676,752 &     189,355,661,460 &      1,397,621,560 &    0:47:15\\
      *$\hypersimplex{6}{3}$             &       1440 &        59,708,427 &      85,793,497,200 &      2,555,523,948 &    0:27:08\\
      $3\standardsimplex{3}$             &         24 &       925,148,763 &      22,201,684,367 &      7,154,329,212 &    0:10:56\\
      $\cyclic{14}{8}$                   &         28 &         2,429,751 &          68,007,706 &        187,209,582 &    0:00:23\\
      $\cyclic{19}{14}$                  &         38 &         6,515,385 &         247,567,074 &      1,265,333,660 &    0:04:09\\
      $\cyclic{16}{3}$                   &          2 &    58,492,955,941 &      116,985,744,91 &    887,659,233,813 &   17:16:57\\
      *$\cyclic{14}{4}$                  &         28 &       244,771,183 &       6,853,476,616 &      8,343,040,544 &    0:12:59\\
      *$\cyclic{15}{4}$                  &         30 &    18,845,509,142 &     565,365,033,880 &    650,520,069,380 &   18:59:44\\
      *$\cyclic{14}{5}$                  &          2 &   328,152,636,588 &     656,305,030,644 & 11,575,302,270,565 &  320:42:24\\
      *$\cyclic{14}{6}$                  &         28 &     8,314,337,199 &     232,797,963,456 &    535,970,897,965 &   18:10:42\\
      *$\cyclic{14}{7}$                  &          2 &    15,813,939,113 &      31,627,843,174 &    937,148,113,630 &   23:36:37\\
      *$\cyclic{15}{9}$                  &          2 &     1,397,895,884 &       2,795,741,709 &    117,124,014,923 &    3:21:53\\
      *$\cyclic{16}{10}$                 &         32 &     1,902,605,255 &      60,881,310,552 &    213,053,994,784 &    9:52:13\\
      *icosahedron                       &        120 &                95 &                8598 &               3813 &    0:00:00\\
      *pseudoicosahedron                 &         24 &              7701 &             182,670 &            147,775 &    0:00:00\\
      *dodecahedron                      &        120 &    12,775,757,027 &   1,533,079,037,570 &    125,333,463,449 &    4:58:15\\
      *pyritohedron                      &         24 & 1,363,918,758,719 &  32,734,029,351,118 & 13,786,801,148,394 &  381:32:27\\
      *$\santospointconfseventeen{0}$ \rlap{[HPC$^{192}$]} &        128 & 2,239,192,475,127 & 282,379,150,400,112 &156,459,603,003,248 &  385:32:29\\
      \bottomrule
    \end{tabular}
  }
  \caption[Computational results for triangulations]{Computational results for the enumeration of triangulations
    using 16 threads (numbers with a ``*'' are new)
  }
  \label{tab:results-triangulations}
\end{table}
\begin{table}[htbp]
  \centering{%
    \sffamily\footnotesize
    \begin{tabular}{lrrrr}
      \toprule
      $\Conf$
      & \multicolumn{4}{c}{CPU times [hh:mm:ss]}\\
      & \multicolumn{2}{c}{\texttt{mptopcom 1.4}}
      & \multicolumn{2}{c}{\texttt{TOPCOM 1.2.0b}}\\
      & subregular   & regular & all & regular\\
      \midrule
      $\hypercube{4}$         & 0:01:46  & 0:02:41 & 0:00:01 & 0:00:54\\
      $\simplexproduct{6}{2}$ & 0:09:25  & 0:13:10 & 0:01:46 & 0:05:36\\
      $\cyclic{12}{4}$        & 0:00:21  &       - & 0:00:01 &       -\\
      \bottomrule
    \end{tabular}
  }
  \caption[Comparison with \texttt{mptopcom} for triangulations]{Comparison of CPU times with the so far fasted software
    \texttt{mptopcom 1.4}
    in~\cite{JordanJoswigKastner_Parallelenumerationtriangulations_2018}
    on small examples using 16 threads (the triangulations output was
    activated in both and then piped to \texttt{/dev/null} for a fair
    comparison); for these examples, all triangulations are
    subregular; since \texttt{mptopcom 1.4} does not support
    regularity checks for cyclic polytopes, the regular-triangulations
    timing was skipped for $\cyclic{12}{4}$ }
  \label{tab:comparison-triangulations}
\end{table}
\begin{table}[htbp]
  \centering{%
    \sffamily\footnotesize
    \begin{tabular}{l@{}rrrrr}
      \toprule
      $\Conf$
      & \# sym's
      & \# reg.~triang's
      & \# reg.~triang's
      & CPU time
      & CPU time\\
      && up to symmetry
      & in total
      & all
      & regular\\
      \midrule
      $\hypercube{4}$                   &        384 &    235,277 &          87,959,448 &    0:00:01 &  0:00:54\\
      $\hypersimplex{7}{2}$             &       5040 & 30,485,496 &     153,209,697,210 &    0:47:15 &  8:42:33\\
      $\hypersimplex{6}{3}$             &       1440 & 42,489,025 &      61,035,863,100 &    0:27:08 & 18:10:37\\
      \bottomrule
    \end{tabular}
  }
  \caption[Comparison for triangulations with regularity
  checks]{Comparison of CPU times for the enumeration of all versus
    regular triangulations using 16 threads for selected examples}
  \label{tab:results-regular-triangulations}
\end{table}

Table~\ref{tab:results-triangulations} presents the results obtained
by applying $\symLSRSFeaswithData$ with lex-breaking and lex-pruning
to the enumeration of all triangulations without prescribed symmetries
of point configurations (see the end of
Section~\ref{sec:applications-preliminaries} for explanations
concerning the point configurations).
Table~\ref{tab:comparison-triangulations} shows for three selected
examples the CPU times compared to \texttt{mptopcom 1.4} (see
\cite{JordanJoswigKastner_Parallelenumerationtriangulations_2018,JoswigKastner_NewCountsNumber_2018}).
Note that, by design, in general \texttt{mptopcom} and \texttt{TOPCOM}
do \emph{not} compute the same thing: \texttt{mptopcom} computes
\emph{all subregular} triangulations, whereas \texttt{TOPCOM} computes
\emph{all} triangulations, for which there was \emph{no} competitive
software \emph{at all} so far.  In all three listed examples, though,
the numbers are known to coincide.  The hypercube $\hypercube{4}$ was
selected as an example with symmetry group of moderate order, the
product of simplices $\simplexproduct{6}{2}$ was chosen as an example
with a rather large symmetry group, and $\cyclic{12}{4}$ represents
cyclic polytopes, for which, first, the symmetry group is tiny,
second, triangulations have many simplices (which is a stress-test for
extension-based enumeration), and, finally, for which \texttt{mptopcom
  1.4} uses a specialized (faster) implementation.

It can be seen that \texttt{TOPCOM} is particularly fast for the
computation of all triangulations of $\hypercube{4}$, where exploring
the enumeration tree is the dominant operation.  For
$\simplexproduct{6}{2}$ the symmetry handling is dominant, where
\texttt{TOPCOM} is still several times faster, but not by the same
margin.  For the regularity checks, \texttt{cddlib} was used as an LP
solver \cite{Fukuda_Cddlibreferencemanual_2003}.  Since these
regularity checks (besides being time-consuming) are advantageous for
\texttt{mptopcom}'s flip-based enumeration and symmetry handling, the
speed-up is less pronounced for the enumeration of regular
triangulations.

Some numbers could be computed for the first time, to the best of my
knowledge.  The largest new examples are the number of triangulations
of the pyritohedron (largest number up to symmetry) and the number of
triangulations of~$\simplexproduct{5}{3}$ (largest total number).  On
high-performance computing resources (HPC$^{192}$), all triangulations
of the Santos-dim6-n17 configuration $\santospointconfseventeen{0}$
with disconnected flip-graph could be computed in around 16 days (a
substantial amount of time invested for checkpointing and restarting).
This configuration has both the largest number of triangulations up to
symmetry and in total compared to what has been enumerated up to now.
It is the first example with a non-connected flip-graph whose
triangulations could be enumerated.

There are also some new counts for cyclic polytopes.  However,
meanwhile the numbers for three-dimensional cyclic polytopes can be
computed much faster via the enumeration of persistent graphs
\cite{FroeseRenken_PersistentGraphsCyclic_2021}.

Table~\ref{tab:results-regular-triangulations} one can see how much
time the regularity check takes compared to the mere enumeration of
triangulations.  For these results, again \texttt{cddlib} was used as
an LP solver inside the regularity checks
\cite{Fukuda_Cddlibreferencemanual_2003}.  The numbers for the
hypersimplices confirm the recently computed numbers from
\cite{CasabellaJoswigKastner_SubdivisionsHypersimplicesView_2024}.

\begin{remark}
  For all examples (with the possible exception of
  $\santospointconfseventeen{0}$) computed in this paper, it is not necessary
  to use a high-performance computing device; all numbers could be
  computed even faster by the eight performance cores of an M1Max
  laptop (see Section \ref{sec:applications-environment} for details
  on the computational environment).
\end{remark}

\begin{table}[htbp]
  \centering{%
    \sffamily\footnotesize
    \begin{tabular}{llrrrrr}
      \toprule
      $\Conf$
      & restriction
      & \llap{\# feas.}
      & \# triang's
      & \# triang's
      & \# nodes\\
      && \llap{symm's} & up to symm. & in total &\\
      \toprule
      $\hypercube{4}$         & none                      &    384 &           247,451 &          92,487,256 &          3,446,659 \\
      *$\hypercube{4}$        & $\mathbb{Z}_2$-inv.       &    384 &               181 &              22,280 &             12,884 \\
      \midrule 
      *pyrit.                 & none                      &     24 & 1,363,918,758,719 &  32,734,029,351,118 & 13,786,801,148,394 \\
      *pyrit.                 & $\mathbb{Z}_2$-inv.       &     24 &         1,313,581 &          15,761,630 &         80,245,911 \\
      *pyrit.                 & $\mathbb{Z}_3$-inv.       &      6 &              7968 &              15,889 &            546,939 \\
      \midrule 
      $3\standardsimplex{3}$  & none                      &     24 &       925,148,763 &      22,201,684,367 &      7,154,329,211 \\
      *$3\standardsimplex{3}$ & $\mathbb{Z}_4$-inv.       &      8 &                98 &                 181 &               4795 \\
      *$3\standardsimplex{3}$ & $\symGroup{4}$-inv.       &     24 &                 3 &                   3 &                100 \\
      \midrule
      $4\standardsimplex{3}$  & none                      &     24 &                 ? &                   ? &                  ? \\
      *$4\standardsimplex{3}$ & $\mathbb{Z}_4$-inv.       &      8 &           836,982 &           1,670,895 &         73,894,796 \\
      *$4\standardsimplex{3}$ & $\symGroup{4}$-inv.       &     24 &                12 &                  12 &               1182 \\
      \midrule
      $3\standardsimplex{4}$  & none                      &    120 &                 ? &                   ? &                  ? \\
      *$3\standardsimplex{4}$ & $\mathbb{Z}_5$-inv.       &     20 &            43,882 &             175,441 &          6,071,031 \\
      *$3\standardsimplex{4}$ & $\symGroup{5}$-inv.       &    120 &                 1 &                   1 &                 56 \\
      \midrule 
      $\simplexproduct{4}{4}$ & none                      & 28,800 & ? & ? & ? \\
      *$\simplexproduct{4}{4}$& $\mathbb{Z}_5$-inv.       &    200 &               317 &                9630 &             61,039 \\
      *$\simplexproduct{4}{4}$& $\mathbb{Z}_2$-inv.       &    240 &        30,327,170 &       3,638,732,520 &      2,016,686,741 \\
      \midrule
      *$\santospointconfseventeen{0}$  & none                            &    128 & 2,239,192,475,127 & 282,379,150,400,112 &156,459,603,003,247 \\
      *$\santospointconfseventeen{0}$  & $\santossymgroupseventeen$-inv. &     64 &                 1 &                   2 &                761 \\
      \midrule
      *$\santospointconftwentysix$     & none                            &     48 &                 ? &                   ? &                  ? \\
      *$\santospointconftwentysix$     & $\santossymgrouptwentysix$-inv. &     48 &               102 &                 204 &             88,817 \\
      \midrule
      *$\santospointconftwentysixmod$  & none                            &     48 &                 ? &                   ? &                  ? \\
      *$\santospointconftwentysixmod$  & $\santossymgrouptwentysix$-inv. &     48 &                90 &                 180 &             89,808 \\
      \bottomrule
    \end{tabular}
  }
  \caption[Computational results for triangulations with prescribed
  symmetries]{Computational results for the enumeration of
    triangulations using 16 threads with and without prescribed
    symmetries (numbers with a ``*'' are new)}
  \label{tab:results-symmetric-triangulations}
\end{table}

Table~\ref{tab:results-symmetric-triangulations} compares the numbers
of nodes and triangulations for various point configurations with and
without prescribed symmetries.  Consider, for example, the dilated
tetrahedron $3\standardsimplex{3}$ with $\symGroup{4}$-symmetry when
the enumeration is restricted to $\symGroup{4}$-invariant
triangulations.  Prior to the new method in this paper, the only way
to deal with prescribed symmetries was to inspect all triangulations
and filter by symmetry in post-processing.  In other words, in order
to count all $\symGroup{4}$-invariant triangulations of
$3\standardsimplex{3}$ up to $\symGroup{4}$-feasible symmetries, it
was necessary to generate all triangulations up to symmetry at some
point in the process.  According to
Table~\ref{tab:results-triangulations}, for the enumeration method in
this paper (which is the fastest known so far) this takes
$7{,}154{,}329{,}212$ enumeration nodes.  Now, with the new method
based on $\symGroup{4}$-feasible simplices that intersect
$\symGroup{4}$-properly, computing the three $\symGroup{4}$-invariant
triangulations up to $\symGroup{4}$-feasible symmetries (in this case,
all $24$ elements of $\symGroup{4}$ are $\symGroup{4}$-feasible) takes
only $100$ enumeration nodes.  The enumeration effort is, therefore,
completely dominated by preprocessing.  While for the case of
$3\standardsimplex{3}$ the post-processing approach would still be
viable (though inefficient), things are different, e.g., for
$4\standardsimplex{3}$: an enumeration of all triangulations seems
currently out of reach.  Still, the new method can compute in seconds
the triangulations with full symmetry.  As an example for a less
restricting prescribed symmetry group, the $\mathbb{Z}_4$-invariant
triangulations (induced by cyclic permutations of coordinates) have
been computed as well in essentially a minute.

The smallest instance for products of simplices with unknown number of
all triangulations is $\simplexproduct{4}{4}$.  A several-months-long
partial enumeration with several interruptions with checkpointing on
up to 192 threads on HPC$^{192}$ (this could not finish because
checkpoints could not be read anymore due to out-of-memory failure)
showed that there are more than $440{,}841{,}041{,}303$ symmetry
classes of triangulations ($12{,}695{,}974{,}068{,}314{,}880$ in total
using $8{,}942{,}032{,}045{,}505$ enumeration nodes).  In contrast to
this, it can now be computed for the first time in less than a minute
that there are, e.g., $317$ triangulations with
$\mathbb{Z}_5$-symmetry (i.e., cyclic symmetry of order~$5$) up to the
$\mathbb{Z}_5$-feasible symmetries of order~$200$.  For the vertices
$v_i, v_j$ of $\standardsimplex{4}$, these symmetries are generated by
the permutation $(v_i, v_j) \mapsto (v_{i+1}, v_{j+1})$ for
$i, j = 0, 1, \dots, 4$ considered modulo~$5$.  This took no more than
$61{,}039$ enumeration nodes.  The total number of such
$\mathbb{Z}_5$-invariant triangulations is $9630$.  The results for
$\mathbb{Z}_2$-symmetries induced by
$(v_i, v_j) \mapsto (v_{j}, v_{i})$ in $\simplexproduct{4}{4}$: There
are $30{,}327{,}170$ $\mathbb{Z}_2$-invariant triangulations up to the
$\mathbb{Z}_2$-feasible symmetries of order~$240$.  The total number
of these is $3{,}638{,}732{,}520$.  This took $2{,}016{,}686{,}741$
enumeration nodes (which is perfectly tractable).

The smallest example of a point configuration (in terms of
cardinality) with a disconnected flip-graph is
$\santospointconfseventeen{0}$ with $17$ points
\cite{Santos_Geometricbistellarflips_2006}.  Santos constructed a
triangulation of it that is invariant with respect to the
group~$\santossymgroupseventeen$ (see
Section~\ref{sec:applications-preliminaries}).  With the methods from
this paper, all its triangulations could be enumerated in
$156{,}459{,}603{,}003{,}247$ enumeration nodes on HPC$^{192}$.  Since
its flip-graph is disconnected, a flip-based enumeration would not be
possible without knowing a triangulation in each component a-priori.
Moreover, it can now be shown computationally that Santos's
triangulation is the only $\santossymgroupseventeen$-invariant
triangulation up to $\santossymgroupseventeen$-feasible symmetries.
This can be computed now in less than a second in only $761$
enumeration nodes.  Another point configuration in dimension five was
presented by Santos in~\cite{Santos_Nonconnectedtoric_2005}.  Santos
constructed a triangulation of it that is invariant with respect to
the group~$\santossymgrouptwentysix$ (see
Section~\ref{sec:applications-preliminaries}); it was claimed that it
could not be flipped to a regular triangulation.  All its
triangulations could so far not be computed.  However, for the first
time, it could be computed in only $88{,}817$ enumeration nodes that
there are exactly $102$ $\santossymgrouptwentysix$-invariant
triangulations of $\santospointconftwentysix$ up to
$\santossymgrouptwentysix$-feasible symmetries.  For the slightly
modified version $\santospointconftwentysixmod$ of
$\santospointconftwentysix$ (see again
Section~\ref{sec:applications-preliminaries}) the same count yields,
with similar effort, only $90$ such triangulations.

\subsection{Enhancements}
\label{sec:triangs:enhancements}

For the enumeration of triangulations, extension-based algorithms can
be equipped with some addititional functionality that is not easily
incorporated in flip-based algorithms.  A large part of this is based
on the following observation.
\begin{observation}
  For any restriction on the triangulations to be enumerated that is a
  restriction on all simplices in the triangulation,
  $\symLSRSFeaswithData$ can be run without modifications except that
  \begin{itemize}
  \item the set of considered simplex indices has to be restricted to
    the indices of the feasible simplices;
  \item the set of considered symmetries has to be restricted to the set-wise
    stabilizer subgroup of the set of feasible simplices.\qed
  \end{itemize}
\end{observation}
This can be applied to full triangulations, i.e., triangulations that
use all the points (exclude any simplex with points beyond its vertex
set in its convex hull), unimodular triangulations (exclude any
simplex with non-minimal volume), triangulations requiring a given
face in each simplex (exclude any simplex not containing the face),
triangulations avoiding some face alltogether (exclude any simplex
containing the face), etc.  In particular, this allows to compute all
triangulations of the boundary of a point configuration, allthough
such triangulations (topological spheres) do not even belong to the
class of triangulations of point configurations (topological balls).
A triangulation where all simplices contain a given point is called a
\emph{conical triangulation} with the given point as its apex.  A
conical triangulation where the apex is a relative interior point and
all other used points are in the boundary is called
\emph{central}. The link of a central triangulation at its apex is
then a boundary triangulation.  Coning any boundary triangulation to a
relative interior point yields a central triangulation.  Thus, there
is a bijection between central triangulations and boundary
triangulations.  In particular, central triangulations do not depend
on the relative interior point that was used to construct them.
\begin{observation}
  \label{thm:triangs:boundary-triang}
  Let $\Conf$ be a point configuration with symmetry group~$\GrG$.
  Then, all boundary triangulations of $\Conf$ can be computed up to
  symmetry as follows:
  \begin{enumerate}[label=(\arabic*)]
  \item Remove all points in the relative interior of~$\Conf$ to
    obtain $\Conf'$ and let $\GrG'$ be the restriction of $\GrG$ to
    the remaining elements.
  \item Let
    $\vectorStyle{b} := \frac{1}{\abs{\Conf'}} \sum_{\vectorStyle{p}
      \in \Conf'} \vectorStyle{p}$ be the barycenter of~$\Conf'$ (or
    any other point in its relative interior) and add it with new
    label~$0$ to $\Conf'$ to obtain $\Conf^{\mathrm{central}}$.
  \item Let $\specialSimpSet{central}$ be the set of all simplices in
    $\Conf^{\mathrm{central}}$ that contain point~$0$, and let
    $\GrG^{\mathrm{central}}$ be the set-wise stabilizer of
    $\specialSimpSet{central}$ in~$\GrG'$.
  \item Apply $\symLSRSFeaswithData$ to $\Conf^{\mathrm{central}}$
    with simplex set~$\specialSimpSet{central}$ and
    symmetries~$\GrG^{\mathrm{central}}$ in order to compute all
    central triangulations of~$\Conf^{\mathrm{central}}$ up to
    symmetry.
  \item For each central triangulation of~$\Conf^{\mathrm{central}}$
    compute the link of~$0$ by deleting the point~$0$ in all its
    simplices to obtain the corresponding boundary triangulation.\qed
  \end{enumerate}
\end{observation}

\begin{table}[htbp]
  \centering{%
    \sffamily\footnotesize
    \begin{tabular}{llrrrrr}
      \toprule
      $\Conf$
      & restriction
      & \llap{\# feas.}
      & \# triang's
      & \# triang's
      & \# nodes\\
      && \llap{symm's} & up to symm. & in total &\\
      \toprule
      $\hypercube{4}$         & none                      &    384 &           247,451 &          92,487,256 &          3,446,659 \\
      $\hypercube{4}$         & unimodular                &    384 &           159,037 &          59,546,240 &          2,375,773 \\
      *$\hypercube{4}$        & $\mathbb{Z}_2$-inv.       &    384 &               181 &              22,280 &             12,884 \\
      *$\hypercube{4}$        & $\mathbb{Z}_2$-inv./unim. &    384 &               154 &              19,520 &             11,414 \\
      \midrule 
      $3\standardsimplex{3}$  & none                      &     24 &       925,148,763 &      22,201,684,367 &      7,154,329,211 \\
      $3\standardsimplex{3}$  & full                      &     24 &        21,302,400 &         511,052,427 &        316,591,002 \\
      $3\standardsimplex{3}$  & unimodular                &     24 &        14,459,488 &         346,903,379 &        207,932,285 \\
      *$3\standardsimplex{3}$ & $\mathbb{Z}_4$-inv.       &      8 &                98 &                 181 &               4795 \\
      *$3\standardsimplex{3}$ & $\mathbb{Z}_4$-inv./full  &      8 &                36 &                  65 &               2365 \\
      *$3\standardsimplex{3}$ & $\mathbb{Z}_4$-inv./unim. &      8 &                18 &                  33 &               1762 \\
      *$3\standardsimplex{3}$ & $\symGroup{4}$-inv.       &     24 &                 3 &                   3 &                100 \\
      *$3\standardsimplex{3}$ & $\symGroup{4}$-inv./full  &     24 &                 1 &                   1 &                 74 \\
      *$3\standardsimplex{3}$ & $\symGroup{4}$-inv./unim. &     24 &                 0 &                   0 &                 29 \\
      \midrule
      $4\standardsimplex{3}$  & none                      &     24 &                 ? &                   ? &                  ? \\
      $4\standardsimplex{3}$  & full                      &     24 &                 ? &                   ? &                  ? \\
      $4\standardsimplex{3}$  & unimodular                &     24 &                 ? &                   ? &                  ? \\
      *$4\standardsimplex{3}$ & $\mathbb{Z}_4$-inv.       &      8 &           836,982 &           1,670,895 &         73,894,796 \\
      *$4\standardsimplex{3}$ & $\mathbb{Z}_4$-inv./full  &      8 &           108,103 &             215,479 &         11,807,847 \\
      *$4\standardsimplex{3}$ & $\mathbb{Z}_4$-inv./unim. &      8 &            79,147 &             157,724 &          8,862,008 \\
      *$4\standardsimplex{3}$ & $\symGroup{4}$-inv.       &     24 &                12 &                  12 &               1182 \\
      *$4\standardsimplex{3}$ & $\symGroup{4}$-inv./full  &     24 &                 5 &                   5 &                785 \\
      *$4\standardsimplex{3}$ & $\symGroup{4}$-inv./unim. &     24 &                 3 &                   3 &                543 \\
      \midrule
      $3\standardsimplex{4}$  & none                      &    120 &                 ? &                   ? &                  ? \\
      $3\standardsimplex{4}$  & full                      &    120 &                 ? &                   ? &                  ? \\
      $3\standardsimplex{4}$  & unimodular                &    120 &                 ? &                   ? &                  ? \\
      *$3\standardsimplex{4}$ & $\mathbb{Z}_5$-inv.       &     20 &            43,882 &             175,441 &          6,071,031 \\
      *$3\standardsimplex{4}$ & $\mathbb{Z}_5$-inv./full  &     20 &            15,841 &              63,306 &          2,436,943 \\
      *$3\standardsimplex{4}$ & $\mathbb{Z}_5$-inv./unim. &     20 &              7720 &              30,832 &          1,345,947 \\
      *$3\standardsimplex{4}$ & $\symGroup{5}$-inv.       &    120 &                 1 &                   1 &                 56 \\
      *$3\standardsimplex{4}$ & $\symGroup{5}$-inv./full  &    120 &                 0 &                   0 &                 45 \\
      *$3\standardsimplex{4}$ & $\symGroup{5}$-inv./unim. &    120 &                 0 &                   0 &                 45 \\
      \bottomrule
    \end{tabular}
  }
  \caption[Computational results for triangulations with
  restrictions]{Computational results for the enumeration of
    triangulations using 16 threads with and without further
    restrictions (numbers with a ``*'' are new)}
  \label{tab:results-restricted-triangulations}
\end{table}

Table~\ref{tab:results-restricted-triangulations} compares the numbers
of nodes and triangulations for various point configurations with and
without certain restrictions.  It can be seen that the restrictions
``full'' and ``unimodular'' further reduce the effort, though not as
much as prescribed symmetries.  With $79$ nodes the unique
$\symGroup{4}$-invariant triangulation using all the lattice points in
$3\standardsimplex{3}$ is found.  Finally, it takes only $29$
enumeration nodes to prove computationally that there is no unimodular
$\symGroup{4}$-invariant triangulation of~$3\standardsimplex{3}$.  For
$4\standardsimplex{3}$, the new method quickly finds the unimodular
$\symGroup{4}$-invariant triangulations.

\begin{table}[htbp]
  \centering{%
    \sffamily\footnotesize
    \begin{tabular}{llrrrrr}
      \toprule
      $\Conf$
      & restriction
      & \# feas.
      & \# triang's
      & \# triang's
      & \# nodes\\
      &
      & symm's
      & up to symm.
      & in total &\\
      \toprule
      $\rootpoly{2}$  & none                        &   12 &                 8 &                  32 &                100 \\
      $\rootpoly{2}$  & central                     &   12 &                 1 &                   1 &                 29  \\
      $\rootpoly{2}$  & central/$\mathbb{Z}_2$-inv. &   12 &                 1 &                   1 &                 26 \\
      \midrule
      $\rootpoly{3}$  & none                        &   48 &              1843 &              79,884 &             51,039 \\
      $\rootpoly{3}$  & central                     &   48 &                 7 &                  64 &                557 \\
      $\rootpoly{3}$  & central/$\mathbb{Z}_2$-inv. &   48 &                 2 &                   8 &                258 \\
      \midrule
      $\rootpoly{4}$  & none                        &  240 &    32,483,441,808 &   7,795,598,797,008 &  1,438,773,642,274 \\
      $\rootpoly{4}$  & central                     &  240 &            15,264 &           3,523,506 &          2,776,349 \\
      $\rootpoly{4}$  & central/$\mathbb{Z}_2$-inv. &  240 &                20 &                1782 &             12,950 \\
      \midrule
      $\rootpoly{5}$  & none                        & 1440 &                 ? &                   ? &                  ? \\
      $\rootpoly{5}$  & central                     & 1440 &                 ? &                   ? &                  ? \\
      *$\rootpoly{5}$ & central/$\mathbb{Z}_2$-inv. & 1440 &           112,234 &          79,216,008 &        183,816,883 \\
      \bottomrule
    \end{tabular}
  }
  \caption[Computational results for full centered root polytopes]{Computational results for the enumeration of
    all/central/$\mathbb{Z}_2$-invariant triangulations of full root
    polytopes using 16 threads (the number with a ``*'' is new, the
    numbers of $\mathbb{Z}_2$-invariant central triangulations up to
    $\rootpoly{4}$ have been computed in
    \cite{DelucchiKuehneMuehlherr_Combinatorialinvariantsfinite_2024}
    by a completely different method before and have been confirmed by
    this computation)}
  \label{tab:results-restricted-triangulations-rootpolytopes}
\end{table}

My original motivation to compute triangulations with restrictions
came from the $n$-dimensional full root polytopes~$\rootpoly{n}$ in
ambient $n+1$-space, where the centrally symmetric central
triangulations have a special meaning
\cite{DelucchiKuehneMuehlherr_Combinatorialinvariantsfinite_2024}.  Up
to now, prior to the result in this paper no numbers for
$\rootpoly{5}$ have been published, restricted or not.  \emph{Nota
  bene:} In
\cite{DelucchiKuehneMuehlherr_Combinatorialinvariantsfinite_2024} it
was already reported that $\rootpoly{5}$ has $25{,}224$ regular
central and centrally symmetric triangulations.  This was based on
preliminary results from this paper, though.

Table~\ref{tab:results-restricted-triangulations-rootpolytopes} shows
the all the results up to $n = 5$. Using the new method in this paper
the number $112{,}234$ of central and centrally symmetric
($\mathbb{Z}_2$-invariant) triangulations of~$\rootpoly{5}$ could be
computed up to $\mathbb{Z}_2$-feasible symmetries for the first time
(in about one hour).  Again, all symmetries of~$\rootpoly{n}$ are
$\mathbb{Z}_2$-feasible. Here, the post-processing approach would have
been way more time-consuming: A several-weeks-long partial enumeration
of the central triangulations of $\rootpoly{5}$ showed that there are
more than $1{,}799{,}917{,}616$ symmetry classes of central
triangulations ($2{,}591{,}706{,}000{,}744$ in total using
$1{,}130{,}442{,}682{,}392$ enumeration nodes).

Meanwhile, all $422{,}664{,}577{,}207$ reflection-symmetric unimodular
triangulations of the dilated standard two-simplex
$9\standardsimplex{2}$ with $55$ points in rank~$3$ have been
enumerated using the method from this paper -- a computation that was
way out of reach for any general-purpose enumeration method
before. This enumeration was motivated by a method for
``patchworking'' algebraic curves and was used to verify new
asymptotic bounds on symmetric unimodular triangulations
of~$k\standardsimplex{2}$ (see
\cite{FerryJoswigRambau_Countingsymmetricunimodular_2026} for
details).

Another enhancement stems from the combination of extension-based
enumeration and flip-graph exploration in one go.  For the following,
recall that a triangulation is \emph{subregular}
\cite{JordanJoswigKastner_Parallelenumerationtriangulations_2018}, if
it can be flipped to a regular triangulation by a some sequence of
GKZ-(lex-)increasing flips (\emph{upflips}).  Call it
\emph{non-subregular} otherwise.  Since deciding the existence of some
upflip-path to some regular triangulation is computationally
difficult, the subregularity-notion is refined as follows in order to
develop more tractable decision routines: A triangulation is
\emph{strongly subregular} if \emph{each} sequence of upflips reaches
the regular triangulations.  A triangulation is a
\emph{greedy-subregular triangulation (gsr-triangulation)} if the
sequence of unique upflips to the respective GKZ-lex-max adjacent
triangulations reaches the regular triangulations.  A triangulation is
a \emph{non-greedy-subregular triangulation (ngsr-triangulation)} if
this flip-sequence ends in a non-regular triangulation without
upflips.  Note that even in a regular triangulation there may exist
upflips to a non-subregular triangulation.

This paper restricts computational subregularity-tests to the
(n)gsr-property.  If a triangulation is a gsr-triangulation, then it
is subregular.  If a triangulation is an nsgr-triangulation then it is
not strongly subregular.  If a configuration has an ngsr-triangulation
then it has at least one non-subregular triangulation, because the
final triangulation in the greedy-upflip path starting at an
ngsr-triangulation is non-regular and has no upflip at all, whence it
is non-subregular.  Therefore, a configuration has non-subregular
triangulations if and only if it has ngsr-triangulations.  The same
conclusions can be derived from any upflip-path, e.g., the respective
first upflip found (which is faster, of course).  The extra cpu-time
for the greedy-GKZ-lex-path was invested in order to make the exact
counts in the computational results unique and reproducible.

Call a triangulation a \emph{regular-component triangulation
  (rc-triangulation)} if it can be flipped to a regular triangulation
by a sequence of arbitrary flips.  Call it a
\emph{non-regular-component triangulation (nrc-triangulation)}
otherwise.

For the following, to avoid duplicate work, some kind of clever
caching is mandatory, but not discussed here.  In order to count all
ngsr-triangulations, proceed as follows:
\begin{itemize}
\item Greedily upflip the (regular) placing triangulation to the
  (possibly locally) GKZ-maximal triangulation, the \emph{target
    triangulation}.
\item Enumerate all triangulations by extension.
\item If the target triangulation is regular\footnote{If all
    triangulations during the upflip-process are regular, then the
    target triangulation must be the \emph{globally} GKZ-maximal
    triangulation (which is regular), however, a GKZ-maximal neighbor
    may a-priori be non-regular, and regularity checks should be
    avoided during upflipping for performance reasons.}, then,
  greedily upflip each found non-regular \emph{seed triangulation}
  dropping all restrictions and ignoring symmetry until no upflip
  exists anymore.  If the resulting triangulation equals the target
  triangulation, then the non-regular seed triangulation is
  sub-regular.
\item For each seed triangulation that was not monotonically flipped
  to the target triangulation (because the target triangulation was
  non-regular or because the locally maximal triangulation reached was
  not the target triangulation), greedily upflip it one more time
  dropping all restrictions and ignoring symmetry until some regular
  triangulation is found.
\item Count all seed triangulations that can neither be upflipped
  greedily to the target triangulation nor to a regular triangulation
  this way.
\end{itemize}
If a triangulation is nrc, then it must be ngsr as well.  However, the
ngsr-check tests properties that are irrelevant for the nrc-property.
In order to check fast some tighter necessary conditions for an
nrc-triangulation, proceed as follows:
\begin{itemize}
\item Generate all non-regular triangulations by extension.
\item For each non-regular triangulation found, heuristically try to
  \begin{itemize}
  \item greedily upflip towards the respective target triangulation
    (regular or not);
  \item greedily downflip towards the respective target triangulation;
  \item greedlily upflip towards the respective target triangulation
    w.r.t.\ $\no$ random permutations of the $\no$~points;
  \item greedily upflip towards some regular triangulation;
  \item greedily downflip towards some regular triangulation;
  \item greedlily upflip towards some regular triangulation w.r.t.\
    $\no$ random permutations of the $\no$~points.
  \end{itemize}
  A non-regular triangulation is \emph{possibly non-regular component
    (pnrc)} if neither the target triangulation nor a regular
  triangulation has been found during these flipping heuristics.
\end{itemize}
Since regularity checks are expensive, the above heuristics tries to
find regularity-check-free certificates for ``rc-triangulation'' in
the first three checks.

In order to conclusively generate only nrc-triangulations, a full
breadth-first-search needs to be employed for each pnrc-triangulation.
Since one can expect to find many triangulations during this
exploration, regularity checks should be avoided alltogether this
time.  The idea is to proceed as follows:
\begin{itemize}
\item For each pnrc-triangulation, run two breadth-first-search
  explorations up to symmetry without any restrictions on
  triangulations and flips, in particular, without (the expensive)
  regularity checks.  One seeded at the triangulation in question, one
  seeded at a (regular) placing triangulation.
\item As soon as the seed of one exploration is found in the other up
  to symmetry, the flip-graph components coincide, and the
  triangulation in question is rc.
\item As soon as one exploration has finished and has not found the
  seed of the other up to symmetry, then the flip-graph components do
  not coincide, and the triangulation in question is nrc.
\end{itemize}
If there are reasons to believe that one component is much smaller
than the other,\footnote{For examples that are currently tractable,
  one can expect that the component of the pnrc-triangulation is
  smaller} explore only the allegedly smaller component.

\begin{remark}
  The selection and order of the steps in the above procedure may
  appear somewhat arbitrary.  The rationale behind the choices made in
  this paper is the following: for an rc-triangulation a flip-path to
  a regular triangulation shall be found fast; for an
  nrc-triangulation the conclusive flip-graph-component exploration
  shall finish fast.
\end{remark}

The new method has been applied to the following configurations:
$3\standardsimplex{3}$ (where it is known that all triangulations are
subregular from combining the results in
\cite{JordanJoswigKastner_Parallelenumerationtriangulations_2018} and
this paper), $4\standardsimplex{3}$, $3\standardsimplex{4}$, the
pyritohedron, $\simplexproduct{4}{4}$, $\santospointconfseventeen{0}$,
$\santospointconftwentysix$, and $\santospointconftwentysixmod$.  The
example $\simplexproduct{4}{4}$ is of special interest because of the
following: It is known that $\simplexproduct{k}{\ell}$ has a connected
flip-graph for all $k \le 3$
\cite{Liu_FlipConnectivityTriangulations_2020} and that
$\simplexproduct{4}{\ell}$ has a non-connected flip-graph for $\ell$
sufficiently large \cite{Liu_ZonotopeProductTwo_2018} (by a
non-constructive probabilistic proof for $\ell \approx 4 \cdot 10^4$).
The result does not exclude that $\simplexproduct{4}{4}$ has a
non-connected flip-graph.  In order to make computations tractable,
symmetries have been prescribed throughout.  Cyclic symmetries were
often chosen, since many famous non-regular triangulations exhibit
cyclic symmetries in one way or another.

\begin{table}[htbp]
  \centering{%
    \sffamily\footnotesize
    \begin{tabular}{llrrrrr}
      \toprule
      $\Conf$
      & restriction
      & \llap{\# feas.}
      & \# triang's
      & \# regularity
      & cpu time\\
      && \llap{symm's} & up to symm. & checks & [hh:mm:ss]\\
      \toprule
      *pyrit.                  & $\mathbb{Z}_2$-inv./non-regular      &  24 & 157,909 &  1,313,581 &  0:14:08 \\
      *pyrit.                  & $\mathbb{Z}_2$-inv./ngsr             &  24 &       0 &  1,471,490 &  0:24:15 \\
      *pyrit.                  & $\mathbb{Z}_2$-inv./pnrc             &  24 &       0 &  1,313,581 &  0:22:51 \\
      *pyrit.                  & $\mathbb{Z}_3$-inv./non-regular      &   6 &    2222 &       7968 &  0:00:05 \\
      *pyrit.                  & $\mathbb{Z}_3$-inv./ngsr             &   6 &       0 &     10,190 &  0:00:13 \\
      *pyrit.                  & $\mathbb{Z}_3$-inv./pnrc             &   6 &       0 &       7968 &  0:00:12 \\
      \midrule
      *$3\standardsimplex{3}$  & $\mathbb{Z}_4$-inv./non-regular      &   8 &      24 &         98 &  0:00:00 \\
      *$3\standardsimplex{3}$  & $\mathbb{Z}_4$-inv./ngsr             &   8 &       0 &        122 &  0:00:00 \\
      *$3\standardsimplex{3}$  & $\mathbb{Z}_4$-inv./pnrc             &   8 &       0 &         98 &  0:00:00 \\
      \midrule
      *$4\standardsimplex{3}$  & $\mathbb{Z}_4$-inv./non-regular      &   8 & 580,117 &    836,982 &  0:48:37 \\
      *$4\standardsimplex{3}$  & $\mathbb{Z}_4$-inv./ngsr             &   8 &     720 &  1,440,695 &  2:46:01 \\
      *$4\standardsimplex{3}$  & $\mathbb{Z}_4$-inv./pnrc             &   8 &       0 &    836,982 &  3:02:40 \\
      \midrule
      *$3\standardsimplex{4}$  & $\mathbb{Z}_5$-inv./non-regular      &  20 &  35,367 &     43,882 &  0:17:49 \\
      *$3\standardsimplex{4}$  & $\mathbb{Z}_5$-inv./ngsr             &  20 &       8 &     79,493 &  0:33:54 \\
      *$3\standardsimplex{4}$  & $\mathbb{Z}_5$-inv./pnrc             &  20 &       0 &     43,882 &  0:42:36 \\
      \midrule
      *$\simplexproduct{4}{4}$ & $\mathbb{Z}_5$-inv./non-regular      & 200 &     247 &        317 &  0:03:50 \\
      *$\simplexproduct{4}{4}$ & $\mathbb{Z}_5$-inv./ngsr             & 200 &       0 &        317 &  0:04:32 \\
      *$\simplexproduct{4}{4}$ & $\mathbb{Z}_5$-inv./pnrc             & 200 &       0 &        317 &  0:04:34 \\
      \midrule
      *$\santospointconfseventeen{0}$   & $\santossymgroupseventeen$-inv./non-regular    &  64 &       1 &          1 &  0:00:01 \\
      *$\santospointconfseventeen{0}$   & $\santossymgroupseventeen$-inv./ngsr           &  64 &       1 &         20 &  0:00:02 \\
      *$\santospointconfseventeen{0}$   & $\santossymgroupseventeen$-inv./pnrc           &  64 &       1 &        295 &  0:00:16 \\
      *$\santospointconfseventeen{0}$   & $\santossymgroupseventeen$-inv./nrc            &  64 &       1 &        332 & 10:19:35 \\
      \midrule
      *$\santospointconftwentysix$      & $\santossymgrouptwentysix$-inv./non-regular    &  48 &     102 &        102 &  0:05:34 \\
      *$\santospointconftwentysix$      & $\santossymgrouptwentysix$-inv./ngsr           &  48 &      41 &       3419 &  0:07:13 \\
      *$\santospointconftwentysix$      & $\santossymgrouptwentysix$-inv./pnrc           &  48 &       0 &        102 &  0:07:01 \\
      \midrule
      *$\santospointconftwentysixmod$   & $\santossymgrouptwentysix$-inv./non-regular    &  48 &      90 &         90 &  0:07:08 \\
      *$\santospointconftwentysixmod$   & $\santossymgrouptwentysix$-inv./ngsr           &  48 &      90 &       9400 &  0:10:20 \\
      *$\santospointconftwentysixmod$   & $\santossymgrouptwentysix$-inv./pnrc           &  48 &      90 &    291,059 &  1:47:45 \\
      \bottomrule
    \end{tabular}
  }
  \caption[Computational results on flip-graph connectivity of
  triangulations]{Computational results using 16 threads for
    triangulations with special symmetries that were checked for
    non-regularity, non-greedy-subregularity (``ngsr'' =
    non-greedy-subregular), and flip-graph connectivity to the
    component of the regular triangulations (``nrc'' =
    non-regular-component, ``pnrc'' = possibly non-regular-component,
    i.e., the heuristic could not flip to a regular triangulation;
    numbers with a ``*'' are new)}
  \label{tab:results-restricted-triangulations-nrc}
\end{table}

Table~\ref{tab:results-restricted-triangulations-nrc} show the results
with some details concerning the effort in terms of the number of
regularity checks (which constitute the dominating factor as soon as
they are applied).  Wherever the number of regularity checks for the
enumeration of pnrc triangulations did not exceed the number of
regularity checks for the enumeration of non-regular triangulations,
the heuristic could greedily upflip to the regular target
triangulation.  It can be seen that this was very often the case.  The
efficiency advantage of avoiding regularity checks as much as possible
is significant for larger instances.  For the first time, to the best
of my knowledge, non-subregular triangulations have been found by
identifying ngsr-triangulations, namely in $4\standardsimplex{3}$ and
in $3\standardsimplex{4}$.  On the other hand, all of the
ngsr-triangulations are rc (they can actually all be flipped to a
regular triangulation by only GKZ-increasing or only GKZ-decreasing
flips). Together this means: \texttt{mptopcom 1.4} on this input with
the given order of points (see the end of
Section~\ref{sec:applications-environment} on
page~\pageref{sec:applications-environment}) would inevitably miss out
on some rc-triangulations of $4\standardsimplex{3}$ and
$3\standardsimplex{4}$.  Note that without prescribed symmetries the
computational approach above would have been intractable for
$4\standardsimplex{3}$, $3\standardsimplex{4}$,
$\simplexproduct{4}{4}$, $\santospointconfseventeen{0}$,
$\santospointconftwentysix$, and $\santospointconftwentysixmod$ (given
today's computational power).  Recall that all
$\santossymgroupseventeen$-invariant nrc-triangulations of
$\santospointconfseventeen{0}$ could be enumerated up to
$\santossymgroupseventeen$-feasible symmetries.  Since there is only
one $\santossymgroupseventeen$-invariant triangulation up to symmetry,
only one triangulation had to be checked for flip-connectivity to a
regular triangulation.  It took less than 17 seconds to find out that
it is pnrc.  It took less than 11 hours (including the pnrc check) to
find out that it is nrc.\footnote{The regularity-check-free
  exploration with a regular target triangulation was several times
  faster than an exploration checking each found triangulation for
  regularity; the latter took several weeks to finish (with the same
  result).}  For $\santospointconftwentysix$, the new method could
confirm that all its $102$ $\santossymgrouptwentysix$-invariant
triangulations are non-regular and $41$ of them are ngsr.  However,
most surprisingly, in about seven minutes all of them could be flipped
to an rc-target triangulation by GKZ-increasing flips for some
permutation (found by choosing $26$ times a permutation uniformly at
random).

Meanwhile, the conflict between this result and the result by Santos
in~\cite{Santos_Nonconnectedtoric_2005} has been sorted out in
personal communication: In one spot, a lemma in
\cite{Santos_Nonconnectedtoric_2005} cannot be applied to
$\santospointconftwentysix$.  If certain $\sfrac{1}{2}$-coordinates
are changed to $\sfrac{1}{3}$ like in~$\santospointconftwentysixmod$,
then all the GKZ-increasing flip-paths to regular triangulations
disappear, and so does the flaw
in~\cite{Santos_Nonconnectedtoric_2005}.  A conclusive
breadth-first-search enumeration of the flip graph component did not
finish (even on HPC$^{24}$ with larger RAM) because the HPC-memory
limit of 2\,TB was hit.

The following prototypical theorem summarizes the findings that
demonstrate the potential of the method.  The result
on~$\santospointconfseventeen{0}$ confirms the result
in~\cite{Santos_Geometricbistellarflips_2006}, also described
in~\cite[Sec.~7.4.2]{DeLoeraRambauSantos_TriangulationsStructuresApplications_2010}.
The result on~$\santospointconftwentysix$ contradicts the result in
\cite{Santos_Nonconnectedtoric_2005}. And the result
on~$\santospointconftwentysixmod$ supports the fix, though not
conclusively.
\begin{theorem}
  There are $317$ $\mathbb{Z}_5$-invariant triangulations of
  $\simplexproduct{4}{4}$ up to $\mathbb{Z}_5$-feasible symmetries
  ($9630$ in total); $247$ of them are non-regular ($8260$ in
  total). All these triangulations are connected to the flip-graph
  component of the regular triangulations by greedy-GKZ-increasing
  flips.  In particular, they are all gsr, hence subregular.

  There are $7968$ $\mathbb{Z}_3$-invariant triangulations of the
  pyritohedron up to $\mathbb{Z}_3$-feasible symmetries ($15{,}889$ in
  total); $2222$ of them are non-regular ($4430$ in total). All these
  triangulations are connected to the flip-graph component of the
  regular triangulations by greedy-GKZ-increasing flips.  In
  particular, they are all gsr, hence subregular.
  
  There are $1{,}313{,}581$ $\mathbb{Z}_2$-invariant triangulations of
  the pyritohedron up to $\mathbb{Z}_2$-feasible symmetries
  ($15{,}761{,}630$ in total); $157{,}909$ of them are non-regular
  ($1{,}894{,}796$ in total). All these triangulations are connected
  to the flip-graph component of the regular triangulations by
  greedy-GKZ-increasing flips.  In particular, they are all gsr, hence
  subregular.
  
  There are $836{,}982$ $\mathbb{Z}_4$-invariant triangulations of
  $4\standardsimplex{3}$ up to $\mathbb{Z}_4$-feasible symmetry
  ($1{,}670{,}895$ in total); $580{,}117$ of them are non-regular
  ($1{,}159{,}626$ in total), and $720$ of them are
  non-greedy-subregular ($1440$ in total).  (In particular, there are
  non-subregular triangulations -- not necessarily
  $\mathbb{Z}_4$-invariant).  All these triangulations are connected
  to the flip-graph component of the regular triangulations (by
  greedy-GKZ-increasing flips w.r.t.\ randomly generated
  permutations).
  
  There are $43{,}882$ $\mathbb{Z}_5$-invariant triangulations of
  $3\standardsimplex{4}$ up to $\mathbb{Z}_5$-feasible symmetries
  ($175{,}441$ in total); $35{,}367$ of them are non-regular
  ($141{,}442$ in total), and $8$ of them are non-greedy-subregular
  ($32$ in total).  (In particular, there are non-subregular
  triangulations -- not necessarily $\mathbb{Z}_5$-invariant.)  All
  these triangulations are connected to the flip-graph component of
  the regular triangulations (by greedy-GKZ-increasing flips w.r.t.\
  randomly generated permutations).

  There is exactly one $\santossymgroupseventeen$-invariant
  triangulation of $\santospointconfseventeen{0}$ up to
  $\santossymgroupseventeen$-feasible symmetries (two in total, one of
  which is Santos's triangulation); it is non-regular and
  non-greedy-subregular.  (In particular, there are non-subregular
  triangulations -- not necessarily
  $\santossymgroupseventeen$-invariant.) The flip-graph component of
  it has $7{,}329{,}890$ triangulations up to symmetry
  ($883{,}875{,}076$ in total), that can all be reached from Santos's
  triangulation in at most $29$ flips.  All components together have
  $2{,}239{,}192{,}475{,}127$ triangulations up to symmetry
  ($282{,}379{,}150{,}400{,}112$ in total).  In particular, the flip
  graph is not connected.

  There are $102$ $\santossymgrouptwentysix$-invariant triangulations
  of $\santospointconftwentysix$ up to
  $\santossymgrouptwentysix$-feasible symmetries ($204$ in total, one
  of which is Santos's triangulation); all of them are non-regular,
  and $41$ of them are non-greedy-subregular ($82$ in total).  (In
  particular, there are non-subregular triangulations -- not
  necessarily $\santossymgrouptwentysix$-invariant.)  All these
  triangulations are connected to the flip-graph component of the
  regular triangulations (by greedy-GKZ-increasing flips w.r.t.\ some
  randomly generated permutations).

  There are $90$ $\santossymgrouptwentysix$-invariant triangulations
  of $\santospointconftwentysixmod$ up to
  $\santossymgrouptwentysix$-feasible symmetries ($180$ in total); all
  of them are non-regular and non-greedy-subregular.  (In particular,
  there are non-subregular triangulations -- not necessarily
  $\santossymgrouptwentysix$-invariant.)  None of them could be
  flipped to a regular triangulation by GKZ-increasing flips w.r.t.\
  randomly generated permutations (using $26$ random permutations for
  each triangulation).  In particular, they are all possibly
  non-regular-component.\qed
\end{theorem}

To the best of my knowledge, a computation of the numbers with
prescribed symmetries has not even been considered with the methods
known prior to this paper, most likely for the following reason: For
all known flip-based algorithms prescribed symmetries have no chance
to reduce the effort of enumeration, since the flip-graph of these
triangulations is, in general, not connected, not even for the regular
triangulations: just consider $2\standardsimplex{2}$ and its two
symmetry classes of regular triangulations with 120-degree rotational
symmetry with one and four simplices, respectively.  Since their
cardinalities differ, they can not be equivalent w.r.t.{} to the
symmetries of $2\standardsimplex{2}$.  Moreover, they are not
connected by a flip.  This marks a systematic advantage of
extension-based enumeration algorithms like \symLSRSFeaswithData.

\begin{table}[htbp]
  \centering{%
    \sffamily\footnotesize
    
    \begin{tabular}{l@{\qquad}rrr@{\qquad\qquad}rrr}
      \toprule
      $\Conf$
      & \multicolumn{6}{c}{regularity checks~\ldots}\\
      & \multicolumn{3}{c}{\ldots~in triangulations only}
      & \multicolumn{3}{c}{\ldots~in all nodes}\\
      & \# nodes & \# reg.~checks & CPU time & \# nodes & \# reg.~checks & CPU time\\
      &&& [hh:mm:ss] &&& [hh:mm:ss]\\
      \midrule
      $\hypercube{4}$         &  3,446,659 &   247,451 & 0:00:54 & 3,382,448 & 1,791,034 & 0:04:59\\
      $(\hypercube{4})^*$     & 18,316,313 &    75,756 & 0:01:38 &   192,448 &    79,266 & 0:01:11\\
      \bottomrule
    \end{tabular}
  }
  \caption[Comparison of regularity-check strategies]{Regularity
    checks in each node versus regularity checks only for
    triangulations on the $4$-cube (235,277 regular out of 247,451
    triangulations) and its Gale dual (490 regular out of 75,756
    triangulations)}
  \label{tab:comparison-earlyregcheck}
\end{table}

Regularity checks can reduce the enumeration effort of extension-based
algorithms if regularity is checked for each partial triangulation.
Because regularity checks are quite expensive compared to the
enumeration operations, the benefits are outweighed by the effort for
the examples seen so far.  However, there are examples where the early
regularity checks pay off.  One such example is the enumeration of
regular triangulations of the (totally cyclic) vector configuration
corresponding to a Gale dual of the $4$-cube, as can be seen in
Table~\ref{tab:comparison-earlyregcheck}.  There, the number of
necessary regularity checks is only ever-so-slightly increased by
checking each node.  Since for this example many triangulations are
non-regular, the resulting reduced effort in the enumeration tree
leads to shorter CPU times.  Note that for these early regularity
checks there is a warm-start opportunity if the LPs are stored withing
the local auxiliary data of a node.  This would speed-up the
regularity checks in each node. Since \texttt{TOPCOM}'s default LP
solver from \texttt{cddlib} does not support warm-starts yet,
\texttt{TOPCOM} does not yet support warm-starts either.

In contrast to most of the other mentioned restrictions, requiring
regularity is known to reduce the effort for flip-based algorithms as
well, since the exploration of the flip-graph can be restricted to the
flip-graph of regular triangulations, which is connected, in contrast
to the flip-graph of all triangulations
\cite[Chapter~5]{DeLoeraRambauSantos_TriangulationsStructuresApplications_2010}.
The flip-graph of all regular conical (and, thus, central)
triangulations of a $\dimension$-dimensional point
configuration~$\Conf$ with $\no$ points is connected, too: if the apex
point is the first point, then the central triangulations form the
face of the secondary polytope induced by the supporting hyperplane
$\{z \in \mathbb{R}^{\no} : z_1 = \vol_{\dimension}(\conv \Conf)
\}$. Moreover, any face of the secondary polytope is a polytope
itself.  Thus, it has a connected edge graph, and each edge in this
graph corresponds to a flip
\cite[Chapter~5]{DeLoeraRambauSantos_TriangulationsStructuresApplications_2010}.

\begin{algorithm}[htbp]
  \TitleOfAlgo{\semiIsNotRightExtLexFlag{$\setStyle{T}', \downsetStyle{D},
      \setsysStyle{F}, \globalData, \localData, \localData'$}}
  
  \KwIn{A subset $\setStyle{T}' = \setStyle{T} \cup \{ s \}$ of
    simplex indices in~$\indexSet{\noOfSimplices}$ with
    $s \in \admissibles(\setStyle{T})$,
    $\setStyle{S} = \lexIndexSimp(s)$, the downset $\downsetStyle{D}$
    of subsets indexing partial triangulations, the set
    $\setsysStyle{F}$ of flag triangulations (given implicitly by
    $\isFeasible$), global data $\globalData$ encompassing the
    interior-facets table $\iftTable$ and the admissibles table
    $\admTable$ of~$\Conf$, local data $\localData$ of~$\setStyle{T}$
    encompassing the uncovered critical cliques $\uncovCliques$
    of~$\setStyle{T}$ and the admissibles $\admSet$ of~$\setStyle{T}$,
    local data $\localData'$ with the admissibles $\admSet'$
    of~$\setStyle{T}'$}
    
  \KwOut{$(\true, - )$ if $\setStyle{T}'$ is not right-extendable to a flag
    triangulation, $(\false, \uncovCliques')$ if $\setStyle{T}'$ may be
    right-extendable to a flag triangulation, where $\uncovCliques'$
    contains the uncovered critical cliques in~$\setStyle{T}'$}
  
  $\uncovCliques' \gets \emptyset$
  \tcc*{initialize}
  \For(\tcc*[f]{for all old uncovered critical cliques}){
    $\setStyle{C} \in \uncovCliques$}{
    \If(\tcc*[f]{critical clique is a face of~$\setStyle{S}$?}){
      $\setStyle{C} \subseteq \setStyle{S}$}{
      continue
      \tcc*{next clique}
    }
    \If(\tcc*[f]{new adm.~lex-too-large?}){
      $\setStyle{C} \lexsmaller \lexksubset{\lexIndexSimp\bigl(\min(\admSet')\bigr)}{\abs{\setStyle{C}}}$}{
      \Return (\true, -)
      \tcc*{no flag extension possible}
    }
    \Else{
      $\uncovCliques' \gets \uncovCliques' \cup \{\setStyle{C}\}$
      \tcc*{add $\setStyle{C}$ to new uncovered cliques}
    }
  }
  \For(\tcc*[f]{for all pairs of indices in~$\setStyle{T}'$}){
    $\{s_1, s_2\} \in \tbinom{\setStyle{T}'}{2}$}{
    $\setStyle{S}_1 = \lexIndexSimp(s_1$),
    $\setStyle{S}_2 = \lexIndexSimp(s_2)$
    \tcc*{generate simplices}
    $\setStyle{C}' = \bigl(\setStyle{S} \cap \setStyle{S}_1\bigr)
    \cup  \bigl(\setStyle{S} \cap \setStyle{S}_2\bigr)
    \cup  \bigl(\setStyle{S}_1 \cap \setStyle{S}_2\bigr)$
    \tcc*{build critical clique for $\setStyle{S}, \setStyle{S}_1, \setStyle{S}_2$}
    \If(\tcc*[f]{clique too large?}){
      $\abs{\setStyle{C}'} > \rank$
    }{
      \Return (\true, -)
      \tcc*{no flag extension possible}
    }
    \If(\tcc*[f]{critical clique is a face?}){
      $\setStyle{C}' \subseteq \setStyle{S}'$ with
      $\lexSimpIndex(\setStyle{S}') \in \setStyle{T}'$}{  
      continue
      \tcc*{next clique}
    }
    \If(\tcc*[f]{new adm.~lex-too-large?}){
      $\setStyle{C}' \lexsmaller \lexksubset{\lexIndexSimp\bigl(\min(\admSet')\bigr)}{\abs{\setStyle{C}'}}$}{
      \Return (\true, -)
      \tcc*{no flag extension possible}
    }

    \Else{
      $\uncovCliques' \gets \uncovCliques' \cup \{\setStyle{C}'\}$
      \tcc*{add $\setStyle{C}'$ to new uncovered cliques}
    }
  }
  \Return (\false, $\uncovCliques'$)
  \tcc*{flag extension might be possible}
  
  \caption[Lex-Flag-Pruning]{The lex-flag-pruning semi-check from
    \cite{LoeraFerroniMoralesRambau_Therearematroid_2026} whether a
    partial triangulation in rank~$\rank$ can certainly not be
    right-extended to a flag triangulation}
  \label{alg:semiIsNotRightExt_triangulations_flaglex}  
\end{algorithm}

A more complicated enhancement concerns \emph{flagness} of a
triangulation.  A
triangulation is \emph{flag} if it is the clique complex of its
edge-graph.  That is: Whenever all edges of a face are present, the
face itself is present in the triangulation.  This condition cannot be
guaranteed by restrictions on the set of admissible simplices alone.
However, the check can be integrated into the right-extension check
\semiIsNotRightExt.  To this end, the following characterization of
flagness from \cite{BetreZhangEdmond_PureSimplicialClique_2024} can be
used:
\begin{lemma}[Critical-Clique Condition~{\cite[Thm.~3.1]{BetreZhangEdmond_PureSimplicialClique_2024}}]
  A triangulation~$\setsysStyle{T}$ of a point configuration~$\Conf$
  is flag if and only if for all triples of simplices
  $\setStyle{S}_1, \setStyle{S}_2, \setStyle{S}_3 \in \setsysStyle{T}$
  there is a simplex $\setStyle{S} \in \setsysStyle{T}$ with
  \begin{equation}
    \setStyle{C}
    :=
    \bigl(\setStyle{S}_1 \cap \setStyle{S}_2\bigr)
    \cup
    \bigl(\setStyle{S}_2 \cap \setStyle{S}_3\bigr)
    \cup
    \bigl(\setStyle{S}_3 \cap \setStyle{S}_1\bigr)
    \subseteq
    \setStyle{S}.
  \end{equation}
  The set~$\setStyle{C}$ is called a \emph{critical clique}.\qed
\end{lemma}

\begin{definition}[Uncovered Critical Cliques]  
  For a partial triangulation~$\setsysStyle{T}$ of a point
  configuration~$\Conf$ a critical clique is \emph{uncovered} if it is
  not a subset of any simplex~$\setStyle{S} \in \setsysStyle{T}$; it
  is \emph{covered} otherwise.  The set
  $\setsysStyle{U}(\setsysStyle{T})$ is the set of uncovered critical
  cliques in~$\setsysStyle{T}$.
\end{definition}

Thus, a triangulation is flag if and only if all its critical cliques
are covered.  The motivation to extend this to partial triangulations
is the following necessary condition for right-extendability to a flag
triangulation.  This was first noticed in a slightly different
phrasing in~\cite{LoeraFerroniMoralesRambau_Therearematroid_2026}.
Recall that $\lexksubset{S}{k}$ the unique $k$-element lex-subset
of~$\setStyle{S}$.
\begin{lemma}[{\cite[App.~A]{LoeraFerroniMoralesRambau_Therearematroid_2026}}]
  \label{thm:lex-flag-pruning}
  Let $\setStyle{T}$ index a partial triangulation $\setsysStyle{T}$
  of a point configuration~$\Conf$.  If there is a right-completion
  of~$\setStyle{T}$ indexing a flag triangulation of~$\Conf$, then the
  following holds for all uncovered critical cliques $\setStyle{C}$
  in~$\setsysStyle{T}$:
  \begin{equation}
    \setStyle{C}
    \lexgreaterequal
    \lexksubset{\lexIndexSimp\Bigl(\min\bigl(\admissibles(\setStyle{T})\bigr)\Bigr)}{\abs{\setStyle{C}}}
  \end{equation}
\end{lemma}
\begin{proof}
  Any right-completion of~$\setStyle{T}$ indexing a flag triangulation
  has all its critical cliques covered.  Thus, each uncovered critical
  clique~$\setStyle{C}$ in $\setsysStyle{U}(\setsysStyle{T})$ is a
  $\abs{\setStyle{C}}$-element subset of an admissible
  simplex~$\setStyle{S}$ with
  $\lexSimpIndex(S) \in \admissibles(\setStyle{T})$.  As such, it must
  be be lex-greater than or equal to the lex-minimal
  $\abs{\setStyle{C}}$-element subset of the lex-minimal admissible
  simplex.  The assertion is exactly this in formula.
\end{proof}
Note that, since in this particular setup the critical cliques have
not been mapped to the integers like the simplices and their facets
(because of their large variety of cardinalities in general), one
needs to evaluate the lexicographic order of subsets explicitly to
make use of this criterion.

Call the application of Lemma~\ref{thm:lex-flag-pruning}
\emph{lex-flag-pruning}.
Algorithm~\ref{alg:semiIsNotRightExt_triangulations_flaglex} shows one
possible implementation that can be used in conjunction with
lex-pruning.  The algorithm then correctly computes for a partial
triangulation
$\setsysStyle{T}' = \setsysStyle{T} \cup \{ \setStyle{S} \}$ its set of
uncovered critical cliques $\uncovCliques'$ or detects that
$\setsysStyle{T}'$ cannot be right-completed to a flag triangulation.

With this method, all flag triangulations of, e.g., the $4$-cube
and~$\simplexproduct{6}{2}$ could be computed using substantially
fewer nodes than without the flag-restriction.  However, for the
$4$-cube the CPU time was substantially larger, whereas for
$\simplexproduct{6}{2}$ the CPU time was substantially smaller.  In
\cite{LoeraFerroniMoralesRambau_Therearematroid_2026} an attempt was
made to enumerate all unimodular flag triangulations of the Fano
matroid base polytope in order to attack the Herzog-Hibi question
in~\cite[p.~242 Question~a)]{HerzogHibi_DiscretePolymatroids_2002} for
this instance.  The computation on HPC$^{192}$ has meanwhile found
$73{,}965$ such triangulations of the Fano matroid base polytope, all
non-regular, but did not finish so far.  The conclusive result in
\cite{LoeraFerroniMoralesRambau_Therearematroid_2026} that there is no
regular unimodular flag triangulation of the Fano matroid base
polytope was obtained by other means, namely a SAT representation of
the question.

Finally, one can use the methods in this paper to design optimization
algorithms as follows.  Any variant of algorithm \symLSRS\ -- as any
enumeration algorithm -- can be turned into a branch-and-bound
optimization algorithm by specifying an objective function and a
dual-bound procedure.  This was used to compute the minimal number of
simplices in a triangulation for some examples.  Details are omitted,
since the method is straight-forward. Example results are: a minimal
triangulation of the product of a square and a triangle has 10
simplices (a result originally proved in
\cite{SeacrestSu_LowerBoundTechnique_2018} with quite some effort,
since in that paper more general dissections are allowed), which took
less than a tenth of a second; a minimal triangulation of the $4$-cube
has 16 simplices, which took less than half a second; a minimal
triangulation of the regular dodecahedron has 23 simplices, which took
less than 10 minutes; a minimal triangulation of the pyritohedron has
23 simplices, too, which took less than 4 hours.  And a minimal
unrestricted triangulation of the full, centered root polytope in
ambient $5$-space has 46 simplices, which took less than 11 hours,
whereas the enumeration of all its triangulations took around 95
hours.  A minimal centrally symmetric central triangulation of it has
70 simplices, which took half a second. All computation times for
optimization are significantly shorter than those for the complete
enumeration.  In branch-and-bound, the exploitation of symmetries is
especially vital, since none of the equivalent branches leading to an
optimal leaf can ever be pruned.

\subsection{Notes on the Use of GKZ Vectors}
\label{sec:triangs:notes}

The switch-table method to compute canonical representatives from
\cite{JordanJoswigKastner_Parallelenumerationtriangulations_2018} is
particularly efficient if it can be based on a representation of
triangulations by the so-called \emph{GKZ-vector}
(Gelfand-Kapranov-Zelevinsky vector, see
\cite[Chapter~5]{DeLoeraRambauSantos_TriangulationsStructuresApplications_2010}).
The GKZ-vector is a vector with a component for each point in the
configuration.  Such vectors can be compared lexicographically like
subsets of simplex indices.  Subsets of simplices (as in this paper)
can be interpreted as characteristic vectors with as many components
as there are simplices.  While each triangulation can be represented
by its characteristic vector, each regular triangulation can be
uniquely represented by its GKZ-vector.  Since, in general, the number
of points is much smaller than the number of simplices spanned by the
points, the representation by GKZ-vectors is significantly more
compact than the representation as characteristic vectors.  Moreover,
when using switch tables the lexicographic comparison of two
GKZ-vectors can be much faster than the lexicographic comparison of
two characteristic vectors.  Thus, one may ask whether the
lexmin-check in $\symLSRSFeaswithData$ could be accelerated by
utilizing GKZ-vectors, at least for the enumeration of regular
triangulations.  The answer is no, at least not in any
straight-forward way.

The reason is the basic justification of $\symLSRS$: a suitable
equivalent of Lemma~\ref{thm:lexmin-lemma} does not hold in any
straight-forward variant for canonicals based on the order of
GKZ-vectors.  Here is why.  The inspection of the subset poset of
partial triangulations of a square with the points ordered as
\begin{equation}
  \Conf = 
  \begin{pmatrix}
    0 & 1 & 0 & 1\\
    0 & 0 & 1 & 1\\
    1 & 1 & 1 & 1
  \end{pmatrix}
\end{equation}
shows that picking the lexmax GKZ-vector as the canonical
representative would result in $\{ 124, 134 \}$ with GKZ-vector
$(2,1,1,2)$ being canonical.  Note that the symmetries of the square
act transitively on points, simplices, and triangulations.  However,
neither the subset $\{124\}$ with GKZ-vector $(1,1,0,1)$ nor the
subset $\{134\}$ with GKZ-vector $(1,0,1,1)$ would be canonical, since
only $\{1,2,3\}$ with GKZ-vector $(1,1,1,0)$ is canonical.  Thus,
there is no way to sort the simplices such that an equivalent of
Lemma~\ref{thm:lexmin-lemma} holds with a canonical defined by the
lexmax GKZ-vector.

Picking instead the lexmin GKZ-vector as the canonical representative
and sorting the simplices by increasing GKZ-vector is, therefore, the
only way to make an equivalent of Lemma~\ref{thm:lexmin-lemma} true
for this point configuration.  Now consider the following six-point
configuration whose symmetries are horizontal and vertical reflection
as well as rotation by 180 degrees.
\begin{equation}
  \Conf' = 
  \begin{pmatrix}
    0 & 1 & 1 & 1 & 0 & 0\\
    0 & 0 & 1 & 2 & 2 & 1\\
    1 & 1 & 1 & 1 & 1 & 1
  \end{pmatrix}    
\end{equation}
For the partial triangulation
$\setsysStyle{T} = \{ 356, 345, 236 \}$ the GKZ-vector is
$(0,1,3,1,2,2)$, which is lexmin (= canonical) in its orbit.
Removing the simplex $236$ with the lexmax GKZ-vector leads to the
partial triangulation $\{ 356, 345 \}$ with GKZ-vector
$(0,0,2,1,2,1)$, which is \emph{not} lexmin in its orbit, since by
horizontal reflection the partial triangulation $\{ 456, 346 \}$ is
equivalent and has the lexsmaller GKZ-vector $(0,0,1,2,1,2)$.

Thus, no straight-forward equivalent of Lemma~\ref{thm:lexmin-lemma}
holds true, and, hence, there is no correct $\symLSRS$ available in
general based on GKZ-vectors.

\section{Conclusions}
\label{sec:conclusions}

The new general enumeration algorithm symmetric lexicographic
symmetric-subset reverse search has been introduced to enumerate
maximal, co-minimal, and feasible symmetric subsets in a downset of
subsets of a finite set up to symmetry.  New versions of
lex-minimality checks and new pruning methods for partial cocircuits
and partial triangulations have been presented and analyzed.  The new
methods allowed for the computation of many new cardinalities, among
them the number, up to symmetry, of cocircuits of the $9$-cube, the
number of circuits of the~$8$-cube, and the number of all
triangulations of the pyritohedron, the dodecahedron,
$\simplexproduct{5}{3}$, and~$\santospointconfseventeen{0}$.
Moreover, it could be computationally confirmed for the first time
that $\santospointconfseventeen{0}$ has a disconnected flip-graph with
the Santos triangulation residing in a purely non-regular flip-graph
component.  In contrast to this, the Santos triangulation
of~$\santospointconftwentysix$ is contained in the flip-graph
component of the regular triangulations.  A small modification
$\santospointconftwentysixmod$ of~$\santospointconftwentysix$ fixes
Santos's construction, and so far no flip-path from the Santos
triangulation of~$\santospointconftwentysixmod$ to a regular
triangulation could be found.

The algorithm \symLSRS can be enhanced.  First, further types of
restrictions imposed on the objects to be enumerated (like
unimodularity for triangulations) can reduce the enumeration effort by
directly invalidating branches in the enumeration tree.  This was
shown for triangulations.  However, the largest efficiency boost stems
from prescribed symmetries.  One major application concerns full root
polytopes.  For ambient dimension six, all central and centrally
symmetric triangulations could be computed for the first time.

It was observed that an objective function and a dual-bound procedure
together with \symLSRS generate a straight-forward optimization
algorithm.  This was applied to find triangulations with a minimal
number of simplices for the $4$-cube and the regular dodecahedron in
much less time than the enumeration takes.  The efficiency of such an
optimization algorithm will depend on the quality of the dual-bound
procedure.  It would be interesting to see where such an approach
would be competitive to other special optimization algorithms (like
the universal-polytope approach
\cite{LoeraHostenSantosSturmfels_polytopealltriangulations_1996} for
the minimal triangulation).

The new methods should have many more applications beyond the three
applications in this paper like the enumeration of maximal cliques in
a graph up to symmetry.

\bibliographystyle{plainurl}
\bibliography{SymLexSubsetRS}

\section*{Statements and Declarations}

\subsection*{Funding}

The author declares that no funds, grants, or other support were
received for the preparation of this manuscript.

\subsection*{Competing Interests}

The author has no conflicts of interest to declare that are relevant
to the content of this article.

\subsection*{Availability of data and materials}

The software package \texttt{TOPCOM} is available online via the professional
web page of the author under
\url{https://www.wm.uni-bayreuth.de/de/team/rambau_joerg/TOPCOM/}.  It
is further installable as a package in several major Linux
distributions.


\end{document}
